\documentclass[12pt]{amsart}
\usepackage{amsmath,eucal,verbatim,amsopn,amsthm,amsfonts,amssymb,mathdots,mathrsfs,esint,mathtools,mnsymbol,enumitem,bbm}
\usepackage{amsmath,amssymb,latexsym}
\usepackage{amscd}
\usepackage[margin=0.9in]{geometry}
\usepackage[hyperfootnotes=false,colorlinks=true,allcolors=blue]{hyperref}

\theoremstyle{plain}
\newtheorem*{theorem*}{Theorem}
\newtheorem{theorem}{Theorem}
\newtheorem{proposition}[theorem]{Proposition}
\newtheorem{corollary}[theorem]{Corollary}
\newtheorem{lemma}[theorem]{Lemma}
\newtheorem{claim}[theorem]{Claim}

\newtheorem{conjecture}[theorem]{Conjecture}
\newtheorem{remark}[theorem]{Remark}
\newtheorem{example}[theorem]{Example}
\newtheorem{apptheorem}{Theorem}[section]
\newtheorem{appproposition}[apptheorem]{Proposition}
\newtheorem{applemma}[apptheorem]{Lemma}

\newtheorem{appremark}[apptheorem]{Remark}

\numberwithin{equation}{section}

\def\N{\mathbb{N}}
\def\Z{\mathbb{Z}}
\def\R{\mathbb{R}}
\def\A{\mathbb{A}}

\def\C{\mathbb{C}}
\def\Mat{\text{Mat}}

\DeclareMathOperator{\tr}{tr}

\DeclareMathOperator{\Hom}{Hom}
\DeclareMathOperator{\Bil}{Bil}
\DeclareMathOperator{\ind}{ind}

\DeclareMathOperator{\Sym}{Sym}

\DeclareMathOperator{\diag}{diag}
\DeclareMathOperator{\Ind}{Ind}
\DeclareMathOperator{\GL}{GL}
\DeclareMathOperator{\SL}{SL}
\DeclareMathOperator{\SO}{SO}
\DeclareMathOperator{\GSpin}{GSpin}
\DeclareMathOperator{\Spin}{Spin}
\DeclareMathOperator{\Sp}{Sp}
\DeclareMathOperator{\Real}{Re}
\newcommand{\bs}{\backslash}

\begin{document}

\title[Doubling for symplectic coverings]{Doubling Constructions and Tensor Product $L$-Functions: coverings of the symplectic group}
\author{Eyal Kaplan}

\address{Department of Mathematics, Bar Ilan University, Ramat Gan 5290002, Israel}
\email{kaplaney@gmail.com}

\thanks{This research was supported by the ISRAEL SCIENCE FOUNDATION (grant No. 421/17).}
\begin{abstract}
In this work we develop an integral representation for the partial $L$-function of a pair $\pi\times\tau$ of genuine irreducible cuspidal automorphic representations, $\pi$ of the $m$-fold covering of Matsumoto of the symplectic group $\Sp_{2n}$, and $\tau$ of a certain covering group of $\GL_k$, with arbitrary $m$, $n$ and $k$. Our construction is based on the recent extension by Cai, Friedberg, Ginzburg and the author, of the classical doubling method of Piatetski-Shapiro and Rallis, from rank-$1$ twists to arbitrary rank twists. We prove a basic global identity for the integral and compute the local integrals with unramified data. Possible applications include an analytic definition of local factors for representations of covering groups, and a Shimura type lift of representations from covering groups to general linear groups.
\end{abstract}

\maketitle

\section*{Introduction}
When Hecke developed his theory of $L$-functions for modular forms of integral weight, he observed that for half-integral weight, his Hecke operators could not ensure the existence of an Euler product expansion for the $L$-series (\cite{GP}). In the adelic setting, these modular forms can be interpreted as functions on global central extensions of $\SL_2$. The work of Weil was perhaps the first application
of covering groups to the theory of modular forms. Weil constructed the double covering of the symplectic group, and a geometric
realization for the theta representation \cite{We}, with the intention of reformulating Siegel's theory of what is now known as theta liftings, a goal pursued in \cite{Weil2}. Arbitrary rank coverings of $\SL_2$ and $\GL_2$ were studied by Kubota \cite{Kubota,Kubota2}. Around the same time, Moore \cite{Moore}, Steinberg \cite{Stein} and Matsumoto \cite{Mats} constructed
covering groups for almost simple simply connected split groups and laid the foundations for their systematic study.

Shimura \cite{Shimura1973} studied modular forms of half-integral weight, and was able to address early problems observed by Hecke. He defined a new family of Hecke operators, which controlled the Euler product factorization (as in the integral case). He proceeded to produce a lift of modular forms of weight $k/2$, where $k\geq3$ is an odd integer, to modular forms of weight $k-1$. To obtain this result, Shimura developed an integral representation involving an Eisenstein series, by generalizing the method of Rankin \cite{Rankin1939}, and used it for the study of the $L$-series attached to a $k/2$ modular form. He then applied the Converse Theorem of Weil \cite{Weil1967}, to produce a weight $k-1$ modular form. We also mention Waldspurger \cite{Waldspurger1980,Waldspurger1981,Waldspurger1989,Waldspurger1991}, who obtained local and global correspondences between representations of the double covering of $\SL_2$ and the group $\SO_3$, using the theta correspondence.
The results of Waldspurger were generalized only recently by Gan and Ichino \cite{GanIchino2018}, who described
the generic part of the automorphic discrete spectrum of $\Sp^{(2)}_{2n}(\A)$, using the theta lift of Li \cite{JSLi1997}
to $\SO_{2l+1}(\A)$ for $l\gg n$ (which is known to be nonzero) and the work of Arthur \cite{Arthur2013} on the endoscopic classification
for $\SO_{2l+1}$. Other works on the theta correspondence include \cite{KudlaRallis2005,GGP,WGS}. Refer to \cite{GanFanWeissman2018} for an historical overview on these topics.
%Another approach for relating between representations of the coverings and the linear groups for general linear groups, based on the trace %formula, was pursued by Flicker \cite{Flicker2} and by Flicker and Kazhdan \cite{FK}.
%For recent works on the trace formula for covering groups see Li \cite{Li2012,Li2013,Li,Li2014}.

In this work we take the first step towards a generalized Shimura lift, of genuine cuspidal automorphic representations of covering groups to automorphic representations of a suitable general linear group, via an integral representation.
For any $m>1$, we construct an integral representation for the partial $L$-function of a pair of genuine irreducible cuspidal automorphic representations, $\pi$ of an $m$-fold covering of $\Sp_{2n}(\A)$ of Matsumoto \cite{Mats}, and $\tau$ of a certain (specific) covering of $\GL_k(\A)$. We prove a global identity, which on decomposable data leads to an ``almost Eulerian product" (an infinite product of local integrals over the unramified places and an integral over finitely many places), and compute the local integrals with unramified data. Our construction is based on the doubling method of Piatetski-Shapiro and Rallis \cite{PSR} and its recent extensions \cite{CFGK2,CFK}, and we expect it to apply to a wide range of coverings. For more precise details and description see below: the linear construction is briefly recalled in \S~\ref{intro linear}, then we describe its extension to coverings in \S~\ref{intro-dbl for cvr}.

In a follow-up to this work (\cite{me13}) we developed the local theory of the integrals and defined local $\gamma$-, $L$- and $\epsilon$-factors. In the particular case of $m=2$ (the double cover) we except the local and global theory to imply the Shimura lift independently of the trace formula.

Piatetski-Shapiro and Rallis \cite{PSR} developed an integral representation for the standard automorphic $L$-function of
an irreducible cuspidal automorphic representation of a classical group, or its rank-$1$ twists. Their construction was different from other constructions for classical groups at the time (\cite{GPS,G}), in the sense that it did not assume any global or local model.
Cuspidal representations of general linear groups are always globally generic, i.e., admit Whittaker--Fourier coefficients; but this is the exceptional case, for other groups cuspidal (or supercuspidal, locally) representations need not admit Whittaker models. The local theory of their integrals was fully developed by Lapid and Rallis \cite{LR}, with additional cases of groups added by Yamana \cite{Yamana} and Gan \cite{Gan} (see \S~\ref{further} below). The construction of
\cite{PSR}, now known as the doubling method, has had numerous applications, including to the study of the theta correspondence \cite{KudlaRallis1994,HKS,WGS,Yamana} and to cohomological automorphic representations \cite{BochererSchmidt2000,HarrisLiSkinner2005,HarrisLiSkinner2006,EischenHarrisLiSkinner}. However, since it was limited to the standard $L$-function, it was not sufficient for the study of functorial lifts.

In the recent works \cite{CFGK2,CFK} the doubling method was extended to produce $L$-functions for arbitrary pairs of
irreducible cuspidal automorphic representations of a classical group and general linear group. This extension was obtained by generalizing the inducing data of the Eisenstein series to arbitrary generalized Speh representations, and applying an additional Fourier coefficient. The goal of \cite{CFGK2} was the basic global identity for the integral and the computation of the local integrals with unramified data, while the main focus of \cite{CFK} was the local theory at all places. The conclusion of \cite{CFK} was the proof of a global functorial lift from a cuspidal representation of a symplectic, split special orthogonal or split general spin group, to the natural general linear group, via the Converse Theorem of Cogdell and Piatetski-Shapiro \cite{CPS3,CPS1999}. This extended the functorial lift of \cite{CKPS2,CKPS,AsgSha} from globally generic representations to arbitrary ones. Of course the novelty here was really the method, since the endoscopic functorial transfer for quasi-split orthogonal or symplectic groups was obtained by Arthur \cite{Arthur2013} via the twisted stable trace formula, and extended to quasi-split unitary groups by Mok \cite{Mok2015}.

Presently, the advantages of the extended doubling method over other integral representations, include its uniform approach (e.g., independent of the ratio between the ranks of the classical and general linear groups), the local theory which is now available at all places, and the applicability to all cuspidal representations, i.e., no model assumption is made. As mentioned above this is already beneficial for linear groups, but is particularly important when considering covering groups. For covering groups local (thereby global) multiplicity one for Whittaker models breaks down, first and foremost, for genuine unramified principal series representations. This is one of the reasons it is intrinsically difficult to develop Eulerian integrals. The doubling construction is free of this assumption, making it suitable for extensions to coverings.

\subsection{The linear case}\label{intro linear}
We briefly introduce the integrals in the linear case first, following \cite{CFGK2}.
Let $F$ be a number field with a ring of adeles $\A$ and $G=\Sp_{2n}$ be the symplectic group on $2n$ variables.
Let $\pi$ be an irreducible cuspidal automorphic representation of $G(\A)$. For any $k\geq1$, let $\tau$ be an irreducible cuspidal automorphic representation of $\GL_k(\A)$. We integrate two cusp forms in the space of $\pi$ against a Fourier coefficient of an Eisenstein series on $H=\Sp_{4kn}$. The integral represents the partial $L$-function $L^S(s,\pi\times\tau)$, where $S$ is a finite set of places of $F$ outside which all data are unramified.

For $c\geq1$, let $\mathcal{E}_{\tau}$ be the generalized Speh representation corresponding to $c$ copies of $\tau$, constructed by Jacquet \cite{Jac4}. It is the residual representation of an Eisenstein series, corresponding to a representation of $\GL_{kc}(\A)$ parabolically induced from $|\det|^{\zeta_1}\tau\otimes\ldots\otimes|\det|^{\zeta_c}\tau$, at the point
$((c-1)/2,(c-3)/2,\ldots, (1-c)/2)$.
The main properties of $\mathcal{E}_{\tau}$ which were used in \cite{CFGK2}, were that its generic Fourier coefficient along the unipotent orbit $(k^c)$, called a (global) $(k,c)$ functional, is not identically zero, and it does not support any Fourier coefficient on an orbit greater than or not comparable with $(k^c)$. These were proved in \cite{G2,JL2013}. The local components of $\mathcal{E}_{\tau}$ admit unique $(k,c)$ models, which are the local analogs of the $(k,c)$ functional: this was proved in \cite{CFGK2} for the unramified components, and refined in \cite{CFK} to all components. We call a representation of $\GL_{kc}(\A)$ with these properties a $(k,c)$ representation. For the precise definition see \S~\ref{speh def}.

For instance if $c=1$, $\mathcal{E}_{\tau}=\tau$ and since $\tau$ is cuspidal, it is globally generic, i.e., admits a (nonzero) Whittaker--Fourier coefficient. In this case $(k,c)$ models are simply Whittaker models, and indeed the local components of $\tau$ all admit unique Whittaker models.

Put $c=2n$. Let $H=\Sp_{2kc}$, $B_H$ be a fixed Borel subgroup of $H$, $K_H$ be a standard maximal compact subgroup of $H(\A)$, and $P=M_P\ltimes U_P$ be a standard Siegel parabolic subgroup of $H$, i.e., the Levi part $M_P$ of $P$ is isomorphic to $\GL_{kc}$. Consider the representation $\Ind_{P(\A)}^{H(\A)}(\mathcal{E}_{\tau}\delta_P^s)$ (normalized induction), where $s\in\C$. For
a standard $K_H$-finite section $f$ of $\Ind_{P(\A)}^{H(\A)}(\mathcal{E}_{\tau}\delta_P^s)$, regarded as a complex-valued function, we have the Eisenstein series
\begin{align}\label{eq:Eisenstein series main}
E(h;s,f)=\sum\limits_{\gamma\in P(F)\backslash H(F)}f(s,\gamma h),\qquad h\in H(\A).
\end{align}
This series converges absolutely for $\Real(s)\gg0$ and admits meromorphic continuation to $\C$.

To describe the Fourier coefficient of $E(h;s,f)$, let $Q$ be a standard parabolic subgroup of $H$, whose Levi part $M_Q$ is isomorphic to $\GL_c\times\ldots\times\GL_c\times\Sp_{2c}$. Let $U=U_Q$ be the unipotent radical of $Q$. Fix a nontrivial character $\psi$ of $F\backslash\A$. One can define an automorphic character $\psi_U$ of $U(\A)$, such that the direct product $G(\A)\times G(\A)$ is embedded in the stabilizer of $\psi_U$ inside $M_Q(\A)$.
Denote the image of this embedding by $(g_1,g_2)\in H(\A)$.
According to the definitions, the Fourier coefficient $E^{U,\psi_U}(h;s,f)$ of $E(h;s,f)$ along $(U,\psi_U)$ is an automorphic form on $G(\A)\times G(\A)$.

Let $\varphi_1$ and $\varphi_2$ be two cusp forms in the space of $\pi$. Let $\iota=\left(\begin{smallmatrix}& I_{c/2}\\ I_{c/2} &\end{smallmatrix}\right)$, $g\mapsto{}^{\iota}g=\iota g\iota^{-1}$ is an outer involution of $G$, and denote
${}^{\iota}\varphi_2(g)=\varphi_2({}^{\iota}g)$.

The global integral in the linear setting was defined in \cite{CFGK2} by
\begin{align*}
Z(s,\varphi_1,\varphi_2,f)=
\int\limits_{G(F)\times G(F)\backslash G({\A})\times G({\A})}
\varphi_1(g_1)\,\overline{{}^{\iota}\varphi_2(g_2)}\,
E^{U,\psi_U}((g_1,g_2);f,s)\,dg_1\,dg_2.
\end{align*}
This integral is absolutely convergent on the plane except perhaps at the poles of the series.

Let $\{\varphi_1,\varphi_2\}$ be the standard $G(\A)$-invariant inner product on the space of $\pi$, and $f_{W({\mathcal E}_{\tau})}(h,s)$ denote the composition of $f$ with the $(k,c)$ functional on $\mathcal{E}_{\tau}$, namely
\begin{align*}
f_{W(\mathcal{E}_{\tau})}(h,s)=\int\limits_{V_{(c^k)}(F)\backslash V_{(c^k)}({\A})}
f(vh,s)\,\psi^{-1}(\mathrm{tr}(\sum_{i=1}^{k-1}v_{i,i+1}))\,dv.
\end{align*}
Here $V_{(c^k)}$ is the unipotent radical of the parabolic subgroup of $\GL_{kc}$ corresponding to the partition $(c^k)$,
$v=(v_{i,j})_{1\leq i,j\leq k}$ where $v_{i,j}$ are $c\times c$ blocks, and $\mathrm{tr}$ is the trace map.
The main global result of \cite{CFGK2} was, that for $\Real(s)\gg0$,
\begin{align}\label{linear global2}
Z(s,\varphi_1,\varphi_2,f)=\int\limits_{G({\A})}\int\limits_{(U_P\cap U)({\A})}
\{\varphi_1,\pi(g)\varphi_2\}f_{W({\mathcal E}_{\tau})}(\delta u_0(1,{}^{\iota}g),s)
\,\psi_U(u_0)\,du_0\,dg.
\end{align}
Here $\delta\in H(F)$ is the product of $\left(\begin{smallmatrix}&I_{kc}\\-I_{kc}\end{smallmatrix}\right)$ and $\diag(I_{(k-1)c},\left(\begin{smallmatrix}I_{c}&I_{c}\\&I_{c}\end{smallmatrix}\right),I_{(k-1)c})$.
On decomposable data, the r.h.s.~(right hand side) can be written as an Euler product of local integrals. In fact, in \cite{CFGK2} we obtained an almost Euler product, in the sense that the places with ramified or archimedean data appear together as one integral; the full Euler product was obtained in \cite{CFK} after proving the uniqueness of $(k,c)$ models everywhere.

Now at almost all places, the local integrals consist of unramified data. The main local result of \cite{CFGK2} was the computation of these integrals. The infinite product over the places with unramified data is $L^S(s,\pi\times\tau)$ divided by a product of partial $L$-functions which constitute the normalizing factor of the series.

The computation of the integrals with unramified data was carried out in \cite{CFGK2} through a sequence of reductions: the $G\times \GL_k$ integral was reduced to a doubling integral for $\GL_n\times \GL_k$ (also defined in \cite{CFGK2}), then to the case of $n=1$, i.e., to a $\GL_1\times \GL_k$ integral. Since the representation of $\GL_k$ is irreducible and generic (it is a local component of $\tau$), it is natural to compute this integral via the theory of Rankin--Selberg integrals for $\GL_1\times \GL_k$ and $\GL_1\times \GL_{2k}$ of
\cite{JS1,JPSS}.

\subsection{Doubling for coverings}\label{intro-dbl for cvr}
Simply put, the purpose of this work is to extend this construction to representations of coverings of $G=\Sp_{2n}$.
For a positive integer $m$, assume $F$ contains the full group $\mu_m$ of $m$-th roots of unity, and let $G^{(m)}(\A)$ be the $m$-fold covering of $G(\A)$ constructed by Matsumoto \cite{Mats} (following \cite{Moore,Stein}), defined using the global $m$-th order Hilbert symbol ($G^{(1)}=G$).
Within the framework of Brylinski and Deligne \cite{BD}, this is the unique covering which corresponds to the integer valued Weyl group invariant quadratic form $\mathcal{Q}$ on the cocharacter group of a maximal torus of $G$, which is $1$ on the short coroots (\cite[Proposition~4.15]{BD}).

Numerous constructions of integral representations including the doubling method, are based on restriction.
In the linear setting we restrict an Eisenstein series or its Fourier coefficient. Uniqueness results in representation theory
are frequently related to Eulerian products and $L$-functions, or their special values (\cite{Bu}).
The local principal at work here is that the Jacquet functor along $(U,\psi_U)$, which is the local analog of the Fourier coefficient of the Eisenstein series, is a multiplicity free representation of
$G\times G$, outside a discrete subset of $s\in\C$. In the set-up of covering groups, it is not even clear this statement is formally well defined.

Restriction is much more subtle for covering groups. For a covering $\widetilde{H}$ of $H$ and a subgroup $L$ of $H$, let
$\widetilde{L}$ be the restriction of the covering to $L$. It is always a subgroup of $\widetilde{H}$, but the restriction $\mathrm{H}^2(H,\mu_m)\rightarrow \mathrm{H}^2(L,\mu_m)$ ($2$-nd cohomology for topological groups) depends on the precise embedding of $L$. For example consider the $2$-fold covering of $\GL_{2d}$ of \cite{KP}: if $\GL_d$ is embedded via
$g\mapsto\diag(g,I_d)$, $\widetilde{\GL}_d$ is a similar $2$-fold covering, but for the map $g\mapsto\diag(g,g)$, $\widetilde{\GL}_d$ is not even among the coverings of \cite{KP}. This means that the embedding $G\times G$ has to be consistent with the choice
$G^{(m)}$ on each of the factors, to begin with.

A second problem concerns the lift of automorphisms of linear groups to their coverings. Since we are dealing with central extensions, $H$ certainly acts on $H^{(m)}$ by conjugation, but a trivial action might lift to a nontrivial action. For instance, direct factors of Levi subgroups of $\GL_d$ do not commute in the coverings of $\GL_d$ of \cite{KP} (similarly for certain coverings of $\GSpin_{2n+1}$).
We must then verify that both copies of $\widetilde{G}$ commute.

A third difficulty is related to local-global issues. Locally we have the block-compatible $2$-cocycle $\sigma$ of
Banks \textit{et. al.} \cite{BLS}, which is convenient for computations, and is compatible with restrictions in a certain strong sense. E.g., the restriction of $H^{(m)}$ to the copy of $\Sp_{2c}$ in $M_Q$ is $\Sp_{2c}^{(m)}$. Globally, however, the product
$\prod_{\nu}\sigma_{\nu}$ is not defined on $H^{(m)}(\A)$, so a global $2$-cocycle must be defined by normalizing $\sigma_{\nu}$ at almost all places, but then the compatibility properties are in general lost.

One precondition for the study of automorphic forms on covering groups is the splitting of the global covering over the rational points of the group. This is guaranteed, e.g., for the central extensions constructed using a Hilbert symbol, by the reciprocity law for the symbol. Here $\prod_{\nu}\sigma_{\nu}$ is defined on $H(F)$, and can be used to define this (unique) splitting.

The first component to extend to the covering is $\mathcal{E}_{\tau}$, the generalized Speh representation.
Let $\tau$ be a genuine irreducible cuspidal automorphic representation of a covering of $\GL_k(\A)$, say, $\widetilde{\GL}_k(\A)$.
For the covering groups under consideration, the coverings are split canonically over unipotent subgroups
(see \cite{Steinberg,Mats} or the general statement of \cite[Appendix~I]{MW2}). Therefore the definitions of Fourier coefficients along unipotent subgroups, and locally Jacquet functors, immediately extend to covering groups. Alas, uniqueness results often break down. For example, the Fourier expansion of Shalika \cite{Sh} implies in particular that $\tau$ is globally generic, as in the linear case;
but in sharp contrast almost all of the local components of $\tau$ admits a number of Whittaker models which is polynomial in the degree of the covering (see e.g., \cite{KP,McNamara,COf}). The reason for this is that (as already observed above) preimages of abelian subgroups need not be abelian in coverings, and
in particular genuine irreducible representations of the preimage of the diagonal torus of $\GL_k$ are finite-dimensional but not one-dimensional. Thus the naive extension of $\mathcal{E}_{\tau}$ to $\widetilde{\GL}_{kc}(\A)$ will not work, already for $c=1$.

Suzuki \cite{Suzuki1998} extended the generalized Speh representations to $r$-fold covering groups $\widetilde{\GL}_k(\A)$ (of \cite{KP}), under certain global (and also local, implicitly) assumptions. Given $\tau$ as above, he produced a residual representation of an $r$-fold covering group $\widetilde{\GL}_{rk}(\A)$, which is globally generic and admits unique Whittaker models at almost all places (at least). The idea behind his construction was the local correspondence obtained in \cite{Suzuki1998} between certain irreducible unramified representations of $\GL_k$ and $\widetilde{\GL}_{rk}$ (for the latter, genuine), where each unramified quasi-character $\chi_i$ was replaced by an exceptional representation (in the sense of \cite{KP}) of $\widetilde{\GL}_r$ twisted by $\chi_i$. To prove the existence of the residue, the computation of the constant term from \cite{Jac4} was carried out on $\widetilde{\GL}_{rk}$, and the analysis of the poles required several assumptions on the partial $L$-functions of $\tau\times\tau^{\vee}$. Basically, one has to assume that the theory of automorphic $L$-functions of $\GL_k\times\GL_k$ of \cite{JS2,JS1} can be applied to $\widetilde{\GL}_k(\A)$.

The construction of \cite{Suzuki1998} provides us with global $(rk,1)$ representations, i.e., the case $c=1$. The extension of his ideas to $c\geq1$ provides a residual representation $\mathcal{E}_{\tau}$ of $\widetilde{\GL}_{rkc}(\A)$, which is $(rk,c)$, and in particular all of its local unramified components admit unique models, a phenomenon which is very rare for covering groups.

As above let $m\geq1$ be an integer, and put $r=m$ if $m$ is odd, otherwise $r=m/2$.
Let $c=2n$, $H=\Sp_{2rkc}$ and take the Siegel parabolic subgroup $P$, now with $M_P=\GL_{rkc}$. The covering $\widetilde{M}_P(\A)$ is obtained by restriction from $H^{(m)}(\A)$, it is the cover of $\GL_{rkc}$ of \cite{BD} defined by letting the quadratic form $\mathcal{Q}$ be $-2$ on all the coroots, in particular it is not one of the coverings from \cite{KP}. This cover was recently studied by
Savin \cite{Savin7} and Gao \cite{Gao4}. Morally, it is an $r$-fold cover, and we denote it by $\GL_{rkc}^{(m,r)}$ (to avoid confusion with
the notation of \cite{KP}). For example the quadratic Hilbert symbol composed with the determinant is a (global or local) $2$-cocycle for $\GL_{kc}^{(2,1)}$. One of the useful properties of $\GL_{rkc}^{(m,r)}$ is that direct factors of Levi subgroups do commute in the cover (\cite{Savin7,Gao4}); another one is that the main involution of $\GL_{rkc}$ admits a trivial extension to $\GL_{rkc}^{(m,r)}$, though this is not the only extension ($\GL_{rkc}$ is not perfect; also cf. \cite{Kable3}).

Now let $\tau$ be a genuine irreducible cuspidal automorphic representation of $\GL_{rkc}^{(m,r)}$. The representation $\mathcal{E}_{\tau}$ is the residual representation of an Eisenstein series attached to a representation of
$\GL_{rkc}^{(m,r)}(\A)$ parabolically induced from $|\det|^{\zeta_1}\tau\otimes\ldots\otimes|\det|^{\zeta_{rc}}\tau$, the multi-residue taken at $((rc-1)/(2r),(rc-3)/(2r),\ldots, (1-rc)/(2r))$.
As explained above, the construction is based on two assumptions: locally
Conjecture~\ref{local Shimura conjecture} and globally Conjecture~\ref{Shimura conjecture}, which are known for $m=1$ (the linear case). The cases $m=2$ or $k=1$ (and all $m\geq1$) are also known. These are the current cases where
$\mathcal{E}_{\tau}\ne0$ unconditionally. In general these conjectures are natural granted a Shimura type correspondence for coverings of general linear groups, and it is perhaps possible to study these questions by extending the aforementioend correspondences \cite{Flicker2,FK}.
Note that Waldspurger, in his works on the theta correspondence (e.g., \cite{Waldspurger1980,Waldspurger1991}), also used \cite{Flicker2} in a substantial way (\cite{PS1984}) and now the results of Bump \textit{et. al.} \cite{BFH1990b,BFH1990a} and Friedberg and Hoffstein \cite{FH1995} can be used in place of the trace formula (see \cite{GanFanWeissman2018}).
Granted the conjectures, we prove $\mathcal{E}_{\tau}$ is an $(rk,c)$ representation.
More precisely, we prove that there is an irreducible summand of $\mathcal{E}_{\tau}$ which is $(rk,c)$, and re-denote it by $\mathcal{E}_{\tau}$. See \S~\ref{speh gbl}.

The Eisenstein series $E(h;s,f)$ is defined for
a standard $\widetilde{K}_H$-finite section $f$ of $\Ind_{\widetilde{P}(\A)}^{H^{(m)}(\A)}(\mathcal{E}_{\tau}\delta_P^s)$, as in \eqref{eq:Eisenstein series main}, taking into account the splitting of $H(F)$ into $H^{(m)}(\A)$. The Fourier coefficient is again simple to define: the parabolic subgroup $Q$, $U$ and $\psi_U$ are defined as above, except that $M_Q$ consists of $rc-1$ copies of $\GL_c$ instead of $c-1$ copies.

The group $G(\A)\times G(\A)$ is still embedded in the stabilizer of $\psi_U$ in $M_Q(\A)$, but as explained above, we have to verify that the restriction to each copy is $G^{(m)}(\A)$ and that they commute
in $H^{(m)}(\A)$. In fact, even more care is needed because while restriction does give $G^{(m)}(\A)$, an equality in
$\mathrm{H}^2(G(\A),\mu_m)$ is only up to a $2$-coboundary. In the computation we must be sensitive to this distinction, because we have one representation $\pi$ of $G^{(m)}(\A)$, but the restrictions to each copy are only isomorphic. At least this isomorphism is canonical, because $G(\A)$ is perfect.

As in the linear case, we require $E^{U,\psi_U}(h;s,f)$ to be an automorphic form on $G^{(m)}(\A)\times G^{(m)}(\A)$.
Here we have three identifications of $G(F)$: one in each copy of $G^{(m)}(\A)$, and a third one through $H(F)$, and our definitions must ensure compatibility between all these.

Last but not least, ${}^{\iota}$ must lift to an involution of $G^{(m)}(\A)$; since it is only an outer involution of $G$, it is not even clear such a lift exists, locally or globally, but this is in fact true.

Now for a genuine irreducible cuspidal automorphic representation $\pi$ of $G^{(m)}(\A)$ and
a pair of cusp forms $\varphi_1$ and $\varphi_2$ in the space of $\pi$, we define the global integral
$Z(s,\varphi_1,\varphi_2,f)$, which is not very different from the linear version.
The first main result of this paper:
\begin{theorem*}[Theorem~\ref{theorem:main theorem classical groups}]
Integral~$Z(s,\varphi_1,\varphi_2,f)$ given by \eqref{global1} is well defined, absolutely convergent away from the poles of
the series and admits meromorphic continuation to $\C$.
\end{theorem*}

The basic identity for this integral, which is the covering analog of \eqref{linear global2}, is obtained using the unfolding technique
of \cite{CFGK2}. Most of the proof in \cite{CFGK2} was based on arguments involving unipotent groups and the vanishing properties of $\mathcal{E}_{\tau}$. These extend naturally to our set-up, with the covering version of $\mathcal{E}_{\tau}$. However, a crucial observation in the linear case underlying the unfolding, is that the summand of $E^{U,\psi_U}(h;s,f)$
corresponding to the double coset $P\delta(G\times G)U$ is left-invariant under the subgroup $\{(g,{}^{\iota}g):g\in G(\A)\}$. As a rule, covering groups do not split over reductive subgroups, even locally.
Carrying out this step in $H^{(m)}(\A)$ requires an elaborate description of splittings. We then state our global identity:
\begin{theorem*}[Theorem~\ref{theorem:main gbl identity}]
For $\Real(s)\gg0$, integral $Z(s,\varphi_1,\varphi_2,f)$ is equal to \eqref{global2}.
\end{theorem*}
On decomposable data, integral \eqref{global2} factorizes into an almost Euler product given by \eqref{eq:almost Euler}.
This is the starting point for the development of the local theory.

The main local result here is the computation of the integrals with unramified data.
The first step is to define $\GL_n^{(m,r)}\times \GL_k^{(m,r)}$ integrals. The theory of covering groups for simply-connected groups such as $\SL_n$ or $\Spin_n$ is considerably simpler than the general reductive case. For example, coverings for $\SL_n$ are unique (given a Steinberg cocycle), while coverings of $\GL_n$ are parameterized by two integers (\cite{Gao4}, see also \cite{KP}). As another example, the splitting of $\GL_n^{(m,r)}$ over a maximal compact open subgroup is not unique (as opposed to the similar splitting of $G^{(m)}$). Also the construction of the integral now involves a diagonal embedding of $\GL_n$ in $\GL_{2rkn}$, and although $\GL_{2rkn}^{(m,r)}$ is split over this subgroup, the splitting is not unique.

The computation of the doubling integral for $\GL_1^{(m,r)}\times \GL_{k}^{(m,r)}$ again involves an extension of the Rankin--Selberg theory, now to $\GL_1^{(m,r)}\times \GL_{rk}^{(m,r)}$ and $\GL_1^{(m,r)}\times \GL_{2rk}^{(m,r)}$ integrals. In the absence of uniqueness for Whittaker models, it is not clear such an extension in general is feasible. However, in this work the representations of $\GL_{rk}^{(m,r)}$ and $\GL_{2rk}^{(m,r)}$ do admit unique Whittaker models, e.g., the representation of $\GL_{rk}^{(m,r)}$ is a local component of $\mathcal{E}_{\tau}$ with $c=1$.

To evaluate these Rankin--Selberg type integrals we need to express the value of the normalized unramified Whittaker function $W$ at the identity, as a special value of an $L$-function, and to relate the value of $W$ on essentially the diagonal elements $\diag(a^l,I_{rk-1})$ to the $l$-th complete symmetric polynomial in the Satake parameters of $\tau_{\nu}$. In the linear case these results are known and follow from the
Casselman--Shalika formula of \cite{S,CS1}. This is another aspect which is difficult to generalize to covering groups.

The extension of the Casselman--Shalika formula to the coverings of $\GL_n$ of \cite{KP} was developed in \cite{KP,COf}, then extended by
McNamara \cite{McNamara3} to the general coverings of reductive groups of \cite{BD}, under a mild assumption on the field to simplify computations and results. These works expressed the normalized unramified Whittaker function as a weighted sum over the Weyl group. A different approach, motivated by the study of Whittaker coefficients of Borel Eisenstein series on covering groups, was pursued by
Brubaker \textit{et. al.} \cite{BBFtwisted,BBF2011,BBF2011b,BrF}. In these works the theory of crystal graphs was applied to express the Whittaker function as a sum over ``decorated" Gelfand--Tsetlin patterns. All of these works considered general unramified principal series representations of covering groups; presumably, the loss of uniqueness for the Whittaker model has made it very difficult to further extend the linear case and relate the Whittaker function to the theory of representations of the $L$-group. By contrast, here we have uniqueness and it is reasonable to expect results that resemble the linear case, at least formally, and this is precisely what we obtain. To compute $W$ we combine the Casselman--Shalika formula with the Gelfand--Tsetlin patterns description. This is perhaps the first application of Gelfand--Tsetlin patterns to the computation of local integrals which arise in the context of an integral representation.
For $k\leq2$ the formula we seek was obtained by \cite{Suzuki1998}, and we treat the general case. We also use a recent result of Yuanqing Cai \cite{CaiLemma39}, which enables us to argue using induction. The following is our main local result:
\begin{theorem*}[Theorem~\ref{theorem:unramified computation for Sp(2n),SO(2n)}]
Assume $\mu_{2m}\subset F^*$. The integral \eqref{eq:local integral at one place} with unramified data equals
\begin{align*}
\frac{L_{\vartheta_{\nu}}(r\alpha s+1/2,\pi_{\nu}\times\tau_{\nu})}
{[L_{\vartheta_{\nu}}(r\alpha s+rn+1/2,\tau_{\nu})]\prod\limits_{1\leq j\leq rn}L_{\vartheta_{\nu}}(2r\alpha s+2j,\tau_{\nu},\wedge^2)
L_{\vartheta_{\nu}}(2r\alpha s+2j-1,\tau_{\nu},\vee^2)}.
\end{align*}
Here $\alpha=rkc+1$ and $L_{\vartheta_{\nu}}(r\alpha s+rn+1/2,\tau_{\nu})$ appears only for odd $m$.
\end{theorem*}
For the notation and definition of the $L$-factors see \S~\ref{unramified reps}.

Note the assumption $\mu_{2m}\subset F^*$, which is stronger than the condition $\mu_m\subset F^*$ needed for the existence of $G^{(m)}$. This assumption is added throughout several parts of this work: it is needed because we use results from
\cite{BBF2011,McNamara2,Gao5,Gao4}, where it was assumed. It also simplifies our work here, but for several claims we do provide a more general computation.

The theory of the local integrals is independent of
Conjectures~\ref{local Shimura conjecture} and \ref{Shimura conjecture}:
Theorem~\ref{theorem:unramified computation for Sp(2n),SO(2n)} holds unconditionally with a slight modification of the construction (see Remark~\ref{remark:unconditionally local}).

\subsection{Further background, extensions and applications}\label{further}
The applicability of the doubling method to representations of $\Sp_{2n}^{(2)}$ was known to Piatetski-Shapiro and Rallis. The local theory of the doubling integrals of \cite{PSR} for $\Sp_{2n}^{(2)}$ (but when $k=1$) was developed by Gan \cite{Gan}. Note that $\Sp_{2n}^{(2)}$ is an exceptional case, in the sense that irreducible representations do admit at most one Whittaker model. This was proved by Waldspurger \cite{Waldspurger1980} for $n=1$ over $\R$, by Bump \textit{et. al.} \cite{BFJ} for principal series representations over $p$-adic fields, and by Szpruch \cite{Szpruch2007} in general. Szpruch \cite{Dani} then developed the Langlands--Shahidi theory for this covering. In this case there are also integral representations, for globally generic representations
by Ginzburg \textit{et. al.} \cite{GRS4}, and arbitrary irreducible cuspidal representations \cite{GJRS}. The local theory was developed in \cite{me6} for the generic case. The integral presented here is the first known to the author, involving representations of $\Sp_{2n}^{(m)}$ for any $m$.

The uniformity of the doubling method was preserved in the extension to arbitrary $k>1$ in \cite{CFGK2,CFK},
and we expect our method and results here to generalize to a wide class of coverings of reductive groups, including coverings of special orthogonal or general spin groups.

Among other works on integral representations, we mention
\cite{JL,GJ,JS1,JPSS,GPS,BF,G,JS4,BG,Soudry,GPSR,GRS4,BS,GJRS,me2,me3,JZ,Tk,BS2,Soudry7,Soudry8}. While several recent works (e.g., \cite{GPSR,BS,GJRS}) were no longer limited to globally generic representations, their local theory has not been fully developed to provide definitions of local factors (outside the unramified case).

Brylinski and Deligne \cite{BD} provided a general framework, functorial in nature, of a covering group for any split reductive group.
Their construction inspired numerous works on extending the Langlands program to those coverings, including
\cite{Savin5,Weissman5,Weissman4,McNamara,Li,Weissman2,GG,Weissman2018a}.
The work here can be regarded as a first attempt to launch the analytic counterpart, namely a theory of integral representations and local factors for covering groups.

In the celebrated works of Cogdell \textit{et al.}\ \cite{CKPS2,CKPS} on functoriality for classical groups, who established the functorial lift via the Converse Theorem, the results on the $L$-functions were derived using the Langlands--Shahidi method.
The theory of Eisenstein series of Langlands (e.g., \cite{La2,La5}) can be used to study the partial $L$-function through the constant term of the series. This part is not limited to globally generic representations; however, the analytic definition of the complete $L$-function and the proof of its functional equation were only obtained by Shahidi (\cite{Sh2,Sh4,Shahidi1983,Shahidi1985,Sh3}), through his method of local coefficients, where uniqueness of the Whittaker model is crucial.

The analysis of the constant term of Langlands has been recently extended by
Gao~\cite{Gao2018} to the covering groups of \cite{BD}. In particular he proved the meromorphic continuation of the partial $L$-function
(e.g., the $L$-function obtained here). Unfortunately, it is impossible to extend the method of local coefficients to covering groups, because
local multiplicity one for Whittaker models breaks down. Consequently, it is not clear how to define the completed $L$-function with this approach. Nonetheless, since the dimension of the space of Whittaker functionals is finite, there have been recent attempts to define and study a generalized local coefficients matrix, see Gao \textit{et. al.} \cite{GaoShahidiSzpruch2018} and Szpruch \cite{Szpruch2018}.

In sharp contrast, the doubling construction can provide a definition of local factors for covering groups through uniqueness, which boils down to the uniqueness of local $(rk,c)$ models. This uniqueness is proved here when data are unramified, i.e., almost everywhere, which is usually more difficult to prove, so that it is reasonable to expect uniqueness at the ramified places as well. In turn, this construction has the potential to provide a definition of a completed $L$-function, with a functional equation.

%Another possible application of this work is within the descent method of Ginzburg \textit{et al.} \cite{GRS10,GRS9,GRS5,GRS7,RGS},
%which uses the theory of integral representations to characterize the image of global functoriality (see also \cite{Soudry6}).
%Parts of this method can be expected to apply to the doubling integrals of \cite{CFGK2}, and a result towards this goal was recently obtained %by Ginzbrug and Soudry \cite{GS2018}. We also mention the recent extension by Hundley and Sayag \cite{HS}, of the global descent to %quasi-split general spin groups.

\subsection{Description of the contents by section}
Section~\ref{Preliminaries} is devoted to preliminaries. We define the groups and provide basic notation in
\S~\ref{Groups}. General conventions and definitions for covering groups are assembled in \S~\ref{Covering groups}.
The embedding of $G\times G$ in $H$, which dictates most of the coverings involved, is presented in
\S~\ref{embedding}. In \S~\ref{local covering} we define the local $2$-cocycles. First we recall the $2$-cocycle of $\GL_{2rkc}$ of \cite{BLS} and describe its restriction to $H=\Sp_{2rkc}$. In Proposition~\ref{proposition:the $2$-cocycle on G times G} we write a formula for the restriction of this $2$-cocycle to $G\times G$ and in particular, deduce that the preimages of both copies of $G$ in $H^{(m)}$ commute. In Corollary~\ref{corollary:lcl splitting of (G,G)} we show that $H^{(m)}$ is split over the group $\{(g,g):g\in G\}$. The global counterpart of these results is described in \S~\ref{global covering}. The global coverings of both copies of $G(\A)$ are defined using $H^{(m)}(\A)$.
Corollary~\ref{corollary:gbl splitting of (G,G)} is the global analog of Corollary~\ref{corollary:lcl splitting of (G,G)}: the splitting
of $H^{(m)}(\A)$ over $\{(g,g):g\in G(\A)\}$. In \S~\ref{extension of the involution} we discuss the local and global extensions of the involution ${}^{\iota}$ to the covering.

The cover $\GL_{d}^{(m,r)}$ is studied in \S~\ref{covering of the Levi}. As mentioned above, this is not one of the coverings of \cite{KP}. Here preimages of direct factors of Levi subgroups do commute. Consequently, the usual tensor product is defined, and we do not need to use the metaplectic versions of the tensor developed in \cite{Kable,Mezo,Tk,Tk2}. While this observation simplifies numerous local arguments, one still needs to define (as in \cite{Tk,Tk2}) a global block-compatible $2$-cocycle. In \S~\ref{unramified reps} we describe the construction of unramified principal series representations and define the $L$-factors, using the definition of the $L$-group of \cite{FL,McNamara,Weissman2018a} (see also \cite{Gao4}). Section~\ref{Casselman--Shalika formula} describes the analog of the Casselman--Shalika formula for $\GL_{d}^{(m,r)}$. Since our covering is different from the coverings of \cite{KP}, we can not use the formulas from \cite{KP,COf}; the formula from \cite{McNamara3} is applicable to $\GL_{d}^{(m,r)}$, but the normalizations of \cite{KP,COf} are more convenient to work with for our purposes. For clarity, we provide the computation of the coefficients appearing in the formula.

Section~\ref{speh} contains the definitions and construction of $(rk,c)$ representations.
Their definition is given in \S~\ref{speh def}, it is similar to the linear case. In the linear case they satisfy an additional invariance property, with respect to the diagonal embedding $\SL_c^{\triangle}$ of $\SL_c$ in $\GL_{kc}$. This property extends to covering groups. First, we show that $\GL_{rkc}^{(m,r)}$ is split over $\SL_{c}^{\triangle}$, locally in Proposition~\ref{proposition:sigma on diagonal embedding of SLc} and globally in Corollary~\ref{corollary:gbl splitting for SL_c}. The global invariance property is proved in Proposition~\ref{proposition:extra invariance}.

In \S~\ref{local theta speh} we construct local $(rk,c)$ representations, roughly denoted $\Theta(\chi)$, from exceptional representations. When $c=1$, $\Theta(\chi)$ affords a unique Whittaker model, and we prove the local results on the normalized unramified Whittaker function $W$ (see \S~\ref{local theta speh for c=1}). In
Theorem~\ref{theorem:CS formula for rk, c} we use Gelfand--Tsetlin patterns to compute the Jacquet integral on $W$, and in
Theorem~\ref{theorem:Whittaker on x,I} we express $W(\diag(a,I_{rk-1}))$ as a symmetric polynomial, by combining the weighted sum formula of \S~\ref{Casselman--Shalika formula}, the inductive method of Suzuki \cite{Suzuki1998} and Gelfand--Tsetlin patterns.

The global construction of $(rk,c)$ representations using residues of Eisenstein series is presented in \S~\ref{speh gbl}. We state Conjectures~\ref{local Shimura conjecture} and \ref{Shimura conjecture}, under which we prove the existence of $\mathcal{E}_{\tau}$ in Theorem~\ref{theorem:prop1}, then in Theorem~\ref{exthspeh1} prove it is an $(rk,c)$ representation. The local components of $\mathcal{E}_{\tau}$ are described in \S~\ref{Local components of rk c speh}, where they are also related to $\Theta(\chi)$.

Section~\ref{global} is devoted to the global integral. In \S~\ref{global classical} we start with gluing together our preliminaries, in order to prove that the global integral \eqref{global1} is indeed well defined, as part of Theorem~\ref{theorem:main theorem classical groups}. In \S~\ref{global symplectic} we carry out the unfolding process and prove Theorem~\ref{theorem:main gbl identity}, the main global identity. We then show how to obtain, from decomposable data, an almost Euler product \eqref{eq:almost Euler}.
The main local result of this work: Theorem~\ref{theorem:unramified computation for Sp(2n),SO(2n)}, namely the computation of the integrals with unramified data, is stated in this section (proved in
\S~\ref{Computation of the local factors with unramified data}).
The local study of the integrals is initiated here, with their formal equivariance properties (Proposition~\ref{proposition:equiv props}) and several other properties (Proposition~\ref{proposition:local props}).
We define the local $\GL_n^{(m,r)}\times\GL_k^{(m,r)}$ integrals in \S~\ref{integarls for GL}.

The local  theory of the integrals with unramified (non-archimedean) data is developed in \S~\ref{Computation of the local factors with unramified data}. Theorem~\ref{theorem:unramified computation for Sp(2n),SO(2n)} is proved through a sequence of reductions in
\S~\ref{Outline of the computation} and \S~\ref{proof of lemma:reduction from GLn to GLa GLb}, eventually leading to the computation of
a $\GL_1^{(m,r)}\times\GL_k^{(m,r)}$ integral in \S~\ref{final reduction n = 1 linear groups}.

Finally for ease of reading and cross-reference, we include a list of common notation and definitions in Appendix~\ref{list of common notation}.

\subsection*{Acknowledgments}
We thank Yuanqing Cai for making a preliminary version of his results available to us,
Solomon Friedberg for numerous fruitful conversations about this work, and
David Ginzburg for his help with the proof of Theorem~\ref{theorem:CS formula for rk, c}.
We are very happy to thank Mikhail Borovoi, Gautam Chinta, Wee Teck Gan, Dmitry Gourevitch, Erez Lapid, Peter McNamara, Gordan Savin and
Dani Szpruch for valuable and inspiring discussions.

\tableofcontents

\section{Preliminaries}\label{Preliminaries}
\subsection{The groups}\label{Groups}
Let $G=\Sp_{2l}$ denote the symplectic group, realized as the subgroup of matrices $g\in \SL_{2l}$ such that
${}^tg\left(\begin{smallmatrix}&J_l\\-J_l\end{smallmatrix}\right)g=\left(\begin{smallmatrix}&J_l\\-J_l\end{smallmatrix}\right)$,
where ${}^tg$ is the transpose of $g$ and $J_l$ is the $l\times l$ permutation matrix with $1$ along the anti-diagonal. We fix the diagonal torus $T_l$ and the Borel subgroup $B_l=T_l\ltimes N_l$ of upper triangular matrices in $G$.
For a parabolic subgroup $P$ of $G$, $\delta_P$ denotes the modulus character.
If $U$ is a unipotent subgroup, $U^-$ denotes the opposite unipotent subgroup.
Denote the Weyl group of $G$ by $W_G$.

In the group $\GL_d$ fix the Borel subgroup $B_{\GL_d}=T_{\GL_d}\ltimes N_{\GL_d}$ of upper triangular invertible matrices,
where $T_{\GL_d}$ is the diagonal torus. A composition of a positive integer $d$ is an ordered sequence of positive integers whose sum is $d$.
For a composition $\beta=(\beta_1,\ldots,\beta_l)$ of $d$, let $P_{\beta}=M_{\beta}\ltimes V_{\beta}$ denote the standard parabolic subgroup with $M_{\beta}=\GL_{\beta_1}\times\ldots\times\GL_{\beta_l}$.

Denote the set of roots of $\GL_d$ by $\Phi_d$, each $\alpha\in\Phi_d$ corresponds to a pair $(i,j)$ with $1\leq i\ne j\leq d$. Let $\Phi_d^+\subset \Phi_d$ denote the positive roots (defined with respect to $B_{\GL_d}$). The simple roots are $(i,i+1)$, $1\leq i<d$.
The Weyl group of $\GL_d$ is denoted $W_{\GL_d}$.
Let $\ell(w)$ be the length of $w\in W_{\GL_d}$. We write $w_{\alpha}$ for the reflection along $\alpha$. For $b\in\GL_d$, denote $b^*=J_d{}^tb^{-1}J_d$.

For two elements $x$ and $y$ in a group $H$, ${}^xy=xyx^{-1}$, and for $Y<H$, ${}^xY=\{{}^{x}y:y\in Y\}$.

Over a local field $F$, we usually identify algebraic groups with their groups of $F$-points, i.e., $G=G(F)$, $N_l=N_l(F)$.
When $F$ is a number field, $\A$ denotes its ring of adeles and we write $F$-points or $\A$-points explicitly, e.g., $G(\A)$. When the discussion applies to both local and global situations, we simply write $G$ for both cases.
For an integer $m$, $F^{*m}=\{a^m:a\in F^*\}$ and similarly denote $\A^{*m}$.

Let $\Mat_{n\times l}$ be the abelian group of $n\times l$ matrices (over a local field or $\A$), and set $\Mat_{n}=\Mat_{n\times n}$. The trace map is denoted $\tr$.

Induction of representations from parabolic subgroups is always implicitly normalized.

Local fields in this work are always of characteristic $0$. All local representations are assumed to act on complex vector spaces.
The action of a group by right-translation is denoted $\cdot$, e.g., $y\cdot f(g)=f(gy)$ where $f$ is a function on $G$ and $Y<G$.
If $\pi$ is a representation of a unipotent subgroup $U<G$ on a space $V$, and $\psi$ is a character of $U$, the Jacquet module
$J_{U,\psi}(\pi)$ is the quotient $V(U,\psi)\backslash V$, where over non-archimedean fields $V(U,\psi)\subset V$ is the subspace spanned by
all vectors of the form $\pi(u)\xi-\psi(u)\xi$, $u\in U$ and $\xi\in V$, and over archimedean fields $V(U,\psi)$ is the closure of
this subspace.

For a local non-archimedean field $F$, we let $\mathcal{O}$ denote its ring of integers, $\mathcal{P}$ be the maximal ideal of
$\mathcal{O}$, $\varpi$ be a uniformizer and $|\varpi|=q^{-1}$.
Then we take a hyperspecial maximal compact subgroup $K_G=G(\mathcal{O})$ of $G$, and additionally $K_{\SL_l}=\SL_l(\mathcal{O})$
and $K_{\GL_l}=\GL_l(\mathcal{O})$. We also
denote by $K_G$ a maximal compact subgroup of $G$ over archimedean fields. In a global context the maximal compact subgroup
is $K_G=\prod_{\nu}K_{G,\nu}$, where the product varies over all places of $F$.

\subsection{Covering groups}\label{Covering groups}
We introduce our notation for covering groups, and include several general observations, to be used throughout.
Our basic reference is \cite{Moore}; see also \cite{GanFanWeissman2018}.
Let $F$ be a local field or a number field. Fix the group $\mu_m\subset\C^*$ of $m$-th roots of unity. Assume $F^*$ contains $m$ $m$-th roots of unity, then we may identify the subgroup of these roots with $\mu_m$, and for simplicity write $\mu_m\subset F^*$.
Let $(\cdot,\cdot)_m$ be the Hilbert symbol of order $m$ in $F$, which is the product of local symbols if $F$ is a number field.

Let $G$ be a linear algebraic group over a local field or over $\A$.
A topological central extension of $G$ by $\mu_m$, is a short exact sequence
of groups
\begin{align*}
1\rightarrow \mu_m\xrightarrow{i} \widetilde{G}\xrightarrow{p} G\rightarrow 1,
\end{align*}
where $p$ is continuous, $i(\mu_m)$ belongs to the center of $\widetilde{G}$ and $i(\mu_m)\backslash \widetilde{G}\cong G$ as topological groups. We call $\widetilde{G}$ an $m$-fold covering group of $G$. A
Borel measurable $\sigma:G\times G\rightarrow\mu_m$ such that
\begin{align}\label{eq:2-cocycle}
\sigma(g,g')\sigma(gg',g'')=\sigma(g,g'g'')\sigma(g',g'')\qquad\forall g,g',g''\in G,
\end{align}
and on the identity element $e$ of $G$, $\sigma(e,e)=1$, is called a $2$-cocycle of $G$. In particular $\sigma(e,g')=\sigma(g,e)=1$. Let $\mathrm{Z}^2(G,\mu_m)$ denote the group of $2$-cocycles. A Borel measurable map $\eta:G\rightarrow\mu_m$
such that $\eta(e)=1$ is called a $1$-cochain; the corresponding $2$-coboundary is given by the function $(g,g')\mapsto \eta(g)\eta(g')/\eta(gg')$. Let $\mathrm{C}^1(G,\mu_m)$ (resp., $\mathrm{B}^2(G,\mu_m)$) be the group of $1$-cochains (resp., $2$-coboundaries). The $2$-nd cohomology $\mathrm{H}^2(G,\mu_m)$ is by definition $\mathrm{B}^2(G,\mu_m)\backslash \mathrm{Z}^2(G,\mu_m)$, and its elements parameterize the topological central extensions of $G$ by $\mu_m$.
Given $\sigma\in \mathrm{Z}^2(G,\mu_m)$, we can realize the group $\widetilde{G}$ as the set of elements $\langle g,\epsilon\rangle$ with $g\in G$, $\epsilon\in\mu_m$, the product given by
\begin{align*}
\langle g,\epsilon\rangle\langle g',\epsilon'\rangle=\langle gg',\epsilon\epsilon'\sigma(g,g')\rangle.
\end{align*}
If $\sigma,\rho\in \mathrm{Z}^2(G,\mu_m)$ are cohomologous, i.e., equal in $\mathrm{H}^2(G,\mu_m)$,
there is $\eta\in\mathrm{C}^1(G,\mu_m)$ such that
\begin{align}\label{eq:cohomologous}
\rho(g,g')=\frac{\eta(g)\eta(g')}{\eta(gg')}\sigma(g,g'),\qquad\forall g,g'\in G.
\end{align}
In this case replacing $\sigma$ by $\rho$ yields an isomorphic group $\widetilde{G}$:
$\langle g,1\rangle \mapsto \langle g,\eta(g)\rangle $ is a topological isomorphism, where the domain is realized using
$\rho$ and the image by $\sigma$.

If $X$ is a closed subgroup of $G$, the restriction of $\sigma$ to $X\times X$, or more briefly the restriction to $X$, is a $2$-cocycle of $X$. The resulting covering $\widetilde{X}$ of $X$, which depends on the embedding of $X$ in $G$, is by definition realized using the restriction of $\sigma$ to $X$. A section of $X$ is a continuous map $x\mapsto\langle x,\eta(x)\rangle$ where $\eta:X\rightarrow\mu_m$ satisfies $\eta(e)=1$, and we call it a splitting of $X$ if it is also a homomorphism, i.e.,
\begin{align*}
\langle x,\eta(x)\rangle\langle x',\eta(x')\rangle=\langle xx',\eta(xx')\rangle,\qquad\forall x,x'\in X.
\end{align*}
In this case we say that $\widetilde{G}$ splits over $X$.
If $x\mapsto\langle x,\eta(x)\rangle$ and $x\mapsto\langle x,\eta'(x)\rangle$ are two splittings, $x\mapsto\eta(x)\eta'(x)^{-1}$ is a homomorphism $X\rightarrow\mu_m$. In particular if $X$ is a unipotent subgroup, such a homomorphism implies a homomorphism of the local field or $\A$ into $\mu_m$, hence must be trivial and $\eta=\eta'$.
With the notation of \eqref{eq:cohomologous},
if $\sigma$ and $\rho$ are trivial on $X$, i.e., $\sigma(x,x')=\rho(x,x')=1$ for all $x,x'\in X$, $\eta$ becomes a homomorphism which may then
also be trivial, depending on $X$.

In a local-global context, $G(\A)$ is the restricted direct product with respect to a family of compact open subgroups $K_{G,\nu}$ defined at almost all places $\nu$. Assume we have a family $(\sigma_{\nu},\rho_{\nu},\eta_{\nu})$ where $\sigma_{\nu},\rho_{\nu}\in \mathrm{Z}^2(G(F_{\nu}),\mu_m)$ and $\eta_{\nu}\in\mathrm{C}^1(G(F_{\nu}),\mu_m)$, related by \eqref{eq:cohomologous}. If we know the r.h.s.~ of \eqref{eq:cohomologous} is trivial on $K_{G,\nu}$ for almost all $\nu$, we can define $\rho=\prod_{\nu}\rho_{\nu}\in \mathrm{Z}^2(G(\A),\mu_m)$.

In the opposite direction, let $X<G$ be a closed algebraic subgroup such that $X(\A)$ is the restricted direct product $\prod'_{\nu} X(F_{\nu})$ with respect to a family $\{X_{0,\nu}\}_{\nu}$, where for almost all $\nu$, $X_{0,\nu}=X(F_{\nu})\cap K_{G,\nu}$
and there is a unique homomorphism $X_{0,\nu}\rightarrow\mu_m$. Assume $\sigma,\rho\in \mathrm{Z}^2(X(\A),\mu_m)$,
write $\sigma=\prod_{\nu}\sigma_{\nu}$, $\rho=\prod_{\nu}\rho_{\nu}$, and assume we have $\eta_{\nu}\in\mathrm{C}^1(X(F_{\nu}),\mu_m)$
as in \eqref{eq:cohomologous} for all $\nu$, and both $\rho_{\nu}$ and $\sigma_{\nu}$ are trivial on $X_{0,\nu}$
for almost all $\nu$. Then almost everywhere, $\eta_{\nu}$ becomes a homomorphism $X_{0,\nu}\rightarrow\mu_m$, thus $\eta_{\nu}(X_{0,\nu})=1$. Under these assumptions we can then define $\eta=\prod_{\nu}\eta_{\nu}\in\mathrm{C}^1(X(\A),\mu_m)$.
Moreover, if $x\mapsto\langle x,\zeta(x)\rangle$ is a splitting of $X(\A)$ in $\widetilde{X}(\A)$ realized using $\sigma$,
\begin{align}\label{eq:cohomologous splitting}
\rho(x,x')=\frac{\eta(x)\eta(x')}{\eta(xx')}\frac{\zeta(xx')}{\zeta(x)\zeta(x')},\qquad\forall x,x'\in X(\A).
\end{align}
Thus $x\mapsto\langle x,\zeta(x)/\eta(x)\rangle$ is a splitting of $X(\A)$ when $\widetilde{X}(\A)$ is realized using $\rho$.

Since $\mu_m$ belongs to the center of $\widetilde{G}$, $G$ acts on $\widetilde{G}$ by conjugation, which is also a homeomorphism. In fact ${}^{\langle x,\epsilon'\rangle}\langle y,\epsilon\rangle=\langle  {}^xy,\sigma(x,y)\sigma(xy,x^{-1})\epsilon\rangle$ (independent of $\epsilon'$).

For a topological automorphism $\chi:G\rightarrow G$, a lift of $\chi$ to $G^{(m)}$ is a topological automorphism
$\widetilde{\chi}:G^{(m)}\rightarrow G^{(m)}$ making the following diagram commute
\begin{align*}
\begin{array}{ccccccccc}
1& \rightarrow & \mu_m & \rightarrow & G^{(m)} & \rightarrow & G & \rightarrow &1\\
&  & \mathrm{id}\downarrow & & \widetilde{\chi}\downarrow &  & \chi\downarrow  &    &\\
1& \rightarrow & \mu_m & \rightarrow & G^{(m)} & \rightarrow & G & \rightarrow &1.
\end{array}
\end{align*}
One can disregard topological considerations, and consider $\chi$ and $\widetilde{\chi}$ as abstract lifts. In this weaker sense
$\Hom(G,\mu_m)$, where $G$ is regarded as an abstract group, acts transitively on the set of lifts (\cite[\S~V.1, Propositions~1, 4]{Weiss1969}, see also \cite[\S~2]{Kable3}). Thus if $G$ is perfect and $\widetilde{\chi}$ exists, it is unique. In particular if
a topological lift exists, it is again unique.

Let $\chi$ and $\widetilde{\chi}$ be as above (again, topological), and $Y$ be a closed subgroup of $G$.
Assume $y\mapsto\langle y,\eta(y)\rangle$ is the unique splitting of $Y$ and
${}^{\chi}y\mapsto\langle {}^{\chi}y,\eta({}^{\chi}y)\rangle$ is a splitting of ${}^{\chi}Y$.
Then $y\mapsto{}^{(\widetilde{\chi})^{-1}}\langle {}^{\chi}y,\eta({}^{\chi}y)\rangle$ is a splitting of $Y$, hence by uniqueness
${}^{(\widetilde{\chi})^{-1}}\langle {}^{\chi}y,\eta({}^{\chi}y)\rangle=\langle y,\eta(y)\rangle$. We deduce that in this case
\begin{align}\label{eq:epsilon for conjugation between split subgroups}
{}^{\widetilde{\chi}}\langle y,\eta(y)\rangle=\langle {}^{\chi}y,\eta({}^{\chi}y)\rangle.
\end{align}
When there is no risk of confusion, we will denote $\widetilde{\chi}$ also by $\chi$.

For a parabolic subgroup $P<G$, induction from $\widetilde{P}$ to $\widetilde{G}$ is implicitly normalized by
$\delta_P^{1/2}$.

Given a faithful character $\varepsilon:\mu_m\rightarrow\C^*$, an $\varepsilon$-genuine representation of $\widetilde{G}$ is a representation where $\mu_m$ acts by $\varepsilon$. We will usually assume this character is fixed, and omit it from the definition and notation. An $\varepsilon^{-1}$-genuine representation will then be called anti-genuine.

For $G=\Sp_{2n}$ or $\SL_l$, let $G^{(m)}$ denote the $m$-fold covering group $\widetilde{G}$ of $G$, over a local field or over $\A$,
defined by \cite{Mats} (following \cite{Moore,Stein}) with the Steinberg symbol constructed from $(\cdot,\cdot)_m^{-1}$.
It is a locally compact group, and an $l$-group (\cite{BZ1} 1.1) over local non-archimedean fields.
At almost all places $\nu$ of a global field $F$, the covering $G^{(m)}$ of $G(F_{\nu})$ splits over $K_{G,\nu}$
(\cite[Lemma~11.3]{Moore} and the proof of \cite[Theorem~12.2]{Moore}), uniquely because $K_{G,\nu}$ is perfect (\cite[Lemma~11.1]{Moore}).
Since $G^{(m)}$ is split canonically over a maximal unipotent subgroup (for more general statements, see \cite{Steinberg,Mats,BLS,McNamara} and \cite[Appendix~I]{MW2}), notions involving unipotent orbits transfer immediately to covering groups.

We point out that one can re-define, once and for all, $(\cdot,\cdot)_m$ to be $(\cdot,\cdot)_m^l$ for any integer $l$ coprime to $m$, our arguments are independent of this choice of Steinberg symbol (e.g., $G^{(m)}$ can be constructed from $(\cdot,\cdot)_m$).

Henceforth throughout most of this work, $G$ will be the symplectic group.
\subsection{Embedding $G\times G$ in $H$}\label{embedding}
Let $n$, $k$ and $m$ be positive integers, and set $c=2n$. If $m$ is even, put $r=m/2$, otherwise $r=m$. Let $G=\Sp_{c}$ and $H=\Sp_{2rkc}$.
Denote by $P=M_P\ltimes U_P$ the standard maximal parabolic subgroup of $H$ with $M_P=\GL_{rkc}$, i.e., the Siegel parabolic subgroup.
Let $Q=M_Q\ltimes U_Q$ be the standard parabolic subgroup of $H$, whose Levi part $M_Q$ is isomorphic to $\GL_{c}\times\ldots\GL_{c}\times \Sp_{2c}$, where $\GL_{c}$ appears $rk-1$ times. Put $U=U_Q$.

For a fixed nontrivial additive character $\psi$, which is a character of a local field $F$, or a character of $F\backslash\A$ when $F$ is a number field, define a character of $U$ as follows. In a local context it is a character of $U(F)$, in a global context it is a character of $U(\A)$ which is trivial on $U(F)$. For $v\in V_{(c^{rk-1})}$, write $v=(v_{i,j})_{i,j}$ where $v_{i,j}\in\Mat_c$, then define
\begin{align}\label{eq:psi on V rk-1}
\psi(v)=\psi(\sum_{i=1}^{rk-2}\tr(v_{i,i+1})).
\end{align}
For $x\in\Mat_{((rk-1)c)\times 2c}$, if $rk>1$, write the bottom
$c\times 2c$ block in the form
\begin{align*}
\left(\begin{smallmatrix} Y_1&Z_1&Y_2\\ Y_3&Z_2&Y_4 \end{smallmatrix}\right),\qquad Y_i\in \Mat_{n},Z_j\in\Mat_{n\times c}.
\end{align*}
Then define
\begin{align}\label{eq:character of U}
\psi_U(\left(\begin{smallmatrix}v&x&y\\&I_{2c}&x'\\&&v^*\end{smallmatrix}\right))=\psi(v)\psi(\tr(Y_1)+\tr(Y_4)), \qquad v\in V_{(c^{rk-1})}.
\end{align}
The stabilizer of $\psi_U$ in $M_Q$ contains $G\times G$. The embedding $G\times G\hookrightarrow H$ is defined by
\begin{align*}
(g_1,g_2)\mapsto\diag(g_1,\ldots,g_1,\left(\begin{smallmatrix} g_{1,1}&&g_{1,2}\\ &g_2&\\ g_{1,3}&&g_{1,4}\end{smallmatrix}\right),g_1^*,\ldots,g_1^*),\qquad
g_1=\left(\begin{smallmatrix}
            g_{1,1} & g_{1,2} \\
            g_{1,3} & g_{1,4}
          \end{smallmatrix}\right),\qquad g_{1,j}\in\Mat_n.
\end{align*}
Here on the r.h.s.~ $g_1^*$ appears $rk-1$ times. Oftentimes it will be convenient to refer to each of the copies separately. We write
for $g\in G$, $\mathfrak{e}_1(g)=(g,1)$ and $\mathfrak{e}_2(g)=(1,g)$.
We call $\mathfrak{e}_1(G)$ the left copy of $G$ in $H$, and $\mathfrak{e}_2(G)$ is the right copy.

\subsection{Local covering}\label{local covering}
We proceed with the notation of \S~\ref{Covering groups} and \S~\ref{embedding}. Consider a local field $F$. Let $H^{(m)}$ denote the $m$-fold covering of $H=H(F)$ of \cite{Mats} defined with the Steinberg symbol constructed from $(\cdot,\cdot)_m^{-1}$. Since $H$ is a subgroup of $\SL_{2rkc}$, it is convenient to use the $2$-cocycle $\sigma_{2rkc}$ of $\GL_{2rkc}$ of Banks \textit{et. al.} \cite[\S~3]{BLS} for local computations (actually, even for a definition, see \cite[p.~114]{Savin5}).

First we recall several properties of this $2$-cocycle, used throughout. For any integer $d>1$,
let $\sigma_{\SL_{d+1}}\in \mathrm{Z}^2(\SL_{d+1},\mu_m)$ be the $2$-cocycle of \cite[\S~2]{BLS} which represents $\SL_{d+1}^{(m)}$
in $\mathrm{H}^2(\SL_{d+1},\mu_m)$ (see \S~\ref{Covering groups}). Let $\sigma_d$ be the $2$-cocycle of $\GL_d$ defined in \cite[\S~3]{BLS} for $b,b'\in \GL_d$ by
\begin{align*}
\sigma_d(b,b')=(\det b,\det b')_m\sigma_{\SL_{d+1}}(\diag(b,\det b^{-1}),\diag(b',\det {b'}^{-1})).
\end{align*}
The image of $\SL_d$ in $\SL_{d+1}$ under the embedding $b\mapsto\diag(b,1)$ is standard in the sense of \cite[\S~2]{BLS},
hence by \cite[\S~2, Theorem~7]{BLS} the restriction of $\sigma_d$ to $\SL_d$ is $\sigma_{\SL_{d}}$.
By \cite[\S~3, Lemma~1]{BLS}, for $t=\diag(t_1,\ldots,t_d)\in T_{\GL_d}$ and $t'\in T_{\GL_d}$ (with similar notation),
\begin{align}\label{eq:sigma on torus of GL}
\sigma_d(t,t')=\prod_{i<j}(t_i,t'_j)_m.
\end{align}
According to \cite[\S~3, Lemma~4]{BLS}, $\sigma_{d}$ is trivial on $N_{\GL_d}$, and for all $b,b'\in\GL_{d}$ and $v,v'\in N_{\GL_d}$,
\begin{align}\label{eq:sigma on h and v}
&\sigma_{d}(b,v')=\sigma_{d}(v,b')=1, \\\label{eq:sigma on vh and h'v'}
&\sigma_{d}(vb,b'v')=\sigma_{d}(b,b').
\end{align}
It follows that if ${}^bv\in N_{\GL_d}$,
\begin{align}\label{eq:sigma conjugate v by h}
{}^b\langle v,1\rangle=\langle {}^bv,1\rangle.
\end{align}
Also if $u^-\mapsto\langle u^-,\varsigma(u^-)\rangle$ is the splitting of $N_{\GL_d}^-$ and for $u^-\in N_{\GL_d}^-$ we have
${}^bu^-\in N_{\GL_d}$,
\begin{align}\label{eq:sigma conjugate v- to v by h}
{}^b\langle u^-,\varsigma(u^-)\rangle=\langle {}^bu^-,1\rangle.
\end{align}
This follows from \eqref{eq:epsilon for conjugation between split subgroups} with $Y=N_{\GL_d}$ and $\chi={}^b$.

One of the key properties of $\sigma_d$ is the block-compatibility formula \cite[Theorem~11]{BLS}: for $0<l<d$,
$a,a'\in\GL_l$ and $b,b'\in\GL_{d-l}$,
\begin{align}\label{eq:BLS block compatible}
\sigma_{d}(\diag(a,b),\diag(a',b'))=(\det a,\det b')_m\sigma_l(a,a')\sigma_{d-l}(b,b').
\end{align}
We also mention that $\sigma_1$ is trivial and $\sigma_2$ coincides with the $2$-cocycle of Kubota \cite{Kubota}:
\begin{align}\label{eq:Kubota formula}
\sigma_2(g,g')=\Big(\frac{\gamma(gg')}{\gamma(g)},\frac{\gamma(gg')}{\gamma(g')\det g}\Big)_m,
\qquad\gamma\left(\begin{smallmatrix}a&b\\c&d\end{smallmatrix}\right)=\begin{cases}c&c\ne0,\\d&c=0.\end{cases}
\end{align}
In addition, there is a well defined action
of $W_{\GL_d}$ on the genuine smooth admissible representations of
$\widetilde{T}_{\GL_d}$ (e.g., \cite[\S~13.6]{McNamara}).

To compute conjugations of torus elements by Weyl elemets, we will repeatedly use the set $\mathfrak{W}_d\subset\GL_d$ as in \cite{BLS}.
For any $\alpha\in\Phi_d$, $w_{\alpha}$ was defined to be the reflection along $\alpha$. We fix concrete representatives for the simple roots $\alpha=(i,i+1)$, by $w_{\alpha}=\diag(I_{i-1},\left(\begin{smallmatrix}&-1\\1\end{smallmatrix}\right),I_{d-i-1})$. Any $w\in W_{\GL_d}$ can be written as a product $w_{\alpha_1}\cdot\ldots \cdot w_{\alpha_{\ell(w)}}$ where $\alpha_1,\ldots,\alpha_{\ell(w)}$ are simple reflections. Define
the set $\mathfrak{W}_d=\{w_{\alpha_1}\cdot\ldots \cdot w_{\alpha_{\ell(w)}}:w\in W_{\GL_d}\}\subset\SL_d$, it is not a group. Any permutation matrix $w\in \GL_d$ can be written uniquely in the form $w=t_0w'$ where $t_0\in T_{\GL_d}$ and its coordinates are $\pm1$, and $w'\in\mathfrak{W}_d$. For any $t\in T_{\GL_d}$, by \cite[\S~3, Theorem~7(b), (d)]{BLS},
\begin{align}\label{eq:sigma w mathcal and t}
&\sigma_d(t,w')=1,\\
\label{eq:sigma t and w mathcal}
&\sigma_d(w',t)=\prod_{(i,j)=\alpha\in\Phi_d^+:w\alpha<0}(-t_j,t_i)_m.
\end{align}
Using \eqref{eq:2-cocycle} (with $g={}^{w'}t$, $g'=w'$, $g''={w'}^{-1}$) and because $\sigma({}^{w'}t,w')=1$,
\begin{align}\label{eq:sigma conjugations}
\sigma_d({}^{w'}tw',{w'}^{-1})=
\sigma_d({}^{w'}t,w'{w'}^{-1})\sigma_d(w',{w'}^{-1})=\sigma_d(w',{w'}^{-1}).
\end{align}
Hence combining \eqref{eq:sigma conjugations} with \eqref{eq:sigma t and w mathcal},
\begin{align}\label{eq:conj mathcal t in GLd}
{}^{w'}\langle t,1\rangle=\langle {}^{w'}t,\sigma_d(w',t)\sigma_d(w't,{w'}^{-1})\sigma_d(w',{w'}^{-1})^{-1}\rangle=
\langle {}^{w'}t,\prod_{(i,j)=\alpha\in\Phi_d^+:w\alpha<0}(-t_j,t_i)_m\rangle.
\end{align}
Then ${}^{w}\langle t,1\rangle={}^{t_0}({}^{w'}\langle t,1\rangle)$ can be computed from this and \eqref{eq:sigma on torus of GL}. It can be a lengthy computation to find $t_0$; things are simplified when we know (a priori) that $t_0$ and ${}^{w'}t$ commute in the cover (e.g., when $-1\in F^{*m}$).

\begin{example}\label{example:delta decomposition t_0 and w'}
Consider the permutation matrix $w=\left(\begin{smallmatrix}&I_d\\I_d\end{smallmatrix}\right)\in\GL_{2d}$.
Denote $\alpha_i=(i,i+1)$ for $1\leq i<2d$. Then
\begin{align*}
w=\diag((-1)^{d}I_{d},I_{d})(w_{\alpha_{d}}\cdot\ldots\cdot w_{\alpha_{2d-1}})
(w_{\alpha_{d-1}}\cdot\ldots\cdot w_{\alpha_{2d-2}})\cdot\ldots\cdot
(w_{\alpha_{1}}\cdot\ldots\cdot w_{\alpha_{d}}),
\end{align*}
and the product is reduced: to see this note that if $w_{0,d}$ is the longest Weyl element of $W_{\GL_d}$,
$\ell(w)=\ell(w_{0,2d})-2\ell(w_{0,d})=d^2$. Hence $t_0=\diag((-1)^{d}I_{d},I_{d})$ and $t_0w=w'\in\mathfrak{W}_{2d}$.
\end{example}
According to the Bruhat decomposition $\GL_d=\bigsqcup_{w''\in\mathfrak{W}_{d}}N_{\GL_d}T_{\GL_d}w''N_{\GL_d}$. For $g\in\GL_d$,
write $g=n't'w''n''$ with $n',n''\in N_{\GL_d}$, $t'\in T_{\GL_d}$ and $w''\in\mathfrak{W}_{d}$. Define $\mathbf{t}(g)=t'$ (see \cite[p.~151]{BLS}). By \eqref{eq:sigma on h and v} and \eqref{eq:sigma w mathcal and t},
\begin{align}\label{eq:Bruhat decomposition sigma}
\langle g,1\rangle=\langle n',1\rangle \langle t',1\rangle \langle w'',1\rangle \langle n'',1\rangle.
\end{align}

As mentioned above $\mathfrak{W}_d$ is not a group, and we usually work with permutations. It will therefore be convenient to define a group containing $\mathfrak{W}_d$ and permutations. Let $\mathfrak{W}^+_d$ be the group generated by $\mathfrak{W}_d$ and the diagonal matrices $\diag(t_1,\ldots,t_d)$ with $t_i=\pm1$ for each $i$.

When $(-1,-1)_m=1$ in $F$, e.g., if $-1\in F^{*m}$, $\sigma_{d}$ is trivial on the permutation matrices and also
on $\mathfrak{W}^+_d$. This follows from the proof of \cite[\S~3, Theorem~7]{BLS}, noting that for $g\in \mathfrak{W}^+_d$, the coordinates of $\mathbf{t}(g)$ are $\pm1$ (see also \cite[\S~5]{BLS}). In particular if we identify $W_{\GL_d}$ with the full subgroup of permutation matrices in $\GL_d$, $w\mapsto \langle w,1\rangle$ is a splitting of $W_{\GL_d}$.

We say that $F$ is unramified if it is non-archimedean, $|m|=1$, $q$ is odd and $q>3$. Assume this is the case.
Then $(\mathcal{O}^*,\mathcal{O}^*)_m=1$, whence $\sigma_d$ is trivial on $T_{\GL_d}\cap K_{\GL_d}$, i.e., on torus elements
with coordinates in $\mathcal{O}^*$ (by \eqref{eq:sigma on torus of GL}), and on $\mathfrak{W}^+_d$.
By \cite[Lemma~11.3]{Moore}, there is a splitting of $K_{\GL_{d}}$ (but $\sigma_d$ is not trivial on $K_{\GL_d}$). As in \cite[p.~43]{KP}, this splitting is made ``canonical"
by taking it to be the restriction of the splitting of $K_{\SL_{d+1}}$ in $\SL_{d+1}^{(m)}$. Denote this splitting by $\langle y,\eta_d(y)\rangle$.
Since $\sigma_d$ is trivial on $N_{\GL_{d}}(\mathcal{O})$, $T_{\GL_d}\cap K_{\GL_d}$ and $\mathfrak{W}^+_d$,
$\eta_d$ is a homomorphism of these subgroups, then the explicit description of $\eta_2$ of \cite[p.~19]{Kubota2}, namely
$\eta_2(\left(\begin{smallmatrix}a&b\\c&d\end{smallmatrix}\right))=1$ if $|c|=0,1$ and $(c,d(ad-bc)^{-1})_m$ otherwise, implies
(in light of \eqref{eq:Kubota formula}) $\eta_d$ is trivial on $N_{\GL_{d}}(\mathcal{O})$, $T_{\GL_d}
\cap K_{\GL_d}$ and $\mathfrak{W}^+_d$ (see \cite[Proposition~0.I.3]{KP} and \cite[(1.3) and p.~183]{Tk}).

Again consider an arbitrary $F$, i.e., not necessarily unramified. The image of $H$ in $\SL_{2rkc}$ is standard
(as defined in \cite[p.~143]{BLS}), hence by \cite[\S~2, Theorem~7]{BLS}, the restriction of $\sigma_{2rkc}$ to $H$
represents $H^{(m)}$ in $\mathrm{H}^2(H,\mu_m)$. In fact this restriction is $\sigma_{\Sp_{2rkc}}$ of \cite[\S~2]{BLS};
but the description of $\sigma_{2rkc}$ in \cite[\S~3]{BLS} is more convenient for matrices.
When we apply \eqref{eq:sigma on torus of GL} to $T_{rkc}<T_{\GL_{2rkc}}$, if
$t=\diag(t_1,\ldots,t_{rkc},t_{rkc}^{-1},\ldots,t_{1}^{-1})\in T_{rkc}$ and $t'\in T_{rkc}$ (with similar notation),
\begin{align}\label{eq:BLS $2$-cocycle on torus}
\sigma_{2rkc}(t,t')=\prod_{i=1}^{rkc}(t_i,t'_i)_m^{-1}.
\end{align}
Since $N_{rkc}<N_{\GL_{2rkc}}$, we can use \eqref{eq:sigma on h and v}--\eqref{eq:sigma conjugate v by h} for $b,b'\in H$ and $v,v'\in N_{rkc}$. When $F$ is unramified,
$K_{H}=H(\mathcal{O})$ is perfect (\cite[Lemma~11.1]{Moore}). Regarding $\eta_{2rkc}$ as
a function on $K_H$ by restriction, it is the unique $1$-cochain such that
\begin{align}\label{eq:splitting of K H}
\sigma_{2rkc}(y,y')=\frac{\eta_{2rkc}(yy')}{\eta_{2rkc}(y)\eta_{2rkc}(y')},\qquad\forall y,y'\in K_{H}.
\end{align}
Indeed if $\eta':K_H\rightarrow\mu_m$ is another such mapping, $y\mapsto\eta_{2rkc}(y)\eta'(y)^{-1}$ is a character of $K_{H}$.

\begin{proposition}\label{proposition:action of W on torus is trivial on Sp}
Assume $-1\in F^{*m}$ and let $w\in H\cap \mathfrak{W}^+_{2rkc}$ be a representative of an element of $W_H$. For any $t\in T_{rkc}$, ${}^w\langle t,1\rangle=\langle {}^wt,1\rangle$.
\end{proposition}
\begin{proof}
Write $w=t_0w'$ with
$t_0\in T_{\GL_{2rkc}}$ and $w'\in\mathfrak{W}_{2rkc}$.
By \eqref{eq:BLS $2$-cocycle on torus} it suffices to prove the result for $t=\diag(I_{i-1},x,I_{2(rkc-i)},x^{-1},I_{i-1})$ and $1\leq i\leq rkc$. Our assumption on $-1$ implies ${}^w\langle t,1\rangle={}^{w'}\langle t,1\rangle$.
By \eqref{eq:conj mathcal t in GLd}, the only roots of unity introduced by the conjugation ${}^{w'}\langle t,1\rangle$ are of the form
$(-x,x)^{\pm1}_m$ and $(1,x)^{\pm1}_m$ (which are always trivial), and $(-1,x)^{\pm1}_m$ which equals $1$ by our assumption on $-1$.
\end{proof}
\begin{remark}
The proposition does not apply to arbitrary $t\in T_{\GL_{2rkc}}$.
\end{remark}
Consider the involution $g\mapsto g^*=J_c{}^tg^{-1}J_c$ of $\GL_c$. Define
\begin{align*}
\sigma^*_{c}(g,g')=\sigma_{c}(g^*,{g'}^*).
\end{align*}
This is again a $2$-cocycle of $\GL_c$, and of $\SL_c$ or $G$ by restriction, because $g\mapsto g^*$ is an automorphism and a homeomorphism.

\begin{proposition}\label{proposition:sigma * and sigma on SLc}
The $2$-cocycles $\sigma^*_{c}$ and $\sigma_{c}$ are cohomologous on $\SL_c$.
\end{proposition}
\begin{proof}
The proof is similar to \cite[Lemma~2]{Kable3}.
Let $T_{\SL_c}=T_{\GL_c}\cap\SL_c$. By \cite[p.~54, Corollary~2]{Moore}, restriction $\mathrm{H}^2(\SL_c,\mu_m)\rightarrow\mathrm{H}^2(T_{\SL_c},\mu_m)$ is injective.
Now by \eqref{eq:sigma on torus of GL}, for $t,t'\in T_{\SL_c}$,
\begin{align*}
\sigma_c(t,t')=\prod_{i=1}^{c-1}\prod_{j=1}^i(t_i,t'_j)_m^{-1},\qquad
\sigma^*_c(t,t')=\prod_{i=1}^{c-1}\prod_{j=i}^{c-1}(t_i,t'_j)_m^{-1}.
\end{align*}
Therefore if we define $\eta\in\mathrm{C}^1(T_{\SL_c},\mu_m)$ by $\eta(t)=\prod_{1\leq i\leq j\leq c-1}(t_i,t_j)_m$,
\begin{align*}
\sigma_c(t,t')\sigma^*_c(t,t')^{-1}=\prod_{i=1}^{c-1}\prod_{j=1}^i(t_i,t'_j)_m^{-1}
\prod_{i=1}^{c-1}\prod_{j=i}^{c-1}(t_i,t'_j)_m=\frac{\eta(tt')}{\eta(t)\eta(t')}.
\end{align*}
Hence $\sigma^*_{c}=\sigma_{c}$ in $\mathrm{H}^2(T_{\SL_c},\mu_m)$, thereby also in $\mathrm{H}^2(\SL_c,\mu_m)$.
\end{proof}
\begin{remark}
The involution $g\mapsto g^*$ is the identity on $T_n$ (in fact on $M_n$). Also, a slightly more careful argument
will extend this proposition to $\GL_c$, see the proof of Proposition~\ref{proposition:sigma * and sigma on GLd}.
\end{remark}
We then have $\varsigma_{*,c}\in\mathrm{C}^1(\SL_c,\mu_m)$ such that
\begin{align}\label{eq:sigma c *}
\sigma^*_{c}(b,b')=\frac{\varsigma_{*,c}(b)\varsigma_{*,c}(b')}{\varsigma_{*,c}(bb')}\sigma_{c}(b,b'),\qquad \forall b,b'\in \SL_c.
\end{align}
We mention that $\sigma^*_{c}\ne\sigma_c$ in $\mathrm{Z}^2(\SL_c,\mu_m)$ (e.g., consider $g=\left(\begin{smallmatrix}a&\\&a^{-1}\end{smallmatrix}\right)$ and $g'=\left(\begin{smallmatrix}&-b^{-1}\\b\end{smallmatrix}\right)$ and use \eqref{eq:Kubota formula}); for another example see below. We introduce a minor correction to $\sigma^*_{c}$ using $\varsigma_{*,c}$: define $\sigma^{*,rk}_{c}\in\mathrm{Z}^2(\SL_c,\mu_m)$ by
\begin{align}\label{eq:sigma c * rk}
\sigma^{*,rk}_{c}(b,b')=\left(\frac{\varsigma_{*,c}(b)\varsigma_{*,c}(b')}{\varsigma_{*,c}(bb')}\right)^{rk}\sigma^*_{c}(b,b')
=\left(\frac{\varsigma_{*,c}(b)\varsigma_{*,c}(b')}{\varsigma_{*,c}(bb')}\right)^{rk+1}\sigma_{c}(b,b').
\end{align}
By definition
$\sigma^{*}_{c}=\sigma^{*,rk}_{c}=\sigma_{c}$ in $\mathrm{H}^2(\SL_c,\mu_m)$, in particular in $\mathrm{H}^2(G,\mu_m)$.
Also if $m|rk$ (e.g., when $r=m$), $\sigma^{*,rk}_{c}=\sigma^*_{c}$ in $\mathrm{Z}^2(\SL_c,\mu_m)$, because $\varsigma_{*,c}^{rk}$ is trivial (being an $m$-th root of unity).

\begin{example}
Consider $c=4$, $t=\diag(t_1,t_2,t_2^{-1},t_1^{-1})$ and $w=\diag(1,\left(\begin{smallmatrix}&-1\\1\end{smallmatrix}\right),1)$. Then
$\sigma_4(w,t)=(-t_2^{-1},t_2)_m=1$ by \eqref{eq:sigma t and w mathcal}. To compute
$\sigma_4^*(w,t)=\sigma_4(w^*,t)$, write $w^*=t_0w'$ where $t_0=\diag(1,-1,-1,1)$ and
$w'$ belongs to the set $\mathfrak{W}_4$. By \eqref{eq:2-cocycle},
\begin{align*}
\sigma_4(t_0,w')\sigma_4(w^*,t)=\sigma_4(t_0,w' t)\sigma_4(w',t).
\end{align*}
Now $\sigma_4(t_0,w')=1$ and $\sigma_4(w',t)=(-t_1^{-1},t_1)_m=1$ by \eqref{eq:sigma w mathcal and t} and \eqref{eq:sigma t and w mathcal},
and
\begin{align*}
\sigma_4(t_0,w' t)=\sigma_4(t_0,{}^{w'}tw')=\sigma_4(t_0,{}^{w'}t)=(-1,t_2)_m^{-1},
\end{align*}
where we used \cite[\S~3, (5)]{BLS} and \eqref{eq:sigma on torus of GL}. Therefore $\sigma_4(w^*,t)=(-1,t_2)_m^{-1}\ne\sigma_4(w,t)$.
\end{example}
We describe the restriction of the $2$-cocycle $\sigma_{2rkc}$ to the image of $G\times G$ in $H$.
\begin{proposition}\label{proposition:the $2$-cocycle on G times G}
For all $g_i,g_i'\in G$,
\begin{align}\label{eq:the $2$-cocycle on G times G formula}
\sigma_{2rkc}((g_1,g_2),(g_1',g_2'))&=\sigma^{*,rk}_{c}(g_1,g_1')^{-1}\sigma_{c}(g_2,g_2').
\end{align}
In other words
\begin{align*}
\sigma_{2rkc}(\mathfrak{e}_1(g_1)\mathfrak{e}_2(g_2),\mathfrak{e}_1(g_1')\mathfrak{e}_2(g_2'))=
\sigma^{*,rk}_{c}(g_1,g_1')^{-1}\sigma_{c}(g_2,g_2').
\end{align*}
In particular $\mathfrak{e}_1(G)$ and $\mathfrak{e}_2(G)$ commute in $H^{(m)}$ (!).
\end{proposition}
\begin{proof}
For $g=\left(\begin{smallmatrix}
         g_1    & g_2 \\
         g_3 & g_4
       \end{smallmatrix}\right)\in G$ with $g_i\in\Mat_n$, denote
\begin{align*}
&\mathfrak{e}_1^{\circ}(g)=\diag(g,\ldots,g,I_{2c},g^*,\ldots,g^*),\qquad
\mathfrak{e}_1^{\bullet}(g)=\diag(I_{(rk-1)c},\left(\begin{smallmatrix} g_{1}&&g_{2}\\ &I_{c}&\\ g_{3}&&g_{4}\end{smallmatrix}\right),I_{(rk-1)c}).
\end{align*}
Then $\mathfrak{e}_1(g)=\mathfrak{e}_1^{\circ}(g)\mathfrak{e}_1^{\bullet}(g)$. The subgroup $\mathfrak{e}_2(G)$ is standard in the sense of \cite[\S~2, Theorem~7]{BLS}, and so is $\mathfrak{e}_1^{\bullet}(G)$ (see the proof of \cite[Lemma~3.5]{Tk2}), but not $\mathfrak{e}_1^{\circ}(G)$. Thus by \cite[\S~2, Theorem~7]{BLS},
\begin{align*}
\sigma_{2rkc}(\mathfrak{e}_1^{\bullet}(g_1)\mathfrak{e}_2(g_2),\mathfrak{e}_1^{\bullet}(g_1')\mathfrak{e}_2(g_2'))&=\sigma_{2rkc}(\mathfrak{e}_1^{\bullet}(g_1),\mathfrak{e}_1^{\bullet}(g_1'))\sigma_{2rkc}(\mathfrak{e}_2(g_2),\mathfrak{e}_2(g_2'))
=\sigma_{c}(g_1,g_1')\sigma_{c}(g_2,g_2').
\end{align*}
Also by \eqref{eq:BLS block compatible},
\begin{align*}
\sigma_{2rkc}((g_1,g_2),(g_1',g_2'))&=\sigma_{2rkc}(
\mathfrak{e}_1^{\circ}(g_1)\mathfrak{e}_1^{\bullet}(g_1)\mathfrak{e}_2(g_2)
,\mathfrak{e}_1^{\circ}(g_1')\mathfrak{e}_1^{\bullet}(g_1')\mathfrak{e}_2(g_2'))\\
&=\prod_{i=1}^{rk-1}\sigma_{c}(g_1,g_1')\sigma_{c}(g_1^*,{g_1'}^*)\sigma_{2rkc}(\mathfrak{e}_1^{\bullet}(g_1)\mathfrak{e}_2(g_2),\mathfrak{e}_1^{\bullet}(g_1')\mathfrak{e}_2(g_2')).
\end{align*}
Hence
\begin{align*}
\sigma_{2rkc}((g_1,g_2),(g_1',g_2'))=\prod_{i=1}^{rk-1}\sigma_{c}(g_1,g_1')\sigma_{c}(g_1^*,{g_1'}^*)\sigma_{c}(g_1,g_1')\sigma_{c}(g_2,g_2').
\end{align*}

It remains to observe
\begin{align*}
\prod_{i=1}^{rk-1}\sigma_{c}(g_1,g_1')\sigma_{c}(g_1^*,{g_1'}^*)
=\sigma_{c}^{2rk-2}(g_1,g_1')\left(\frac{\varsigma_{*,c}(g_1)\varsigma_{*,c}(g_1')}{\varsigma_{*,c}(g_1g_1')}
\right)^{rk-1}
=\sigma_{c}^{-2}(g_1,g_1')\left(\frac{\varsigma_{*,c}(g_1)\varsigma_{*,c}(g_1')}{\varsigma_{*,c}(g_1g_1')}
\right)^{rk-1}.
\end{align*}
Here we used the fact that $\sigma_{c}^{2r}$ is trivial, because the image of $\sigma_{c}$ is in $\mu_m$.
\end{proof}
We realize the left copy of $G$ using the $2$-cocycle $\sigma_c^{*,rk}$, and the right copy using $\sigma_c$.
Then by \eqref{eq:the $2$-cocycle on G times G formula}, we can lift the embedding $G\times G \hookrightarrow H$ to an embedding (of topological groups)
\begin{align*}
\{(\epsilon_1,\epsilon_2)\in\mu_m^2:\epsilon_1=\epsilon_2\}\backslash G^{(m)}\times G^{(m)} \hookrightarrow H^{(m)}
\end{align*}
($\mu_m^2=\mu_m\times\mu_m$), via
\begin{align}\label{eq:embeddings coverings G and G into H}
\langle g,\epsilon\rangle\mapsto\langle \mathfrak{e}_1(g),\epsilon^{-1}\rangle,\qquad
\langle g,\epsilon\rangle\mapsto\langle \mathfrak{e}_2(g),\epsilon\rangle.
\end{align}
The group action is preserved, and we have
\begin{align}\label{eq:g_1 and g_2 product in H}
\langle\mathfrak{e}_1(g_1),\epsilon_1^{-1}\rangle\langle \mathfrak{e}_2(g_2),\epsilon_2\rangle=\langle (g_1,g_2),\epsilon_1^{-1}\epsilon_2\rangle.
\end{align}

While it is more natural to work with $\sigma_c$ for both copies of $G^{(m)}$, if we realize the left copy
using $\sigma_c$, then by \eqref{eq:sigma c * rk}, the embedding \eqref{eq:embeddings coverings G and G into H} of the left copy changes into
\begin{align}\label{eq:lcl embeding left copy using sigma c into H}
\langle g,\epsilon\rangle\mapsto\langle \mathfrak{e}_1(g),\varsigma_{*,c}^{rk+1}(g)\epsilon^{-1}\rangle,
\end{align}
which complicates matters on the $H^{(m)}$ side. To enjoy the best of both alternatives,
momentarily denote the realization of $G^{(m)}$ using $\sigma_c$ by $G^{(m)}[\sigma_c]$, and similarly denote
$G^{(m)}[\sigma_c^{*,rk}]$. The map
\begin{align}\label{eq:iso sigma sigma * rk}
G^{(m)}[\sigma_c^{*,rk}]\rightarrow G^{(m)}[\sigma_c],\qquad \langle g,\epsilon\rangle\mapsto \langle g,\varsigma_{*,c}^{rk+1}(g)\epsilon\rangle
\end{align}
is an isomorphism, which is canonical in the sense that it is the only isomorphism ($G$ is perfect) which projects to the identity map of $G$. Dualizing, for a function $\varphi$ on
$G^{(m)}[\sigma_c]$ the function $\varphi^{\varsigma_{*,c}^{rk+1}}$ on $G^{(m)}[\sigma_c^{*,rk}]$ is defined by
\begin{align}\label{eq:iso sigma sigma * rk for functions}
\varphi^{\varsigma_{*,c}^{rk+1}}(\langle g,\epsilon\rangle)=\varphi(\langle g,\varsigma_{*,c}^{rk+1}(g)\epsilon\rangle).
\end{align}

\begin{proposition}\label{proposition:local toy integral}
Let $\varphi_1,\varphi_2$ be continuous genuine functions on $G^{(m)}$, realized using $\sigma_c$, and $f$ be a continuous genuine function on $H^{(m)}$.
The integral
\begin{align*}
\int\limits_{G\times G}\varphi_1^{\varsigma_{*,c}^{rk+1}}(\langle g_1,1\rangle)\overline{\varphi_2(\langle g_2,1\rangle)}f(\langle(g_1,g_2),1\rangle)\,dg_1\,dg_2
\end{align*}
is well defined, provided it is absolutely convergent.
\end{proposition}
\begin{proof}
First note that since $\varphi_i(\langle g_i,\epsilon_i\rangle)=\epsilon_i\varphi_i(\langle g_i,1\rangle)$, and \eqref{eq:embeddings coverings G and G into H}--\eqref{eq:g_1 and g_2 product in H} imply
\begin{align*}
f(\langle \mathfrak{e}_1(g_1),\epsilon_1^{-1}\rangle\langle \mathfrak{e}_2(g_2),\epsilon_2\rangle)=\epsilon_1^{-1}\epsilon_2 f(\langle (g_1,g_2),1\rangle),
\end{align*}
the integrand is well defined as a function on $G\times G$. Also
\begin{align*}
&\varphi_1^{\varsigma_{*,c}^{rk+1}}(
\langle g_1,1\rangle\langle h_1,1\rangle)=\varphi_1^{\varsigma_{*,c}^{rk+1}}(\langle
g_1h_1,\sigma_{c}^{*,rk}(g_1,h_1)\rangle),\\
&\varphi_2(\langle g_2,1\rangle\langle h_2,1\rangle)=\varphi_2(\langle
g_2h_2,\sigma_{c}(g_2,h_2)\rangle)
\end{align*}
and
\begin{align*}
&\langle (g_1,g_2),1\rangle\langle (h_1,h_2),1\rangle
=\langle (g_1h_1,g_2h_2),\sigma_{c}^{*,rk}(g_1,h_1)^{-1}\sigma_{c}(g_2,h_2)\rangle.
\end{align*}
Therefore the integral is a right-invariant functional on $G\times G$.
This completes the proof.
\end{proof}
Consider the subgroup $\{(g,g):g\in G\}$ of $H$. By \eqref{eq:the $2$-cocycle on G times G formula} and since $\sigma_c=\sigma^{*,rk}_{c}$ in $\mathrm{H}^2(G,\mu_m)$, the restriction of $H^{(m)}$ to this subgroup is the identity in $\mathrm{H}^2(G,\mu_m)$, hence
$H^{(m)}$ is split over $\{(g,g):g\in G\}$. We determine the splitting.
\begin{corollary}\label{corollary:lcl splitting of (G,G)}
The map $(g,g)\mapsto\langle (g,g),\varsigma_{*,c}^{rk+1}(g)\rangle$ is the splitting
of $\{(g,g):g\in G\}$ in $H^{(m)}$.
\end{corollary}
\begin{proof}
Let $g,g'\in G$. By Proposition~\ref{proposition:the $2$-cocycle on G times G},
\begin{align*}
&\langle(g,g),1\rangle\langle(g',g'),1\rangle=\langle (gg',gg'),\sigma^{*,rk}_{c}(g,g')^{-1}\sigma_{c}(g,g')\rangle.
\end{align*}
Now \eqref{eq:sigma c * rk} implies
\begin{align*}
\varsigma_{*,c}^{rk+1}(g)\varsigma_{*,c}(g')^{rk+1}\sigma^{*,rk}_{c}(g,g')^{-1}
=\varsigma_{*,c}(gg')^{rk+1}\sigma_{c}(g,g')^{-1},
\end{align*}
therefore
\begin{align*}
&\langle(g,g),\varsigma_{*,c}^{rk+1}(g)\rangle\langle(g',g'),\varsigma_{*,c}(g')^{rk+1}\rangle=\langle (gg',gg'),\varsigma_{*,c}(gg')^{rk+1}\rangle.
\end{align*}
The prescribed map is continuous because $\sigma^{*,rk}_{c}\in\mathrm{C}^1(G,\mu_m)$.
\end{proof}

\subsection{Global covering}\label{global covering}
Let $F$ be a number field and $\nu$ be a place of $F$. Let $\sigma_{2rkc,\nu}$ be the local $2$-cocycle of \S~\ref{local covering},
regarded as a $2$-cocycle of $H(F_{\nu})$ by restriction from $\GL_{2rkc}(F_{\nu})$.
Denote $\sigma_{\nu}=\sigma_{2rkc,\nu}$. The product $\sigma=\prod_{\nu}\sigma_{\nu}$ is trivial on $N_{rkc}(\A)$. It is also well defined on $B_{rkc}(\A)$: to see this use \eqref{eq:sigma on vh and h'v'}, \eqref{eq:sigma on torus of GL} and note that for $t\in T_{rkc}(\A)$, at almost all places, all coordinates of $t$ belong to $\mathcal{O}_{\nu}^*$ and $(\mathcal{O}_{\nu}^*,\mathcal{O}_{\nu}^*)_m=1$. While it is well known that $\sigma$ is undefined on $H(\A)$, we can define $\rho\in\mathrm{Z}^2(H(\A),\mu_m)$ such that at any $\nu$, $\rho_{\nu}=\sigma_{\nu}$ in $\mathrm{H}^2(H(F_{\nu}),\mu_m)$; $\rho$ represents $H^{(m)}(\A)$ in $\mathrm{H}^2(H(\A),\mu_m)$ (see \S~\ref{Covering groups}).

To achieve this, first note that for almost all $\nu$ there is a unique splitting $y\mapsto\langle y,\eta_{\nu}(y)\rangle$ of $K_{H,\nu}$, where $\eta_{\nu}=\eta_{2rkc,\nu}$ (see \S~\ref{local covering}). We can extend $\eta_{\nu}$ to an element of $\mathrm{C}^1(H(F_{\nu}),\mu_m)$
(though \eqref{eq:splitting of K H} will no longer hold for arbitrary $y,y'\in H(F_{\nu})$).
Now define, for almost all $\nu$,
\begin{align}\label{eq:nu and sigma for covering of H}
\rho_{\nu}(h,h')=\frac{\eta_{\nu}(h)\eta_{\nu}(h')}{\eta_{\nu}(hh')}\sigma_{\nu}(h,h'),\qquad \forall h,h'\in H(F_{\nu}).
\end{align}
Then $\rho_{\nu}$ is trivial on $K_{H,\nu}$. For the remaining places we can simply take $\rho_{\nu}=\sigma_{\nu}$ and $\eta_{\nu}=1$. Then
$\rho=\prod_{\nu}\rho_{\nu}$ is well defined on $H(\A)$.

Realize $H^{(m)}(\A)$ using $\rho$. We then use the embedding $G(\A)\times G(\A)\hookrightarrow H(\A)$ to realize the $2$-cocycles on the copies of $G(\A)$. Let
\begin{align}\label{gbl rho on left and right}
\rho_L(g,g')=\rho^{-1}(\mathfrak{e}_1(g),\mathfrak{e}_1(g')),\qquad\rho_R(g,g')=\rho(\mathfrak{e}_2(g),\mathfrak{e}_2(g')).
\end{align}
Now the product on the left copy of $G^{(m)}(\A)$ is defined by $\rho_L(g,g')$, i.e.,
\begin{align*}
\langle g,\epsilon\rangle\langle g',\epsilon'\rangle=\langle gg',\epsilon\epsilon'\rho_L(g,g')\rangle,
\end{align*}
and on the right copy by $\rho_R(g,g')$. In addition
\begin{align}\label{eq:gbl commuting G and G}
\langle \mathfrak{e}_1(g_1),1\rangle \langle \mathfrak{e}_2(g_2),1\rangle =
\langle \mathfrak{e}_2(g_2),1\rangle \langle \mathfrak{e}_1(g_1),1\rangle,\qquad\forall g_1,g_2\in G(\A),
\end{align}
because this holds locally by Proposition~\ref{proposition:the $2$-cocycle on G times G} (this does not mean the global version of \eqref{eq:the $2$-cocycle on G times G formula} holds with $\rho$).
Therefore we may lift the embedding $G(\A)\times G(\A) \hookrightarrow H(\A)$ to an embedding
\begin{align*}
\{(\epsilon_1,\epsilon_2)\in\mu_m^2:\epsilon_1=\epsilon_2\}\backslash G^{(m)}(\A)\times G^{(m)}(\A) \hookrightarrow H^{(m)}(\A),
\end{align*}
and we obtain the global analog of \eqref{eq:embeddings coverings G and G into H}, namely,
$\langle g,\epsilon\rangle\mapsto\langle \mathfrak{e}_1(g),\epsilon^{-1}\rangle$ for the left copy and
$\langle g,\epsilon\rangle\mapsto\langle \mathfrak{e}_2(g),\epsilon\rangle$
for the right.
Also by Proposition~\ref{proposition:the $2$-cocycle on G times G},
$\rho_{L,\nu}=\sigma^{*,rk}_{c,\nu}=\sigma_{c,\nu}=\rho_{R,\nu}$ in $\mathrm{H}^2(G(F_{\nu}),\mu_m)$, whence
both copies of $G^{(m)}(\A)$ are cohomologous.

As in the local setting, we would like to explicate the relation between the copies of $G^{(m)}(\A)$, so that we can
work with the same $2$-cocycle for both, yet still use the global embedding \eqref{eq:gbl commuting G and G}.
At any place $\nu$,
\eqref{eq:the $2$-cocycle on G times G formula} implies
\begin{align*}
\sigma_{\nu}(\mathfrak{e}_1(g),\mathfrak{e}_1(g'))=\sigma^{*,rk}_{c,\nu}(g,g')^{-1},\qquad
\sigma_{\nu}(\mathfrak{e}_2(g),\mathfrak{e}_2(g'))=\sigma_{c,\nu}(g,g').
\end{align*}
Whence by \eqref{eq:sigma c * rk},
\begin{align*}
\sigma_{\nu}(\mathfrak{e}_1(g),\mathfrak{e}_1(g'))=\left(\frac{\varsigma_{*,c,\nu}(gg')}{\varsigma_{*,c,\nu}(g)\varsigma_{*,c,\nu}(g')}\right)^{rk+1}\sigma_{\nu}(\mathfrak{e}_2(g),\mathfrak{e}_2(g'))^{-1}.
\end{align*}
Then for all $\nu$,
\begin{align*}
&\rho_{\nu}(\mathfrak{e}_1(g),\mathfrak{e}_1(g'))\\&=\frac{\eta_{\nu}(\mathfrak{e}_1(g))\eta_{\nu}(\mathfrak{e}_1(g'))}{\eta_{\nu}(\mathfrak{e}_1(gg'))}
\sigma_{\nu}(\mathfrak{e}_1(g),\mathfrak{e}_1(g'))\\
&=\frac{\eta_{\nu}(\mathfrak{e}_1(g))\eta_{\nu}(\mathfrak{e}_1(g'))}{\eta_{\nu}(
\mathfrak{e}_1(gg'))}\frac{\eta_{\nu}(\mathfrak{e}_2(g))\eta_{\nu}(\mathfrak{e}_2(g'))}{\eta_{\nu}(\mathfrak{e}_2(gg'))}
\left(\frac{\varsigma_{*,c,\nu}(gg')}{\varsigma_{*,c,\nu}(g)\varsigma_{*,c,\nu}(g')}\right)^{rk+1}\rho_{\nu}(
\mathfrak{e}_2(g),\mathfrak{e}_2(g'))^{-1}.
\end{align*}
Hence if we define $\eta^{\times}_{\nu}\in\mathrm{C}^1(G(F_{\nu}),\mu_m)$ by
\begin{align*}
\eta^{\times}_{\nu}(g)=\eta_{\nu}(\mathfrak{e}_1(g))\eta_{\nu}(\mathfrak{e}_2(g))/\varsigma_{*,c,\nu}^{rk+1}(g),
\end{align*}
\begin{align*}
\rho_{\nu}(\mathfrak{e}_1(g),\mathfrak{e}_1(g'))&=\frac{\eta^{\times}_{\nu}(g)\eta^{\times}_{\nu}(g')}{\eta^{\times}_{\nu}(gg')}\rho_{\nu}(\mathfrak{e}_2(g),\mathfrak{e}_2(g'))^{-1}.
\end{align*}
Now $\eta^{\times}=\prod_{\nu}\eta^{\times}_{\nu}\in\mathrm{C}^1(G(\A),\mu_m)$ is well defined, since for almost all $\nu$,
\begin{align*}
\rho_{\nu}(\mathfrak{e}_1(y),\mathfrak{e}_1(y'))=\rho_{\nu}(\mathfrak{e}_2(y),\mathfrak{e}_2(y'))=1,\qquad\forall y,y'\in K_{G,\nu},
\end{align*}
and then for these places $\eta^{\times}_{\nu}$ becomes a homomorphism of $K_{G,\nu}$.
Thus we deduce the global relation, for all $g,g'\in G(\A)$,
\begin{align}\label{eq:eta times inverse}
\rho(\mathfrak{e}_1(g),\mathfrak{e}_1(g'))=
\frac{\eta^{\times}(g)\eta^{\times}(g')}{\eta^{\times}(gg')}\rho(\mathfrak{e}_2(g),\mathfrak{e}_2(g'))^{-1}.
\end{align}
Then by definition
\begin{align}\label{eq:eta times inverse for rho}
\rho_R(g,g')=\frac{\eta^{\times}(g)\eta^{\times}(g')}{\eta^{\times}(gg')}\rho_L(g,g').
\end{align}
Now we can state the global analogs of \eqref{eq:iso sigma sigma * rk} and \eqref{eq:iso sigma sigma * rk for functions}. First,
if $G^{(m)}(\A)[\rho_R]$ denotes the realization of $G^{(m)}(\A)$ using $\rho_R$ and similarly for $G^{(m)}(\A)[\rho_L]$,
we have the (canonical) isomorphism
\begin{align}\label{eq:gbl rhoR rhoL}
G^{(m)}(\A)[\rho_L]\rightarrow G^{(m)}(\A)[\rho_R],\qquad \langle g,\epsilon\rangle\mapsto \langle g,(\eta^{\times})^{-1}(g)\epsilon\rangle.
\end{align}
Then for a function $\varphi$ on $G^{(m)}(\A)[\rho_R]$, the function $\varphi^{(\eta^{\times})^{-1}}$ on $G^{(m)}(\A)[\rho_L]$ is defined by
\begin{align}\label{eq:gbl iso rhoR rhoL for functions}
\varphi^{(\eta^{\times})^{-1}}(\langle g,\epsilon\rangle)=
\varphi(\langle g,(\eta^{\times})^{-1}(g)\epsilon\rangle).
\end{align}
Note that locally, ignoring the correction using $\eta_{\nu}$,
$(\eta_{\nu}^{\times})^{-1}(g_{\nu})$ becomes $\varsigma_{*,c,\nu}^{rk+1}(g_{\nu})$ from \eqref{eq:iso sigma sigma * rk for functions}.

Next we state the analog of Proposition~\ref{proposition:local toy integral}.
\begin{proposition}\label{proposition:global toy integral}
Let $\varphi_1,\varphi_2$ be continuous genuine functions on $G^{(m)}(\A)$, realized using $\rho_R$, and $f$ be a continuous genuine function on $H^{(m)}(\A)$.
The integral
\begin{align*}
\int\limits_{G(\A)\times G(\A)}\varphi_1^{(\eta^{\times})^{-1}}(\langle g_1,1\rangle)\overline{\varphi_2(\langle g_2,1\rangle)}f(
\langle \mathfrak{e}_1(g_1),1\rangle\langle \mathfrak{e}_2(g_2),1\rangle)\,dg_1\,dg_2
\end{align*}
is well defined, provided it is absolutely convergent.
\end{proposition}
\begin{proof}
The integrand is a well defined function on $G(\A)\times G(\A)$ by the global analog of \eqref{eq:embeddings coverings G and G into H}. To see that it is a right-invariant functional on the domain, observe that by \eqref{gbl rho on left and right} and \eqref{eq:gbl iso rhoR rhoL for functions},
\begin{align*}
&\varphi_1^{(\eta^{\times})^{-1}}(\langle g_1,1\rangle\langle h_1,1\rangle)=\rho_L(g_1,h_1)\varphi_1^{(\eta^{\times})^{-1}}(\langle g_1h_1,1\rangle)=\rho^{-1}(\mathfrak{e}_1(g_1),\mathfrak{e}_1(h_1))\varphi_1^{(\eta^{\times})^{-1}}(\langle g_1h_1,1\rangle),\\
&\varphi_2(\langle g_2,1\rangle\langle h_2,1\rangle)=\rho_R(g_2,h_2)\varphi_2(\langle g_2h_2,1\rangle)
=\rho(\mathfrak{e}_2(g_2),\mathfrak{e}_2(h_2))\varphi_2(\langle 1,g_2h_2\rangle).
\end{align*}
Now although \eqref{eq:g_1 and g_2 product in H} no longer holds,
still by \eqref{eq:gbl commuting G and G},
\begin{align*}
&\langle \mathfrak{e}_1(g_1),1\rangle\langle \mathfrak{e}_2(g_2),1\rangle
\langle \mathfrak{e}_1(h_1),1\rangle\langle \mathfrak{e}_2(h_2),1\rangle
\\&=
\langle \mathfrak{e}_1(g_1),1\rangle\langle \mathfrak{e}_1(h_1),1\rangle\langle \mathfrak{e}_2(g_2),1\rangle
\langle \mathfrak{e}_2(h_2),1\rangle
\\&=
\rho(\mathfrak{e}_1(g_1),\mathfrak{e}_1(h_1))\rho(\mathfrak{e}_2(g_2),\mathfrak{e}_2(h_2))\langle
\mathfrak{e}_1(g_1h_1),1\rangle\langle\mathfrak{e}_2(g_2h_2),1\rangle.
\end{align*}
We see that both $\rho(\mathfrak{e}_i(g_i),\mathfrak{e}_i(h_i))$ are cancelled (we integrate against $\overline{\varphi_2}$).
The result follows.
\end{proof}
Consider now the subgroup $\{(g,g):g\in G\}$ of $H$. Locally the covering is split over this group (see Corollary~\ref{corollary:lcl splitting of (G,G)}), hence it is also split globally. Again, we determine the splitting.
\begin{corollary}\label{corollary:gbl splitting of (G,G)}
The map $(g,g)\mapsto\langle (g,g),(\eta^{\times})^{-1}(g)\rho(\mathfrak{e}_1(g),\mathfrak{e}_2(g))\rangle$ is the splitting
of $\{(g,g):g\in G(\A)\}$ in $H^{(m)}$ (!).
\end{corollary}
\begin{proof}
Let $g,g'\in G(\A)$.
Since
\begin{align*}
\langle(g,g),1\rangle=\langle \mathfrak{e}_1(g),1\rangle \langle\mathfrak{e}_2(g),\rho(\mathfrak{e}_1(g),\mathfrak{e}_2(g))^{-1}\rangle,
\end{align*}
by \eqref{eq:gbl commuting G and G} we have
\begin{align*}
&\langle(g,g),\rho(\mathfrak{e}_1(g),\mathfrak{e}_2(g))\rangle\langle(g',g'),\rho(\mathfrak{e}_1(g'),\mathfrak{e}_2(g'))\rangle\\&=
\langle \mathfrak{e}_1(gg'),\rho(\mathfrak{e}_1(g),\mathfrak{e}_1(g'))\rangle
\langle \mathfrak{e}_2(gg'),\rho(\mathfrak{e}_2(g),\mathfrak{e}_2(g'))\rangle
\\&=\langle (gg',gg'),\rho(\mathfrak{e}_1(gg'),\mathfrak{e}_2(gg'))\rho(\mathfrak{e}_1(g),\mathfrak{e}_1(g'))\rho(\mathfrak{e}_2(g),\mathfrak{e}_2(g'))\rangle.
\end{align*}
Hence by \eqref{eq:eta times inverse},
\begin{align*}
&\langle(g,g),(\eta^{\times})^{-1}(g)\rho(\mathfrak{e}_1(g),\mathfrak{e}_2(g))\rangle\langle(g',g'),(\eta^{\times})^{-1}(g')\rho(\mathfrak{e}_1(g'),\mathfrak{e}_2(g'))\rangle
\\&=\langle (gg',gg'),(\eta^{\times})^{-1}(gg')\rho(\mathfrak{e}_1(gg'),\mathfrak{e}_2(gg'))\rangle.
\end{align*}
The continuity follows since $\eta^{\times}$ and $g\mapsto \rho(\mathfrak{e}_1(g),\mathfrak{e}_2(g))$ belong to $\mathrm{C}^1(G(\A),\mu_m)$.
\end{proof}

As mentioned above $\sigma=\prod_{\nu}\sigma_{\nu}$ is defined on certain subgroups of $H(\A)$, e.g., on $N_{rkc}(\A)$. It is also defined on $H(F)$: indeed for $h,h'\in H(F)$, $\sigma_{\nu}(h_{\nu},h'_{\nu})=1$ for almost all $\nu$, because the local $2$-cocycle is written as a finite product of Hilbert symbols $(x,x')_{m,\nu}$, with elements $x,x'\in F^*$ that are independent of $\nu$, and $(x,x')_{m,\nu}=1$ for almost all $\nu$ (see the proof of \cite[\S~3, Theorem~7]{BLS}, and also \cite [\S~0.2]{KP} and \cite[Proposition~1.7]{Tk}).
Moreover $\eta_{\nu}$ given by \eqref{eq:nu and sigma for covering of H} is trivial on $h_{\nu}$ for almost all $\nu$.
This was shown in \cite[Proposition~1.8]{Tk} and we follow the argument. Write $h=n't'w'n''$ according to the Bruhat decomposition in $H(F)$, i.e., $n',n''\in N_{rkc}(F)$, $t'\in T_{rkc}(F)$ and $w'\in H(F)\cap\mathfrak{W}_{2rkc}^+$ (with $\mathfrak{W}_{2rkc}^+$ defined over $F$). For almost all $\nu$, we have $n'_{\nu},t'_{\nu},w'_{\nu},n''_{\nu}\in K_{G,\nu}$, then by \eqref{eq:Bruhat decomposition sigma}, $\eta_{\nu}(h_{\nu})=1$ because
$\eta_{\nu}(n'_{\nu})=\eta_{\nu}(t'_{\nu})=\eta_{\nu}(w'_{\nu})=\eta_{\nu}(n''_{\nu})=1$ (see \S~\ref{local covering}). Therefore we can define $\eta=\prod_{\nu}\eta_{\nu}$ on $H(F)$, and deduce $\sigma=\rho$ in $\mathrm{H}^2(H(F),\mu_m)$. Furthermore, because of the product formula $\prod_{\nu}(x,x')_{m,\nu}=1$ for any $x,x'\in F^*$, $\sigma$ is in fact trivial on $H(F)$, thus $\langle h,\eta^{-1}(h)\rangle$ is the splitting of $H(F)$ in $H^{(m)}(\A)$ (see \cite[Proposition~1.7]{Tk}).

Since $G(F)\times G(F)<H(F)$, we can use $\eta$ to define splittings of $G(F)$. A direct verification shows that
$g\mapsto\langle g,\eta(\mathfrak{e}_1(g))\rangle$ is the splitting of $G(F)$ in the covering $G^{(m)}(\A)$ realized using $\rho_L$, and
$g\mapsto\langle g,\eta^{-1}(\mathfrak{e}_2(g))\rangle$ is the splitting when the covering is realized via $\rho_R$ (i.e., the right copy).
Having fixed these splittings, we can now consider spaces of automorphic forms on $G^{(m)}(\A)$ and $H^{(m)}(\A)$, which are in particular functions on $G(F)\backslash G^{(m)}(\A)$ and $H(F)\backslash H^{(m)}(\A)$ (resp.). We must specify whether we are considering the left or right copy of $G^{(m)}(\A)$, because the $2$-cocycles differ (up to a $2$-coboundary) and so do the
splittings of $G(F)$ (see \eqref{eq:cohomologous splitting}).
Our next goal is to show that the map \eqref{eq:gbl iso rhoR rhoL for functions} preserves the notion of automorphic forms.

Observe that \eqref{eq:the $2$-cocycle on G times G formula} also implies $\sigma_{\nu}(\mathfrak{e}_1(g),\mathfrak{e}_2(g))=1$ for $g\in G(F_{\nu})$.
Therefore
\begin{align}\label{eq:some local splitting by eta times}
\rho_{\nu}(\mathfrak{e}_1(g),\mathfrak{e}_2(g))=
\frac{\eta_{\nu}(\mathfrak{e}_1(g))\eta_{\nu}(\mathfrak{e}_2(g))}{\eta_{\nu}(\mathfrak{e}_1(g)\mathfrak{e}_2(g))}
=\frac{\eta^{\times}_{\nu}(g)\varsigma_{*,c,\nu}^{rk+1}(g)}{\eta_{\nu}(\mathfrak{e}_1(g)\mathfrak{e}_2(g))},\qquad\forall g\in G(F_{\nu}).
\end{align}
Now $\eta$ is well defined on $\{(g,g):g\in G(F)\}$, because it is defined on $H(F)$; $\rho$ is well defined on $G(F)$ since it is defined on $H(\A)$; and $\eta^{\times}$ is well defined on $G(F)$ because it is defined on $G(\A)$. Thus if $g\in G(F)$, for almost all $\nu$ by \eqref{eq:some local splitting by eta times}, $\varsigma_{*,c,\nu}^{rk+1}(g_{\nu})=1$, so that
$\varsigma_{*,c}^{rk+1}=\prod_{\nu}\varsigma_{*,c,\nu}^{rk+1}$ is well defined on $G(F)$
(but if $m$ does not divide $rk$, $\varsigma_{*,c}$ might not be) and we can write globally
\begin{align}\label{eq:some splitting by eta times}
\rho(\mathfrak{e}_1(g),\mathfrak{e}_2(g))=\frac{\eta^{\times}(g)\varsigma_{*,c}^{rk+1}(g)}{\eta(\mathfrak{e}_1(g)\mathfrak{e}_2(g))},\qquad\forall g\in G(F).
\end{align}
\begin{proposition}\label{proposition:varsigma rk+1 is trivial on G(F)}
The section $\varsigma_{*,c}^{rk+1}$ is trivial on $G(F)$.
\end{proposition}
\begin{proof}
For each $\nu$, raising \eqref{eq:sigma c *} to the power $rk+1$ we have
\begin{align*}
(\sigma^*_{c,\nu})^{rk+1}(g,g')=\left(\frac{\varsigma_{*,c,\nu}(g)\varsigma_{*,c,\nu}(g')}{\varsigma_{*,c,\nu}(gg')}\right)^{rk+1}\sigma_{c,\nu}^{rk+1}(g,g').
\end{align*}
Since $\sigma_c$ is trivial on $G(F)$ and $g\mapsto g^*$ is an involution of $G(F)$, $\sigma_c^*$ is also trivial on $G(F)$. Also
$\varsigma_{*,c}^{rk+1}$ is well defined on $G(F)$. Therefore we can globalize this equality and deduce that
$\varsigma_{*,c}^{rk+1}:G(F)\rightarrow\mu_m$ is a homomorphism, which must be trivial because $G(F)$ is perfect.
\end{proof}
\begin{corollary}\label{corollary:eta times takes automorphic to automorphic}
Let $\varphi_1$ be a continuous genuine function on $G(F)\backslash G^{(m)}(\A)$, where $G^{(m)}(\A)$ is realized using $\rho_R$. Then $\varphi_1^{(\eta^{\times})^{-1}}$ is a similar
function on $G(F)\backslash G^{(m)}(\A)$ (realized using $\rho_L$).
\end{corollary}
\begin{proof}
By Proposition~\ref{proposition:varsigma rk+1 is trivial on G(F)},
$\eta^{\times}(y)=\eta(\mathfrak{e}_1(y))\eta(\mathfrak{e}_2(y))$ for all $y\in G(F)$, hence for any $h\in G^{(m)}(\A)$,
\begin{align*}
\varphi_1^{(\eta^{\times})^{-1}}(\langle y,\eta(\mathfrak{e}_1(y))\rangle h)
=\varphi_1(\langle y,(\eta^{\times})^{-1}(y)\eta(\mathfrak{e}_1(y))\rangle h)
=\varphi_1(\langle y,\eta^{-1}(\mathfrak{e}_2(y))\rangle h)=\varphi_1(h),
\end{align*}
where we used the left-invariance of $\varphi_1$.
\end{proof}
\begin{proposition}\label{proposition:global toy integral automorphic}
Let $\varphi_1,\varphi_2$ be continuous genuine functions on $G(F)\backslash G^{(m)}(\A)$, where $G^{(m)}(\A)$ is realized using $\rho_R$.  Let $f$ be a
continuous genuine function on the image of $G(F)\times G(F)\backslash G^{(m)}(\A)\times G^{(m)}(\A)$ in
$H(F)\backslash H^{(m)}(\A)$, with the above identifications (e.g., $f$ on $H(F)\backslash H^{(m)}(\A)$). Then
\begin{align*}
\int\limits_{G(F)\times G(F)\backslash G(\A)\times G(\A)}\varphi_1^{(\eta^{\times})^{-1}}(\langle g_1,1\rangle)\overline{\varphi_2(\langle g_2,1\rangle)}f(\langle\mathfrak{e}_1(g_1),1\rangle\langle\mathfrak{e}_2(g_2),1\rangle)\,dg_1\,dg_2
\end{align*}
is well defined, provided it is absolutely convergent.
\end{proposition}
\begin{proof}
By Proposition~\ref{proposition:global toy integral}, the integrand is well defined on $G(\A)\times G(\A)$ and provided it is defined on the quotient, it is also right-invariant on $G(\A)\times G(\A)$. It remains to show that the integrand is well defined with respect to the quotient.

Let $y_1,y_2\in G(F)$ and $g_1,g_2\in G(\A)$. Put $\epsilon_i=\eta_i(\mathfrak{e}_i(y_i))$. Then
\begin{align*}
&\varphi_1^{(\eta^{\times})^{-1}}(\langle y_1g_1,1\rangle)\overline{\varphi_2(\langle y_2g_2,1\rangle)}f(
\langle\mathfrak{e}_1(y_1g_1),1\rangle\langle\mathfrak{e}_2(y_2g_2),1\rangle)
\\&=\varphi_1^{(\eta^{\times})^{-1}}(\langle y_1,1\rangle \langle g_1,1\rangle )\overline{\varphi_2(\langle y_2,1\rangle \langle g_2,1\rangle )} f(
\langle\mathfrak{e}_1(y_1),1\rangle\langle\mathfrak{e}_1(g_1),1\rangle\langle\mathfrak{e}_2(y_2),1\rangle\langle\mathfrak{e}_2(g_2),1\rangle)\nonumber
\\&=\varphi_1^{(\eta^{\times})^{-1}}(\langle y_1,\epsilon_1\rangle \langle g_1,1\rangle)
\overline{\varphi_2(\langle y_2,\epsilon_2^{-1}\rangle\langle g_2,1\rangle)}f(
\langle\mathfrak{e}_1(y_1),\epsilon_1^{-1}\rangle\langle\mathfrak{e}_1(g_1),1\rangle\langle\mathfrak{e}_2(y_2),
\epsilon_2^{-1}\rangle\langle\mathfrak{e}_2(g_2),1\rangle)\nonumber
\\&=\varphi_1^{(\eta^{\times})^{-1}}(\langle y_1,\epsilon_1\rangle \langle g_1,1\rangle)
\overline{\varphi_2(\langle y_2,\epsilon_2^{-1}\rangle\langle g_2,1\rangle)}f(
\langle\mathfrak{e}_1(y_1),\epsilon_1^{-1}\rangle\langle\mathfrak{e}_2(y_2),
\epsilon_2^{-1}\rangle\langle\mathfrak{e}_1(g_1),1\rangle\langle\mathfrak{e}_2(g_2),1\rangle)\\
&=\varphi_1^{(\eta^{\times})^{-1}}(\langle g_1,1\rangle)
\overline{\varphi_2(\langle g_2,1\rangle)}f(\langle\mathfrak{e}_1(g_1),1\rangle\langle\mathfrak{e}_2(g_2),1\rangle)\nonumber.
\end{align*}
For the last equality we used the left invariance under $G(F)$ (resp., $H(F)$) of $\varphi_i$ (resp., $f$), and note that this left invariance is with respect to the particular section for each copy of $G(F)$.
\end{proof}
As explained above, on $H(F)$ we have global definitions of $\sigma$ and $\eta$, so that the global analog of
\eqref{eq:nu and sigma for covering of H} is valid, and $h\mapsto\langle h,\eta^{-1}(h)\rangle$ is the splitting of $H(F)$ under $\rho$. A simpler argument applies to $N_{rkc}(\A)$: since $\eta_{\nu}$ is trivial on $N_{rkc}(\mathcal{O}_{\nu})$ for all $\nu$ where $F_{\nu}$ is unramified (see \S~\ref{local covering}), $\eta\in\mathrm{C}^1(N_{rkc}(\A),\mu_m)$, and because $\sigma_{\nu}$ is trivial on $N_{rkc}(F_{\nu})$ for all $\nu$, the global analog of \eqref{eq:nu and sigma for covering of H} holds and $u\mapsto\langle u,\eta^{-1}(u)\rangle$ is the splitting of $N_{rkc}(\A)$ in $H^{(m)}(\A)$.

The following lemma is the extension of \eqref{eq:sigma conjugate v by h} to the global cover.
\begin{lemma}\label{lemma:conjugation of N by H}
Let $h\in H(\A)$, $v\in N_{rkc}(\A)$. If ${}^hv\in N_{rkc}(\A)$,
${}^h\langle v,\eta^{-1}(v)\rangle=\langle {}^hv,\eta^{-1}({}^hv)\rangle$.
\end{lemma}
\begin{proof}
Let $Y<N_{rkc}$ be the unipotent subgroup generated by $v$. Then
$u\mapsto \langle u,\eta^{-1}(u)\rangle$ (the unique splitting of $N_{rkc}(\A)$) is also the unique splitting
of $Y(\A)$. Since ${}^hv\in N_{rkc}(\A)$, we have ${}^hY(\A)<N_{rkc}(\A)$ and hence
${}^hu\mapsto \langle {}^hu,\eta^{-1}({}^hu)\rangle$ is a splitting of ${}^hY(\A)$. Now the result follows from
\eqref{eq:epsilon for conjugation between split subgroups} (with $\chi={}^h$).
\end{proof}

\begin{corollary}\label{corollary:rho and eta on H and N without conjugation}
For any $h\in H(F)$ and $u\in N_{rkc}(\A)$,
\begin{align*}
\langle h,\eta^{-1}(h)\rangle\langle u,\eta^{-1}(u)\rangle=
\langle hu,\eta^{-1}(hu)\rangle,\qquad
\langle u,\eta^{-1}(u)\rangle\langle h,\eta^{-1}(h)\rangle=
\langle uh,\eta^{-1}(uh)\rangle.
\end{align*}
\end{corollary}
\begin{proof}
By \eqref{eq:sigma on h and v}, $\sigma(h'_{\nu},u'_{\nu})=1$ for any $h'\in H(F)$ and $u'\in N_{rkc}(\A)$. Hence \eqref{eq:nu and sigma for covering of H} gives
\begin{align*}
\rho_{\nu}(h'_{\nu},u'_{\nu})=\eta_{\nu}(h'_{\nu})\eta_{\nu}(u'_{\nu})/\eta_{\nu}(h'_{\nu}u'_{\nu}).
\end{align*}
Since the l.h.s.~(left hand side) is $1$ for almost all $\nu$, and so are $\eta_{\nu}(h'_{\nu})$ and $\eta_{\nu}(u'_{\nu})$, we deduce that
$\eta_{\nu}(h'_{\nu}u'_{\nu})=1$ almost everywhere. Therefore $\eta(hu)$ is well defined and we have
\begin{align*}
\rho(h,u)=\frac{\eta(h)\eta(u)}{\eta(hu)},
\end{align*}
proving the first equality. The symmetric argument (for $\rho(u,h)$) implies the second formula.
\end{proof}

\subsection{The lift of the involution ${}^{\iota}$ of $G$}\label{extension of the involution}
For the construction of the integral we will repeatedly use the outer involution ${}^{\iota}$ of $G$, given locally and globally by
$g\mapsto{}^{\iota}g=\iota g\iota^{-1}$ where $\iota=\left(\begin{smallmatrix}&I_{c/2}\\I_{c/2}\end{smallmatrix}\right)$. In this section we discuss its lift to $G^{(m)}$.

First consider the local setting and realize $G^{(m)}$ with $\sigma_c$. Since ${}^{\iota}$ is also a (continuous) involution of $\GL_c$, we can define $\sigma_c^{\iota}\in\mathrm{Z}^2(\GL_c,\mu_m)$ by $\sigma_c^{\iota}(g,g')=\sigma_c({}^{\iota}g,{}^{\iota}g')$ for $g,g'\in\GL_c$.
\begin{proposition}\label{proposition:sigma iota and sigma cohomologous}
We have $\sigma^{\iota}_{c}=\sigma_{c}$ in $\mathrm{H}^2(\SL_c,\mu_m)$, in particular in $\mathrm{H}^2(G,\mu_m)$.
\end{proposition}
\begin{proof}
As in the proof of Proposition~\ref{proposition:sigma * and sigma on SLc}, by
\cite[p.~54, Corollary~2]{Moore} it is enough to consider $t,t'\in T_{\SL_c}$
(by \textit{loc. cit.} restriction $\mathrm{H}^2(G,\mu_m)\rightarrow\mathrm{H}^2(T_n,\mu_m)$ is $2$-to-$1$).
By \eqref{eq:sigma on torus of GL}, $\sigma_c^{\iota}(t,t')$ equals
\begin{align*}
\prod_{i=n+1}^{c-1}\prod_{j=n+1}^i(t_i,t'_j)_m^{-1}
\prod_{i=n}^{c-1}\prod_{j=1}^n(t_i^{-1},t'_j)_m\prod_{i=1}^{n-1}\prod_{j=1}^n(t_i^{-1},t'_j)_m
\prod_{i=1}^{n-1}\prod_{j=i+1}^n(t_i,t'_j)_m
=\prod_{i=1}^{c-1}\prod_{j=1}^i(t_i,t'_j)_m^{-1}=\sigma_c(t,t').
\end{align*}
The result follows.
\end{proof}
Consequently there is $\varsigma_{\iota,c}\in\mathrm{C}^1(\SL_c,\mu_m)$ such that
\begin{align}\label{eq:sigma iota c}
\sigma^{\iota}_{c}(g,g')=\frac{\varsigma_{\iota,c}(g)\varsigma_{\iota,c}(g')}{\varsigma_{\iota,c}(gg')}\sigma_{c}(g,g'),\qquad \forall g,g'\in G.
\end{align}
\begin{proposition}\label{proposition:iota lifts locally to an involution}
The involution ${}^{\iota}$ lifts (uniquely) to an outer involution of $G^{(m)}$, also denoted ${}^{\iota}$, and moreover
\begin{align}\label{eq:iota on the local coverings}
{}^{\iota}\langle g,\epsilon\rangle=\langle {}^{\iota}g,\varsigma_{\iota,c}^{-1}(g)\epsilon\rangle.
\end{align}
\end{proposition}
\begin{proof}
First we show \eqref{eq:iota on the local coverings} is an abstract automorphism of $G^{(m)}$. Indeed
the definition of $\sigma_c^{\iota}$ and \eqref{eq:sigma iota c} imply
\begin{align*}
{}^{\iota}(\langle g,1\rangle\langle g',1\rangle)
&={}^{\iota}\langle gg',\sigma_c(g,g')\rangle=\langle {}^{\iota}(gg'),\varsigma_{\iota,c}^{-1}(gg')\sigma_c(g,g')\rangle
\\&=\langle {}^{\iota}g,\sigma_c({}^{\iota}g,{}^{\iota}g')^{-1}\varsigma_{\iota,c}^{-1}(gg')\sigma_c(g,g')\rangle
\langle {}^{\iota}g',1\rangle
\\&=\langle {}^{\iota}g,\sigma_c^{\iota}(g,g')^{-1}\varsigma_{\iota,c}^{-1}(gg')\sigma_c(g,g')\rangle
\langle {}^{\iota}g',1\rangle
\\&=\langle {}^{\iota}g,\varsigma_{\iota,c}^{-1}(g)\rangle
\langle {}^{\iota}g',\varsigma_{\iota,c}^{-1}(g')\rangle
={}^{\iota}\langle g,1\rangle
{}^{\iota}\langle g',1\rangle.
\end{align*}
Therefore \eqref{eq:iota on the local coverings} is an abstract lift of ${}^{\iota}$ to $G^{(m)}$ and the unique one (see \S~\ref{Covering groups}).

Since $\iota\in\GL_c$, ${}^{\iota}g$ is simply conjugation in $\GL_c$. Our realization of $G^{(m)}$ using restriction from $\sigma_c$ allows us to compute ${}^{\iota}\langle g,\epsilon\rangle=\langle {}^{\iota}g,\epsilon_g\epsilon\rangle$ by regarding $\iota$ and $g$ as elements of $\GL_c$ and using the formulas for $\sigma_c$ as a $2$-cocycle of $\GL_c$. The lift of ${}^{\iota}$ to an involution of $G^{(m)}$ is unique, hence $\epsilon_g=\varsigma_{\iota,c}^{-1}(g)$. This implies \eqref{eq:iota on the local coverings} is continuous, i.e.,
the lift is a topological automorphism.

Since ${}^{\iota}$ is an involution of $G$, the map
\begin{align*}
\langle g,\epsilon\rangle\mapsto{}^{\iota}({}^{\iota}\langle g,\epsilon\rangle)=
\langle g,\varsigma_{\iota,c}^{-1}({}^{\iota}g)\varsigma_{\iota,c}^{-1}(g)\epsilon\rangle
\end{align*}
is the lift of the identity map, which (by uniqueness) coincides with $\langle g,1\rangle\mapsto\langle g,1\rangle$. Hence
\begin{align}\label{eq:varsigma varsigme}
\varsigma_{\iota,c}^{-1}({}^{\iota}g)\varsigma_{\iota,c}^{-1}(g)=1
\end{align}
and
${}^{\iota}$ is a also an involution of $G^{(m)}$.
\end{proof}
\begin{remark}
The proposition applies to $\SL_c^{(m)}$ as well, but this will not be needed.
\end{remark}

The embedding $\mathfrak{e}_2:G\rightarrow H$ was lifted to an embedding $G^{(m)}\hookrightarrow H^{(m)}$ by
\eqref{eq:embeddings coverings G and G into H}. Then we can define ${}^{\iota}$ as an involution of this image of $G^{(m)}$ by
\begin{align}\label{eq:action of local iota on right image of G}
{}^{\iota}\langle\mathfrak{e}_2(g),1\rangle=\langle\mathfrak{e}_2({}^{\iota}g),\varsigma_{\iota,c}^{-1}(g)\rangle.
\end{align}
By Proposition~\ref{proposition:the $2$-cocycle on G times G}, $\mathfrak{e}_2(G)$ and $\mathfrak{e}_1(G)$ commute in $H^{(m)}$. Therefore
\begin{align}\label{eq:iota image commutes with G1}
{}^{\iota}\langle\mathfrak{e}_2(g),1\rangle\langle\mathfrak{e}_1(g'),1\rangle=\langle\mathfrak{e}_1(g'),1\rangle
{}^{\iota}\langle\mathfrak{e}_2(g),1\rangle.
\end{align}
Also by \eqref{eq:the $2$-cocycle on G times G formula} and \eqref{eq:g_1 and g_2 product in H}, we can lift ${}^{\iota}$ uniquely to an involution of the image of $G^{(m)}\times G^{(m)}$ in $H^{(m)}$, which is given by
\begin{align}\label{eq:iota image on product}
{}^{\iota}\langle (g_1,g_2),\epsilon\rangle=\langle (g_1,{}^{\iota}g_2),\varsigma_{\iota,c}(g_2)^{-1}\epsilon\rangle.
\end{align}

When $F$ is unramified, $y\mapsto\langle y,\eta_c(y)\rangle$ is the unique splitting of $K_G$ in $G^{(m)}$.
Since ${}^{\iota}$ is in particular an automorphism of $K_G$, the map ${}^{\iota}y\mapsto\langle {}^{\iota}y,\eta_c({}^{\iota}y)\rangle$ is a splitting of ${}^{\iota}K_G=K_G$, and hence by
\eqref{eq:epsilon for conjugation between split subgroups} (with $Y=K_G$, $\chi={}^{\iota}$),
\begin{align}\label{eq:iota image on K}
{}^{\iota}\langle y,\eta_c(y)\rangle=\langle {}^{\iota}y,\eta_c({}^{\iota}y)\rangle,\qquad
\varsigma_{\iota,c}(y)^{-1}\eta_c(y)=\eta_c({}^{\iota}y),\qquad\forall y\in K_G.
\end{align}

Given a genuine smooth admissible representation $\pi$ of $G^{(m)}$, the representation $\pi^{\iota}$ is defined to be the genuine representation of $G^{(m)}$ acting on the same space as $\pi$, where the action is defined by $\pi^{\iota}(g)=\pi({}^{\iota}g)$.
Observe that since ${}^{\iota}(g^{-1})=({}^{\iota}g)^{-1}$,
\begin{align*}
{}^{\iota}(\langle g,1\rangle^{-1})&=
{}^{\iota}\langle g^{-1},\sigma_c(g,g^{-1})^{-1}\rangle\\&=
\langle {}^{\iota}(g^{-1}),\varsigma_{\iota,c}^{-1}(g^{-1})\sigma_c(g,g^{-1})^{-1}\rangle\\&=
\langle {}^{\iota}g,\varsigma_{\iota,c}(g^{-1})\sigma_c(g,g^{-1})\sigma_c({}^{\iota}g,{}^{\iota}g^{-1})^{-1}\rangle^{-1}\\&=
\langle {}^{\iota}g,\varsigma_{\iota,c}(g^{-1})\sigma_c(g,g^{-1})\sigma_c^{\iota}(g,g^{-1})^{-1}\rangle^{-1}=
\langle {}^{\iota}g,\varsigma_{\iota,c}^{-1}(g)\rangle^{-1}=
({}^{\iota}\langle g,1\rangle)^{-1},
\end{align*}
where we used $\langle ({}^{\iota}g)^{-1},\epsilon^{-1}\rangle=\langle {}^{\iota}g,\epsilon\sigma_c({}^{\iota}g,{}^{\iota}g^{-1})^{-1}\rangle^{-1}$ for the third equality; and \eqref{eq:sigma iota c} with $g'=g^{-1}$ one equality before the last (note that $\varsigma_{\iota,c}(gg^{-1})=\varsigma_{\iota,c}(I_c)=1$).
Since $\pi^{\vee}$ and $(\pi^{\vee})^{\iota}$ act on the same space
and ${}^{\iota}(\langle g,1\rangle^{-1})=({}^{\iota}\langle g,1\rangle)^{-1}$, the definition of $\pi^{\vee}$ implies
$(\pi^{\vee})^{\iota}=(\pi^{\iota})^{\vee}$.
Also when $\pi$ is unramified, so is $\pi^{\iota}$.

Consider the global setting. Recall the $2$-cocycle $\rho_R$ defined by \eqref{gbl rho on left and right}, which we use for the realization  of the right copy of $G^{(m)}(\A)$. Define $\rho_R^{\iota}(g,g')=\rho_R(^{\iota}g,^{\iota}g')$. Since
$\rho_{R,\nu}=\sigma_{c,\nu}=\sigma_{c,\nu}^{\iota}$, we have
$\rho_{R,\nu}^{\iota}=\sigma_{c,\nu}^{\iota}=\rho_{R,\nu}$ (all in $\mathrm{H}^2(G(F_{\nu}),\mu_m)$), thus there is $\eta_{\iota,R}\in\mathrm{C}^1(G(\A),\mu_m)$ such that
\begin{align}\label{eq:eta iota R}
\rho_R^{\iota}(g,g')=\frac{\eta_{\iota,R}(g)\eta_{\iota,R}(g')}{\eta_{\iota,R}(gg')}\rho_R(g,g'),\qquad \forall g,g'\in G(\A).
\end{align}
Repeating the arguments of Proposition~\ref{proposition:iota lifts locally to an involution} (now with \eqref{eq:eta iota R} instead of \eqref{eq:sigma iota c}) we deduce
\begin{align}\label{eq:iota on the gbl coverings}
{}^{\iota}\langle g,\epsilon\rangle=\langle {}^{\iota}g,\eta_{\iota,R}^{-1}(g)\epsilon\rangle \qquad (g\in G(\A)),
\end{align}
is the unique abstract lift of ${}^{\iota}$. It is also a topological lift because the local lifts are topological and
by \eqref{eq:iota image on K}. Hence we can define
\begin{align}\label{eq:action of gbl iota on right image of G}
{}^{\iota}\langle\mathfrak{e}_2(g),1\rangle=\langle\mathfrak{e}_2({}^{\iota}g),\eta_{\iota,R}^{-1}(g)\rangle
\end{align}
(cf. \eqref{eq:iota on the local coverings}  and \eqref{eq:action of local iota on right image of G}).
Then \eqref{eq:iota image commutes with G1} holds globally and we can lift
${}^{\iota}$ (uniquely) to an involution of the image of $G^{(m)}(\A)\times G^{(m)}(\A)$ in $H^{(m)}(\A)$, by defining
\begin{align}\label{eq:iota gbl image on product}
{}^{\iota}(\langle \mathfrak{e}_1(g_1),\epsilon_1\rangle\langle \mathfrak{e}_2(g_2),\epsilon_2\rangle)
=\langle \mathfrak{e}_1(g_1),\epsilon_1\rangle\,{}^{\iota}\langle \mathfrak{e}_2(g_2),\epsilon_2\rangle
=\langle (g_1,{}^{\iota}g_2),\rho(\mathfrak{e}_1(g_1),\mathfrak{e}_2({}^{\iota}g_2))\eta_{\iota,R}^{-1}(g_2)\epsilon_1\epsilon_2\rangle.
\end{align}

We can also extend the local argument on $K_{G,\nu}$ above, to $G(F)$. The unique splitting of $G(F)$ is
$y\mapsto\langle y,\eta^{-1}(\mathfrak{e}_2(y))\rangle$, and since ${}^{\iota}G(F)=G(F)$,
${}^{\iota}y\mapsto\langle {}^{\iota}y,\eta^{-1}(\mathfrak{e}_2({}^{\iota}y))\rangle$ is also a splitting.
Hence by \eqref{eq:epsilon for conjugation between split subgroups},
\begin{align}\label{eq:gbl iota compatible with G(F)}
{}^{\iota}\langle y,\eta^{-1}(\mathfrak{e}_2(y))\rangle=\langle {}^{\iota}y,
\eta^{-1}(\mathfrak{e}_2({}^{\iota}y))\rangle.
\end{align}
Now we can define, for an automorphic function $\varphi$ on the right copy of $G^{(m)}(\A)$, ${}^{\iota}\varphi(g)=\varphi({}^{\iota}g)$, $g\in G^{(m)}(\A)$. By \eqref{eq:gbl iota compatible with G(F)}, this function is still left invariant on
$\{\langle y,\eta^{-1}(\mathfrak{e}_2(y))\rangle:y\in G(F)\}$, hence it is still an automorphic function on $G^{(m)}(\A)$.

\subsection{The covering $\GL_{rkc}^{(m,r)}$}\label{covering of the Levi}
Recall from \S~\ref{embedding} that $m$ is a positive integer and $r=m$ for odd $m$, otherwise $r=m/2$.
In a local or global context, $\widetilde{M}_P$ is the covering obtained by restriction from $H^{(m)}$.
Identify $M_P$ with $\GL_{rkc}$ via $g\mapsto\diag(g,g^*)$. Then $\widetilde{M}_P$ can be regarded as a covering
$\widetilde{\GL}_{rkc}$ of $\GL_{rkc}$. Denote $\GL_{rkc}^{(m,r)}=\widetilde{\GL}_{rkc}$.
Further restricting to $\SL_{rkc}$, we obtain a covering $\widetilde{\SL}_{rkc}$ of the latter.
Locally, by \eqref{eq:BLS block compatible} and Proposition~\ref{proposition:sigma * and sigma on SLc}, $\widetilde{\SL}_{rkc}$ is (topologically isomorphic to) the covering defined in \cite{Mats} with $(\cdot,\cdot)_m^{-2}$, i.e.,
an $r$-fold covering. We see that $\GL_{rkc}^{(m,r)}$ is ``morally" an $r$-fold covering, but since it is $m$ which uniquely determines $r$ and not the other way around, we keep both in the notation.

The group $\GL_{rkc}^{(m,r)}$ is not one of the coverings studied by Kazhdan and Patterson \cite{KP}; it was recently studied in a local context by Savin \cite{Savin7} (see \cite{Gao4}).
In this section we describe several properties of this cover. The description does not depend on
the rank of the general linear group, so we take an integer $d$ and discuss $\GL_{d}^{(m,r)}$, obtained by restriction from $\Sp_{2d}^{(m)}$.

Consider the local setting first.
For brevity denote
\begin{align}\label{eq:sigma square}
\sigma^{\diamondsuit}_{d}(b,b')=\sigma_{2d}(\diag(b,b^*),\diag(b',{b'}^*)),\qquad b,b'\in\GL_d.
\end{align}
For quick reference we rewrite \eqref{eq:BLS $2$-cocycle on torus} for $\sigma^{\diamondsuit}_{d}$,
\begin{align}\label{eq:Nice GL $2$-cocycle on torus}
\sigma^{\diamondsuit}_{d}(\diag(t_1,\ldots,t_{d}),\diag(t_1',\ldots,t_{d}'))=\prod_{i=1}^{d}(t_i,t'_i)_m^{-1}.
\end{align}
The properties \eqref{eq:sigma on h and v}--\eqref{eq:sigma conjugate v by h} remain valid with
$\sigma^{\diamondsuit}_{d}$ instead of $\sigma_{d}$, because the mapping $b\mapsto b^*$ (see \S~\ref{Groups}) is an automorphism of $\GL_d$ which restricts to an automorphism of $N_{\GL_d}$, and for $v\in N_{\GL_d}$, $\diag(v,v^*)\in N_{\GL_{2d}}$. However, more care
is needed with formulas involving $w'\in\mathfrak{W}_d$ (e.g., \eqref{eq:sigma w mathcal and t} and \eqref{eq:sigma t and w mathcal}), because in general it is not true that $\diag(w',{w'}^*)\in\mathfrak{W}_{2d}$
(even if $w'$ represents a simple reflection), only $\diag(w',{w'}^*)\in\mathfrak{W}^+_{2d}$. Note that if $-1\in\mu_m$, we do have ${}^w\langle t,1\rangle=\langle {}^wt,1\rangle$ for $w\in\mathfrak{W}^+_{d}$ (e.g., a permutation matrix) and $t\in T_{\GL_d}$, by Proposition~\ref{proposition:action of W on torus is trivial on Sp}.

According to \eqref{eq:BLS $2$-cocycle on torus}, $\widetilde{T}_{\GL_{d}}$ is a $2$-step nilpotent group and its center is
the preimage of the subgroup of torus elements with coordinates in $F^{*r}$, unless $r=1$, then it is abelian. Denote
$C_{r,d}=\{xI_d:x\in F^{*r}\}$. By \eqref{eq:Nice GL $2$-cocycle on torus}, $t\in C_{r,d}$ commutes with any $t'\in T_{\GL_d}$ in $\GL_{d}^{(m,r)}$, then by \eqref{eq:conj mathcal t in GLd} and since
$(x,-x)_m=1$ for all $x\in F^*$, and also by \eqref{eq:sigma on h and v}--\eqref{eq:sigma conjugate v by h}, $\widetilde{C}_{r,d}$ is the center of $\GL_d^{(m,r)}$ (as opposed to coverings of \cite{KP}, here the parity of $r$ does not play a role).

Let $\beta=(\beta_1,\ldots,\beta_l)$ be a composition of $d$. For $b=\diag(b_1,\ldots,b_l)\in M_{\beta}$ and
$b'\in M_{\beta}$, by \eqref{eq:BLS block compatible} and since $\det b_i^* = \det b_i^{-1}$,
\begin{align*}
\sigma_{2d}(\diag(b,b^*),\diag(b',{b'}^*))=\prod_{i=1}^l\sigma_{2\beta_i}(\diag(b_i,b_i^*),\diag(b_i',{b_i'}^*)),
\end{align*}
which we can write in the form
\begin{align}\label{eq:block compatibility on Levi of P}
\sigma^{\diamondsuit}_{d}(b,b')=\prod_{i=1}^l\sigma^{\diamondsuit}_{\beta_i}(b_i,b_i').
\end{align}
In particular, the direct factors of $M_{\beta}$ commute in $\GL_{d}^{(m,r)}$, which is a special property of this covering, as opposed to the coverings of \cite{KP}. This formula also implies that the embedding $b\mapsto\diag(I_i,b,I_{d-j-i})$ of $\GL_j$ in $\GL_{d}$ induces the same covering on $\GL_j$, i.e.,
$\widetilde{\GL}_{j}=\GL_{j}^{(m,r)}$. Therefore we can study representations induced from parabolic subgroups using the usual tensor product, when we identify
\begin{align}\label{eq:M beta as a quotient}
\widetilde{M}_{\beta}=\{(\epsilon_1,\ldots,\epsilon_l)\in\mu_m^l:\prod_{i=1}^l\epsilon_i=1\}\backslash \GL_{\beta_1}^{(m,r)}\times\ldots\times \GL_{\beta_l}^{(m,r)}.
\end{align}
In particular, we can construct genuine irreducible representations of $\widetilde{T}_{\GL_d}$ by
tensoring $d$ genuine representations of $\GL_1^{(m,r)}$; to construct genuine principal series representations we extend to $\widetilde{B}_{\GL_d}$ by letting the image of $N_{\GL_d}$ act trivially, then induce as usual (see \S~\ref{unramified reps} below).

If $F$ is unramified, the splitting of $K_{\GL_d}$ is chosen to be
$y\mapsto\langle y,\eta^{\diamondsuit}_{d}(y)\rangle$ where $\eta^{\diamondsuit}_{d}(y)=\eta_{2d}(\diag(y,y^*))$. This is compatible with the choice in $\Sp_{2d}$: if a function on $\Sp_{2d}^{(m)}$ is right-invariant on $\{\langle y,\eta_{2d}(y)\rangle:y\in K_{\Sp_{2d}}\}$, its restriction to $\GL_{d}^{(m,r)}$ is right-invariant on
$\{\langle y,\eta^{\diamondsuit}_{d}(y)\rangle:y\in K_{\GL_{d}}\}$.

Using $(b^*)^*=b$ and $(\det{b},\det{b'}^*)_m=(\det{b}^*,\det{b'})_m$ (because $\det{b^*}=\det{b}^{-1}$ and $(x^{-1},y)_m=(x,y^{-1})_m$), \eqref{eq:BLS block compatible} implies $\sigma^{\diamondsuit}_{d}(b^*,{b'}^*)=\sigma^{\diamondsuit}_{d}(b,b')$. Thus
the involution $b\mapsto b^*$ preserves $\sigma^{\diamondsuit}_{d}$ (!), hence lifts to an abstract involution of $\GL_d^{(m,r)}$ by
\begin{align}\label{eq:involution b*0}
{}^*\langle b,\epsilon\rangle =\langle b^*,\epsilon\rangle.
\end{align}
This lift is not unique; since $\Hom(\GL_d,\mu_m)=\Hom(F^*,\mu_m)$, any other abstract lift of ${}^*$ takes the form
\begin{align*}
{}^*\langle b,\epsilon\rangle =\langle b^*,\varrho(\det{b})\epsilon\rangle,
\end{align*}
for some abstract $\varrho\in\Hom(F^*,\mu_m)$ (see \S~\ref{Covering groups}). Since $F^{*m}$ is open in $F^*$, any such $\varrho$ is automatically continuous and we claim all those lifts are topological. Indeed let $w=\left(\begin{smallmatrix}&I_d\\-I_d\end{smallmatrix}\right)\in\Sp_{2d}$, then one of the lifts $\langle b^*,\varrho(\det{b})\epsilon\rangle$ is a homeomorphism, namely the lift corresponding to
${}^*\langle b,\epsilon\rangle={}^w\langle \diag(b,b^*),\epsilon\rangle$. Now for any $\varrho'\in\Hom(F^*,\mu_m)$, the map $b\mapsto{\varrho'}^{-1}(\det b)\varrho(\det b)$ is in $\mathrm{C}^1(\GL_d,\mu_m)$ hence
$\langle b^*,\varrho(\det{b})\epsilon\rangle\rightarrow\langle b^*,\varrho'(\det{b})\epsilon\rangle$ is a homeomorphism, and
therefore $\langle b,\epsilon\rangle\rightarrow\langle b^*,\varrho'(\det{b})\epsilon\rangle$ is a topological lift. In particular \eqref{eq:involution b*0} is topological.

Since we are interested in a lift which is also an involution, we must have $\varrho^2=1$. Hence if $m$ is odd, \eqref{eq:involution b*0} is the
only lift of ${}^*$ to an involution.

Fixing a lift of ${}^*$, we can define for a genuine smooth admissible representation $\pi$ of $\GL_d^{(m,r)}$, the representation $\pi^*$, which acts on the space of $\pi$ by $\pi^*(\langle b,\epsilon\rangle)=\pi({}^*\langle b,\epsilon\rangle)$. If $\pi$ is unramified and $\varrho$ is trivial on $\mathcal{O}^*$, then $\pi^{*}$ is also unramified. Henceforth we only use \eqref{eq:involution b*0}.

We mention that Kable \cite{Kable3} studied the lifts of the main involution for the coverings of \cite{KP};
at least when $-1$ is not a square, there is no $2$-cocycle which is cohomologous to (a twist of) $\sigma_d$, and fixed by any of those lifts (\cite[Proposition~2]{Kable3}).

\begin{proposition}\label{proposition:sigma * and sigma on GLd}
The $2$-cocycles $\sigma^{\diamondsuit}_{d}(b,b')$ and $\sigma^2_{d}(b,b')(\det b,\det b')_m$ are cohomologous.
In particular for $m=2$, $\sigma^{\diamondsuit}_{d}(b,b')$ is cohomologous to the $2$-cocycle given by $(\det b,\det b')_2$.
\end{proposition}
\begin{proof}
First we claim $\sigma_d^*=\sigma_d$ in $\mathrm{H}^2(\GL_d,\mu_m)$. This is \cite[Lemma~2]{Kable3}, but the $2$-cocycle $\tau$ in the notation of \textit{loc. cit.} is not precisely $\sigma_d$. Briefly, as in \cite[\S~4]{Kable3} for $b_0\in\SL_{d+1}$ put
$b_0^{\circ}=\diag(J_k,1){}^tb_0^{-1}\diag(J_k,1)$ and $\sigma_{\SL_{d+1}}^{\circ}(b_0,b_0')=\sigma_{\SL_{d+1}}(b_0^{\circ},{b_0'}^{\circ})$.
The formulas from the proof of Proposition~\ref{proposition:sigma * and sigma on SLc} remain true with
$(\sigma_c,\sigma_c^*,c-1)$ replaced by $(\sigma_{\SL_{d+1}},\sigma_{\SL_{d+1}}^{\circ},d)$, hence $\sigma_{\SL_{d+1}}$ and $\sigma_{\SL_{d+1}}^{\circ}$ are cohomologous and note that $\sigma_d^*(b,b')=\sigma_{\SL_{d+1}}^{\circ}
(\left(\begin{smallmatrix}b\\&\det b^{-1}\end{smallmatrix}\right),
\left(\begin{smallmatrix}b'\\&\det {b'}^{-1}\end{smallmatrix}\right))$.
Now the proposition follows from \eqref{eq:sigma square} and \eqref{eq:BLS block compatible}.
\end{proof}

Consider the global setting. Let $\rho_{2d}\in\mathrm{Z}^2(\Sp_{2d}^{(m)}(\A),\mu_m)$ as in \S~\ref{global covering}.
Then we may define
\begin{align}\label{eq:rho square}
\rho^{\diamondsuit}_{d}(b,b')=\rho_{2d}(\diag(b,b^*),\diag(b',{b'}^*)),\qquad b,b'\in\GL_d(\A).
\end{align}
The covering $\GL_d^{(m,r)}(\A)$ is by definition realized using $\rho^{\diamondsuit}_{d}$. Since $\rho_{2d,\nu}=\sigma_{2d,\nu}$
in $\mathrm{H}^2(\Sp_{2d}(F_{\nu}),\mu_m)$, $\rho^{\diamondsuit}_{d,\nu}=\sigma^{\diamondsuit}_{d,\nu}$ in $\mathrm{H}^2(\GL_d(F_{\nu}),\mu_m)$ and we can write
\begin{align}\label{eq:cohomologous $2$-cocycles on GL}
\rho^{\diamondsuit}_{d,\nu}(b,b')=\frac{\eta^{\diamondsuit}_{d,\nu}(b)\eta^{\diamondsuit}_{d,\nu}(b')}{\eta^{\diamondsuit}_{d,\nu}(bb')}\sigma^{\diamondsuit}_{d,\nu}(b,b'),
\qquad\eta^{\diamondsuit}_{d,\nu}(b)=\eta_{2d,\nu}(\diag(b,b^*))\in\mathrm{C}^1(\GL_d(F_{\nu}),\mu_m)
\end{align}
($\eta_{2d,\nu}$ was used to relate $\rho_{2d,\nu}$ to $\sigma_{2d,\nu}$, see \eqref{eq:nu and sigma for covering of H}).
One can then globalize \eqref{eq:cohomologous $2$-cocycles on GL} on the
subgroups $\GL_d(F)$ and $N_{\GL_d}(\A)$ (see the paragraph before Lemma~\ref{lemma:conjugation of N by H}). We deduce
$b\mapsto\langle b,(\eta^{\diamondsuit}_{d})^{-1}(b)\rangle$ is a splitting of $\GL_d(F)$ (not a perfect group, as opposed to $\Sp_{2d}(F)$), and the splitting of $N_{\GL_d}(\A)$.

For example when $m=2$, by Proposition~\ref{proposition:sigma * and sigma on GLd} and the quadratic reciprocity
$\rho^{\diamondsuit}_d$ is cohomologous to the $2$-cocycle given by $(\det,\det)_2$ (with the global quadratic Hilbert symbol).

Since the direct factors of $M_{\beta}$ commute locally in $\GL_d^{(m,r)}$, they also commute globally. As observed by Takeda
\cite{Tk,Tk2} (for coverings of \cite{KP}), to define a global tensor product we also need to construct a global block-compatible $2$-cocycle. We closely follow his arguments.

Define $\rho_{\beta}\in\mathrm{Z}^2(M_{\beta}(\A),\mu_m)$ by
\begin{align}\label{eq:rho beta}
\rho_{\beta}(b,b')=\prod_{i=1}^l\rho^{\diamondsuit}_{\beta_i}(b_i,b_i'),\qquad b=\diag(b_1,\ldots,b_l).
\end{align}
It is block-compatible by definition. We show $\rho_{\beta}=\rho^{\diamondsuit}_d$ in $\mathrm{H}^2(M_{\beta}(\A),\mu_m)$. Let
\begin{align}\label{eq:eta beta nu}
\eta_{\beta,\nu}\in\mathrm{C}^1(M_{\beta}(F_{\nu}),\mu_m),\qquad \eta_{\beta,\nu}(b)=\frac{\prod_{i=1}^l\eta^{\diamondsuit}_{\beta_i,\nu}(b_i)}{\eta^{\diamondsuit}_{d,\nu}(b)},\qquad b\in M_{\beta}(F_{\nu}).
\end{align}
\begin{proposition}\label{proposition:rho beta and rho are cohomologous locally}
For all $\nu$, $\rho_{\beta,\nu}$ and $\rho^{\diamondsuit}_{d,\nu}$ are cohomologous:
\begin{align}\label{eq:rho beta and rho square locally}
\rho_{\beta,\nu}(m,m')=\frac{\eta_{\beta,\nu}(m)\eta_{\beta,\nu}(m')}{\eta_{\beta,\nu}(mm')}\rho^{\diamondsuit}_{d,\nu}(m,m').
\end{align}
\end{proposition}
\begin{proof}
Indeed, $\rho_{\beta,\nu}(m,m')$ equals
\begin{align*}
&\prod_{i=1}^l\rho^{\diamondsuit}_{\beta_i,\nu}(m_i,m_i')
=\prod_{i=1}^l\frac{\eta^{\diamondsuit}_{\beta_i,\nu}(m_i)\eta^{\diamondsuit}_{\beta_i,\nu}(m_i')}{\eta^{\diamondsuit}_{\beta_i,\nu}(m_im_i')}\sigma^{\diamondsuit}_{\beta_i,\nu}(m_i,m_i')
=\prod_{i=1}^l\frac{\eta^{\diamondsuit}_{\beta_i,\nu}(m_i)\eta^{\diamondsuit}_{\beta_i,\nu}(m_i')}{\eta^{\diamondsuit}_{\beta_i,\nu}(m_im_i')}\sigma^{\diamondsuit}_{d,\nu}(m,m')
\\&=\prod_{i=1}^l\frac{\eta^{\diamondsuit}_{\beta_i,\nu}(m_i)\eta^{\diamondsuit}_{\beta_i,\nu}(m_i')}{\eta^{\diamondsuit}_{\beta_i,\nu}(m_im_i')}
\frac{\eta^{\diamondsuit}_{d,\nu}(mm')}{\eta^{\diamondsuit}_{d,\nu}(m)\eta^{\diamondsuit}_{d,\nu}(m')}\rho^{\diamondsuit}_{d,\nu}(m,m')
=\frac{\eta_{\beta,\nu}(m)\eta_{\beta,\nu}(m')}{\eta_{\beta,\nu}(mm')}\rho^{\diamondsuit}_{d,\nu}(m,m').
\end{align*}
Here for the third equality we used \eqref{eq:block compatibility on Levi of P}.
\end{proof}
\begin{proposition}\label{proposition:gbl rho beta and rho are cohomologous}
The $2$-cocycles $\rho_{\beta}$ and $\rho^{\diamondsuit}_{d}$ are cohomologous.
\end{proposition}
\begin{proof}
By virtue of Proposition~\ref{proposition:rho beta and rho are cohomologous locally}, to deduce $\rho_{\beta}=\rho^{\diamondsuit}_{d}$, we need to show $\eta_{\beta}=\prod_{\nu}\eta_{\beta,\nu}$ is well defined. This follows at once if we prove that for almost all $\nu$,
\begin{align}\label{eq:eta is locally stable}
\eta_{\beta,\nu}(y)=1,\qquad\forall y\in M_{\beta}(\mathcal{O}_{\nu}).
\end{align}
Fix $\nu$ such that $F_{\nu}$ is unramified. Since for each $i$, $\rho^{\diamondsuit}_{\beta_i,\nu}$ is trivial on $K_{\GL_{\beta_i},\nu}$,
so is $\rho_{\beta,\nu}$, thus by \eqref{eq:rho beta and rho square locally}, $\eta_{\beta,\nu}$ is a homomorphism of $M_{\beta}(\mathcal{O}_{\nu})$. It is therefore enough to prove
\eqref{eq:eta is locally stable} for $y\in K_{\GL_{\beta_i},\nu}$, in other words we must show
\begin{align*}
\eta^{\diamondsuit}_{\beta_i,\nu}(y)=\eta^{\diamondsuit}_{d,\nu}(\diag(I_{\sum_{j=1}^{i-1}\beta_j},y,I_{\sum_{j=i+1}^{l}\beta_j})).
\end{align*}
Note that because $K_{\GL_{\beta_i},\nu}$ is not perfect, it is not enough to prove that both sides give rise to a splitting of
$K_{\GL_{\beta_i},\nu}$ with respect to the same $2$-cocycle.

Using the definitions, we can write the last equality in the form
\begin{align}\label{eq:stable eta square 2}
\eta_{2\beta_i,\nu}(\diag(y,y^*))=\eta_{2d,\nu}(\diag(I_{\sum_{j=1}^{i-1}\beta_j},m_i,I_{2\sum_{j=i+1}^{l}\beta_j},m_i^*,I_{\sum_{j=1}^{i-1}\beta_j}).
\end{align}
Consider the embedding $\Sp_{2\beta_i}\hookrightarrow\Sp_{2d}$ given by
\begin{align}\label{eq:embedding Sp in Sp}
x^{\blacksquare}=\diag(I_{\sum_{j=1}^{i-1}\beta_j},
\left(\begin{smallmatrix}x_1&&x_2\\&I_{2\sum_{j=i+1}^{l}\beta_j}\\x_3&&x_4\end{smallmatrix}\right),I_{\sum_{j=1}^{i-1}\beta_j}),\qquad x=\left(\begin{smallmatrix}x_1&x_2\\x_3&x_4\end{smallmatrix}\right),\quad x_i\in\Mat_{\beta_i}.
\end{align}
The image in $\Sp_{2d}$ is a standard group in the sense of \cite[\S~2]{BLS}, hence by \cite[\S~2, Lemma~5]{BLS} (the strong
block-compatibility of the $2$-cocycle on standard subgroups),
\begin{align*}
\sigma_{2\beta_i,\nu}(x,x')=\sigma_{2d,\nu}(x^{\blacksquare},{x'}^{\blacksquare}),\qquad\forall x,x'\in\Sp_{2\beta_i}.
\end{align*}
Since for $x,x'\in K_{\Sp_{2\beta_i},\nu}$,
$\rho_{2\beta_i,\nu}(x,x')=\rho_{2d,\nu}(x^{\blacksquare},{x'}^{\blacksquare})=1$,
equality~\eqref{eq:nu and sigma for covering of H} implies
\begin{align*}
\frac{\eta_{2\beta_i,\nu}(xx')}{\eta_{2\beta_i,\nu}(x)\eta_{2\beta_i,\nu}(x')}
=\sigma_{2\beta_i,\nu}(x,x')=\sigma_{2d,\nu}(x^{\blacksquare},{x'}^{\blacksquare})=
\frac{\eta_{2d,\nu}(x^{\blacksquare}{x'}^{\blacksquare})}{\eta_{2d,\nu}(x^{\blacksquare})
\eta_{2d,\nu}({x'}^{\blacksquare})}.
\end{align*}
Thus we obtain two splittings of $K_{\Sp_{2\beta_i},\nu}$, when the covering is realized using $\sigma_{2\beta_i,\nu}$, and we deduce
$\eta_{2\beta_i,\nu}(x)=\eta_{2d,\nu}(x^{\blacksquare})$ on $K_{\Sp_{2\beta_i},\nu}$, in particular \eqref{eq:stable eta square 2} holds, and thereby
\eqref{eq:eta is locally stable}.
\end{proof}

Now we can define the tensor representation of $\widetilde{M}_{\beta}(\A)$ using the block-compatible global $2$-cocycle $\rho_{\beta}$,
and with the global version of \eqref{eq:M beta as a quotient}. A genuine irreducible admissible representation of $\widetilde{M}_{\beta}(\A)$ can then be written as a tensor of genuine irreducible admissible representations of $\widetilde{M}_{\beta_i}(\A)$. The splitting of $M_{\beta}(F)$
is given by
\begin{align}\label{eq:splitting of M beta F under block compatible $2$-cocycle}
b=\diag(b_1,\ldots,b_l)\mapsto\prod_{i=1}^l\langle b_i,(\eta_{\beta_i}^{\diamondsuit})^{-1}(b_i)\rangle.
\end{align}
Therefore if $\tau_i$ are genuine irreducible automorphic representations of $\GL_{\beta_i}^{(m,r)}(\A)$, the tensor
$\otimes_{i=1}^l\tau_i$ is a genuine irreducible automorphic representation of $\widetilde{M}_{\beta}(\A)$.

Let $\tau_{\beta}=\otimes_{i=1}^l\tau_i$ be a genuine automorphic representation of $\widetilde{M}_{\beta}(\A)$. The space of
\begin{align}\label{rep:gbl parabolic induction on GL}
\Ind_{\widetilde{P}_{\beta}(\A)}^{\GL_{d}^{(m,r)}(\A)}(\tau_{\beta})
\end{align}
is the space of genuine functions $\xi$ on $\GL_{d}^{(m,r)}(\A)$ taking values in the space of $\tau_{\beta}$, such that
\begin{align}\label{rep:gbl parabolic induction on GL properties}
\xi(\langle b,1\rangle \langle v,(\eta_d^{\diamondsuit})^{-1}(v)\rangle g)=\eta_{\beta}^{-1}(b)\delta_{P_{\beta}}^{1/2}(b)\tau_{\beta}(\langle b,1\rangle)\xi(g),
\quad \forall b\in M_{\beta}(\A), v\in V_{\beta}(\A), g\in \GL_{d}^{(m,r)}(\A).
\end{align}
The function $\eta_{\beta}$ is included to compensate for the change of the $2$-cocycle $\rho_d^{\diamondsuit}$ to the (cohomologous) $2$-cocycle $\rho_{\beta}$ on $M_{\beta}(\A)$ (see \cite[p.~204]{Tk}). Given a $\widetilde{K}_{\GL_d}$-finite vector in the space of \eqref{rep:gbl parabolic induction on GL}, as in the linear case we can extend it to a standard section in a complex parameter $\boldsymbol{\zeta}\in\C^l$, i.e., to an element of
\begin{align}\label{rep:gbl parabolic induction on GL with twist}
\Ind_{\widetilde{P}_{\beta}(\A)}^{\GL_{d}^{(m,r)}(\A)}(\otimes_{i=1}^l|\det|^{\boldsymbol{\zeta}_i}\tau_i).
\end{align}
We re-denote the new section by $\xi$, and regard it as a function on $\GL_{d}^{(m,r)}(\A)\times\C^l$. By definition, the section $\xi$ is standard, in the sense that its restriction to $\widetilde{K}_{\GL_d}$ is independent of $\boldsymbol{\zeta}$. We may then consider the Eisenstein series
\begin{align}\label{eq:Eisenstein series on GL}
E(g;\xi,\boldsymbol{\zeta})=\sum\limits_{y\in P_{\beta}(F)\backslash \GL_d(F)}\xi(\langle y,(\eta_{d}^{\diamondsuit})^{-1}(y)\rangle g,\boldsymbol{\zeta}).
\end{align}
To see that the summation is formally well defined, first note that we can globalize \eqref{eq:eta beta nu} to $y\in M_{\beta}(F)$.
Indeed for any integer $j$, $\eta_{2j}$ is defined on $\Sp_{2j}(F)$ hence
each $\eta_{\beta_i}^{\diamondsuit}$ and $\eta_d^{\diamondsuit}$ are defined on
$\GL_{\beta_i}(F)$ and $\GL_{d}(F)$. Thus
$\eta_{\beta}(y)=\prod_{i=1}^l\eta^{\diamondsuit}_{\beta_i}(y_i)/\eta^{\diamondsuit}_{d}(y)$.
Now by \eqref{rep:gbl parabolic induction on GL properties} and \eqref{eq:splitting of M beta F under block compatible $2$-cocycle},
\begin{align*}
\xi(\langle y,(\eta_{d}^{\diamondsuit})^{-1}(y)\rangle)&=
\eta_{\beta}^{-1}(y)\tau_{\beta}(\langle y,(\eta_{d}^{\diamondsuit})^{-1}(y)\rangle )\xi(\langle I_d,1\rangle)\\&
=\tau_{\beta}(\langle y,\prod_{i=1}^l(\eta_{\beta_i}^{\diamondsuit})^{-1}(y_i)\rangle)\xi(\langle I_d,1\rangle)
=\xi(\langle I_d,1\rangle).
\end{align*}
Now the general theory of the Eisenstein series (e.g., \cite{La2,La5}, and \cite{MW2} who also treated
covering groups) implies, among other properties, that the sum is absolutely convergent for $\Real(\boldsymbol{\zeta})$ in a certain cone, and the series admits meromorphic continuation.

\subsection{Unramified representations and $L$-factors}\label{unramified reps}
Let $F$ be unramified (see \S~\ref{local covering}). Recall that $r=m$ if $m$ is odd, otherwise $r=m/2$, $c=2n$ and $G=\Sp_{c}$.
The preimage of the torus $T_n$ in $G^{(m)}$ is a $2$-step nilpotent group
(unless $m\leq2$). Thus its irreducible representations are constructed using
Stone-von Neumann Theory, i.e., by extending a genuine character of the center to a maximal abelian subgroup, then inducing to
$\widetilde{T}_n$ (see e.g., \cite{Savin5} and \cite[\S~13.5--13.6]{McNamara}). The isomorphism class of the representation is determined by the action of the center, i.e., independent of the choice of the maximal abelian subgroup and the extension.

To make this construction uniform in $n$, first note that by \eqref{eq:BLS $2$-cocycle on torus} and \eqref{eq:Nice GL $2$-cocycle on torus},
\begin{align*}
\widetilde{T}_{n}=\{(\epsilon_1,\ldots,\epsilon_n)\in\mu_m^n:\prod_{i=1}^n\epsilon_i=1\}\backslash \GL_{1}^{(m,r)}\times\ldots\times\GL_{1}^{(m,r)},
\end{align*}
where on $\GL_{1}^{(m,r)}$ the group operation is given by
\begin{align*}
\langle x,1\rangle\langle x',1\rangle=\langle xx',(x,x')_m^{-1}\rangle.
\end{align*}
Thus to construct a genuine irreducible representation of $\widetilde{T}_{n}$, one may tensor $n$ genuine irreducible representations of
$\GL_{1}^{(m,r)}$. Consider the $i$-th copy of $\GL_{1}^{(m,r)}$. The center of this group is $\widetilde{C}_{r,1}$, which is the preimage of $F^{*r}$. Let $A=F^{*r}\mathcal{O}^*$, $\widetilde{A}$ is a maximal abelian subgroup of $\GL_{1}^{(m,r)}$. A genuine irreducible representation of $\widetilde{A}$ is called unramified if it is trivial on the preimage of $\mathcal{O}^*$. Let $\mu_i$ be an unramified quasi-character of $F^*$.

If $r=m$, $(\cdot,\cdot)_m$ is trivial on $A\times A$, hence we may form the genuine character of
$\widetilde{A}$ by $\varepsilon\otimes\mu_i(\langle x,\epsilon\rangle)=\varepsilon(\epsilon)\mu_i(x)$.
Then we obtain a genuine irreducible
unramified representation of $\GL_{1}^{(m,r)}$ by (non-normalized) induction from $\widetilde{A}$ and $\varepsilon\otimes\mu_i$.

If $r=m/2$ the definition depends on a choice of a root of unity, as we now explain. First we recall the definition of the Weil factor.
For a nontrivial additive character $\psi'$ of $F$ and $a\in F^*$, let $\psi'_a(x)=\psi'(ax)$. The Weil index of $x\mapsto\psi'(x^2)$ is denoted $\gamma(\psi')$, and the Weil factor of $\psi'$ is $\gamma_{\psi'}(a)=\gamma(\psi_a')/\gamma(\psi')$ (see \cite[p.~176]{We} and \cite{Rao}). The following formulas are well known (see e.g., \cite[Appendix]{Rao}):
\begin{align}\label{eq:some props of gamma psi'}
\gamma_{\psi'}^4=1,\quad\gamma_{\psi'}(a^2)=1,\quad \gamma_{\psi'}(ab)=\gamma_{\psi'}(a)\gamma_{\psi'}(b)(a,b)_2,\quad \gamma_{\psi_a'}(x)=(a,x)\gamma_{\psi'}(x).
\end{align}
In particular there are $[F^*:F^{*2}]=4$ (e.g., \cite[p.~32]{We2})
choices for $\gamma_{\psi'}$, and only two are unramified, i.e., trivial on $\mathcal{O}^*$.
Also recall the formula $(\cdot,\cdot)_m^{r}=(\cdot,\cdot)_2$ relating the $m$-th Hilbert symbol to the quadratic one (when $r=m/2$).
% https://mathoverflow.net/questions/211645/computation-of-hilbert-symbol-of-order-4/211650

Now if $r=m/2$ is odd, $(x^r,y^r)_m=(x^r,y^r)_2$, so that $\varepsilon\otimes\gamma_{\psi'}\mu_i$ defines a genuine character of
$\widetilde{C}_{r,1}$. Moreover, since $(x^r,y)_m=(x,y)_2=(x^r,y)_2$ and both $m$-th and quadratic Hilbert symbols are trivial
on $\mathcal{O}^*\times\mathcal{O}^*$, $\varepsilon\otimes\gamma_{\psi'}\mu_i$ is also a genuine character of $\widetilde{A}$. We take $\psi'$ such that $\gamma_{\psi'}$ is unramified, this choice is unique up to $\pm\gamma_{\psi'}(\varpi^r)$ ($\gamma_{\psi'}(\varpi^r)$ is determined up to a sign, and this determines $\gamma_{\psi'}$ on $A$). Then $\varepsilon\otimes\gamma_{\psi'}\mu_i$ is a genuine unramified character of $\widetilde{A}$, and we induce to $\GL_{1}^{(m,r)}$ as in the case $r=m$ above. For $m\leq2$, $r=1$ and $A=F^*$, so the induction is trivial.

Moreover, for this case (even $m$, odd $r$) since any genuine character of $\GL_1^{(2,1)}$ can be written in the form $\varepsilon\otimes\gamma_{\psi'}\mu_i'$ for some $\psi'$ and quasi-character $\mu_i'$ of $F^*$, we can write
any genuine character of $\widetilde{C}_{r,1}$ in the form $\varepsilon\otimes\gamma_{\psi'}\mu_i'$. Since $x\mapsto(a,x)_2$ is a non-genuine character of $\widetilde{C}_{r,1}$, which is unramified if $|a|=1$, we may a priori fix an unramified $\gamma_{\psi'}$ (in one of two ways), then write any genuine unramified character of $\widetilde{C}_{r,1}$ as $\varepsilon\otimes\gamma_{\psi'}\mu_i'$ for some unramified $\mu_i'$. In addition, we will not lose any generality by tensoring $n$ genuine irreducible representations of $\GL_{1}^{(m,r)}$ using the same $\psi'$, so that when $\psi'$ is fixed, we can effectively parameterize the genuine unramified principal series representations using linear data.

Now consider the case of even $r$. Then $(x^r,y^r)_m=(x^r,y)_2=1$, so that
$\varepsilon\otimes\mu_i$ defines a genuine character of $\widetilde{C}_{r,1}$.
For a fixed element $\varpi_1\in\mathcal{O}$ with $|\varpi_1|=q^{-1}$, we can write any $a\in A$ uniquely
in the form $\varpi_1^{rl}u$ with an integer $l$ and $|u|=1$.
Define
\begin{align*}
\varepsilon\otimes\gamma_{\varpi_1}(\langle \varpi_1^{rl}u,\epsilon\rangle)=\varepsilon(\epsilon\cdot(\varpi_1^l,u)_2).
\end{align*}
The definition is independent of the choice of $\varpi_1$, because if $\varpi=\varpi_1o$ with $|o|=1$, $\varpi_1^{rl}u=\varpi^{rl}o^{rl}u$
then $(\varpi_1^l,u)_2=(\varpi^l,o^{rl}u)_2$, since $r$ is even.
This is a genuine character of $\widetilde{A}$, because
\begin{align*}
&\varepsilon\otimes\gamma_{\varpi_1}(\langle \varpi_1^{rl_1}u_1,1\rangle\langle \varpi_1^{rl_2}u_2,1\rangle)
=\varepsilon((\varpi_1^{l_1r}u_1,\varpi_1^{l_2r}u_2)_m^{-1}(\varpi_1^{(l_1+l_2)},u_1u_2)_2)
\\&=\varepsilon((\varpi_1^{l_1},u_2)_2(u_1,\varpi_1^{l_2})_2(\varpi_1^{l_1},u_1)_2
(\varpi_1^{l_1},u_2)_2(\varpi_1^{l_2},u_1)_2(\varpi_1^{l_2},u_2)_2)
\\&=\varepsilon((\varpi_1^{l_1},u_1)_2(\varpi_1^{l_2},u_2)_2)=
\varepsilon\otimes\gamma_{\varpi_1}(\langle \varpi_1^{rl_1}u_1,1\rangle)\varepsilon\otimes\gamma_{\varpi_1}(\langle \varpi_1^{rl_2}u_2,1\rangle).
\end{align*}
It is also unramified and trivial on $\{\langle x,1\rangle:x\in F^{*r}\}$.
Thus we can extend
$\varepsilon\otimes\mu_i$ to a genuine unramified character of $\widetilde{A}$ by
\begin{align*}
\varepsilon\otimes\gamma_{\varpi_1}\mu_i(\langle a,\epsilon\rangle)=\varepsilon(\epsilon)\mu_i(a)\gamma_{\varpi_1}(\langle a,1\rangle).
\end{align*}
(We do not claim this extension is unique.)

To unify the notation, denote the representation of $\widetilde{A}$ by
$\varepsilon\otimes\vartheta\mu_i$. Specifically, $\vartheta=1$ if $m$ is odd; $\vartheta=\gamma_{\psi'}$ if $m\equiv2\,(4)$ ($m$ is even, $r$ is odd), in which case the isomorphism class of the induced representation of $\GL_{1}^{(m,r)}$ depends on $\psi'$; and $\vartheta=\gamma_{\varpi_1}$ when $m\equiv0\,(4)$, but then the isomorphism class of the induced representation of $\GL_{1}^{(m,r)}$ does not depend on the choice of extension of $\varepsilon\otimes\mu_i$ to $\widetilde{A}$ and in particular, it is independent of $\varpi_1$.

The value $\varepsilon\otimes\vartheta\mu_i(\langle \varpi^r,1\rangle)$ is independent of the choice of uniformizer $\varpi$ except when
$m\equiv2\,(4)$, in which case it depends both on this choice and on $\psi'$, but only up to a sign $(\varpi,u)_2=\pm1$
where $|u|=1$ (use \eqref{eq:some props of gamma psi'}).

Denote the tensor of $n$ such representations, with unramified quasi-characters $\mu_1,\ldots,\mu_n$ of $F^*$, by
$\varepsilon\otimes\vartheta\mu$ (we use the same $\vartheta$ for all $i$). Let $T_{n,*}=\{t\in T_n:t_i\in A,\forall i\}$, $\widetilde{T}_{n,*}$ is a maximal abelian subgroup of $\widetilde{T}_n$. Then
\begin{align*}
\varepsilon\otimes\vartheta\mu(\langle t,\epsilon\rangle)=\varepsilon(\epsilon)\prod_{i=1}^n\vartheta(t_i)\mu_i(t_i),\qquad\forall t\in T_{n,*}.
\end{align*}
The restriction of $\varepsilon\otimes\vartheta\mu$ to the center of $\widetilde{T}_n$ exhausts its genuine irreducible unramified representations, therefore when we induce to $\widetilde{T}_n$ we obtain all such representations of $\widetilde{T}_n$. Now extending to $\widetilde{B}_n$ then inducing to $G^{(m)}$, we obtain the genuine unramified principal series representation
\begin{align*}
\mathrm{I}_{G^{(m)}}(\vartheta,\mu)=\Ind_{\widetilde{B}_{n}}^{G^{(m)}}(\Ind_{\widetilde{T}_{n,*}}^{\widetilde{T}_n}
(\varepsilon\otimes\vartheta\mu)).
\end{align*}
(Only the outer induction is normalized.) We usually regard elements in the space of this representation as complex-valued functions, by evaluating at the identity.

The dual group of a covering group has been defined and studied in several works,
e.g., \cite{Savin5,McNamara,Weissman2,GG,Weissman2018a}. For $G^{(m)}$, the dual group ${G^{(m)}}^{\vee}$ is
$\SO_{2n+1}(\C)$ when $m$ is odd, and $\Sp_{2n}(\C)$ if it is even (see e.g., \cite[\S~5.1]{WWLi2017}).
If $\pi$ is the irreducible unramified constituent of $\mathrm{I}_{G^{(m)}}(\vartheta,\mu)$, the Satake parameter of $\pi$ is the semi-simple conjugacy class in ${G^{(m)}}^{\vee}$ of
\begin{align}\label{eq:Satake symplectic}
t_{\pi,\vartheta}=
\begin{cases}
\diag(\mu_1(\varpi^r),\ldots,\mu_n(\varpi^r),1,\mu_n^{-1}(\varpi^r),\ldots,\mu_1^{-1}(\varpi^r))&r=m,\\
\diag(\mu_1(\varpi^r),\ldots,\mu_n(\varpi^r),\mu_n^{-1}(\varpi^r),\ldots,\mu_1^{-1}(\varpi^r))&r=m/2.
\end{cases}
\end{align}
The parameter $t_{\pi,\vartheta}$ depends on $\vartheta$ only when $m\equiv2\,(4)$, since only in this case $\vartheta$ affects the choice of
$\mu$. The $L$-function of $\pi$ is now defined by
\begin{align*}
L_{\vartheta}(s,\pi)=\det(1-t_{\pi,\vartheta}q^{-s})^{-1}.
\end{align*}
Again, this function depends on $\vartheta$ only when $m\equiv2\,(4)$.

For the covering $\GL_d^{(m,r)}$ the situation is similar in light of \eqref{eq:M beta as a quotient} and \eqref{eq:Nice GL $2$-cocycle on torus}, with $T_{\GL_d}$ instead of $T_n$. The maximal abelian subgroup $\widetilde{T}_{\GL_d,*}$ of $\widetilde{T}_{\GL_d}$ is the preimage of $d$ copies of $A$.
For unramified quasi-characters $\chi_1,\ldots,\chi_d$ of $F^*$ we denote $\chi=\otimes_{i=1}^d\chi_i$, then the genuine unramified principal series representation is
\begin{align*}
\mathrm{I}_{\GL_d^{(m,r)}}(\vartheta,\chi)=\Ind_{\widetilde{B}_{\GL_d}}^{\GL_d^{(m,r)}}(\Ind_{\widetilde{T}_{\GL_d,*}}^{\widetilde{T}_{\GL_d}}
(\varepsilon\otimes\vartheta\chi)).
\end{align*}

The following argument regarding the support of unramified functions in the space of an unramified principal series appeared in the proof of \cite[Lemma~2]{McNamara}, which generalized \cite[Lemma~I.1.3]{KP}.
\begin{proposition}\label{proposition:support of unr is in T*}
Let $\xi$ be an unramified (nonzero) element in $\mathrm{I}_{\GL_d^{(m,r)}}(\vartheta,\chi)$. Write $b\in\GL_d$ in the form $b=uty$ where $u\in N_{\GL_d}$, $t\in T_{\GL_d}$ and $y\in K_{\GL_d}$. Then
$\xi(\langle b,\epsilon\rangle)\ne0$ if and only if $t\in T_{\GL_d,*}$.
\end{proposition}
\begin{proof}
By the definition of the space of $\mathrm{I}_{\GL_d^{(m,r)}}(\vartheta,\chi)$, $\xi(\langle uty,\epsilon\rangle)\ne0$ if and only if $\xi(\langle t,1\rangle)\ne0$. If $t\notin T_{\GL_d,*}$, choose $t_x=\diag(I_{i-1},x,I_{n-i})$ with $x\in\mathcal{O}^*$ such that $(x,t_i)^2_m\ne1$. Then $t_x\in T_{\GL_d,*}\cap K_{\GL_d}$ and $\eta_d^{\diamondsuit}(t_x)=1$, because $\eta_{2d}$ is trivial on
$T_{\GL_{2d}}\cap K_{\GL_{2d}}$. Thus $\langle t_x,1\rangle\in\eta_d^{\diamondsuit}(T_{\GL_d,*}\cap K_{\GL_d})$ whence $(\varepsilon\otimes\vartheta\chi)(\langle t_x,1\rangle)=1$. Now by \eqref{eq:Nice GL $2$-cocycle on torus},
\begin{align*}
\xi(\langle t,1\rangle)=\xi(\langle t,1\rangle\langle t_x,1\rangle)
=(t_i,x)_m^{-2}\xi(\langle t,1\rangle),
\end{align*}
hence $\xi(\langle t,1\rangle)=0$. Thus $t\in T_{\GL_d,*}$ and then
$\xi(\langle t,1\rangle)=\delta_{B_{\GL_d}}^{1/2}(t)(\varepsilon\otimes\vartheta\chi)(\langle t,1\rangle)\xi(\langle I_{d},1\rangle)\ne0$.
\end{proof}

The dual group ${\GL_d^{(m,r)}}^{\vee}$ is $\GL_d(\C)$ (\cite[\S~2.3]{Gao4}).
If $\tau$ is the irreducible unramified constituent of $\mathrm{I}_{\GL_d^{(m,r)}}(\vartheta,\chi)$, the Satake parameter
of $\tau$ is the semi-simple conjugacy class in ${\GL_d^{(m,r)}}^{\vee}$ of
\begin{align*}
t_{\tau,\vartheta}=\diag(\chi_1(\varpi^r),\ldots,\chi_d(\varpi^r)).
\end{align*}
Then for any finite-dimensional complex representation $\sigma$ of
$\GL_d(\C)$, define
\begin{align*}
L_{\vartheta}(s,\tau,\sigma)=\det(1-\sigma(t_{\tau,\vartheta})q^{-s})^{-1}.
\end{align*}
For the identity representation $\mathrm{id}:\GL_d(\C)\rightarrow\GL_d(\C)$, denote
$L_{\vartheta}(s,\tau)=L_{\vartheta}(s,\tau,\mathrm{id})$.
In the linear case this is the standard $L$-function of $\tau$.
If $\tau'$ is a similar constituent of $\mathrm{I}_{\GL_{d'}^{(m,r)}}(\vartheta',\chi')$,
$L_{\vartheta,\vartheta'}(s,\tau\times\tau')=\det(1-(t_{\tau,\vartheta}\otimes t_{\tau',\vartheta'})q^{-s})^{-1}$.
In particular the definitions imply
\begin{align}\label{eq:standard sym and ext}
L_{\vartheta,\vartheta}(s,\tau\times\tau)=L_{\vartheta}(s,\tau,\mathrm{Sym}^2)
L_{\vartheta}(s,\tau,\wedge^2),
\end{align}
where $\mathrm{\Sym}^2$ is the symmetric square and $\wedge^2$ is the exterior square representation.

Finally assume $\pi$ and $\tau$ are the irreducible unramified constituents of $\mathrm{I}_{G^{(m)}}(\vartheta_{\pi},\mu)$ and
$\mathrm{I}_{\GL_d^{(m,r)}}(\vartheta_{\tau},\chi)$.
Put $n'=2n+1$ if $r=m$, otherwise $n'=2n$, and regard $t_{\pi,\vartheta_{\pi}}$ as an element of $\GL_{n'}(\C)$ (see \eqref{eq:Satake symplectic}). Using the definition of the $L$-factor for $\GL_{n'}\times\GL_d$,
\begin{align*}
L_{\vartheta_{\pi},\vartheta_{\tau}}(s,\pi\times\tau)=\det(1-(t_{\pi,\vartheta_{\pi}}\otimes t_{\tau,\vartheta_{\tau}})q^{-s})^{-1}.
\end{align*}

We extend a known identity from the linear set-up, relating the $L$-function of $\pi\times\tau$
to the product of $L$-functions for representations of $\GL_n\times\GL_k$.
One can always write $\pi$ as the irreducible unramified constituent of $\Ind_{\widetilde{R}}^{G^{(m)}}(\pi_n)$,
where $\pi_n$ is the irreducible unramified constituent of $\mathrm{I}_{\GL_n^{(m,r)}}(\vartheta_{\pi},\mu)$ and $R$ is a Siegel parabolic subgroup of $G$. The representation $\pi_n^{\vee}$ (which is anti-genuine) is the irreducible unramified constituent of the unramified principal series with inducing data $\otimes_{i=1}^d\varepsilon^{-1}\otimes\vartheta_{\pi}^{-1}\mu^{-1}$, and we can define
$L_{\vartheta_{\pi}^{-1},\vartheta_{\tau}}(s,\pi_n^{\vee}\times\tau)$ as above. By
definition
\begin{align}\label{eq:standard L of pi and tau identity}
L_{\vartheta_{\pi},\vartheta_{\tau}}(s,\pi\times\tau)=
L_{\vartheta_{\pi},\vartheta_{\tau}}(s,\pi_n\times\tau)[L_{\vartheta_{\tau}}(s,\tau)]L_{\vartheta_{\pi}^{-1},\vartheta_{\tau}}(s,\pi_n^{\vee}\times\tau),
\end{align}
where the factor in square brackets appears only when $r=m$ or equivalently, $m$ is odd. Note that we have to
parameterize $\pi_n^{\vee}$ using $\vartheta_{\pi}^{-1}$ (when $m\equiv2\,(4)$) to deduce
\eqref{eq:standard L of pi and tau identity}.

Assume $m\equiv2\,(4)$ (the following remarks becomes trivial otherwise). Since we will usually be interested in the $L$-function
of $\pi\times\tau$, it is convenient to use one parametrization parameter for both, and we do not lose any generality by doing that. For any two parameters $\vartheta_{\pi}$ and $\vartheta_{\tau}$, $\vartheta_{\pi}\vartheta_{\tau}^{-1}=\vartheta$ for some quadratic unramified character $\vartheta$ of $F^*$. Now $t_{\tau,\vartheta_{\tau}}=t_{\tau,\vartheta_{\pi}\vartheta}$, and since
$\mathrm{I}_{\GL_d^{(m,r)}}(\vartheta_{\pi}\vartheta,\chi)=\mathrm{I}_{\GL_d^{(m,r)}}(\vartheta_{\pi},\vartheta\chi)$, we also have
$t_{\tau,\vartheta_{\tau}}=t_{\vartheta\tau,\vartheta_{\pi}}$, so that
$L_{\vartheta_{\pi},\vartheta_{\tau}}(s,\pi\times\tau)=
L_{\vartheta_{\pi},\vartheta_{\pi}}(s,\pi\times\vartheta\tau)$.
For brevity, denote $L_{\vartheta_{\pi}}(\cdots)=L_{\vartheta_{\pi},\vartheta_{\pi}}(\cdots)$.
In particular when $\mu_{2m}\subset F^*$, $(x,x)_2=1$ for all $x\in F^*$ and $\gamma_{\psi'}=\gamma_{\psi'}^{-1}$ (use \eqref{eq:some props of gamma psi'}).
Hence $\vartheta_{\pi}=\vartheta_{\pi}^{-1}$ and \eqref{eq:standard L of pi and tau identity} becomes
\begin{align*}
L_{\vartheta_{\pi}}(s,\pi\times\vartheta\tau)=
L_{\vartheta_{\pi}}(s,\pi_n\times\vartheta\tau)[L_{\vartheta_{\tau}}(s,\vartheta\tau)]L_{\vartheta_{\pi}}(s,\pi_n^{\vee}\times\vartheta\tau).
\end{align*}
\begin{example}
Since $\mathrm{I}_{G^{(m)}}(\vartheta_{\pi},\mu)=\mathrm{I}_{G^{(m)}}(\vartheta_{\pi}\vartheta,\vartheta\mu)$, for $n=1$ we have
$t_{\pi,\vartheta_{\pi}}=\diag(\mu_1(\varpi),\mu_1^{-1}(\varpi))$ while
$t_{\pi,\vartheta\vartheta_{\pi}}=\diag(\vartheta\mu_1(\varpi),\vartheta\mu_1^{-1}(\varpi))$.
\end{example}
\begin{proposition}\label{proposition:vartheta of *}
With $\tau$ as above,
$\tau^*$ defined by \eqref{eq:involution b*0} is the
irreducible unramified constituent of $\mathrm{I}_{\GL_d^{(m,r)}}(\vartheta_{\tau},\chi^{-1})$. Consequently
$t_{\tau^{\vee},\vartheta_{\tau}^{-1}}=t_{\tau^*,\vartheta_{\tau}}$.
\end{proposition}
\begin{proof}
By definition $\tau^*$ is realized in the space of $\tau$, and $\tau^*$ is irreducible and unramified. Hence to prove that
$\tau^*$ is the irreducible unramified constituent of $\mathrm{I}_{\GL_d^{(m,r)}}(\vartheta_{\tau},\chi^{-1})$, it suffices to
show $(\mathrm{I}_{\GL_d^{(m,r)}}(\vartheta_{\tau},\chi))^*=\mathrm{I}_{\GL_d^{(m,r)}}(\vartheta_{\tau},\chi^{-1})$. The representation
$(\mathrm{I}_{\GL_d^{(m,r)}}(\vartheta_{\tau},\chi))^*$ is an unramified principal series, and we need to check that
for all $a\in A$, $\vartheta_{\tau}({}^*\langle a,1\rangle)=\vartheta_{\tau}(\langle a,1\rangle)$. Note that
${}^*\langle a,1\rangle=\langle a^{-1},1\rangle$.

This is clear if $r=m$. If $m\equiv2\,(4)$, $\vartheta_{\tau}=\gamma_{\psi'}$ and by
\eqref{eq:some props of gamma psi'}, $\gamma_{\psi}(a)=\gamma_{\psi}(a^{-1})$.
For the remaining case $\vartheta_{\tau}=\gamma_{\varpi_1}$, if we write $a=\varpi_1^{rl}u$ as above,
\begin{align*}
\gamma_{\varpi_1}(\langle a^{-1},1\rangle)=\varepsilon((\varpi_1^{-l},u^{-1})_2)=\varepsilon((\varpi_1^l,u)_2)=\gamma_{\varpi_1}(\langle a,1\rangle).
\end{align*}
Now $t_{\tau^{\vee},\vartheta_{\tau}^{-1}}=t_{\tau^*,\vartheta_{\tau}}$ is immediate because the inducing data for $\tau^{\vee}$ is
$\otimes_{i=1}^d\varepsilon^{-1}\otimes\vartheta_{\tau}^{-1}\chi^{-1}$.
\end{proof}

\subsection{Whittaker functionals}\label{whittaker functionals}
Assume $F$ is a local non-archimedean field. Recall that $\psi$ is a nontrivial character of $F$.
Define a generic character $\psi$ of $N_{\GL_d}$ by
\begin{align*}
\psi(v)=\psi(\sum_{i=1}^{d-1}v_{i,i+1}).
\end{align*}
Consider a genuine unramified representation $\mathrm{I}(\vartheta,\chi)=\mathrm{I}_{\GL_{d}^{(m,r)}}(\vartheta,\chi)$ with the notation of \S~\ref{unramified reps}. A Whittaker functional on $\mathrm{I}(\vartheta,\chi)$ is a (nonzero) morphism in
\begin{align}\label{Hom:Whittaker}
\Hom_{N_{\GL_d}}(\mathrm{I}(\vartheta,\chi),\psi).
\end{align}
The dimension of the vector space \eqref{Hom:Whittaker} is $|T_{\GL_d,*}\backslash T_{\GL_d}|$, by an application of the Geometric Lemma of Bernstein and Zelevinsky \cite[Theorem~5.2]{BZ2} (see also \cite[Lemma~I.3.2]{KP}). For a function $f$ in the space of $\mathrm{I}(\vartheta,\chi)$ regarded as a complex-valued function, consider the Whittaker functional $\Lambda_t$ defined for each fixed $t\in T_{\GL_d}$ by
\begin{align*}
\Lambda_t(f)=\int\limits_{N_{\GL_d}}f(\langle t,1\rangle\langle J_d,1\rangle
\langle u,1\rangle)\psi^{-1}(u)\,du.
\end{align*}
Note that because in this context the $2$-cocycle is $\sigma^{\diamondsuit}_{d}$, $u\mapsto\langle u,1\rangle$ is the canonical splitting of $N_{\GL_d}$.
This integral is defined as an absolutely convergent integral for $\Real(\chi)$ in a certain cone, then by meromorphic continuation in general. In fact, by Banks \cite{Banks} this continuation is analytic (his argument - Bernstein's continuation principle - is applicable to $\GL_d^{(m,r)}$ as well). The functionals $\Lambda_t$ as $t$ varies over
the set of representatives of $T_{\GL_d,*}\backslash T_{\GL_d}$ form a linear basis of \eqref{Hom:Whittaker}
(see \cite[Lemma~I.3.1]{KP}). In particular denote $\Lambda=\Lambda_{I_d}$.

\subsection{The Casselman--Shalika formula for $\GL_d^{(m,r)}$}\label{Casselman--Shalika formula}
Assume $F$ is unramified and $\mu_{2m}\subset F^*$, in particular $(-1,x)_m=1$ for all $x\in F^*$. Fix a uniformizer $\varpi$ (see \S~\ref{Groups}). Let $\mathrm{I}_{\GL_d^{(m,r)}}(\vartheta,\chi)$ be an unramified principal series representation and
denote
\begin{align*}
\mathbf{x}=(x_1,\ldots,x_d)\in\C^d,\qquad x_i=\chi_i(\varpi^r).
\end{align*}
While $\mathbf{x}$ is independent of $\varpi$, the formulas in this section are developed for a given fixed $\varpi$.
Let $\xi^0$ be the normalized unramified vector in the space of $\mathrm{I}_{\GL_d^{(m,r)}}(\vartheta,\chi)$. For each $t\in\GL_d$, we have an unramified Whittaker function
$W_t(g)=\Lambda_t(g\cdot \xi^0)$ ($g\in \GL_d^{(m,r)}$). It is clear that this function is determined by its values on $\widetilde{T}_{\GL_d}$, and the purpose of the Casselman--Shalika formula is to describe $W_t$ on $\widetilde{T}_{\GL_d}$ in terms of $\mathbf{x}$, $t$ and the input
(in the linear case $\mathbf{x}$ would be replaced by the Satake parameter). As mentioned in the introduction, the metaplectic analog of the Casselman--Shalika formula \cite{CS2} has been developed through the works \cite{KP,McNamara2,COf,McNamara3}. In this section we describe this formula.

Assume the conductor of $\psi$ is $0$, i.e., $\psi$ is unramified.
Fix the Haar measure on $F$ which is self-dual with respect to $\psi$, then $\mathrm{vol}(\mathcal{O})$ (the volume of $\mathcal{O}$) is $1$.
Define the Gauss sum by
\begin{align}\label{eq:Gauss Sum}
\mathfrak{g}(l)=\int\limits_{\mathcal{O}^*}(o,\varpi)_m^l\psi(\varpi^{-1}o)\,do,\qquad l\in\Z.
\end{align}
Here $do$ is the restriction of the Haar measure on $F$, in particular
$\int_{\mathcal{O}^*}do=1-q^{-1}$. A direct computation implies the following properties (see e.g., \cite[\S~3.2]{Gao4}):
\begin{enumerate}[leftmargin=*]
\item If $l\equiv 0\,(m)$, $\mathfrak{g}(l)=-q^{-1}$.
\item Otherwise $|\mathfrak{g}(l)|=q^{-1/2}$ and $\overline{\mathfrak{g}(l)}=\mathfrak{g}(-l)$, hence $\mathfrak{g}(l)\mathfrak{g}(-l)=q^{-1}$.
\end{enumerate}

For $\mathbf{a}=(a_1,\ldots,a_d)\in\Z^r$, denote
\begin{align*}
\mathbf{a}^*=(-a_1,\ldots,-a_d),\qquad \varpi^{\mathbf{a}}=\diag(\varpi^{a_1},\ldots,\varpi^{a_d}),\qquad t_{\mathbf{a}}=\langle\varpi^{\mathbf{a}},1\rangle.
\end{align*}
The map $\mathbf{a}\mapsto t_{\mathbf{a}}$ is a homomorphism, by \eqref{eq:Nice GL $2$-cocycle on torus} and because $(\varpi,\varpi)_m=1$.
If $\mathbf{b}\in\Z^r$, write $\mathbf{a}\equiv \mathbf{b}$ if $\mathbf{a}-\mathbf{b}\in r\Z^d$.
The function $W_{t}$ is uniquely determined by its values on $t_{\mathbf{b}}$.
We define
\begin{align}\label{eq:def W Lambda}
W_{\mathbf{a}}(\mathbf{b},\vartheta,\chi)=\Lambda_{t_{\mathbf{a}}}(t_{\mathbf{b}}\cdot \xi^0).
\end{align}

For $t\in T_{\GL_d}$, let $\mathbf{v}(t)\in \Z^d$ denote the vector of valuations of $t$.
Recall that $T_{\GL_d,*}<T_{\GL_d}$ consists of the torus elements
with coordinates in $F^{*r}\mathcal{O}^*$. Then for $t_0\in T_{\GL_d,*}$, $\mathbf{v}(t_0)\in r\Z^d$.

Recall from \S~\ref{Groups} that $W_{\GL_d}$ is the Weyl group of $\GL_d$, $\Phi_d$ (resp., $\Phi_d^+$) is the set of roots (resp., positive roots) and the simple roots are $(i,i+1)$, $1\leq i<d$. We extend the notation $t_{\mathbf{a}}$ to $t_{l\alpha}$ for any $l\in\Z$ and $\alpha\in\Phi_d$, regarding $\alpha$ as the natural element in $\Z^d$.

Denote $\mathbf{x}_{\alpha}=x_ix_j^{-1}$. Since $w\in W_{\GL_d}$ is a permutation on $1,\ldots,d$, we can define
${}^w\mathbf{x}$ by $({}^w\mathbf{x})_i=x_{w^{-1}(i)}$. Also set for $\mathbf{a}\in r\Z^d$, $\mathbf{x}(\mathbf{a})=(x_1^{a_1/r},\ldots,x_d^{a_d/r})$.

As in \cite[p.~425]{COf}, define an action $w[\mathbf{a}]$ of $W_{\GL_d}$ on $\Z^d$ such that
\begin{align*}
w_{\alpha}[\mathbf{a}]=(a_1,\ldots,a_{i+1}-1,a_{i}+1,\ldots,a_d),\qquad \alpha=(i,i+1).
\end{align*}

Recall that for any $w\in W_{\GL_d}$ and a representative $y_w\in K_{\GL_d}$ of $w$, we have a standard intertwining operator
defined as follows. For $\alpha\in\Phi_d^+$, let $U_{\alpha}$ be the root subgroup of $\alpha$. Denote
\begin{align}\label{eq:N_w}
N_w=\prod_{\{\alpha\in\Phi_d^+:w^{-1}\alpha\notin\Phi_d^+\}}U_{\alpha}.
\end{align}
For $\xi$ in the space of $\mathrm{I}_{\GL_d^{(m,r)}}(\vartheta,\chi)$,
The integral
\begin{align*}
M(y_w)\xi(g)=\int\limits_{N_{w}}\xi(\langle y_w,1\rangle^{-1}\langle u,1\rangle g)\,du,\qquad g\in \GL_d^{(m,r)},
\end{align*}
is absolutely convergent when $\Real(\chi_1)\gg\ldots\gg\Real(\chi_d)$ and admits meromorphic continuation in the variables
$\chi_1,\ldots,\chi_d$ (see e.g., \cite[\S~13.7]{McNamara}, this is a straightforward extension of the results in linear case, e.g., \cite{CS1}). Since $\varepsilon\otimes\vartheta\chi$ is unramified as a genuine representation of $\widetilde{T}_{\GL_d,*}$;
$\eta_d^{\diamondsuit}$ is trivial on $T_{\GL_d}\cap K_{\GL_d}$; and
$\sigma_d^{\diamondsuit}(t,y_w)=1$ for any $t\in T_{\GL_d}\cap K_{\GL_d}$, the integral is independent of the choice of representative $y_w$,
hence we can simply denote $M(w)=M(y_w)$.
By Proposition~\ref{proposition:action of W on torus is trivial on Sp} (recall that $\GL_d^{(m,r)}$ is defined via restriction
from $\Sp_{2d}^{(m)}$), ${}^{w}(\varepsilon\otimes\vartheta\chi)=\varepsilon\otimes\vartheta({}^w\chi)$ as a representation of
$\widetilde{T}_{\GL_d,*}$, hence (away from the poles)
\begin{align*}
M(w):\mathrm{I}_{\GL_d^{(m,r)}}(\vartheta,\chi)\rightarrow \mathrm{I}_{\GL_d^{(m,r)}}(\vartheta,{}^w\chi).
\end{align*}

In fact we can assume that $\chi$ is regular, i.e., ${}^{w}(\varepsilon\otimes\vartheta\chi)\ne\varepsilon\otimes\vartheta\chi$, equivalently
${}^{w}\chi\ne\chi$ as characters of $(F^{*r})^d$, for all $1\ne w\in W_{\GL_d}$. Then the operators $M(w)$ are holomorphic (\cite[\S~13.7]{McNamara}).
Since the Jacquet integrals $\Lambda_t$ admit analytic continuation, the results of this section are true regardless of this assumption,
by analytic continuation.

Since $M(w)\xi^0$ is unramified, it is a constant multiple of the normalized unramified vector $\xi_w^0$ of $\mathrm{I}_{\GL_d^{(m,r)}}(\vartheta,{}^w\chi)$. The constant is given by the
Gindikin--Karpelevich formula \cite[Corollary~7.4]{Gao2018}, which is the extension of \cite[Theorem~3.1]{CS1} to coverings. We have
\begin{align}\label{eq:GK formula for GL}
M(w)\xi^0(g)=\prod_{\{\alpha\in\Phi_d^+:w\alpha\notin\Phi_d^+\}}\frac{1-q^{-1}\mathbf{x}_{\alpha}}{1-\mathbf{x}_{\alpha}}
\xi_w^0(g).
\end{align}
Here the appearance of $\mathbf{x}_{\alpha}$ in \eqref{eq:GK formula for GL} is the consequence of our definition of
$\mathrm{I}_{\GL_d^{(m,r)}}(\vartheta,\chi)$ and the fact that ${}^{w}(\varepsilon\otimes\vartheta\chi)=\varepsilon\otimes\vartheta({}^w\chi)$; in general $\mathbf{x}_{\alpha}$ would be replaced by the restriction of the genuine unramified character of the center of $\widetilde{T}_{\GL_d}$ to
$\langle\diag(I_{i-1},\varpi^r,I_{j-i-1},\varpi^{-r},I_{d-j}),1\rangle$ where $\alpha=(i,j)$, and one must then show that this restriction is independent of the choice of uniformizer $\varpi$ (see \cite[\S~7.1 and Proposition~7.3]{Gao2018}, and also \cite{COf}).
We also mention that \cite[Corollary~7.4]{Gao2018} is applicable even without the assumption $\mu_{2m}\subset F^*$. For previous extensions of this formula to coverings see \cite{KP,McNamara2,COf,McNamara3,me8}.

The Casselman--Shalika formula for covering groups depends on certain coefficients which we present next, following \cite[\S~4, \S~5]{COf}.
Given $w\in W_{\GL_d}$, these coefficients are functions $(t,t')\mapsto\tau_{t,t'}(w,\vartheta,\chi)$ ($t,t'\in \widetilde{T}_{\GL_d}$).
Let $\mathcal{A}$ and $\mathcal{A}'$ be two sets of representatives for $\widetilde{T}_{\GL_d,*}\backslash \widetilde{T}_{\GL_d}$. Then for any $w\in W_{\GL_d}$, we have the following functional equation:
\begin{align}\label{eq:general equation defining tau a b}
(\Lambda_t\circ M(w))_{t\in \mathcal{A}}=
(\tau_{t,t'}(w,\vartheta,\chi))_{t\in\mathcal{A},t'\in\mathcal{A}'}
(\Lambda_{t'})_{t'\in \mathcal{A}'}.
\end{align}
The coefficients $\tau_{t,t'}(w,\vartheta,\chi)$ do not depend on the choice of representative in $K_{\GL_d}$ for $w$.
While \eqref{eq:general equation defining tau a b} defines them uniquely, it is simpler to compute them inductively.
First note that for any
$t,t'\in \widetilde{T}_{\GL_d}$ and $t_0,t_0'\in \widetilde{T}_{\GL_d,*}$,
\begin{align}\label{eq:tau on maximal abelian}
\tau_{t_0t,t_0't'}(w,\vartheta,\chi)=\delta_{B_{\GL_d}}^{1/2}(t_0(t_0')^{-1})
(\varepsilon\otimes\vartheta{}^{w}\chi)(t_0)
(\varepsilon\otimes\vartheta\chi)(t_0')^{-1}\tau_{t,t'}(w,\vartheta,\chi).
\end{align}
In particular these coefficients do depend on the choices of representatives, and thereby on $\varpi$ itself.
Now if $w_1,w_2\in W_{\GL_d}$ satisfy $\ell(w_1w_2)=\ell(w_1)\ell(w_2)$,
\begin{align}\label{eq:cocycle for tau}
\tau_{t,t'}(w_1w_2,\vartheta,\chi)=\sum_{\mathbf{e}\in r\Z^d\backslash \Z^d}\tau_{t,t_{\mathbf{e}}}(w_1,\vartheta,{}^{w_2}\chi)\tau_{t_{\mathbf{e}},t'}(w_2,\vartheta,\chi).
\end{align}
The sum is well defined by \eqref{eq:tau on maximal abelian}. To determine
$\tau_{t,t'}(w,\vartheta,\chi)$ on $w_{\alpha}$, $\alpha=(i,i+1)$, write
\begin{align}\label{eq:tau for simple ref}
\tau_{t,t'}(w_{\alpha},\vartheta,\chi)=\tau_{t,t'}^1(w_{\alpha},\vartheta,\chi)+\tau_{t,t'}^2(w_{\alpha},\vartheta,\chi).
\end{align}
\begin{proposition}\label{proposition:coefficients tau t t' 1 2}
Let $\alpha=(i,i+1)$, $\mathbf{a},\mathbf{b}\in\Z^d$. Then
$\tau_{t_{\mathbf{a}},t_{\mathbf{b}}}^1(w_{\alpha},\vartheta,\chi)=0$ unless $\mathbf{b}\equiv\mathbf{a}$,
$\tau_{t_{\mathbf{a}},t_{\mathbf{b}}}^2(w_{\alpha},\vartheta,\chi)=0$ unless $\mathbf{b}\equiv w_{\alpha}[\mathbf{a}]$,
\begin{align}\label{eq:tau 1 nonzero}
&\tau_{t_{\mathbf{a}},t_{\mathbf{a}}}^1(w_{\alpha},\vartheta,\chi)=(1-q^{-1})\frac{\mathbf{x}_{\alpha}^{\lceil (a_{i+1}-a_{i})/r\rceil}}{1-\mathbf{x}_{\alpha}},\\
\label{eq:tau 2 nonzero}
&\tau_{t_{\mathbf{a}},t_{w_{\alpha}[\mathbf{a}]}}^2(w_{\alpha},\vartheta,\chi)=q^{a_{i+1}-a_i-1}\mathfrak{g}(2(a_{i+1}-a_i-1)).
\end{align}
Here for any $x\in\R$, $\lceil x \rceil$ is the smallest integer greater than or equal to $x$.
Also with the notation of \eqref{eq:tau on maximal abelian}, for $j=1,2$,
\begin{align}\label{eq:tau i on maximal abelian}
\tau_{t_0t,t_0't'}^j(w_{\alpha},\vartheta,\chi)=\delta_{B_{\GL_d}}^{1/2}(t_0(t_0')^{-1})
(\varepsilon\otimes\vartheta{}^{w}\chi)(t_0)
(\varepsilon\otimes\vartheta\chi)(t_0')^{-1}\tau_{t,t'}^j(w_{\alpha},\vartheta,\chi).
\end{align}
\end{proposition}
\begin{proof}
Apply the arguments of \cite[\S~9, \S~13]{McNamara3} (based on \cite[\S~I.1.3]{KP}).
\end{proof}
\begin{remark}
Our conventions are slightly different from \cite{McNamara3} in order to facilitate the application of results from \cite{Suzuki1997}.
\end{remark}

By \eqref{eq:tau on maximal abelian}, $\tau_{t,t'}(w,\vartheta,\chi)$ is determined by elements $t=t_{\mathbf{a}}$ and $t'=t_{\mathbf{b}}$, and we define
\begin{align*}
\tau_{\mathbf{a},\mathbf{b}}(w,\vartheta,\chi)=\tau_{t_{\mathbf{a}},t_{\mathbf{b}}}(w,\vartheta,\chi).
\end{align*}

The Casselman--Shalika formula for $\GL_d^{(m,r)}$ takes the following form:
\begin{align}\label{eq:CS formula}
&W_{\mathbf{a}}(\mathbf{b},\vartheta,\chi)=\delta_{B_{\GL_{d}}}(t_{\mathbf{b}})\sum_{w\in W_{\GL_d}}\prod_{\{\alpha\in\Phi_d^+:w\alpha\in\Phi_d^+\}}\frac{1-q^{-1}({}^{w^{-1}}\mathbf{x})_{\alpha}}
{1-({}^{w^{-1}}\mathbf{x})_{\alpha}}\tau_{\mathbf{a},\mathbf{b}^*}(w,\vartheta,{}^{w^{-1}}\chi).
\end{align}

We also mention the following observation of \cite[Remark~3.2]{Suzuki1997}: for a nontrivial $w\in W_{\GL_d}$,
if $w=w_{\alpha_1}\cdot\ldots\cdot w_{\alpha_{\ell(w)}}$, by definition
$\tau_{t,t'}(w,\vartheta,\chi)=0$ unless there is $w'=w_{\alpha_1}^{i_1}\cdot\ldots\cdot w_{\alpha_{\ell(w)}}^{i_1}$ for some $i_1,\ldots,i_{\ell(w)}\in\{0,1\}$, such that $w'[\mathbf{v}(t')]\equiv \mathbf{v}(t)$.

As in Suzuki \cite[\S~4.1]{Suzuki1997} (see also \cite[\S~7.1 (2)]{Suzuki1998}), we can rewrite the summation in \eqref{eq:CS formula} to obtain an iterative formula. Regard $W_{\GL_{d-1}}$ as the subgroup of $W_{\GL_d}$ generated by $w_{\alpha}$ for $\alpha=(i,i+1)$ and $1\leq i<d-1$. Identify $\chi$ with $(\chi_1,\ldots,\chi_d)$, where each $\chi_i$ is an unramified quasi-character of $F^*$, then let
$\chi'[i]$ denote the $d-1$ tuple $(\chi_1,\ldots,\chi_{i-1},\chi_{i+1},\ldots,\chi_d)$, which then defines the representation
$\mathrm{I}_{\GL_{d-1}^{(m,r)}}(\vartheta,\chi'[i])$. Also denote $\mathbf{b}=(b_1,\mathbf{b}')$ ($\mathbf{b}'\in\Z^{d-1}$).
For $1\leq i\leq d-1$, let $\omega_{i}=w_{\alpha_{d-i}}\cdot\ldots\cdot w_{\alpha_{d-1}}$ ($\omega_{i}$ takes $i$ positive roots of $P_{(d-1,1)}$ into negative roots), and take $\omega_0=I_d$. Then we can write each $w\in W_{\GL_d}$ uniquely in the form
$w=w'\omega_{i}^{-1}$, where $w'\in W_{\GL_{d-1}}$ and $0\leq i\leq d-1$. Now \eqref{eq:CS formula} can be restated in the form
\begin{align}\label{eq:iterative CS formula of Suzuki}
&W_{\mathbf{a}}(\mathbf{b},\vartheta,\chi)=
\\\nonumber&\delta_{B_{\GL_d}}(t_{\mathbf{b}})
\sum_{i=0}^{d-1}
\prod_{j=1}^{d-i-1}\frac{1-q^{-1}\mathbf{x}_{(j,d-i)}}{1-\mathbf{x}_{(j,d-i)}}
\sum_{\mathbf{e'}\in r\Z^{d-1}\backslash \Z^{d-1}}
\tau_{\mathbf{a}^*,(\mathbf{e'}^*,-b_1)}(\omega_i,\vartheta,{}^{\omega_i^{-1}}\chi)
\delta_{B_{\GL_{d-1}}}^{-1}(t_{\mathbf{b'}})W_{\mathbf{e'}}(\mathbf{b'},\vartheta,\chi'[d-i]).
\end{align}
\begin{remark}
Gao \textit{et. al.} \cite{GaoShahidiSzpruch2018} and Szpruch \cite{Szpruch2018} have recently developed
a reinterpretation of the coefficients of \eqref{eq:general equation defining tau a b}: the coefficients were presented as linear combinations of Tate $\gamma$-factors, or ``metaplectic" $\gamma$-factors (defined in \cite{Szpruch11}), where the scalars were computed via harmonic analysis on $F^{*r}\backslash F^*$. The determinant of their local coefficients matrix was related to the Plancherel measure and may shed new light on the Plancherel formula for covering groups.
\end{remark}
\subsection{Exceptional representations}\label{exceptional}
We briefly describe the exceptional representations of coverings of general linear groups, which play a role in our construction.
These representations were introduced by \cite{KP} (for the coverings studied there). Gao \cite{Gao5} extended their construction to the coverings of arbitrary connected split reductive groups over non-archimedean fields, developed in \cite{BD}, and in particular to $\GL_d^{(m,r)}$ (see also \cite{Gao4}).

Assume $F$ is unramified and $\mu_{2m}\subset F^*$. Define the exceptional representation $\Theta_{d,m,r,\vartheta}$ as the unique irreducible subrepresentation of $\mathrm{I}_{\GL_{d}^{(m,r)}}(\vartheta,\delta_{B_{\GL_{d}}}^{-1/(2r)})$, equivalently the unique irreducible quotient of $\mathrm{I}_{\GL_{d}^{(m,r)}}(\vartheta,\delta_{B_{\GL_{d}}}^{1/(2r)})$. The existence of this representation follows from Langlands' Quotient Theorem for covering groups (\cite{BJ}). Also $\Theta_{d,m,r,\vartheta}$ is unramified.

The representation $\Theta_{r,m,r,\vartheta}$ affords a unique (up to scaling) Whittaker functional, while
$\Theta_{d,m,r,\vartheta}$ does not afford a Whittaker functional for any $d>r$. This follows from
\cite[Proposition~3.5]{Gao4}. Specifically, the condition $Y_{\GL_r,Q,n}=n_{\alpha}Y_{\GL_r}$ in the notation of \textit{loc. cit.} becomes
$Y_{\GL_d,Q,m}=rY_{\GL_d}$ here, which is simple to verify directly. These results of \cite{Gao4} were proved under the assumptions $|2m|=1$ and $\mu_{2m}\subset F^*$, both satisfied here (the condition $p\nmid 2m$ of \cite[\S~2]{Gao4}, where $p$ is the characteristic of the residue field of $F$, is not needed for these assertions). See also \cite[Theorem~I.3.5]{KP} for the proof of this statement
in their setting.

\begin{proposition}\label{proposition:* of Theta}
Let ${}^*$ be defined by \eqref{eq:involution b*0}. Then
$\Theta_{d,m,r,\vartheta}^*=\Theta_{d,m,r,\vartheta}$.
\end{proposition}
\begin{proof}
By Proposition~\ref{proposition:vartheta of *}, $\Theta_{d,m,r,\vartheta}^*$ is a constituent
of $\mathrm{I}_{\GL_{d}^{(m,r)}}(\vartheta,\delta_{B_{\GL_{d}}}^{-1/(2r)})$, and since it is also irreducible and unramified,
it coincides with $\Theta_{d,m,r,\vartheta}$.
\end{proof}

\begin{remark}
For exceptional representations of groups other than $\GL_d$, Jacquet modules may depend on $\vartheta$ (see e.g., \cite{Gao5}).
\end{remark}

\subsection{Fourier coefficients on unipotent orbits}\label{fourier coeff on orbits}
The unipotent orbits of $\GL_d$ are in bijection with the partitions of $d$. There is a partial ordering defined on the partitions, and for two partitions $\beta$ and $\beta'$, we write $\beta\succsim \beta'$ if $\beta$ is greater than or not comparable with $\beta'$. Given a partition $\beta$ there is a corresponding unipotent subgroup $V(\beta)$ and a set of generic characters of $V(\beta)$. See \cite[\S~2]{G2} for these definitions or \cite{Cr,CM} for the standard reference. Over a local field, for one such character $\psi$ and a smooth (complex) representation $\rho$ of $\GL_d$, define $\mathcal{O}(\rho,\beta,\psi)=\Hom_{V(\beta)}(\rho,\psi)$ (continuous morphisms over archimedean fields). The global counterpart is the
Fourier coefficient
\begin{align}\label{int:general Fourier coeff}
\int\limits_{V(\beta)(F)\backslash V(\beta)(\A)}\phi(v)\psi^{-1}(v)\,dv,
\end{align}
defined for an automorphic form $\phi$ in the space of an automorphic representation $\rho$ of $\GL_d(\A)$.
We then let $\mathcal{O}(\rho,\beta,\psi)$ denote the set of Fourier coefficients as $\phi$ varies in the space of $\rho$.

These notions extend to $\GL_{d}^{(m,r)}$ and any other covering of $\GL_d$ from \cite{KP,BD}, because these coverings are split canonically over a given unipotent subgroup (see \S~\ref{Covering groups}).

\subsection{Semi--Whittaker coefficients}\label{Semi Whittaker coeff}
Let $d$ be a positive integer. For a composition $\lambda$ of $d$, let $\psi_{\lambda}$ be the character of $N_{\GL_{d}}$ which restricts to $\psi$ on the simple root subgroups of $M_{\lambda}$ and acts trivially otherwise. Globally, starting with a nontrivial character of $F\backslash \A$, we obtain
a character of $N_{\GL_d}(\A)$ which is trivial on $N_{\GL_d}(F)$. For an automorphic form $\phi$ on $\GL_d(\A)$, the global Fourier coefficients along $(N_{\GL_{d}},\psi_{\lambda})$ is given by the integral
\begin{align}\label{lambda semi coefficient}
\int\limits_{N_{\GL_d}(F)\bs N_{\GL_d}(\A)}\phi(u)\psi_{\lambda}^{-1}(u) \, du.
\end{align}
The definition extends to autormorphic forms on $\GL_d^{(m,r)}(\A)$ or any other covering of $\GL_d(\A)$ from \cite{KP,BD}, except that $u$ is replaced by
its image in the covering under the canonical splitting. These coefficients
are called semi--Whittaker coefficients, and have strong relations with Fourier coefficients associated to unipotent orbits.
They were studied in both global and local contexts in several works, including \cite{MW3,GourevitchSahi2013,AGS2015a,AGS2015,GGS,Cai2,GGS2}.

Extend the partial order on partitions to compositions by comparing their underlying partitions.
We will use the following two results, which are particular cases of results of Cai \cite[Proposition~5.3]{Cai2} and
\cite[Proposition~5.5]{Cai2}, respectively.
\begin{lemma}\label{lemma:gbl semi Whittaker and WSS}
Let $\mathcal{E}$ be an automorphic representation of $\GL_{rkc}(\A)$. If \eqref{lambda semi coefficient} is identically zero on $\mathcal{E}$ for any $\lambda\succsim((rk)^c)$, and is nonzero
for $\lambda=((rk)^c)$, then $\mathcal{O}(\mathcal{E},((rk)^c),\psi)\ne0$.
\end{lemma}
\begin{lemma}\label{lemma:local semi Whittaker and WSS}
Let $\mathcal{E}$ be a smooth representation of $\GL_{rkc}$ over a local non-archimedean field. Then
$J_{N_{\GL_{rkc}},\psi_{\lambda}}(\mathcal{E})=0$ for all $\lambda\succsim((rk)^c)$,
if and only if $\mathcal{O}(\mathcal{E},\lambda_0,\psi)=0$ for all $\lambda_0\succsim ((rk)^c)$. Moreover,
\begin{align*}
\dim J_{N_{\GL_{rkc}},\psi_{((rk)^c)}}(\mathcal{E})=\dim \mathcal{O}(\mathcal{E},((rk)^c),\psi).
\end{align*}
\end{lemma}
\begin{remark}\label{remark:applicability of Cai results to coverings}
The proofs in \cite{Cai2} were stated for general linear groups, but are based on the ``root exchange" technique which is also applicable to covering groups: in a global context the fundamental root exchange result is \cite[Lemma~7.1]{RGS} (following \cite{G,GRS3,Soudry5}); the local version of root exchange is \cite[\S~2.2]{GRS5} (see also the proof of \cite[Theorem~3.1]{me11}).
\end{remark}
\section{Representations of type $(rk,c)$}\label{speh}
\subsection{Definition}\label{speh def}
Let $r,k,c\geq1$ be integers. We use the notation of \S~\ref{fourier coeff on orbits} for the group $\GL_{rkc}$.
For the partition $\beta=((rk)^c)$, $V(\beta)=V_{(c^{rk})}$ (the unipotent radical of $P_{(c^{rk})}$). For $v\in V_{(c^{rk})}$ write $v=(v_{i,j})_{1\leq i,j\leq rk}$ with $v_{i,j}\in\Mat_c$,
and define
\begin{align}\label{eq:rk c character}
\psi(v)=\psi(\sum_{i=1}^{rk-1}\tr(v_{i,i+1})).
\end{align}
Let $\mathcal{E}$ be a genuine smooth representation of $\GL_{rkc}^{(m,r)}$ over a local field. We say that $\mathcal{E}$ is an $(rk,c)$ representation if the following holds:
\begin{enumerate}
\item $\mathcal{O}(\mathcal{E},\beta',\psi')=0$ for any $\beta'\succsim((rk)^c)$.
\item For $\psi$ given by \eqref{eq:rk c character}, $\dim\mathcal{O}(\mathcal{E},((rk)^c),\psi)=1$.
\end{enumerate}
Any (nonzero) $\Lambda\in \mathcal{O}(\mathcal{E},((rk)^c),\psi)$ is called an $(rk,c)$ functional on $\mathcal{E}$. If $\mathcal{E}$ is an $(rk,c)$ representation, we fix one such $\Lambda$ and let $\mathcal{W}(\mathcal{E})$ be the unique $(rk,c)$ model of $\mathcal{E}$ with respect to $\psi$, i.e., the space spanned by functions $g\mapsto\Lambda(\mathcal{E}(g)\xi)$ where $g\in\GL_{rkc}^{(m,r)}$ and $\xi$ varies in the space of $\mathcal{E}$.

Now for a genuine irreducible automorphic representation $\mathcal{E}$ of $\GL_{rkc}^{(m,r)}(\A)$, $\mathcal{E}$ is an $(rk,c)$ representation if it satisfies the following conditions:
\begin{enumerate}
\item\label{def:Whittaker--Speh--Shalika 1} $\mathcal{O}(\mathcal{E},\beta',\psi')=0$ for any $\beta'\succsim((rk)^c)$.
\item\label{def:Whittaker--Speh--Shalika 1.5} $\mathcal{O}(\mathcal{E},((rk)^c),\psi)\ne0$.
\item\label{def:Whittaker--Speh--Shalika 2} All the unramified components $\mathcal{E}_{\nu}$ of $\mathcal{E}$ are $(rk,c)$ representations.
\end{enumerate}
\begin{remark}
In the linear case, the generalized Speh representation of $\GL_{kc}(\A)$ attached to an irreducible cuspidal automorphic representation of $\GL_k(\A)$ is $(k,c)$ and all of its local components are $(k,c)$ as well (globally - \cite{G2,JL2013}, locally - \cite[Theorem~5]{CFK}).
\end{remark}
\begin{remark}
These definitions make sense for any covering of $\GL_{rkc}$, as long as it is split canonically over $N_{\GL_{rkc}}$, and in a global setting it is also split over $\GL_{rkc}(F)$.
\end{remark}

The global vanishing condition is superfluous: the local vanishing properties of $\mathcal{E}_{\nu}$, even in one place, already imply global vanishing by a local-global principle (see e.g., \cite[Proposition~1]{JR}).

For $\GL_{rkc}^{(m,r)}(\A)$ with the $2$-cocycle $\rho_{rkc}^{\diamondsuit}$, $\beta=((rk)^c)$ and $\psi$ given by \eqref{eq:rk c character},
\eqref{int:general Fourier coeff} becomes
\begin{align}\label{int:a c general Fourier coeff}
\Lambda(\phi)=\int\limits_{V_{(c^{rk})}(F)\backslash V_{(c^{rk})}(\A)}\phi(\langle v,(\eta_{rkc}^{\diamondsuit})^{-1}(v)\rangle)\psi^{-1}(v)\,dv.
\end{align}
We call this the global $(rk,c)$ functional. If $\mathcal{E}$ is an $(rk,c)$ representation of $\GL_{rkc}^{(m,r)}(\A)$, then
\eqref{int:a c general Fourier coeff} does not vanish on the space of $\mathcal{E}$. The (global) $(rk,c)$ model of $\mathcal{E}$ is then
by definition the space of functions $g\mapsto\Lambda(\mathcal{E}(g)\phi)$.

We also mention that for $c=1$, $\Lambda(\phi)$ is the Whittaker--Fourier coefficient on $\mathcal{E}$ and
we call $\mathcal{E}$ globally generic when $\mathcal{O}(\mathcal{E},rk,\psi)\ne0$. If
$\mathcal{E}$ is an irreducible globally generic automorphic representation of $\GL_k(\A)$ (the linear setting),
the global condition already implies the local condition at all places.

One may expect all the local components of an $(rk,c)$ representation to be $(rk,c)$, but since we can not prove this for the $(rk,c)$ representations we construct, we settle for uniqueness at all the unramified places. Unfortunately, in the absence of local uniqueness everywhere we do not obtain a decomposable functional; we obtain the weaker statement \eqref{eq:partial decomp of Lambda} below, but it can still be used to obtain results on the partial $L$-function. See \S~\ref{global symplectic}.

Let $\mathcal{E}$ be an $(rk,c)$ representation of $\GL_{rkc}^{(m,r)}(\A)$.
We can identify $\mathcal{E}$ with the restricted tensor product $\otimes_{\nu}'\mathcal{E}_{\nu}$, with respect to a finite set $S$ of places of $F$ such that $\mathcal{E}_{\nu}$ is unramified for all $\nu\notin S$, and a collection $\{\xi_{\nu}^0\}_{\nu\notin S}$ of vectors where each $\xi_{\nu}^0$ belongs to the space of $\mathcal{E}_{\nu}$ and is fixed by $\{\langle y,1\rangle:y\in K_{\GL_{rkc},\nu}\}$
($\rho_{rkc,\nu}^{\diamondsuit}$ is trivial on $K_{\GL_{rkc},\nu}$). For
$\nu\notin S$, $\Hom_{V_{(c^k)}(F_\nu)}(\mathcal{E}_\nu,\psi_\nu)$ is one-dimensional. We fix the $(rk,c)$ functional $\Lambda_{\nu}^0$ on $\mathcal{E}_{\nu}$ by demanding
$\Lambda_{\nu}^0(\xi_{\nu}^0)=1$. Denote $F_S=\prod_{\nu\in S} F_{\nu}$. For $g\in\GL_{rkc}^{(m,r)}(\A)$, write $g=(g_{\nu})_{\nu}$ with
$g_{\nu}\in\GL_{rkc}^{(m,r)}(F_{\nu})$ and set
$g_S=(g_{\nu})_{\nu\in S}$. Also let $\mathcal{E}_S$ and $\psi_S$ denote the tensor products over the components in $S$. For the places in $S$ we can define $\Lambda[S]\in{\Hom}_{V_{(c^k)}(F_S)}(\mathcal{E}_S,\psi_S)$ by
\begin{align}\label{eq:partial Lambda S}
\Lambda[S](\xi_S)=\Lambda(\xi_S\otimes'_{\nu\notin S}\xi_{\nu}^0).
\end{align}
Then for a decomposable vector $\phi=\xi_S\otimes'_{\nu\notin S}\xi_{\nu}$ in the space of $\otimes_{\nu}'\mathcal{E}_{\nu}$, for all $g\in\GL_{rkc}^{(m,r)}(\A)$,
\begin{align}\label{eq:partial decomp of Lambda}
\Lambda(\mathcal{E}(g)\varphi)=\Lambda[S](\mathcal{E}_S(g_S)\xi_S)\prod_{\nu\notin S}\Lambda_{\nu}(\mathcal{E}_{\nu}(g_{\nu})\xi_{\nu}).
\end{align}
Here $\Lambda_{\nu}$ is a scalar multiple of $\Lambda_{\nu}^0$ for all $\nu\notin S$. To prove this one argues exactly as in
\cite[Proposition~3.14]{Tk}, which is the adaptation of the decomposition result of Shalika
\cite[\S~4]{Sh} when uniqueness holds everywhere (see also \cite{PS1979} and \cite[Theorem~3.5.2]{Bump1997}).

Next we prove that $\Lambda(\phi)$ enjoys an additional invariance property, which is important for the global construction of the integral. The stabilizer of the character \eqref{eq:rk c character} inside $M_{(c^{rk})}$ is the diagonal embedding $\GL_c^{\Delta}=\{b^{\triangle}:b\in\GL_c\}$ in $\GL_{rkc}$, where $b^{\triangle}=\diag(b,\ldots,b)$. Our first step is to construct a splitting of $\SL_c^{\Delta}(\A)$ under $\rho_{rkc}^{\diamondsuit}$ (for $c=1$, $\SL_c$ is trivial).

\begin{proposition}\label{proposition:sigma on diagonal embedding of SLc}
Locally for all $b,b'\in\SL_c$, $\sigma_{rkc}^{\diamondsuit}(b^{\triangle},{b'}^{\triangle})=
(\varsigma_{*,c}(b)\varsigma_{*,c}(b')/\varsigma_{*,c}(bb'))^{rk}$. Hence
$b^{\triangle}\mapsto\langle b^{\triangle},\varsigma_{*,c}^{-rk}(b)\rangle$
is the unique splitting of $\SL_c^{\triangle}$ in $\GL_{rkc}^{(m,r)}$, when we realize $\GL_{rkc}^{(m,r)}$ using
$\sigma_{rkc}^{\diamondsuit}$.
\end{proposition}
\begin{proof}
For $b,b'\in \GL_c$, by \eqref{eq:block compatibility on Levi of P}, \eqref{eq:sigma square} and \eqref{eq:BLS block compatible},
\begin{align*}
\sigma_{rkc}^{\diamondsuit}(b^{\triangle},{b'}^{\triangle})=
\prod_{i=1}^{rk}\sigma_{2c}(\diag(b,b^*),\diag(b',{b'}^*))
=\left((\det b,\det {b'})_m^{-1}
\sigma_{c}(b,b')\sigma^*_{c}(b,b')\right)^{rk}.
\end{align*}
In particular when $b,b'\in\SL_c$, we can apply \eqref{eq:sigma c *} and since $\sigma_{c}^{2rk}=1$, the r.h.s.~ equals
$(\varsigma_{*,c}(b)\varsigma_{*,c}(b')/\varsigma_{*,c}(bb'))^{rk}$. The assertion regarding the splitting follows at once.
\end{proof}
\begin{corollary}\label{corollary:invariance wrt SLc}
Let $\mathcal{E}$ be an $(rk,c)$ representation of $\GL_{rkc}^{(m,r)}$ over a local non-archimedean field, and
$\Lambda$ be an $(rk,c)$ functional on $\mathcal{E}$. Then $\Lambda(\mathcal{E}(\langle b^{\triangle},\varsigma_{*,c}^{-rk}(b)\rangle)\xi)=
\Lambda(\xi)$ for any $b\in\SL_c$ and $\xi$ in the space of $\mathcal{E}$.
\end{corollary}
\begin{proof}
Since $\dim\mathcal{O}(\mathcal{E},((rk)^c),\psi)=1$ and the field is non-archimedean, $\dim J_{V_{(c^{rk})},\psi}(\mathcal{E})=1$.
Hence the stabilizer acts by a character, which is trivial on $\{\langle b^{\triangle},\varsigma_{*,c}^{-rk}(b)\rangle:b\in\SL_c\}\cong\SL_c$.
\end{proof}
\begin{corollary}\label{corollary:gbl splitting for SL_c}
For each $\nu$ and $b\in\SL_c(F_{\nu})$, define $\eta_{rkc,\nu}^{\triangle}(b)=\eta_{rkc,\nu}^{\diamondsuit}(b^{\triangle})\varsigma_{*,c,\nu}^{rk}(b)$.
The product $\eta_{rkc}^{\triangle}=\prod_{\nu}\eta_{rkc,\nu}^{\triangle}$ is well defined on $\SL_c(\A)$ and $b^{\triangle}\rightarrow\langle b^{\triangle},(\eta_{rkc}^{\triangle})^{-1}(b)\rangle$ is the unique splitting of $\SL_c^{\triangle}(\A)$ in $\GL_{rkc}^{(m,r)}(\A)$.
\end{corollary}
\begin{proof}
According to \eqref{eq:cohomologous $2$-cocycles on GL} and Proposition~\ref{proposition:sigma on diagonal embedding of SLc}, for all $\nu$
and $b,b'\in\SL_c(F_{\nu})$,
\begin{align}\label{eq:local rho eta varsigma}
\rho_{rkc,\nu}^{\diamondsuit}(b^{\triangle},{b'}^{\triangle})=
\frac{\eta_{rkc,\nu}^{\triangle}(b)\eta_{rkc,\nu}^{\triangle}(b')}{\eta_{rkc,\nu}^{\triangle}(bb')}.
\end{align}

At almost all places the l.h.s.~ is $1$ on $\SL_c(\mathcal{O}_{\nu})$, then
$b\mapsto \eta_{rkc,\nu}^{\triangle}(b)$ is a homomorphism of
$\SL_c(\mathcal{O}_{\nu})$ which is trivial. Thus $\eta_{rkc}^{\triangle}$ is well defined on $\SL_c(\A)$. Now
globalize \eqref{eq:local rho eta varsigma}.
\end{proof}

Since both
$b^{\triangle}\mapsto\langle b^{\triangle},(\eta_{rkc}^{\diamondsuit})^{-1}(b^{\triangle})\rangle$ and
$b^{\triangle}\mapsto\langle b^{\triangle},(\eta_{rkc}^{\triangle})^{-1}(b)\rangle$ are splittings of
$\SL_c^{\triangle}(F)$ into $\GL_{rkc}^{(m,r)}(\A)$, and there is a unique such splitting, we have
$(\eta_{rkc}^{\diamondsuit})^{-1}(b^{\triangle})=(\eta_{rkc}^{\triangle})^{-1}(b)$ for $b\in\SL_c(F)$.
The same identity holds with $b\in N_{\GL_c}(\A)$, for the same reason.
Note that $G<\SL_c$, whence the corollary also provides the unique splitting of $G^{\triangle}(\A)$ in $\GL_{rkc}^{(m,r)}(\A)$.
\begin{proposition}\label{proposition:extra invariance}
Let $\mathcal{E}$ be an $(rk,c)$ representation of $\GL_{rkc}^{(m,r)}(\A)$, realized by right-translation, and $\phi$ be an automorphic form in the space of $\mathcal{E}$. Then
\begin{align*}
\Lambda(\langle b^{\triangle},(\eta_{rkc}^{\triangle})^{-1}(b)\rangle\cdot\phi)=\Lambda(\phi),\qquad\forall
b\in \SL_c({\A}).
\end{align*}
\end{proposition}
\begin{proof}
The group $\SL_c({\A})$ is generated as an abstract group by its unipotent subgroups, or alternatively
by $U_{\alpha}(\A)$ for a single root subgroup $U_{\alpha}<N_{\GL_c}$ and a set of representatives in $\SL_c(F)$ for the Weyl group. Since $\phi$ is an automorphic form, it is in particular left-invariant with respect to the image of $\SL_c(F)$ under $b\mapsto\langle b^{\triangle},(\eta_{rkc}^{\diamondsuit})^{-1}(b^{\triangle})\rangle
=\langle b^{\triangle},(\eta_{rkc}^{\triangle})^{-1}(b)\rangle$.
Therefore it suffices to prove invariance under $\{\langle u^{\triangle},(\eta_{rkc}^{\triangle})^{-1}(u)\rangle :u\in U_{\alpha}(\A)\}$. As in \cite[Proposition~3]{FG2}, we write the Fourier
expansion of $b\mapsto \Lambda(\langle b^{\triangle},(\eta_{rkc}^{\triangle})^{-1}(b)\rangle \cdot\phi)$ along $U_{\alpha}$ and observe that the coefficients corresponding to nontrivial characters are all associated with unipotent orbits which are greater than or not comparable with $((rk)^c)$, hence vanish on $\mathcal{E}$ by definition.
\end{proof}

\begin{corollary}\label{corollary:extra invariance on Lambda S}
Let $\mathcal{E}$ be an $(rk,c)$ representation of $\GL_{rkc}^{(m,r)}(\A)$, realized by right-translation. Let $S$ be a finite set of places of $F$ containing (at least) the archimedean places and the places $\nu$ where $\mathcal{E}_{\nu}$ is ramified. Realize $\GL_{rkc}^{(m,r)}(F_{S})$ using the product $\prod_{\nu\in S}\sigma^{\diamondsuit}_{rkc,\nu}$. The $(rk,c)$ functional
$\Lambda[S]$ defined by \eqref{eq:partial Lambda S} also satisfies, for all $\xi_S$ in the space of $\mathcal{E}_S$,
\begin{align*}
\Lambda[S](\langle b^{\triangle},\prod_{\nu\in S}\varsigma_{*,c,\nu}^{-rk}(b_{\nu})\rangle\cdot\xi_S)=\Lambda[S](\xi_S),\qquad\forall
b\in \SL_c(F_S).
\end{align*}
\end{corollary}

\subsection{Local construction of $(rk,c)$ representations}\label{local theta speh}
Let $F$ be unramified and assume $\mu_{2m}\subset F^*$. In this section we construct
local unramified $(rk,c)$ representations from twisted copies of an exceptional representation.

Recall the representation $\Theta_{rc,m,r,\vartheta}$ of \S~\ref{exceptional}, which is the
unique irreducible unramified subrepresentation of $\mathrm{I}_{\GL_{rc}^{(m,r)}}(\vartheta,\delta_{B_{\GL_{rc}}}^{-1/(2r)})$.
For an unramified character $\chi$ of $T_{\GL_k}$, write $\chi=\chi_1\otimes\ldots\otimes\chi_k$ where
$\chi_1,\ldots,\chi_k$ are unramified quasi-characters of $F^*$. Consider the unramified representation
\begin{align}\label{gen of Suzuki new}
\Theta_{rc,m,r,\vartheta}(\chi)=\Ind_{\widetilde{P}_{((rc)^k)}}^{\GL_{rkc}^{(m,r)}}(
\chi_1\Theta_{rc,m,r,\vartheta}\otimes\ldots\otimes\chi_k\Theta_{rc,m,r,\vartheta}).
\end{align}
Since $\delta_{B_{\GL_{rc}}}(\diag(t_1,\ldots,t_{rc}))=\prod_{j=1}^{rc}|t_j|^{rc-2j+1}$,
$\Theta_{rc,m,r,\vartheta}(\chi)$ is a subrepresentation of
\begin{align}\label{eq:Thete rc m as a subrep}
\mathrm{I}_{\GL_{rkc}^{(m,r)}}(\vartheta,\chi_{\Theta,c}),\qquad
\chi_{\Theta,c}=(\otimes_{j=1}^{rc}\vartheta\chi_1|~|^{-(rc-2j+1)/(2r)})\otimes \ldots \otimes
(\otimes_{j=1}^{rc}\vartheta\chi_k|~|^{-(rc-2j+1)/(2r)}),
\end{align}
where $\otimes_{j=1}^{rc}|~|^{-(rc-2j+1)/(2r)}=|~|^{-(rc-1)}\otimes\ldots\otimes |~|^{-(1-rc)}$,
and a quotient of
\begin{align}\label{eq:Thete rc m as a quotient}
&\Ind_{\widetilde{B}_{\GL_{rkc}}}^{\GL_{rkc}^{(m,r)}}\left(
\Ind_{\widetilde{T}_{\GL_{rkc},*}}^{\widetilde{T}_{\GL_{rkc}}}(
(\otimes_{j=1}^{rc}\vartheta\chi_1|~|^{(rc-2j+1)/(2r)})\otimes \ldots \otimes
(\otimes_{j=1}^{rc}\vartheta\chi_k|~|^{(rc-2j+1)/(2r)}))\right).
\end{align}
Since we will usually use $\chi_{\Theta,1}$, we denote $\chi_{\Theta}=\chi_{\Theta,1}$.
\begin{proposition}\label{proposition:local-semi}
The representation $\Theta_{rc,m,r,\vartheta}(\chi)$ is $(rk,c)$.
\end{proposition}
\begin{proof}
According to Lemma~\ref{lemma:local semi Whittaker and WSS} and Remark~\ref{remark:applicability of Cai results to coverings}, to deduce the vanishing property it is enough to prove that for any composition $\lambda\succsim((rk)^c)$, $J_{N_{\GL_{rkc}},\psi_\lambda}(\Theta_{rc,m,r,\vartheta}(\chi))=0$, and to prove the result on the dimension, it is enough to show $\dim J_{N_{\GL_{rkc}},\psi_{((rk)^c)}}(\Theta_{rc,m,r,\vartheta}(\chi))=1$.

We prove vanishing first. Argue by induction on $k$.
When $k=1$, $\Theta_{rc,m,r,\vartheta}(\chi)=\chi_1\Theta_{rc,m,r,\vartheta}$ and the assumption on $\lambda$ implies $\lambda_i>r$ for some $i$. Then the result follows from \cite[Corollary~3.34]{Cai1}, which was stated for the coverings of \cite{KP}, but the proof is actually simpler for $\GL_{rc}^{(m,r)}$. Indeed, the main technical difficulty handled in \textit{loc. cit.} was the lack of a standard tensor product construction (the ``metaplectic tensor" was used), because for the coverings considered there, direct factors of Levi subgroups do not commute in the cover.

For the inductive step, because the direct factors of $M_{((rk)^c)}$ do commute in $\GL_{rkc}^{(m,r)}$ and the tensor is the standard one, several results from \cite{BZ2} extend immediately to the covering, and the proof now proceeds similarly to
the linear case (\cite[Claim~10]{CFGK2}). Denote $\chi^-=\chi_2\otimes\ldots\otimes\chi_k$. Let $\times$ denote the parabolic induction functor of \cite{BZ2}.
By \cite[4.14]{BZ2} for any composition $\lambda$,
\begin{align*}
J_{N_{\GL_{rkc}},\psi_{\lambda}}(\Theta_{rc,m,r,\vartheta}(\chi))=J_{N_{\GL_{rkc}},\psi_{\lambda}}(\Theta_{rc,m,r,\vartheta}(\chi^-)\times\Theta_{rc,m,r,\vartheta}(\chi_1))
\end{align*}
(normalized Jacquet functors)
is glued from the representations
\begin{align}\label{rep:filtration rho k}
J_{N_{\GL_{r(k-1)c}},\psi_{\lambda'}}(\Theta_{rc,m,r,\vartheta}(\chi^-))\times J_{N_{\GL_{rc}},\psi_{\lambda''}}(\Theta_{rc,m,r,\vartheta}(\chi_1)).
\end{align}
Here $\lambda'$ and $\lambda''$ vary over the compositions of $r(k-1)c$ and $rc$, with
$\lambda_i=\lambda_i'+\lambda_i''$ for all $i$.

Now assuming $\lambda>((rk)^c)$, either $\lambda_i'>r(k-1)$ or $\lambda_i''>r$, thus each constituent \eqref{rep:filtration rho k} vanishes by the inductive hypothesis. This proves $J_{N_{\GL_{kc}},\psi_{\lambda}}(\Theta_{rc,m,r,\vartheta}(\chi))=0$, hence $\mathcal{O}(\Theta_{rc,m,r,\vartheta}(\chi),\beta',\psi)=0$ for any partition $\beta'\succsim((rk)^c)$.

We turn to prove $\dim J_{N_{\GL_{rkc}},\psi_{((rk)^c)}}(\Theta_{rc,m,r,\vartheta}(\chi))=1$. For $k=1$, the Jacquet functor factors through $J_{V_{(r^c)}}$, and $J_{V_{(r^c)}}(\Theta_{rc,m,r,\vartheta})=\delta_{P_{(r^c)}}^{-1/(2r)}\Theta_{r,m,r,\vartheta}\otimes\ldots\otimes\Theta_{r,m,r,\vartheta}$
(this follows as in \cite[Theorem~5.1]{Kable} for double coverings of \cite{KP}, the proof is simpler with the usual tensor). Thus
\begin{align*}
\delta_{P_{(r^c)}}^{1/(2r)}J_{N_{\GL_{rc}},\psi_{(r^c)}}(\Theta_{rc,m,r,\vartheta}(\chi_1))=\chi_1\otimes_{i=1}^cJ_{N_{\GL_{r}},\psi_{(r)}}(\Theta_{r,m,r,\vartheta})
\end{align*}
and the dimension is $1$, since $\Theta_{r,m,r,\vartheta}$ affords a unique Whittaker model (see \S~\ref{exceptional}).

For the general case consider the filtration described above, the only nonzero term \eqref{rep:filtration rho k} corresponds to $\lambda'=((r(k-1))^c)$ and $\lambda''=(r^c)$, and the dimension is $1$ using induction.
\end{proof}

\begin{proposition}\label{proposition:* of Theta chi}
Let ${}^*$ be defined by \eqref{eq:involution b*0}. Then
$\Theta_{rc,m,r,\vartheta}(\chi)^*=\Theta_{rc,m,r,\vartheta}(\chi^*)$, where $\chi^*=\chi_k^{-1}\otimes\ldots\otimes\chi_1^{-1}$.
\end{proposition}
\begin{proof}
Definitions \eqref{eq:involution b*0} and \eqref{gen of Suzuki new} imply
\begin{align*}
\Theta_{rc,m,r,\vartheta}(\chi)^*=\Ind_{\widetilde{P}_{((rc)^k)}}^{\GL_{rkc}^{(m,r)}}(
\chi_k^{-1}\Theta_{rc,m,r,\vartheta}^*\otimes\ldots\otimes\chi_1^{-1}\Theta_{rc,m,r,\vartheta}^*).
\end{align*}
Now the result follows from Proposition~\ref{proposition:* of Theta}.
\end{proof}

We also mention that one expects a local correspondence between $\Theta_{rc,m,r,\vartheta}(\chi)$ and the unramified principal series representation $\Ind_{B_{\GL_k}}^{\GL_k}(\chi^r)$; this was studied in \cite{Suzuki1998} (in the context of \cite{KP}) where a partial correspondence was obtained.

\subsection{The representation $\Theta_{r,m,r,\vartheta}(\chi)$}\label{local theta speh for c=1}
We continue with the assumptions of \S~\ref{local theta speh}: $F$ is unramified and $\mu_{2m}\subset F^*$.
When $c=1$, $(rk,1)$-functionals are Whittaker functionals. The representation $\Theta_{r,m,r,\vartheta}(\chi)$ is unramified, and in this section we develop the Casselman--Shalika formula for the normalized unramified vector in its space.
In the following we use the notation and results of \S~\ref{whittaker functionals} and \S~\ref{Casselman--Shalika formula}.

The representation $\Theta_{r,m,r,\vartheta}(\chi)$ is an unramified quotient of \eqref{eq:Thete rc m as a quotient}, and an unramified subrepresentation of \eqref{eq:Thete rc m as a subrep} (with $c=1$).
The functionals $\Lambda_t$ form a basis of the space of Whittaker functionals on each of these genuine unramified principal series representations.

On the one hand, regarding $\Theta_{r,m,r,\vartheta}(\chi)$ as a quotient, any Whittaker functional on $\Theta_{r,m,r,\vartheta}(\chi)$ extends to \eqref{eq:Thete rc m as a quotient} where it is a linear combination of functionals $\Lambda_t$, but it is not clear which (if any) $\Lambda_t$ factors through this quotient. On the other hand, taking the subrepresentation point of view, each $\Lambda_t$ restricts to
$\Theta_{r,m,r,\vartheta}(\chi)$, but a priori there is no guarantee that a Whittaker functional on \eqref{eq:Thete rc m as a subrep} would be nonzero on a subrepresentation (\eqref{eq:Thete rc m as a subrep} is reducible).

Nonetheless, we will prove that $\Lambda$ ($\Lambda=\Lambda_{I_{rk}}$) is nonzero on the normalized unramified vector $\xi^0$ in the space of
\eqref{eq:Thete rc m as a subrep}. Since $\xi^0$ belongs to the space of $\Theta_{r,m,r,\vartheta}(\chi)$, it follows that the $(rk,1)$-model $\mathcal{W}(\Theta_{r,m,r,\vartheta}(\chi))$ (which is now the Whittaker model) can be realized using $\Lambda$.

Regarding $\Theta_{r,m,r,\vartheta}(\chi)$ as a subrepresentation of \eqref{eq:Thete rc m as a subrep}, denote
\begin{align*}
&\mathbf{x}=(x_1,\ldots,x_k)\in\C^k,\qquad x_i=\chi_i(\varpi^r),
\end{align*}
\begin{align}\label{def:mathbf y}
\mathbf{y}=(q^{(r+1)/2-1}x_1,q^{(r+1)/2-2}x_1,\ldots,q^{(r+1)/2-r}x_1,\ldots,
q^{(r+1)/2-1}x_{k},\ldots,q^{(r+1)/2-r}x_{k})\in\C^{rk},
\end{align}
\begin{align*}
&\mathbf{0}=(0,\ldots,0)\in\Z^{rk},\qquad \mathbf{0'}=(0,\ldots,0)\in\Z^{rk-1}.
\end{align*}
Define for any $s\in\C$,
\begin{align}\label{eq:value of W at unit}
C(s,\mathbf{x})=
\prod_{\alpha\in\Phi_{k}^+}(1-q^{-s}\mathbf{x}_{\alpha}).
\end{align}
\begin{theorem}\label{theorem:CS formula for rk, c}
$W_{\mathbf{0}}(\mathbf{0},\vartheta,\chi_{\Theta})=\prod_{j=1}^{r}C(j,\mathbf{x})$.
\end{theorem}
\begin{proof}
For $m=1$ this was proved in \cite{CS2}, assume $m>1$.
Since in general positive roots are spanned by the simple roots with positive integer coefficients, the product of factors $C(j,\mathbf{x})$ in the statement of the lemma belongs to the unique factorization domain
$\C[\{\mathbf{x}_{(i,i+1)}\}_{1\leq i<k}]$. It can also be written as the product of distinct irreducible polynomials
\begin{align}\label{eq:irred poly factors for identity at 1}
1-q^{-l}\prod_{i_0=i}^{j-1}{\mathbf{x}_{(i_0,i_0+1)}},\qquad i<j, \qquad 1\leq l\leq r.
\end{align}
The total degree of $\prod_{j=1}^{r}C(j,\mathbf{x})$ is $rk(k^2-1)/6$ and the monomial attaining this degree is
\begin{align}\label{desired monomial}
(-1)^{rk(k-1)/2}q^{-r(r+1)k(k-1)/4}\prod_{i=1}^{k-1}\mathbf{x}_{(i,i+1)}^{ri(k-i)}.
\end{align}

According to the formula of McNamara \cite[\S~8]{McNamara2},
$\Lambda(t_{\mathbf{a}}\cdot \xi^0)\in\C[\{\mathbf{x}_{(i,i+1)}\}_{1\leq i<k}]$.
We will show that each factor \eqref{eq:irred poly factors for identity at 1} divides
$W_{\mathbf{0}}(\mathbf{0},\vartheta,\chi_{\Theta})$, which implies that $\prod_{j=1}^{r}C(j,\mathbf{x})$ also divides $W_{\mathbf{0}}(\mathbf{0},\vartheta,\chi_{\Theta})$.
Then, it is sufficient to prove that the total degree of $W_{\mathbf{0}}(\mathbf{0},\vartheta,\chi_{\Theta})$ is
also $rk(k^2-1)/6$ and the corresponding monomial is \eqref{desired monomial}.

To prove that \eqref{eq:irred poly factors for identity at 1} divides $W_{\mathbf{0}}(\mathbf{0},\vartheta,\chi_{\Theta})$, it is enough to show
that when $\mathbf{x}_{(i,j)}=q^l$ and the variables $x_{i_0}$ for $i_0\in \{1,\ldots,k\}-\{i,j\}$ are in ``general position",
$W_{\mathbf{0}}(\mathbf{0},\vartheta,\chi_{\Theta})=0$. (E.g., if both $(1-q^{-1}\mathbf{x}_{(1,2)})$ and $(1-q^{-1}\mathbf{x}_{(2,3)})$ divide
$W_{\mathbf{0}}(\mathbf{0},\vartheta,\chi_{\Theta})$, then already $W_{\mathbf{0}}(\mathbf{0},\vartheta,\chi_{\Theta})=0$ when $\mathbf{x}_{(1,2)}=\mathbf{x}_{(2,3)}=q$
so that $\mathbf{x}_{(1,3)}=q^2$, yet $(1-q^{-2}\mathbf{x}_{(1,3)})$ might not divide $W_{\mathbf{0}}(\mathbf{0},\vartheta,\chi_{\Theta})$; but if
$W_{\mathbf{0}}(\mathbf{0},\vartheta,\chi_{\Theta})=0$ when $\mathbf{x}_{(1,3)}=q^2$ and $x_{2}$ is arbitrary, then
$(1-q^{-2}\mathbf{x}_{(1,3)})$ divides $W_{\mathbf{0}}(\mathbf{0},\vartheta,\chi_{\Theta})$.)

Assume $\mathbf{x}_{(i,j)}=q^l$ and suppose $W_{\mathbf{0}}(\mathbf{0},\vartheta,\chi_{\Theta})\ne0$, we will arrive at a contradiction.
The representation $\Theta_{r,m,r,\vartheta}(\chi)$ is a subrepresentation of \eqref{eq:Thete rc m as a subrep}, namely of
\begin{align*}
&\Ind_{\widetilde{B}_{\GL_{rk}}}^{\GL_{rk}^{(m,r)}}\left(
\Ind_{\widetilde{T}_{\GL_{rk},*}}^{\widetilde{T}_{\GL_{rk}}}(
(\otimes_{j=1}^{r}\vartheta\chi_1|~|^{-(r-2j+1)/(2r)})\otimes \ldots \otimes
(\otimes_{j=1}^{r}\vartheta\chi_k|~|^{-(r-2j+1)/(2r)}))\right).
\end{align*}
Since $x_i=x_jq^l$, $\chi_i=\chi_j|~|^{-l/r}$ ($x_i=\chi_i(\varpi^r)$).
Thus the $l$-th genuine character of $\widetilde{A}$ in the block of $\vartheta\chi_i\delta_{B_{\GL_r}}^{-1/(2r)}$, i.e.,  $\vartheta\chi_i|~|^{-(r-2l+1)/(2r)}=\vartheta\chi_j|~|^{-(r+1)/(2r)}$, and the $r$ genuine characters of $\widetilde{A}$ in the block of $\vartheta\chi_j\delta_{B_{\GL_r}}^{-1/(2r)}$ constitute the inducing character
$\vartheta\chi_j|~|^{-1/(2r)}\delta_{B_{\GL_{r+1}}}^{-1/(2r)}$ of $\chi_j|~|^{-1/(2r)}\Theta_{r+1,m,r,\vartheta}$.

Consider the unramified principal series representation obtained by permuting
the aforementioned character $\vartheta\chi_i|~|^{-(r-2l+1)/(2r)}$ to the right. Using transitivity of induction, it can be written in the form \begin{align}\label{eq:sub with pi1 delta and pi2}
&\Ind_{\widetilde{P}_{(a,r+1,rk-r-1-a)}}^{\GL_{rk}^{(m,r)}}
\left(\pi_1\otimes\chi_j|~|^{-1/(2r)}\mathrm{I}_{\GL_{r+1}^{(r,k)}}(\vartheta,\delta_{B_{\GL_{r+1}}}^{-1/(2r)})\otimes\pi_2\right),
\end{align}
for suitable genuine unramified principal series representations $\pi_1$ and $\pi_2$. Then \eqref{eq:sub with pi1 delta and pi2} contains
\begin{align}\label{eq:sub with pi1 theta and pi2}
&\Ind_{\widetilde{P}_{(a,r+1,rk-r-1-a)}}^{\GL_{rk}^{(m,r)}}
\left(\pi_1\otimes\chi_j|~|^{-1/(2r)}\Theta_{r+1,m,r,\vartheta}\otimes\pi_2\right),
\end{align}
which is again unramified. Since $\Theta_{r+1,m,r,\vartheta}$ does not admit any Whittaker functional (see \S~\ref{exceptional}),
the representation \eqref{eq:sub with pi1 theta and pi2} does not admit one either (this follows from \cite[Theorem~5.2]{BZ2}; see also \cite{Banks1998}).
Therefore any Whittaker functional on \eqref{eq:sub with pi1 delta and pi2} must vanish on the normalized unramified vector $\xi^0$ in the space of \eqref{eq:sub with pi1 delta and pi2}, because it is contained in \eqref{eq:sub with pi1 theta and pi2}.

However, if $M(w)$ is the intertwining operator from \eqref{eq:sub with pi1 delta and pi2}, regarded as an unramified principal series
representation, back to \eqref{eq:Thete rc m as a subrep}, then we claim that $M(w)$ is holomorphic, and also nonzero on $\xi^0$.
To see this, decompose $M(w)$ into rank-$1$ intertwining operators which now take the form
\begin{align*}
M(w_{\alpha}):&\mathrm{I}_{\GL_2^{(m,r)}}(\vartheta,\chi_{i_0}|~|^{-(r-2l_0+1)/(2r)}\otimes\chi_j|~|^{-(r+1)/(2r)})\rightarrow \\& \mathrm{I}_{\GL_2^{(m,r)}}(\vartheta,\chi_j|~|^{-(r+1)/(2r)}\otimes\chi_{i_0}|~|^{-(r-2l_0+1)/(2r)}),
\end{align*}
where either $i<i_0<j$ and $l_0>0$, or $i_0=i$ but then $l_0>l$. The inducing character is regular: for $i<i_0$, this is because we are assuming
$x_{i_0}$ is in general position; if $i_0=i$, this is because $q^{-l_0}\mathbf{x}_{(i_0,j)}=q^{-l_0+l}$ and $l_0>l$. Hence
each $M(w_{\alpha})$ is holomorphic. If ${\xi'}^0$ (resp., ${\xi''}^0$) is the normalized unramified vector in the domain (resp., image) of $M(w_{\alpha})$,
\eqref{eq:GK formula for GL} implies
\begin{align*}
M(w_{\alpha}){\xi'}^0=\frac{1-q^{-1-l_0}\mathbf{x}_{(i_0,j)}}{1-q^{-l_0}\mathbf{x}_{(i_0,j)}}{\xi''}^0.
\end{align*}
This is nonzero (and well defined), again for $i<i_0$ because $x_{i_0}$ is in general position, and for $i_0=i$
the quotient equals $(1-q^{-1-l_0+l})/(1-q^{-l_0+l})$ (with $l_0>l$). Hence $M(w)$ is well defined, so that the map
$\xi\mapsto\Lambda(M(w)\xi)$ (if not identically zero) is a Whittaker functional on \eqref{eq:sub with pi1 delta and pi2}, and also
$M(w)\xi^0\ne0$. Now if $W_{\mathbf{0}}(\mathbf{0},\vartheta,\chi_{\Theta})\ne0$, $\Lambda(M(w)\xi^0)\ne0$ - contradiction.

We turn to compute the total degree and relevant coefficient of $W_{\mathbf{0}}(\mathbf{0},\vartheta,\chi_{\Theta})$ using the results of
\cite{BBF2011,McNamara2} on Gelfand--Tsetlin patterns.
Recall that a (strict) Gelfand--Tsetlin pattern $\mathfrak{T}$ is a triangular array $\{a_{i,j}\}$ of non-negative integers,
such that $a_{i,j}>a_{i,j+1}$ and $a_{i,j}\geq a_{i+1,j}\geq a_{i,j+1}$ for all $i,j$ such that all entries exist (we enumerate the elements such that the leftmost entry at row $i$ is $a_{i,1}$). For each monomial of $W_{\mathbf{0}}(\mathbf{0},\vartheta,\chi_{\Theta})$, the coefficient is read off a finite set of Gelfand--Tsetlin patterns.
We briefly adapt the exposition of \cite{McNamara2} and \cite{BBCFG} to the case at hand, and in particular plug in the parameters $\mathbf{y}$ of our principal series representation.

Here $1\leq i,j\leq rk$, and since we are evaluating the Whittaker function at the identity, the first row is $a_{1,j}=rk-j$, $1\leq j\leq rk$. The entries at row $i$ are $a_{i,1},\ldots,a_{i,rk-i+1}$.

We claim that the only patterns to consider are those where the differences
\begin{align}\label{eq:di}
d_{i}=\sum_{j=1}^{rk-i+1}a_{i,j}-rk+i-1+j
\end{align}
are divisible by $r$, for all $i\geq1$. This condition did not appear in \cite{McNamara2,BBCFG}, we explain how to obtain it. Since $d_1=0$, assume $i>1$.
The value $d_i$ is $k_{i-1}(\mathfrak{T})$ of \cite{McNamara2} and \cite[\S~3.3]{BBCFG} (in those works the rows were enumerated starting from $0$).
Given an element $\mathrm{m}$ in the root lattice of $\GL_{rk}$, denote its representation
as a combination of simple roots by $\mathrm{m}_{\Delta}$. To deduce $d_i\equiv0\,(r)$ we relate between \cite{McNamara2} and \cite[\S~3.3]{BBCFG}, freely using their notation. Denote $\mathbf{k}(\mathfrak{T})=(k_1(\mathfrak{T}),\ldots,k_{rk-1}(\mathfrak{T}))$. Comparing coefficients between
\cite[(8), Theorem~1]{BBCFG} and \cite[Theorem~8.6]{McNamara2} we see that for a given $\mathbf{k}\in\N^{rk-1}$, $H(\varpi^{\mathbf{k}};\lambda)$ is the sum over $\mathfrak{T}$ such that $\mathbf{k}(\mathfrak{T})=\mathbf{k}$, and
$H(\varpi^{\mathbf{k}};\lambda)$ is also equal to the sum of products
$\sum_{(\mathrm{i},\mathrm{m})\in B(\lambda+\rho)}\prod_{\alpha\in\Phi_{rk}^+}w(\mathrm{m},\alpha)$ where
$\mathrm{m}\in\N^{|\Phi_{rk}^+|}$ and $\mathrm{m}_{\Delta}=\mathbf{k}$. Each summand here is obtained from an integral
$\int_{C^{\mathrm{i}}_{\mathrm{m}}}\xi^0(u)\psi(u)du$ where $C^{\mathrm{i}}_{\mathrm{m}}\subset N_{\GL_{rk}}^-$ (called a ``cell").
If $\mathrm{m}\notin r\N^{|\Phi_{rk}^+|}$, then for any element from $C^{\mathrm{i}}_{\mathrm{m}}$, in its Iwasawa decomposition given in \cite[Theorem~4.5, Proposition~4.6]{McNamara2} the torus part does not belong to $T_{\GL_{rk},*}$. In this case $\xi^0$ vanishes on $C^{\mathrm{i}}_{\mathrm{m}}$, by Proposition~\ref{proposition:support of unr is in T*}. Hence only cells $C^{\mathrm{i}}_{\mathrm{m}}$ with $\mathrm{m}\in r\N^{|\Phi_{rk}^+|}$ contribute,
so that $\mathbf{k}\in r\N^{rk-1}$, i.e., $r|d_i$ for all $i\leq rk$.

The resulting monomial for such a pattern is
\begin{align}\label{monomial}
C\left(\,\prod_{l=1}^{k}\,\prod_{i=r(l-1)+2}^{rl}\,q^{d_i/r}\right)\left(\prod_{l=1}^{k-1}\,q^{-(r-1)d_{rl+1}/r}\right)\left(\prod_{l=1}^{k-1}
\,\mathbf{x}_{(l,l+1)}^{d_{rl+1}/r}\right),
\end{align}
where $C$ is a constant to be determined below, depending on certain products of Gauss sums.

Since we are looking for the monomial of highest total degree, we consider patterns with the maximal entries $a_{rl+1,j}$
for $1\leq l\leq k-1$ and $1\leq j\leq rk-(rl+1)+1=r(k-l)$ possible. We will show that there exists exactly one pattern satisfying this condition, with a nonzero contribution. This will imply that it is the only pattern for the highest monomial, and in turn, determine the total degree and the coefficient.

First note that $a_{rl+1,1}\leq rk-1$ for all $1\leq l\leq k-1$, because $a_{1,1}=rk-1$. Consider patterns such that
$a_{rl+1,1}=rk-1$ for all $1\leq l\leq k-1$. Since the numbers are strictly decreasing, $d_{rl+1}$ will attain its maximum when
\begin{align}\label{eq:rl+1 j}
a_{rl+1,j}=rk-j,\qquad \forall 1\leq j\leq r(k-l).
\end{align}
Then $d_{rl+1}=r^2l(k-l)$, and the total degree of the monomial is $rk(k^2-1)/6$ as required.

Since at each row $rl+1$ the entries are the maximal, we see that $a_{i,j}=a_{rl+1,j}$ for all pairs $(i,j)$ in the set
\begin{align*}
\mathcal{G}=\{(i,j):1\leq l\leq k-1,\quad r(l-1)+1<i\leq rl+1,\quad 1\leq j\leq r(k-l)\}.
\end{align*}
The remaining entries to determine are $a_{i,j}$ where $(i,j)$ belongs to
\begin{align*}
\mathcal{B}=\{(i,j):1\leq l\leq k,\quad r(l-1)+1<i<rl+1,\quad r(k-l)<j\leq rk-i+1\}.
\end{align*}
This is addressed by the following claim, proved below.
\begin{claim}\label{claim:a bc in the remaining range}
For $(i,j)\in\mathcal{B}$, $a_{i,j}=a_{i-1,j+1}$.
\end{claim}
We have shown that there is a unique pattern for the monomial of
$\prod_{l=1}^{k-1}\,\mathbf{x}_{(l,l+1)}^{rl(k-l)}$. It remains to compute its coefficient.

For $(i,j)\in\mathcal{B}$ with $r(l-1)+1<i<rl+1$, a repeated application of Claim~\ref{claim:a bc in the remaining range}
shows
\begin{align}\label{eq:a_bc in B}
a_{i,j}=a_{r(l-1)+1,j+i-r(l-1)-1}=r(k+l-1)-i-j+1.
\end{align}
Now computing $d_i$ as two summations, one over $1\leq j\leq r(k-l)$ for which $a_{i,j}=a_{rl+1,j}=rk-j$, the other with $j>r(k-l)$ where we use
\eqref{eq:a_bc in B}, we obtain
\begin{align*}
d_i=r(k-l)(i-1)+r(l-1)(rl-i+1).
\end{align*}
Thus
\begin{align*}
\sum_{l=1}^{k}\,\sum_{i=r(l-1)+2}^{rl}d_i=r^2(r-1)k(k^2-1)/6=\sum_{l=1}^{k-1}(r-1)d_{rl+1}
\end{align*}
and we conclude that the powers of $q$ vanish in \eqref{monomial}. It remains to compute $C$.

For $i>1$, put
\begin{align*}
h_{i,j}=\sum_{v=j}^{rk-i+1}a_{i,v}-a_{i-1,v+1}.
\end{align*}
Recall the Gauss sum $\mathfrak{g}(a)$ given by \eqref{eq:Gauss Sum}, where $\mathfrak{g}(a)=-q^{-1}$ if
$a\equiv0\,(m)$ otherwise $\mathfrak{g}(a)\mathfrak{g}(-a)=q^{-1}$.
The coefficient $C$ is the product of $\mathfrak{g}(2h_{i,j})$ over all $(i,j)$ such that $a_{i,j}=a_{i-1,j}$ (in particular, $i>1$).
This follows from \cite[\S~3.3]{BBCFG} and note that the $2$-cocycle used there for a covering of $\GL_{rk}$ is our $\sigma_{rk}$
(see \cite[\S~3.3, (4)]{BBCFG} and \eqref{eq:sigma on torus of GL}); here we use $\sigma_{rk}^{\diamondsuit}$ and since
$\sigma_{2}(\left(\begin{smallmatrix}t\\&t^{-1}\end{smallmatrix}\right),\left(\begin{smallmatrix}t'\\&{t'}^{-1}\end{smallmatrix}\right))=(t,t')_m^{-1}$ and
$\sigma_{2}^{\diamondsuit}(\left(\begin{smallmatrix}t\\&t^{-1}\end{smallmatrix}\right),\left(\begin{smallmatrix}t'\\&{t'}^{-1}\end{smallmatrix}\right))=(t,t')_m^{-2}$,
the function $g(b)$ of \cite[\S~3.3, (9)]{BBCFG} which is $\mathfrak{g}(b)$ here (the normalization of the measure is the same) is replaced by $\mathfrak{g}(2b)$.
Also for the pattern at hand, the set of such pairs $(i,j)$ is $\mathcal{G}$.

Let $(i,j)\in\mathcal{G}$ with $r(l-1)+1<i\leq rl+1$. Then for
$v\leq r(k-l)$, $a_{i,v}=rk-v$. Claim~\ref{claim:a bc in the remaining range}
shows $a_{i,v}=a_{i-1,v+1}$ when $(i,v)\in\mathcal{B}$ (i.e., $v>r(k-l)$). Also
since $(i-1,r(k-l)+1)\in\mathcal{B}$, \eqref{eq:a_bc in B} implies
$a_{i-1,r(k-l)+1}=r(2l-1)-i+1$.
Therefore
\begin{align*}
h_{i,j}=r(k-2l+1)+i-1-j\equiv i-1-j\quad(r).
\end{align*}
Then for each $1\leq j\leq r(k-l)$,
\begin{align*}
\prod_{i=r(l-1)+2}^{rl+1}\mathfrak{g}(2h_{i,j})=\begin{cases}-q^{-(r+1)/2}&\text{odd $r$},\\-q^{-(r+1)/2}q^{1/2}\mathfrak{g}(r)&\text{even $r$}.\end{cases}
\end{align*}
Note that $2a\equiv 2b\,(m)$ if and only if $a\equiv b\,(r)$, and if $r$ is even, $-r\equiv r\,(m)$ whence
$(q^{1/2}\mathfrak{g}(r))^r=1$.
When we multiply over $r(k-l)$ columns and $1\leq l\leq k-1$ rows we see that $C$ is the coefficient appearing in \eqref{desired monomial} (!).
\end{proof}
\begin{proof}[Proof of Claim~\ref{claim:a bc in the remaining range}]
Fix $1\leq l\leq k$.
Let $r(l-1)+1<i<rl+1$.
We split the sum $d_i$ (given by \eqref{eq:di}) into two summations, one for $1\leq j\leq r(k-l)$, the other for $j>r(k-l)$. For the first sum, by definition
$(i,j)\in\mathcal{G}$ so that $a_{i,j}=a_{rl+1,j}$, which is $rk-j$ by \eqref{eq:rl+1 j}. Hence this sum equals
$r(k-l)(i-1)$ which is divisible by $r$, so that
\begin{align*}
d_i\equiv\sum_{j=r(k-l)+1}^{rk-i+1}a_{i,j}-rk+i-1+j\equiv0\quad(r).
\end{align*}

Put $v=i-r(l-1)-1>0$.
According to \eqref{eq:rl+1 j}, we have in particular
\begin{align*}
a_{r(l-1)+1,j+v}=rk-j-v,\qquad\forall\quad r(k-l)<j \leq r(k-(l-1))-v.
\end{align*}
Then for each $r(k-l)< j\leq rk-i+1$,
\begin{align*}
a_{i,j}-rk+i-1+j=a_{i,j}-a_{r(l-1)+1,j+v}+i-1-v=a_{i,j}-a_{r(l-1)+1,j+v}+r(l-1).
\end{align*}
Hence
\begin{align}\label{eq:di before using induction}
d_{i}\equiv\sum_{j=r(k-l)+1}^{rk-i+1}a_{i,j}-a_{r(l-1)+1,j+v}\quad(r),
\end{align}
and note that $a_{r(l-1)+1,j+v}$ is defined, because in the summation $j+v\leq rk-i+1+v=rk-r(l-1)$.
Since $i>r(l-1)+1$, this is a sum of less than $r$ integers which vanishes modulo $r$.

We now use induction on $i$ to show $a_{i,j}=a_{i-1,j+1}$ for all $r(k-l)<j\leq rk-i+1$. For the base case
$i=r(l-1)+2$ and $v=1$. Since $rk-i+1<r(k-(l-1))$, \eqref{eq:rl+1 j} implies that for any $j$ in the range, $a_{r(l-1)+1,j}=a_{r(l-1)+1,j+1}+1$, hence we must have
$a_{i,j}=a_{r(l-1)+1,j}$ or $a_{r(l-1)+1,j+1}$ (according to the conditions defining the Gelfand--Tsetlin pattern). Thus
\eqref{eq:di before using induction} is a sum of less than $r$ integers which are either $0$ or $1$, and is divisible by
$r$. We conclude $a_{i,j}=a_{r(l-1)+1,j+1}$ for all $r(k-l)<j\leq rk-i+1$.

For $r(l-1)+2<i<rl+1$, $a_{i,j}-a_{r(l-1)+1,j+(i-r(l-1)-1)}$ equals
\begin{align*}
a_{i,j}+\sum_{u=1}^{i-r(l-1)-2}(-a_{i-u,j+u}+a_{i-u,j+u})-a_{r(l-1)+1,j+(i-r(l-1)-1)}
=a_{i,j}-a_{i-1,j+1}
\end{align*}
where the second equality follows from the inductive hypothesis. Therefore \eqref{eq:di before using induction} becomes
\begin{align*}
d_{i}\equiv\sum_{j=r(k-l)+1}^{rk-i+1}a_{i,j}-a_{i-1,j+1}\quad(r).
\end{align*}
The inductive hypothesis also implies $a_{i-1,j}=a_{r(l-1)+1,j+v-1}$ for any $j$ in the range, and because $j+v-1<r(k-(l-1))$ we can use \eqref{eq:rl+1 j} to deduce
$a_{i-1,j}=a_{i-1,j+1}+1$.
Then as in the base case we have $a_{i,j}=a_{i-1,j}$ or $a_{i-1,j+1}$ for all $j$, and the vanishing
of $d_i$ modulo $r$ implies $a_{i,j}=a_{i-1,j+1}$. In this manner we argue for each $l$.
\end{proof}
\begin{example}
The unique Gelfand--Tsetlin pattern when $r=k=2$ ($m=4$) is
\begin{align*}
\begin{array}{cccccccccc}
3&&2&&1&&0\\&3&&2&&0\\&&3&&2\\&&&2
\end{array}.
\end{align*}
Here $\mathcal{G}=\{(2,1),(2,2),(3,1),(3,2)\}$, $h_{2,1}=2$, $h_{2,2}=1$,
$h_{3,1}=3$, $h_{3,2}=2$ and $C=\prod_{{i,j}\in\mathcal{G}}\mathfrak{g}(2h_{i,j})=q^{-3}$.
When $r=k=3$ ($m=3$ or $6$) the pattern is
\begin{align*}
\begin{array}{ccccccccccccccccc}
8&&7&&6&&5&&4&&3&&2&&1&&0\\
&8&&7&&6&&5&&4&&3&&1&&0\\
&&8&&7&&6&&5&&4&&3&&0\\
&&&8&&7&&6&&5&&4&&3\\
&&&&8&&7&&6&&4&&3\\
&&&&&8&&7&&6&&3\\
&&&&&&8&&7&&6\\
&&&&&&&7&&6\\
&&&&&&&&6
\end{array}.
\end{align*}
Here $d_4=d_7=18$; $\mathcal{B}=\{(2,7),(2,8),(3,7),(5,4),(5,5),(6,4),(8,1),(8,2),(9,1)\}$;
the contribution to $C$ of rows $2,3$ and $4$ is $q^{-4}$ each; and rows $5,6$ and $7$ contribute $-q^{-2}$ each, e.g., the product for the $7$-th row is $\mathfrak{g}(2\cdot 5)\cdot\mathfrak{g}(2\cdot 4)\cdot(-q^{-1})$. Hence the monomial of highest total degree is $-q^{-18}\mathbf{x}_{(1,2)}^6\mathbf{x}_{(2,3)}^6$.
\end{example}

Define a function $\mathscr{G}:W_{\GL_{rk-1}}\times \Z^{rk-1}\rightarrow\C^*$ as follows.
\begin{enumerate}[leftmargin=*]
\item If $w$ is the identity element, $\mathscr{G}(w,\mathbf{a})=1$.
\item For $w=w_{\alpha}$ with $\alpha=(i,i+1)$, $\mathscr{G}(w_{\alpha},\mathbf{a})=q^{-a_{i+1}+a_{i}+1}\mathfrak{g}(-2(a_{i+1}-a_{i}-1))$.
\item If $\ell(w_1w_2)=\ell(w_1)+\ell(w_2)$, $\mathscr{G}(w_1w_2,\mathbf{a})=\mathscr{G}(w_1,w_2[\mathbf{a}])\mathscr{G}(w_2,\mathbf{a})$.
\end{enumerate}
This is similar to Suzuki \cite[\S~3.2]{Suzuki1997}, but the definition of the action $w[\mathbf{a}]$ is different.

For the rest of this section assume
\begin{align}\label{assumption:local C product for nonvanishing of Whittaker}
\prod_{j=1}^{r}C(j,\mathbf{x})\ne0.
\end{align}
Then according to Theorem~\ref{theorem:CS formula for rk, c}, since $\Lambda(\xi^0)=W_{\mathbf{0}}(\mathbf{0},\vartheta,\chi_{\Theta})\ne0$,
$\mathcal{W}(\Theta_{r,m,r,\vartheta}(\chi))$ can be realized using $\Lambda$.
Now define the normalized unramified Whittaker function $W^0$ in this model by
\begin{align*}
W^0(g)=\frac{\Lambda(g\cdot\xi^0)}{\Lambda(\xi^0)}.
\end{align*}
This function is determined by its values on $t_{\mathbf{b}}$ and the notation implies
\begin{align*}
W^0(t_{\mathbf{b}})=\frac{W_{\mathbf{0}}(\mathbf{b},\vartheta,\chi_{\Theta})}{W_{\mathbf{0}}(\mathbf{0},\vartheta,\chi_{\Theta})}.
\end{align*}

Let $p_l(\mathbf{x})$ be the $l$-th complete symmetric polynomial in $\mathbf{x}$.
\begin{theorem}\label{theorem:Whittaker on x,I}
For all integers $l\geq0$ and $a\in F^*$,
\begin{align*}
W^0(\langle\diag(a^{rl},I_{rk-1}),1\rangle)=q^{(-lr(rk-1)/2)+l(r-1)/2}\vartheta(a^{rl})p_l(\mathbf{x}).
\end{align*}
\end{theorem}
\begin{remark}
The factor $p_l(\mathbf{x})$ is independent of the choice of uniformizer $\varpi$.
\end{remark}
\begin{remark}
This is a special phenomenon which reflects the deep connection
between $\Theta_{r,m,r,\vartheta}(\chi)$ and $\Ind_{B_{\GL_k}}^{\GL_k}(\chi^r)$ (cf. \S~\ref{local theta speh}) observed
by Suzuki \cite[\S~5.5]{Suzuki1998}, who proved a stronger version of Theorem~\ref{theorem:Whittaker on x,I} but when $k\leq2$, for the coverings of \cite{KP}. The value $p_l(\mathbf{x})$ is the one obtained from the formula of Shintani \cite{S} for the normalized unramified Whittaker function of $\Ind_{B_{\GL_k}}^{\GL_k}(\chi^r)$ on $\diag(\varpi^{l},I_{k-1})$.
\end{remark}
\begin{proof}
For $m=1$ this follows from \cite{S}, assume $m>1$.
Fix a uniformizer $\varpi$ and write $a=\varpi^{rl}o$ with $|o|=1$. Then by \eqref{eq:Nice GL $2$-cocycle on torus},
\begin{align*}
\langle\diag(a^{rl},I_{rk-1}),1\rangle
=t_{(rl,\mathbf{0'})}\langle \diag(o,I_{rk-1}),(\varpi,o)^{rl}_m\rangle.
\end{align*}
Since $W^0$ is unramified and $\eta_{d}^{\diamondsuit}(\diag(o,I_{rk-1}))=1$,
\begin{align*}
W^0(\langle\diag(a^{rl},I_{rk-1}),1\rangle)=
(\varpi,o)^{rl}_m W^0(t_{(rl,\mathbf{0'})}).
\end{align*}
At the same time $\vartheta(a^{rl})=(\varpi,o)^{-rl}_m\vartheta(\varpi^{rl})=(\varpi,o)^{rl}_m\vartheta(\varpi^{rl})$, because $\vartheta$ is a genuine unramified character of $\widetilde{A}$. Therefore having fixed $\varpi$, the statement of the theorem is equivalent to
\begin{align*}
W^0(t_{(rl,\mathbf{0'})})=q^{(-lr(rk-1)/2)+l(r-1)/2}\vartheta(\varpi^{rl})p_l(\mathbf{x}).
\end{align*}
When $r=1$ this follows from \cite{S}. Specifically, by Proposition~\ref{proposition:sigma * and sigma on GLd}
$\GL_{k}^{(2,1)}$ is isomorphic to the $2$-fold covering with the $2$-cocycle given by $(\det,\det)_2$, and
this isomorphism takes $t_{(rl,\mathbf{0'})}$ to itself (because $(x,x)_m=1$, see the proof of Proposition~\ref{proposition:sigma * and sigma on SLc}). Then we can use the bijection $\tau\mapsto\vartheta\otimes\tau$ between representations of $\GL_{k}$ and genuine representations of $\GL_{k}^{(2,1)}$, where $\vartheta\otimes\tau(\langle g,\epsilon\rangle)=\epsilon\vartheta(\det g)\tau(g)$
(the isomorphism of Proposition~\ref{proposition:sigma iota and sigma cohomologous} is implicit). Thus we can assume $r>1$ (the arguments below remain valid for $r=1$, but most of them trivialize).

Recall
\begin{align*}
\mathbf{y}=(q^{(r+1)/2-1}x_1,q^{(r+1)/2-2}x_1,\ldots,q^{(r+1)/2-r}x_1,\ldots,
q^{(r+1)/2-1}x_{k},\ldots,q^{(r+1)/2-r}x_{k})\in\C^{rk},
\end{align*}
it is uniquely defined given $\chi_{\Theta}$.
According to \eqref{eq:iterative CS formula of Suzuki} and with the same notation,
\begin{align}\label{eq:iterative CS formula of Suzuki applied}
W_{\mathbf{0}}((rl,\mathbf{0'}),\vartheta,\chi_{\Theta})
&=q^{-lr(rk-1)}
\sum_{i=0}^{rk-1}\,
\prod_{j=1}^{rk-i-1}\,\frac{1-q^{-1}\mathbf{y}_{(j,rk-i)}}{1-\mathbf{y}_{(j,rk-i)}}
\\\nonumber&\quad
\sum_{\mathbf{e'}\in r\Z^{rk-1}\backslash \Z^{rk-1}}
\tau_{\mathbf{0},(\mathbf{e'}^*,-rl)}(\omega_i,\vartheta,{}^{\omega_i^{-1}}\chi_{\Theta})
W_{\mathbf{e'}}(\mathbf{0'},\vartheta,\chi_{\Theta}'[rk-i]).
\end{align}

By McNamara \cite[\S~8]{McNamara2},
the l.h.s.~ of \eqref{eq:iterative CS formula of Suzuki applied}
belongs to $\C[\{\mathbf{x}_{(i,i+1)}\}_{1\leq i<k}]$, hence we can prove the result under the assumption that $\mathbf{x}$ is in general position. Also $W_{\mathbf{e'}}(\mathbf{0'},\vartheta,\chi_{\Theta}'[rk-i])\in\C[\{\mathbf{x}_{(i,i+1)}\}_{1\leq i<k}]$ (\cite[\S~8]{McNamara2}).
The coefficients $\tau_{\ldots}(\ldots)$ may a priori have poles,
but by \eqref{eq:tau on maximal abelian}--\eqref{eq:tau for simple ref} and
Proposition~\ref{proposition:coefficients tau t t' 1 2}, when $\mathbf{y}_{(i,j)}\ne1$ for all $i<j$, they are holomorphic
(poles can only occur from the denominator in \eqref{eq:tau 1 nonzero}). Hence if the product over $j$ vanishes, the summand is zero. Looking at $\mathbf{y}$, we see that the product vanishes
unless $i=r(k-u)-1$ for $0\leq u\leq k-1$. For such $i$, we note that
$\mathbf{y}_{rk-i}=x_{u+1}q^{(r-1)/2}$ and $\omega_{i}(rk)=ru+1$. Then by \eqref{eq:tau on maximal abelian},
\begin{align*}
\tau_{\mathbf{0},(\mathbf{e'}^*,-rl)}(\omega_{i},\vartheta,{}^{\omega_{i}^{-1}}\chi_{\Theta})=q^{lr(rk-1)/2+l(r-1)/2}
\vartheta(\varpi^{-rl})x_{u+1}^{l}
\tau_{\mathbf{0},(\mathbf{e'}^*,0)}(\omega_{i},\vartheta,{}^{\omega_{i}^{-1}}\chi_{\Theta}).
\end{align*}
Note that $\vartheta$ is either trivial on $\varpi^{-rl}$ if $m\not\equiv2\,(4)$, otherwise $\vartheta(\varpi^{-rl})=\vartheta(\varpi^{rl})$
by \eqref{eq:some props of gamma psi'}. Hence we can write \eqref{eq:iterative CS formula of Suzuki applied} in the form
\begin{align}\label{eq:partial W formula}
&W_{\mathbf{0}}((rl,\mathbf{0'}),\vartheta,\chi_{\Theta})=q^{-lr(rk-1)/2+l(r-1)/2}
\vartheta(\varpi^{rl})\sum_{u=0}^{k-1}\,x_{u+1}^{l}\Gamma_u,
\end{align}
where $\Gamma_u=\Gamma_u(\vartheta,\chi_{\Theta})$ equals
\begin{align}\label{eq:exp for Gamma u}
\prod_{j=1}^{ru}\,\frac{1-q^{-1}\mathbf{y}_{(j,ru+1)}}{1-\mathbf{y}_{(j,ru+1)}}
\sum_{\mathbf{e'}\in r\Z^{rk-1}\backslash \Z^{rk-1}}
\tau_{\mathbf{0},(\mathbf{e'}^*,0)}(\omega_{r(k-u)-1},\vartheta,{}^{\omega_{r(k-u)-1}^{-1}}\chi_{\Theta})
W_{\mathbf{e'}}(\mathbf{0'},\vartheta,\chi_{\Theta}'[ru+1]).
\end{align}
If $k=1$, there is only one coefficient $\Gamma_0$, which by \eqref{eq:partial W formula} is equal to
$W_{\mathbf{0}}(\mathbf{0},\vartheta,\chi_{\Theta})$. Thus in this case the theorem follows immediately from
\eqref{eq:partial W formula} and Theorem~\ref{theorem:CS formula for rk, c}. Now we can take $k>1$.

The coefficients $\Gamma_u$ are independent of $l$, $\Gamma_u\in \C(\{\mathbf{x}_{(i,i+1)}\}_{1\leq i<k})$ for each $u$ and
all of these factors are holomorphic, for $\mathbf{x}$ in general position (see \eqref{eq:tau 1 nonzero}). Therefore if
$\sum_{u=0}^{k-1}\,x_{u+1}^{l}\Gamma_u=0$ for all $l$, we obtain $\Gamma_u=0$ for all $u$.

The only coefficient we can evaluate directly is $\Gamma_{k-1}$. However, since $\Theta_{r,m,r,\vartheta}(\chi)$ admits a unique Whittaker functional (up to scaling), $W_{\mathbf{0}}((rl,\mathbf{0'}),\vartheta,\chi_{\Theta})$ is invariant under any permutation of $\chi$. Indeed,
if $w\in W_{\GL_k}$, using the Gindikin--Karpelevich formula \eqref{eq:GK formula for GL} as in \cite[Lemma~5.2]{CS2}, one sees that
\begin{align*}
W_{\mathbf{0}}(\mathbf{b},\vartheta,({}^w\chi)_{\Theta})=\frac{W_{\mathbf{0}}(\mathbf{0},\vartheta,({}^w\chi)_{\Theta})}
{W_{\mathbf{0}}(\mathbf{0},\vartheta,\chi_{\Theta})}W_{\mathbf{0}}(\mathbf{b},\vartheta,\chi_{\Theta}),
\qquad\forall \mathbf{b}\in\Z^{rk}.
\end{align*}
Therefore
\begin{align*}
\sum_{u=0}^{k-1}\,({}^w\mathbf{x})_{u+1}^{l}\frac{\Gamma_u(\vartheta,({}^w\chi)_{\Theta})}{W_{\mathbf{0}}(\mathbf{0},\vartheta,({}^w\chi)_{\Theta})}=
\sum_{u=0}^{k-1}\,(\mathbf{x})_{u+1}^{l}\frac{\Gamma_u(\vartheta,\chi_{\Theta})}{W_{\mathbf{0}}(\mathbf{0},\vartheta,\chi_{\Theta})}.
\end{align*}
As explained above and by Theorem~\ref{theorem:CS formula for rk, c}, the coefficients of $x_i$ appearing on both sides belong to
$\C(\{\mathbf{x}_{(i,i+1)}\}_{1\leq i<k})$ and are holomorphic (simultaneously) when $\mathbf{x}$ is in general position; regarding
$W_{\mathbf{0}}(\mathbf{0},\vartheta,\chi_{\Theta})$ this is nontrivial,
because $\mathbf{y}$ is not ``general", but we use Theorem~\ref{theorem:CS formula for rk, c}. Thus
\begin{align}\label{eq:func eq for Gamma coeff}
\frac{\Gamma_u(\vartheta,({}^w\chi)_{\Theta})}{W_{\mathbf{0}}(\mathbf{0},\vartheta,({}^w\chi)_{\Theta})}=
\frac{\Gamma_u(\vartheta,\chi_{\Theta})}{W_{\mathbf{0}}(\mathbf{0},\vartheta,\chi_{\Theta})},\qquad\forall w\in W_{\GL_k}.
\end{align}
Hence it suffices to compute $\Gamma_{k-1}$.

In this case the product over $1\leq j\leq r(k-1)$ simplifies to
\begin{align}\label{eq:prod simplified in Gamma k-1}
\prod_{j=1}^{k-1}\,\frac{1-q^{-r}\mathbf{x}_{(j,k)}}{1-\mathbf{x}_{(j,k)}}.
\end{align}

By \eqref{eq:cocycle for tau} and since $\omega_{r-1}=w_{\alpha_{rk-r+1}}\cdot\ldots\cdot w_{\alpha_{rk-1}}$ is a reduced expression,
\begin{align*}
\tau_{\mathbf{0},(\mathbf{e'}^*,0)}(\omega_{r-1},\vartheta,{}^{\omega_{r-1}^{-1}}\chi_{\Theta})
=\sum_{\mathbf{d}\in r\Z^{rk}\backslash \Z^{rk}}
\tau_{\mathbf{0},\mathbf{d}}(w_{\alpha_{rk-r+1}},\vartheta,{}^{w_{\alpha_{rk-r+1}}}\chi_{\Theta})
\tau_{\mathbf{d},(\mathbf{e'}^*,0)}(\omega_{r-2},\vartheta,{}^{\omega_{r-1}^{-1}}\chi_{\Theta}).
\end{align*}
Now $\tau_{\mathbf{0},\mathbf{d}}(w_{\alpha_{rk-r+1}},\vartheta,{}^{w_{\alpha_{rk-r+1}}}\chi_{\Theta})$ vanishes unless
$\mathbf{d}\equiv w_{\alpha_{rk-r+1}}^{i_{rk-r+1}}[\mathbf{0}]$ for $i_{rk-r+1}\in\{0,1\}$. Hence
\begin{align*}
&\tau_{\mathbf{0},(\mathbf{e'}^*,0)}(\omega_{r-1},\vartheta,{}^{\omega_{r-1}^{-1}}\chi_{\Theta})
\\&=\sum_{i_{rk-r+1}}
\tau_{\mathbf{0},w_{\alpha_{rk-r+1}}^{i_{rk-r+1}}[\mathbf{0}]}(w_{\alpha_{rk-r+1}},\vartheta,{}^{w_{\alpha_{rk-r+1}}}\chi_{\Theta})
\tau_{w_{\alpha_{rk-r+1}}^{i_{rk-r+1}}[\mathbf{0}],(\mathbf{e'}^*,0)}(\omega_{r-2},\vartheta,{}^{\omega_{r-1}^{-1}}\chi_{\Theta}),
\end{align*}
where the sum is over $i_{rk-r+1}\in\{0,1\}$. If $r=2$, $\omega_{r-2}=I_{rk}$ and
$\tau_{w_{\alpha_{rk-r+1}}^{i_{rk-r+1}}[\mathbf{0}],(\mathbf{e'}^*,0)}(\omega_{r-2},\vartheta,{}^{\omega_{r-1}^{-1}}\chi_{\Theta})=0$ unless
$i_{rk-r+1}=i_{rk-1}=0$, hence $\mathbf{e'}^*$ can be taken to be $\mathbf{0}'$.
Similarly if $r>2$,
\begin{align*}
&\tau_{w_{\alpha_{rk-r+1}}^{i_{rk-r+1}}[\mathbf{0}],(\mathbf{e'}^*,0)}(\omega_{r-2},\vartheta,{}^{\omega_{r-1}^{-1}}\chi_{\Theta})
\\&=\sum_{\mathbf{d}\in r\Z^{rk}\backslash \Z^{rk}}
\tau_{w_{\alpha_{rk-r+1}}^{i_{rk-r+1}}[\mathbf{0}],\mathbf{d}}(w_{\alpha_{rk-r+2}},\vartheta,{}^{w_{\alpha_{rk-r+2}}w_{\alpha_{rk-r+1}}}\chi_{\Theta})
\tau_{\mathbf{d},(\mathbf{e'}^*,0)}(\omega_{r-2},\vartheta,{}^{\omega_{r-1}^{-1}}\chi_{\Theta})
\\&=\sum_{i_{rk-r+2}}
\tau_{w_{\alpha_{rk-r+1}}^{i_{rk-r+1}}[\mathbf{0}],
w_{\alpha_{rk-r+2}}^{i_{rk-r+2}}w_{\alpha_{rk-r+1}}^{i_{rk-r+1}}[\mathbf{0}]
}(w_{\alpha_{rk-r+2}},\vartheta,{}^{w_{\alpha_{rk-r+2}}w_{\alpha_{rk-r+1}}}\chi_{\Theta})
\\&\qquad\times\tau_{w_{\alpha_{rk-r+2}}^{i_{rk-r+2}}w_{\alpha_{rk-r+1}}^{i_{rk-r+1}}[\mathbf{0}],(\mathbf{e'}^*,0)}(\omega_{r-2},\vartheta,{}^{\omega_{r-1}^{-1}}\chi_{\Theta}).
\end{align*}
Proceeding $r-2$ times we obtain
\begin{align}\label{eq:exp for all tau e'}
&\tau_{\mathbf{0},(\mathbf{e'}^*,0)}(\omega_{r-1},\vartheta,{}^{\omega_{r-1}^{-1}}\chi_{\Theta})
\\&=\sum_{i_{rk-r+1}}\tau_{\mathbf{0},\vartheta,w_{\alpha_{rk-r+1}}^{i_{rk-r+1}}[\mathbf{0}]}(w_{\alpha_{rk-r+1}},\vartheta,{}^{w_{\alpha_{rk-r+1}}}\chi_{\Theta})\nonumber
\\&\qquad\times\sum_{i_{rk-r+2}}\tau_{w_{\alpha_{rk-r+1}}^{i_{rk-r+1}}[\mathbf{0}],\nonumber
w_{\alpha_{rk-r+2}}^{i_{rk-r+2}}w_{\alpha_{rk-r+1}}^{i_{rk-r+1}}[\mathbf{0}]
}(w_{\alpha_{rk-r+2}},\vartheta,{}^{w_{\alpha_{rk-r+2}}w_{\alpha_{rk-r+1}}}\chi_{\Theta})
\\&\qquad\ldots\nonumber
\\&\qquad\times\sum_{i_{rk-2}}\tau_{w_{\alpha_{rk-3}}^{i_{rk-3}}\cdot\ldots\cdot w_{\alpha_{rk-r+1}}^{i_{rk-r+1}}[\mathbf{0}],
w_{\alpha_{rk-2}}^{i_{rk-2}}\cdot\ldots\cdot w_{\alpha_{rk-r+1}}^{i_{rk-r+1}}[\mathbf{0}]
}(w_{\alpha_{rk-2}},\vartheta,{}^{w_{\alpha_{rk-2}}\cdot\ldots\cdot w_{\alpha_{rk-r+2}}w_{\alpha_{rk-r+1}}}\chi_{\Theta})\nonumber
\\&\qquad\times
\tau_{w_{\alpha_{rk-2}}^{i_{rk-2}}\cdot\ldots\cdot w_{\alpha_{rk-r+1}}^{i_{rk-r+1}}[\mathbf{0}],
(\mathbf{e'}^*,0)}(w_{\alpha_{rk-1}},\vartheta,{}^{\omega_{r-1}^{-1}}\chi_{\Theta}).\nonumber
\end{align}
Momentarily proceed assuming $r>2$.
To compute $\tau_{w_{\alpha_{rk-2}}^{i_{rk-2}}\cdot\ldots\cdot w_{\alpha_{rk-r+1}}^{i_{rk-r+1}}[\mathbf{0}],
(\mathbf{e'}^*,0)}(w_{\alpha_{rk-1}},\vartheta,{}^{\omega_{r-1}^{-1}}\chi_{\Theta})$ note that
since the rightmost coordinate of $(\mathbf{e'}^*,0)$ is zero,
\begin{align}\label{eq:exp for last tau w}
\tau_{w_{\alpha_{rk-2}}^{i_{rk-2}}\cdot\ldots\cdot w_{\alpha_{rk-r+1}}^{i_{rk-r+1}}[\mathbf{0}],
(\mathbf{e'}^*,0)}(w_{\alpha_{rk-1}},\vartheta,{}^{\omega_{r-1}^{-1}}\chi_{\Theta})
=\tau^1_{w_{\alpha_{rk-2}}^{i_{rk-2}}\cdot\ldots\cdot w_{\alpha_{rk-r+1}}^{i_{rk-r+1}}[\mathbf{0}],
(\mathbf{e'}^*,0)}(w_{\alpha_{rk-1}},\vartheta,{}^{\omega_{r-1}^{-1}}\chi_{\Theta}).
\end{align}
This implies $(\mathbf{e'}^*,0)\equiv w_{\alpha_{rk-2}}^{i_{rk-2}}\cdot\ldots\cdot w_{\alpha_{rk-r+1}}^{i_{rk-r+1}}[\mathbf{0}]$, and we can assume this equivalence is an equality and omit the summation over $\mathbf{e'}$ in \eqref{eq:exp for Gamma u}. In turn each summand in \eqref{eq:exp for all tau e'} is multiplied by
\begin{align}\label{eq:W on e'}
W_{\mathbf{e'}}(\mathbf{0'},\vartheta,\chi_{\Theta}'[r(k-1)+1]),\qquad
\mathbf{e'}^*=w_{\alpha_{rk-2}}^{i_{rk-2}}\cdot\ldots\cdot w_{\alpha_{rk-r+1}}^{i_{rk-r+1}}[\mathbf{0}'].
\end{align}
Note that $\mathbf{e'}^*$ is well defined, because $rk-r+1<rk-1$.
To treat $r\geq2$ simultaneously, we fix the convention that if $rk-2<rk-r+1$, $w_{\alpha_{rk-2}}^{i_{rk-2}}\cdot\ldots\cdot w_{\alpha_{rk-r+1}}^{i_{rk-r+1}}$ is taken to be the identity.

The value of \eqref{eq:W on e'} was computed by Cai \cite[Theorem~8.1]{CaiLemma39} (using the methods of \cite{Suzuki1997,Suzuki1998}) and is equal to
\begin{align}\label{lemma:mk-1 Whittaker}
\prod_{j=1}^{k-1}(1-q^{-r}\mathbf{x}_{(j,k)})^{-1}\prod_{j=1}^{r}C(j,\mathbf{x})
\mathscr{G}(w_{\alpha_{rk-2}}^{i_{rk-2}}\cdot\ldots\cdot w_{\alpha_{rk-r+1}}^{i_{rk-r+1}},\mathbf{0}').
\end{align}

Next we compute the factors appearing in \eqref{eq:exp for all tau e'}.
Fix $1\leq l\leq r-1$. If $1<l<r-1$, denote
\begin{align*}
\mathbf{a}=\mathbf{a}^l=w_{\alpha_{rk-l-1}}^{i_{rk-l-1}}\cdot\ldots\cdot w_{\alpha_{rk-r+1}}^{i_{rk-r+1}}[\mathbf{0}],\qquad
\mathbf{b}=\mathbf{b}^l=w_{\alpha_{rk-l}}^{i_{rk-l}}\cdot\ldots\cdot w_{\alpha_{rk-r+1}}^{i_{rk-r+1}}[\mathbf{0}];
\end{align*}
for $l=r-1$, $\mathbf{a}=\mathbf{0}$ and $\mathbf{b}=w_{\alpha_{rk-r+1}}^{i_{rk-r+1}}[\mathbf{0}]$;
and if $l=1$, $\mathbf{a}=\mathbf{b}=w_{\alpha_{rk-2}}^{i_{rk-2}}\cdot\ldots\cdot w_{\alpha_{rk-r+1}}^{i_{rk-r+1}}[\mathbf{0}]$
(for $r=2$, $\mathbf{a}=\mathbf{b}=\mathbf{0}$ either way).
The corresponding factor to evaluate is
\begin{align}\label{eq:one tau in the chain}
\tau_{\mathbf{a},\mathbf{b}}(w_{\alpha_{rk-l}},\vartheta,{}^{w_{\alpha_{rk-l}}\cdot\ldots\cdot w_{\alpha_{rk-r+1}}}\chi_{\Theta}).
\end{align}

If $l>1$ and $i_{rk-l}=0$, then $\mathbf{a}=\mathbf{b}$ and
\begin{align*}
\tau_{\mathbf{a},\mathbf{b}}(w_{\alpha_{rk-l}},\vartheta,{}^{w_{\alpha_{rk-l}}\cdot\ldots\cdot w_{\alpha_{rk-r+1}}}\chi_{\Theta})
=\tau_{\mathbf{a},\mathbf{b}}^1(w_{\alpha_{rk-l}},\vartheta,{}^{w_{\alpha_{rk-l}}\cdot\ldots\cdot w_{\alpha_{rk-r+1}}}\chi_{\Theta}).
\end{align*}
Similarly for $l=1$, we have \eqref{eq:exp for last tau w}.
Hence in both cases we can use \eqref{eq:tau 1 nonzero} to compute \eqref{eq:one tau in the chain}. Since
$0\leq a_{rk-l}\leq r-1-l$ and $a_{rk-l+1}=0$, we obtain
$\lceil (a_{rk-l+1}-a_{rk-l})/r\rceil=0$, thus \eqref{eq:one tau in the chain} equals
\begin{align}\label{eq:tau 1 for each l}
(1-q^{-1})/(1-q^{l-r}).
\end{align}

The remaining case is $l>1$ and $i_{rk-l}=1$. We appeal to \eqref{eq:tau 2 nonzero} and deduce
that \eqref{eq:one tau in the chain} equals
\begin{align}\label{eq:tau 2 for each l}
q^{a_{rk-l+1}-a_{rk-l}-1}\mathfrak{g}(2(a_{rk-l+1}-a_{rk-l}-1)).
\end{align}
At the same time
\begin{align}\label{eq:tau 2 for each l mathscr G}
\mathscr{G}(w_{\alpha_{rk-l}},w_{\alpha_{rk-l-1}}\cdot\ldots\cdot w_{\alpha_{rk-r+1}}[\mathbf{0}'])
=q^{-a_{rk-l+1}+a_{rk-l}+1}\mathfrak{g}(-2(a_{rk-l+1}-a_{rk-l}-1)),
\end{align}
where we replace $w_{\alpha_{rk-l-1}}\cdot\ldots\cdot w_{\alpha_{rk-r+1}}[\mathbf{0}']$ by $\mathbf{0}'$ if $l=r-1$.
Since $2(a_{rk-l+1}-a_{rk-l}-1)\equiv0\,(m)$ if and only if $a_{rk-l}+1\equiv0\,(r)$ ($a_{rk-l+1}=0$) and
$1\leq a_{rk-l}+1\leq r-l\leq r-1$, we deduce $2(a_{rk-l+1}-a_{rk-l}-1)\not\equiv0\,(m)$ hence the product of
\eqref{eq:tau 2 for each l} and \eqref{eq:tau 2 for each l mathscr G} is $q^{-1}$.

Now when we combine \eqref{eq:exp for Gamma u} (for $u=k-1$ and with $\mathbf{e'}$ fixed, as explained above), \eqref{eq:prod simplified in Gamma k-1}, \eqref{eq:exp for all tau e'}, \eqref{eq:one tau in the chain} and \eqref{lemma:mk-1 Whittaker}, we obtain
\begin{align*}
\Gamma_{k-1}=&\prod_{j=1}^{k-1}(1-\mathbf{x}_{(j,k)})^{-1}\prod_{j=1}^{r} C(j,\mathbf{x})
\sum_{i_{rk-r+1},\ldots,i_{rk-2}\in\{0,1\}}\\&\prod_{l=r-1}^{1}
\tau_{\mathbf{a}^l,\mathbf{b}^l}(w_{\alpha_{rk-l}},\vartheta,{}^{w_{\alpha_{rk-l}}\cdot\ldots\cdot w_{\alpha_{rk-r+1}}}\chi_{\Theta})\mathscr{G}(w_{\alpha_{rk-2}}^{i_{rk-2}}\cdot\ldots\cdot w_{\alpha_{rk-r+1}}^{i_{rk-r+1}},\mathbf{0}').
\end{align*}
Using \eqref{eq:tau 1 for each l}, \eqref{eq:tau 2 for each l} and the definition of $\mathscr{G}$, the summation
over $i_{rk-r+1},\ldots,i_{rk-2}$ becomes
\begin{align*}
\prod_{l=1}^{r-2}\left(\frac{1-q^{-1}}{1-q^{-l}}+q^{-1}\right)\frac{1-q^{-1}}{1-q^{-r+1}}=
\prod_{l=1}^{r-2}\left(\frac{1-q^{-l-1}}{1-q^{-l}}\right)\frac{1-q^{-1}}{1-q^{-r+1}}=1.
\end{align*}
Hence by Theorem~\ref{theorem:CS formula for rk, c},
\begin{align*}
\Gamma_{k-1}=\prod_{j=1}^{k-1}(1-\mathbf{x}_{(j,k)})^{-1}\prod_{j=1}^{r} C(j,\mathbf{x})=
\prod_{j=1}^{k-1}(1-\mathbf{x}_{(j,k)})^{-1}W_{\mathbf{0}}(\mathbf{0},\vartheta,\chi_{\Theta}).
\end{align*}
Then $W^0(t_{(rl,\mathbf{0'})})$ equals
\begin{align*}
\frac{W_{\mathbf{0}}((rl,\mathbf{0'}),\vartheta,\chi_{\Theta})}{W_{\mathbf{0}}(\mathbf{0},\vartheta,\chi_{\Theta})}
&=q^{-lr(rk-1)/2+l(r-1)/2}\vartheta(\varpi^{rl})\left(
\frac{x_k^l}{\prod_{j=1}^{k-1}(1-x_jx_k^{-1})}+\sum_{u=0}^{k-2}\,\frac{x_{u+1}^{l}\Gamma_u}{W_{\mathbf{0}}(\mathbf{0},\vartheta,\chi_{\Theta})}\right),
\end{align*}
and the proof of the theorem is complete by \eqref{eq:func eq for Gamma coeff} and the definition of $p_l(\mathbf{x})$.
\end{proof}
\begin{remark}
For the cases $r=1$ or $k=1$, \cite{CaiLemma39} was not used in the proof.
\end{remark}
\subsection{Global construction of $(rk,c)$ representations}\label{speh gbl}
Let $\tau$ be a genuine unitary irreducible cuspidal automorphic representation of $\GL_{k}^{(m,r)}(\A)$ (see \S~\ref{covering of the Levi}). We construct global $(rk,c)$ representations by means of residues of Eisenstein series.
As explained in the introuction, for the construction we rely on certain assumptions on local intertwining operators and global partial $L$-functions. We begin by stating our local and global working assumptions.

Let $\nu$ be a place of $F$. For $l\geq2$, let $\boldsymbol{\zeta}=(\zeta_1,\ldots,\zeta_{l})\in\C^{l}$, and
\begin{align*}
w_l=\left(\begin{smallmatrix}&&I_{k}\\&\udots\\I_{k}\end{smallmatrix}\right)\in\GL_{lk}(F).
\end{align*}
We have the standard intertwining operator
\begin{align*}
M(w_l,\boldsymbol{\zeta}):\Ind_{\widetilde{P}_{(k^{l})}(F_{\nu})}^{\GL_{lk}^{(m,r)}(F_{\nu})}
(|\det|^{\zeta_1}\tau_{\nu}\otimes\ldots\otimes|\det|^{\zeta_l}\tau_{\nu})
\rightarrow
\Ind_{\widetilde{P}_{(k^{l})}(F_{\nu})}^{\GL_{lk}^{(m,r)}(F_{\nu})}
(|\det|^{\zeta_l}\tau_{\nu}\otimes\ldots\otimes|\det|^{\zeta_1}\tau_{\nu}).
\end{align*}
For $\Real(\zeta_1)\gg\ldots\gg\Real(\zeta_l)$ it is given by the absolutely convergent integral
\begin{align*}
M(w_l,\boldsymbol{\zeta})\xi(g,\boldsymbol{\zeta})=\int\limits_{V_{(k^l)}(F_{\nu})}\xi(\langle w_l,1\rangle^{-1}\langle v,1\rangle g,\boldsymbol{\zeta})\,dv,
\end{align*}
and in general by meromorphic continuation. Our local conjecture is the following.
\begin{conjecture}\label{local Shimura conjecture}
For each place $\nu$, $M(w_2,\boldsymbol{\zeta})$ is holomorphic for $\Real(\zeta_1-\zeta_2)\geq1/r$, and if $1<l\leq r$, the image of $M(w_l,\boldsymbol{\zeta})$ is irreducible at $\boldsymbol{\zeta}=((l-1)/(2r),(l-3)/(2r),\ldots,(1-l)/(2r))$.
\end{conjecture}

Globally, we define the partial $L$-function of $\tau\times\tau^{\vee}$. Let $S$ be a finite set of places of $F$, such that outside $S$ all data are unramified. The set $S$ depends only on $F$, $\psi$ and $\tau$.
If $m\equiv2\,(4)$, we need to fix the parametrization of the local components of $\tau$ outside $S$, because the local $L$-factors depend on this parametrization (see \S~\ref{unramified reps}). Choose a nontrivial character $\psi'$ of $F\backslash\A$ which is unramified outside $S$ (e.g., $\psi$). Let $\vartheta=\gamma_{\psi'}=\prod_{\nu}\gamma_{\psi'_{\nu}}$, $\vartheta_{\nu}=\gamma_{\psi'_{\nu}}$ for all $\nu$. When $m\not\equiv2\,(4)$, $\vartheta$ can be ignored. Then we can define for $s\in\C$,
\begin{align*}
L^S_{\vartheta,\vartheta^{-1}}(s,\tau\times\tau^{\vee})=\prod_{\nu\notin S}L_{\vartheta_{\nu},\vartheta_{\nu}^{-1}}(s,\tau_{\nu}\times\tau_{\nu}^{\vee}).
\end{align*}
This product is absolutely convergent for $\Real(s)\gg0$, and admits meromorphic continuation (\cite{La2,La5,MW2} and \cite{Gao2018}).
\begin{conjecture}\label{Shimura conjecture}
$L^S_{\vartheta,\vartheta^{-1}}(s,\tau\times\tau^{\vee})$ has a simple pole at $s=1$ and is holomorphic and nonzero for $\Real(s)>1$.
\end{conjecture}

In the linear case ($m=1$) both conjectures are known: the local conjecture follows from the results of
\cite{Jac4} (see also \cite[Proposition~I.10]{MW4}) and \cite{JS1,JPSS,JS3}, and the global was proved in \cite{JS2,JS1}.
\begin{proposition}\label{proposition:conjectures hold in minimal cases}
Assume $r=1$ or $k=1$. Then Conjectures~\ref{local Shimura conjecture} and \ref{Shimura conjecture} hold.
\end{proposition}
\begin{proof}
For $m=2$, we can assume that the $2$-cocycles for $\GL_k^{(2,1)}(F_{\nu})$ and $\GL_k^{(2,1)}(\A)$ are given by a local or global Hilbert symbol (see Proposition~\ref{proposition:sigma * and sigma on GLd}). In this case there is a local and global bijection between representations of $\GL_k$ and genuine representations of $\GL_k^{(2,1)}$. Specifically, for
a representation $\tau_0$ of $\GL_k$, define a genuine representation of $\GL_k^{(2,1)}$ by
$(\gamma_{\psi'}\otimes\tau_0)(\langle g,\epsilon\rangle)=\epsilon\gamma_{\psi'}(\det g)\tau_0(g)$, where $\gamma_{\psi'}$ is a local or global Weil factor. In particular irreducible representations of $\GL_k^{(2,1)}(F_{\nu})$ admit at most one Whittaker model.
Now Conjecture~\ref{local Shimura conjecture} holds because we can identify between
\begin{align*}
\Ind_{P_{(k^{2})}(F_{\nu})}^{\GL_{2k}(F_{\nu})}
(|\det|^{\zeta_1}\tau_{\nu}\otimes|\det|^{\zeta_2}\tau_{\nu})\qquad\mathrm{and}\qquad
\gamma_{\psi'}\Ind_{\widetilde{P}_{(k^{2})}(F_{\nu})}^{\GL_{2k}^{(m,r)}(F_{\nu})}
(|\det|^{\zeta_1}\gamma_{\psi'}\tau_{\nu}\otimes|\det|^{\zeta_2}\gamma_{\psi'}\tau_{\nu})
\end{align*}
via $\xi\mapsto \xi_{\psi'}$, where
$\xi_{\psi'}(\langle h,\epsilon\rangle,\boldsymbol{\zeta})=\epsilon\gamma_{\psi'}(\det h)\xi(h,\boldsymbol{\zeta})$, so that the analytic properties of the intertwining operator follow from the linear case.
Globally, given a genuine unitary irreducible cuspidal automorphic representation $\tau$ of $\GL_k^{(2,1)}(\A)$, there is a unique
unitary irreducible cuspidal automorphic representation $\tau_0$ of $\GL_k(\A)$
such that $\tau_0\otimes\gamma_{\psi'}=\tau$. Then $\tau_0^{\vee}\otimes\gamma_{\psi'}^{-1}=\tau^{\vee}$,
this follows from a direct verification of the local definitions, note that $\langle a,\epsilon\rangle^{-1}=\langle a^{-1},\epsilon^{-1}(a,a)_2\rangle$ and
use \eqref{eq:some props of gamma psi'}. Then $L^S(s,\tau_0\times\tau_0^{\vee})=L^S_{\gamma_{\psi'},\gamma_{\psi'}^{-1}}(s,\tau\times\tau^{\vee})$
and Conjecture~\ref{Shimura conjecture} follows from \cite{JS2,JS1}.

When $k=1$, the local (resp., irreducible automorphic) representations of $\GL_1^{(m,r)}$ are constructed using quasi-characters of $F_{\nu}^*$ (resp., automorphic characters of $\A^*$) restricted to $F_{\nu}^{*r}$ (resp., $\A^{*r}$), by Stone-von Neumann Theory (see \cite{KP} and \S~\ref{unramified reps}; for this there is no reason to assume data are unramified).

The first statement of Conjecture~\ref{local Shimura conjecture}: for $p$-adic fields it follows from a straightforward adaptation of
\cite[Proposition~I.2.3 (b)]{KP}; for archimedean fields the case $m=2$ was explained above and for $m>2$, the underlying field is $\C$ and the results over archimedean fields are immediate from the linear case ($\GL_2^{(m,r)}(\C)$ is split over $\GL_2(\C)$).
Note that since $k=1$, $\tau_{\nu}$ is now tempered so that $M(w_2,\boldsymbol{\zeta})$ is known to be holomorphic already for
$\Real(\zeta_1-\zeta_2)>0$ (again by \cite{JPSS,Jac4,MW4,JS3}).

The irreducibility statement ($r>1$) holds because in this case the image of
$M(w_l,\boldsymbol{\zeta})$ is the exceptional representation (see \S~\ref{exceptional}), then over $p$-adic fields one can modify
\cite[Theorem~I.2.9 (a)]{KP} to $\GL_l^{(m,r)}$ to obtain irreducibility, and over archimedean fields it is a direct consequence of
Langlands' Quotient Theorem \cite{La3} ($r>1$ implies $m>2$).

Conjecture~\ref{Shimura conjecture} again follows from the linear case.
\end{proof}

Henceforth until the end of this section we assume these conjectures.

We turn to describe our construction.
Let $\boldsymbol{\zeta}=(\zeta_1,\ldots,\zeta_{rc})\in\C^{rc}$.
As explained in \S~\ref{covering of the Levi}, we can construct the induced representation
\begin{align}\label{ind1}
\Ind_{\widetilde{P}_{(k^{rc})}({\A})}^{\GL_{rkc}^{(m,r)}({\A})}
(|\det|^{\zeta_1}\tau\otimes\ldots\otimes |\det|^{\zeta_{rc}}\tau).
\end{align}
For a standard $\widetilde{K}_{\GL_{rkc}}$-finite vector $\xi$ in the space of \eqref{ind1}, let $E(g;\xi,\boldsymbol{\zeta})$ denote the Eisenstein series, as in \eqref{eq:Eisenstein series on GL}.
Let
\begin{align}\label{eq:zeta_0}
\boldsymbol{\zeta}^{(rc)}=((rc-1)/(2r),(rc-3)/(2r),\ldots, (1-rc)/(2r))\in\C^{rc}.
\end{align}
In particular $\zeta^{(rc)}_i-\zeta^{(rc)}_{i+1}=1/r$ for all $1\leq i\leq {rc}-1$.

Consider the following multi-residue of the series at \eqref{eq:zeta_0},
\begin{align}\label{eq:limit}
E_{-1}(g;\xi)=
\lim\limits_{\boldsymbol{\zeta}\to\boldsymbol{\zeta}^{(rc)}}\prod_{i=1}^{rc-1}(r(\zeta_i-\zeta_{i+1})-1)E(g;\xi,\boldsymbol{\zeta}).
\end{align}
As we explain in the proof of Theorem~\ref{theorem:prop1} below, our assumptions imply that this limit is finite. Let $\mathcal{L}_{\tau,c}$ denote the residual representation generated by the
functions $E_{-1}(\cdot;\xi)$, it is a genuine automorphic representation of $\GL_{rkc}^{(m,r)}({\A})$. Recall that
for the notion of automorphic representations of $\GL_{rkc}^{(m,r)}({\A})$ we identify $\GL_{rkc}(F)$ with its image under the splitting
$g\mapsto\langle g,(\eta_{rkc}^{\diamondsuit})^{-1}(g)\rangle$.

For any $d\geq1$, let $C_{d}=\{xI_d:x\in\A^{*r}\}$. Then $\widetilde{C}_{d}$ is the center of $\GL_{d}^{(m,r)}(\A)$ (see \S~\ref{covering of the Levi}). For a genuine unitary character $\varrho$ of $\widetilde{C}_{rkc}(\A)$ which is trivial on
the subgroup of elements $\langle xI_{rkc},(\eta_{rkc}^{\diamondsuit})^{-1}(xI_{rkc})\rangle$ with $x\in F^{*r}$,
let $L^2(\GL_{rkc}(F)\backslash \GL_{rkc}^{(m,r)}(\A),\varrho)$ be the space of genuine square-integrable automorphic forms which translate by $\varrho$ on $\widetilde{C}_{rkc}$.

Let $\varrho_{\tau}$ be the central character of $\tau$. Since $C_{rkc}<C_{rc}\times\ldots\times C_{rc}$ , the representation $\mathcal{L}_{\tau,c}$ admits a central character, which is $\varrho_{\tau}^{rc}$.
\begin{theorem}\label{theorem:prop1}
The limit \eqref{eq:limit} is finite. The genuine automorphic representation
$\mathcal{L}_{\tau,c}$ lies in the discrete spectrum of
$L^2(\GL_{rkc}(F)\backslash \GL_{rkc}^{(m,r)}({\A}),\varrho_{\tau}^{rc})$. Furthermore, for any irreducible summand $\mathcal{E}$ of $\mathcal{L}_{\tau,c}$, $\mathcal{E}=\otimes'_\nu\,\mathcal{E}_\nu$ where for all $\nu$, $\mathcal{E}_{\nu}$ is an irreducible subrepresentation of
\begin{align}\label{rep:unr 1 subrep}
\Ind_{\widetilde{P}_{(k^{rc})}(F_\nu)}^{\GL^{(m,r)}_{rkc}(F_\nu)}((\tau_\nu\otimes \ldots \otimes \tau_\nu)\delta_{P_{(k^{rc})}}^{-1/(2rk)})
\end{align}
and a quotient of
\begin{align}\label{rep:unr 1 quotient}
\Ind_{\widetilde{P}_{(k^{rc})}(F_\nu)}^{\GL^{(m,r)}_{rkc}(F_\nu)}((\tau_\nu\otimes \ldots \otimes \tau_\nu)\delta_{P_{(k^{rc})}}^{1/(2rk)}),
\end{align}
and for almost all $\nu$, $\mathcal{E}_{\nu}$ is the unique irreducible unramified subrepresentation (resp., quotient) of \eqref{rep:unr 1 subrep} (resp., \eqref{rep:unr 1 quotient}).
\end{theorem}
\begin{proof}
The proof follows from the computation of the constant term along $V_{(k^{rc})}$, which in the linear case was
carried out by Jacquet \cite[\S~2.1--2.4]{Jac4}. Suzuki \cite[\S~8]{Suzuki1998} extended the result of \cite{Jac4} to coverings of \cite{KP} under certain restrictions (see also \cite{KP,Gao2018}).

Briefly, according to \cite[\S~II.1.7]{MW2}, the constant term of $E(g;\xi,\boldsymbol{\zeta})$ along $V_{(k^{rc})}$ is the sum of images of intertwining operators
$M(w,\boldsymbol{\zeta})$, where $w$ is a representative in $\GL_{rkc}(F)$ of an element of $W_{\GL_{rkc}}$ such that ${}^wM_{(k^{rc})}=M_{(k^{rc})}$, and
$M(w,\boldsymbol{\zeta})$ is given by the meromorphic continuation of the integral
\begin{align*}
M(w,\boldsymbol{\zeta})\xi(g,\boldsymbol{\zeta})=\int\limits_{({}^w(V_{(k^{rc})}(\A))\cap V_{(k^{rc})}(\A))\backslash V_{(k^{rc})}(\A)}\xi(\langle w,(\eta_{rkc}^{\diamondsuit})^{-1}(w)\rangle^{-1}\langle
u,(\eta_{rkc}^{\diamondsuit})^{-1}(u)\rangle g,\boldsymbol{\zeta})\,du.
\end{align*}
By Conjecture~\ref{local Shimura conjecture}, the poles of these operators are of global origin and are contained in the poles of quotients of partial $L$-functions
\begin{align*}
\frac{L^S_{\vartheta,\vartheta^{-1}}(r(\zeta_i-\zeta_j),\tau\times\tau^{\vee})}
{L^S_{\vartheta,\vartheta^{-1}}(r(\zeta_i-\zeta_j)+1,\tau\times\tau^{\vee})}.
\end{align*}
As in the linear case, these quotients are determined directly from the Gindikin--Karpelevich formula \eqref{eq:GK formula for GL}.
Using Conjecture~\ref{Shimura conjecture} we see that all of these intertwining operators are holomorphic at \eqref{eq:zeta_0}, except the one corresponding to
\begin{align*}
w_0=\left(\begin{smallmatrix}&&I_{k}\\&\udots\\I_{k}\end{smallmatrix}\right)\in\GL_{rkc}(F).
\end{align*}
Now Conjecture~\ref{Shimura conjecture} also implies that
\begin{align*}
\lim\limits_{\boldsymbol{\zeta}\to \boldsymbol{\zeta}^{(rc)}}\prod_{i=1}^{{rc}-1}(r(\zeta_i-\zeta_{i+1})-1)M(w_0,\boldsymbol{\zeta})\xi
\end{align*}
is finite and nonzero. When we take the limit \eqref{eq:limit},
\begin{align}\label{eq:const-term-residue}
\int\limits_{V_{(k^{rc})}(F)\bs V_{(k^{rc})}(\A)}E_{-1}(\langle u,(\eta_{rkc}^{\diamondsuit})^{-1}(u)\rangle g;\xi)\,du=\lim\limits_{\boldsymbol{\zeta}\to \boldsymbol{\zeta}^{(rc)}}\prod_{i=1}^{{rc}-1}(r(\zeta_i-\zeta_{i+1})-1)M(w_0,\boldsymbol{\zeta})\xi(g,\boldsymbol{\zeta})
\end{align}
and it follows that \eqref{eq:limit} is finite and nonzero.
By applying Jacquet's criterion as stated in \cite[Lemma~I.4.11]{MW2}, we deduce that the automorphic forms
$E_{-1}(\cdot;\xi)$ are square-integrable, and it also follows that
$\mathcal{L}_{\tau,c}$ belongs to the discrete spectrum of $L^2(\GL_{rkc}(F)\backslash \GL_{rkc}^{(m,r)}({\A}),\varrho_{\tau}^{rc})$.

Note that $\mathcal{L}_{\tau,1}$ is irreducible, because for $r>1$, the image of $M(w_0,\boldsymbol{\zeta})$ at $\boldsymbol{\zeta}^{(r)}$ is irreducible
by Conjecture~\ref{local Shimura conjecture} (in this case $w_0$ is $w_r$ defined above).

According to \eqref{eq:const-term-residue}, the application of the constant term yields an embedding of $\mathcal{L}_{\tau,c}$ into the image of $M(w_0,\boldsymbol{\zeta})$. Thus each local component
$(\mathcal{L}_{\tau,c})_{\nu}$ is embedded in \eqref{rep:unr 1 subrep}, and at almost all places $(\mathcal{L}_{\tau,c})_{\nu}$ and \eqref{rep:unr 1 subrep} contain an unramified constituent. Since $\mathcal{E}_{\nu}$ is irreducible and embeds in \eqref{rep:unr 1 subrep}, at almost all places it is the unique irreducible unramified subrepresentation of \eqref{rep:unr 1 subrep}. Similar reasoning applies to the assertion on the quotient of \eqref{rep:unr 1 quotient}.
\end{proof}
We construct an $(rk,c)$ representation from
$\mathcal{L}_{\tau,c}$. Assume $\mu_{2m}\subset F^*$, in the rest of this section.
\begin{theorem}\label{exthspeh1}
The representation $\mathcal{L}_{\tau,c}$ has at least one irreducible $(rk,c)$ summand $\mathcal{E}_{\tau}$.
\end{theorem}
\begin{proof}
Let $\mathcal{E}$ be any irreducible summand of $\mathcal{L}_{\tau,c}$.
First consider the local setting. By Theorem~\ref{theorem:prop1}, at almost all places $\nu$, $\mathcal{E}_{\nu}$ is an irreducible unramified subrepresentation of
\eqref{rep:unr 1 subrep}. Also assume that $\tau_{\nu}$ is the genuine irreducible unramified constituent of $\mathrm{I}_{\GL_k^{(m,r)}}(\vartheta,\chi)$. By transitivity of induction \eqref{rep:unr 1 subrep} is a genuine unramified constituent of a principal series representation, and when we permute the inducing characters of $\tau_{\nu}$ in this full induced representation and use
the formula for $\delta_{B_{\GL_{rc}}}$, we obtain \eqref{eq:Thete rc m as a subrep}:
\begin{align*}
&\Ind_{\widetilde{B}_{\GL_{rkc}}(F_\nu)}^{\GL_{rkc}^{(m,r)}(F_\nu)}\left(
\Ind_{\widetilde{T}_{\GL_{rkc},*}(F_\nu)}^{\widetilde{T}_{\GL_{rkc}}(F_\nu)}(
(\otimes_{j=1}^{rc}\vartheta\chi_1|~|^{-(rc-2j+1)/(2r)})\otimes \ldots \otimes
(\otimes_{j=1}^{rc}\vartheta\chi_k|~|^{-(rc-2j+1)/(2r)}))\right).
\end{align*}
By \cite[Theorem~2.9]{BZ2}, $\mathcal{E}_{\nu}$ is the unique irreducible unramified constituent of \eqref{eq:Thete rc m as a subrep}.
The representation $\Theta_{rc,m,r,\vartheta}(\chi)$ is also an unramified constituent (a subrepresentation) of \eqref{eq:Thete rc m as a subrep}, see \S~\ref{local theta speh}. Therefore $\mathcal{E}_{\nu}$ is a constituent of $\Theta_{rc,m,r,\vartheta}(\chi)$.
Since the Jacquet functor is
exact, by Proposition~\ref{proposition:local-semi} we have $\mathcal{O}(\mathcal{E}_{\nu},\beta,\psi)=0$ for any $\beta\succsim((rk)^{c})$, and $\dim\mathcal{O}(\mathcal{E}_\nu,((rk)^c),\psi)\leq1$.

Thus far we have not assumed anything on $\mathcal{E}$.
Next we show that \eqref{lambda semi coefficient} with $\lambda=((rk)^c)$ is not identically zero on
$\mathcal{L}_{\tau,c}$. For $c=1$ the result can be proved by adapting the argument of \cite[Theotam~II.2.5]{KP} or \cite[Theorem~1]{FG2017}
(using the irreducibility assumption of Conjecture~\ref{local Shimura conjecture} for all $1<l\leq r$).
Assume $c>1$. We claim
\begin{align}\label{eq:global-nonvanishing}
\int\limits_{N_{\GL_{rkc}}(F)\bs N_{\GL_{rkc}}(\A)} E_{-1}(
\langle u,(\eta_{rkc}^{\diamondsuit})^{-1}(u)\rangle;\xi)\psi_{((rk)^c)}(u)\,du\neq 0
\end{align}
for some $\xi$. By definition $\psi_{((rk)^c)}$ is trivial on $V_{((rk)^c)}$ (see \S~\ref{Semi Whittaker coeff}), hence
\eqref{eq:global-nonvanishing} factors through the constant term along $V_{((rk)^c)}$. This constant term can be computed as in the proof of Theorem~\ref{theorem:prop1}, and we deduce that the mapping
\begin{align*}
b\mapsto \int\limits_{V_{((rk)^c)}(F)\bs V_{((rk)^c)}(\A)} &E_{-1}(\langle v,(\eta_{rkc}^{\diamondsuit})^{-1}(v)\rangle b;\xi)\,dv,\qquad b\in \widetilde{M}_{((rk)^c)}(\A)
\end{align*}
belongs to the space of $\mathcal{L}_{\tau,1}\otimes\ldots\otimes\mathcal{L}_{\tau,1}$ ($c$ times). For more details see
\cite[Lemma~4.1]{JL2013}, where the constant term of $E_{-1}(g;\xi)$ in the linear case was computed along any $V_{(lk,(c-l)k)}$ with $0<l<c$.

Now we note that $N_{\GL_{rkc}}=(M_{((rk)^c)}\cap N_{\GL_{rkc}})\ltimes V_{((rk)^c)}$,
$M_{((rk)^c)}\cap N_{\GL_{rkc}}$ is the direct product of $c$ copies of $N_{\GL_{rk}}$, and $\psi_{((rk)^c)}$ is the product of the
$(rk,1)$ characters on each these copies. Therefore \eqref{eq:global-nonvanishing} becomes the product of $c$ applications of
\eqref{lambda semi coefficient} with $\lambda=(rk)$ on vectors in the space of $\mathcal{L}_{\tau,1}$, each of these is nonzero by the case $c=1$.

Since $\mathcal{L}_{\tau,1}$ is irreducible (by Conjecture~\ref{local Shimura conjecture}, see the proof of Theorem~\ref{theorem:prop1}),
we can conclude that \eqref{lambda semi coefficient} with $\lambda=((rk)^c)$ is nonzero on
$\mathcal{L}_{\tau,c}$, and we let $\mathcal{E}_{\tau}$ be an irreducible summand of $\mathcal{L}_{\tau,c}$ on which this semi-Whittaker coefficient is nonzero. As explained above, for almost all $\nu$, $\mathcal{O}((\mathcal{E}_{\tau})_{\nu},\beta,\psi)=0$ for any $\beta\succsim((rk)^{c})$,
then by Lemma~\ref{lemma:local semi Whittaker and WSS}, $J_{N_{\GL_{rkc}},\psi_{\lambda}}((\mathcal{E}_{\tau})_{\nu})=0$ for all $\lambda\succsim((rk)^c)$. Hence $\mathcal{E}_{\tau}$ does not support any semi-Whittaker functionals with respect to $\lambda\succsim((rk)^c)$
(local vanishing implies global). Now we can conclude from Lemma~\ref{lemma:gbl semi Whittaker and WSS} that
$\mathcal{O}(\mathcal{E}_{\tau},((rk)^c),\psi)\ne0$. This immediately implies $\dim\mathcal{O}((\mathcal{E}_{\tau})_{\nu},((rk)^c),\psi)>0$ for all the local components of $\mathcal{E}_{\tau}$, and so for almost all $\nu$, $(\mathcal{E}_{\tau})_{\nu}$ is $(rk,c)$, and also $\mathcal{E}_{\tau}$ is $(rk,c)$ ($\mathcal{O}(\mathcal{E}_{\tau},\beta,\psi)=0$ for $\beta\succsim((rk)^{c})$ follows from the local assertion).
\end{proof}
When we combine Theorem~\ref{theorem:prop1} with the proof of Theorem~\ref{exthspeh1} we obtain the following corollary.
\begin{corollary}\label{corollary:local spaces are almost always good}
Outside finitely many places $\nu$, the representation \eqref{rep:unr 1 subrep} (resp., \eqref{rep:unr 1 quotient}) contains a unique irreducible unramified subrepresentation (resp., quotient), which is in addition the unique $(rk,c)$ constituent of a representation
$\Theta_{rc,m,r,\vartheta}(\chi)$ for some $\chi$ and $\vartheta$.
\end{corollary}
\begin{proof}
The fact that \eqref{rep:unr 1 subrep} contains a unique irreducible unramified subrepresentation was proved in Theorem~\ref{theorem:prop1}. In
the proof of Theorem~\ref{exthspeh1} we deduced that this subrepresentation is a constituent of $\Theta_{rc,m,r,\vartheta}(\chi)$ and is $(rk,c)$. Since $\Theta_{rc,m,r,\vartheta}(\chi)$ is also $(rk,c)$ by Proposition~\ref{proposition:local-semi}, it affords precisely one irreducible $(rk,c)$ constituent, by the exactness of the Jacquet functor.
\end{proof}

\subsection{The local components of $\mathcal{E}_{\tau}$}\label{Local components of rk c speh}
We describe several properties of the local components at unramified places, of the $(rk,c)$ representation produced
by Theorem~\ref{exthspeh1}. To avoid confusion, we re-denote the
genuine unitary irreducible cuspidal automorphic representation of $\GL_{k}^{(m,r)}(\A)$ by $\tau'$,
then the $(rk,c)$ representation is $\mathcal{E}_{\tau'}$.
We let $\tau=\tau'_{\nu}$ at an unramified place $\nu$, in particular $F$ unramified, and we also assume $\mu_{2m}\subset F^*$.

Assume that $\tau$ is the unique irreducible unramified constituent of $\mathrm{I}_{\GL_k^{(m,r)}}(\vartheta,\chi)$. Then
by Proposition~\ref{proposition:vartheta of *}, $\tau^*$ defined by \eqref{eq:involution b*0}
is the unique irreducible unramified constituent of $\mathrm{I}_{\GL_k^{(m,r)}}(\vartheta,\chi^{-1})$.
Put
$\mathbf{x}=(x_1,\ldots,x_d)\in\C^d$, $x_i=\chi_i(\varpi^r)$.
By Theorem~\ref{theorem:prop1}, $(\mathcal{E}_{\tau'})_{\nu}$ is the unique irreducible unramified subrepresentation of
\begin{align}\label{nontwisted subrep rho of tau}
\Ind_{\widetilde{P}_{(k^{rc})}}^{\GL^{(m,r)}_{rkc}}((\tau\otimes \ldots \otimes \tau)\delta_{P_{(k^{rc})}}^{-1/(2rk)}),
\end{align}
and we denote $\rho_c(\tau)=(\mathcal{E}_{\tau'})_{\nu}$. The representation $\rho_c(\tau)$ is irreducible and $(rk,c)$, hence isomorphic to its $(rk,c)$ model $\mathcal{W}(\rho_c(\tau))$.

If we fix $c$, we can further assume that for all $0<l\leq c$, the representation
\begin{align*}
\Ind_{\widetilde{P}_{(k^{rl})}}^{\GL^{(m,r)}_{rkl}}((\tau\otimes \ldots \otimes \tau)\delta_{P_{(k^{rl})}}^{-1/(2rk)})
\end{align*}
contains a unique irreducible unramified subrepresentation $\rho_l(\tau)$, which is furthermore an $(rk,l)$ representation,
so that $\rho_l(\tau)\cong \mathcal{W}(\rho_l(\tau))$. This is because we can apply Corollary~\ref{corollary:local spaces are almost always good} for all $l$ in the given range.

\begin{corollary}\label{corollary:realization space for 0 < l < c}
For any $0<l<c$, the representation $\rho_c(\tau)$ is embedded in
\begin{align}\label{rep:composition lemma output space before W}
\Ind_{\widetilde{P}_{(rkl,rk(c-l))}}^{\GL_{rkc}^{(m,r)}}((
\mathcal{W}(\rho_{l}(\tau))\otimes \mathcal{W}(\rho_{c-l}(\tau)))\delta_{P_{(rkl,rk(c-l))}}^{-1/(2rk)}).
\end{align}
\end{corollary}
\begin{proof}
By transitivity of induction, the representation \eqref{nontwisted subrep rho of tau} is contained in
\begin{align}\label{nontwisted subrep rho of tau with l and c}
\Ind_{\widetilde{P}_{(rkl,rk(c-l))}}^{\GL_{rkc}^{(m,r)}}(
(\Ind_{\widetilde{P}_{(k^{rl})}}^{\GL^{(m,r)}_{rkl}}((\tau\otimes \ldots \otimes \tau)\delta_{P_{(k^{rl})}}^{-1/(2rk)})\otimes
\Ind_{\widetilde{P}_{(k^{r(c-l)})}}^{\GL^{(m,r)}_{rk(c-l)}}((\tau\otimes \ldots \otimes \tau)\delta_{P_{(k^{r(c-l)})}}^{-1/(2rk)}))
\delta_{P_{(rkl,rk(c-l))}}^{-1/(2rk)}).
\end{align}
Therefore $\rho_c(\tau)$ is the unique irreducible unramified subrepresentation of \eqref{nontwisted subrep rho of tau with l and c}. The representation \eqref{rep:composition lemma output space before W} is also contained in \eqref{nontwisted subrep rho of tau with l and c} and contains $\rho_c(\tau)$, because it is unramified.
\end{proof}
\begin{proposition}\label{proposition:local L function condition claim}
Conjecture~\ref{Shimura conjecture} implies
\begin{align}\label{eq:local assumption on tau}
L_{\vartheta,\vartheta^{-1}}(s,\tau\times\tau^{\vee})^{-1}\ne0,\qquad \forall \Real(s)\geq1.
\end{align}
\end{proposition}
\begin{proof}
If $m\equiv2\,(4)$, then since $\vartheta$ is unitary, replacing it with the parameter
$(\vartheta_{\tau'})_{\nu}$ appearing in Conjecture~\ref{Shimura conjecture} does not change the condition
\eqref{eq:local assumption on tau}.
Now if \eqref{eq:local assumption on tau} does not hold at a place $\nu$, applying the conjecture twice, first with a finite set of places $S$ which does not contain $\nu$, then with $S\cup\nu$, we arrive at a contradiction.
\end{proof}
By definition $L_{\vartheta,\vartheta^{-1}}(s,\tau\times\tau^{\vee})=\prod_{1\leq i,j\leq k}(1-q^{-s}\chi_i\chi_j^{-1}(\varpi^r))^{-1}$
(so if $m\equiv2\,(4)$, replacing $(\vartheta,\vartheta^{-1})$ by some $(\vartheta',\vartheta'')$ means multiplying $\chi_i\chi_j^{-1}(\varpi^r)$ by $(o,\varpi^r)_2$ for some $o$, $|o|=1$).
Thus by \eqref{eq:local assumption on tau},
\begin{align}\label{eq:local assumption on tau implies on all roots}
\mathbf{x}_{\alpha}=\chi_i(\varpi^r)\chi_j^{-1}(\varpi^r)\ne q^a,\qquad\forall \alpha=(i,j)\in \Phi_k, 0\ne a\in\Z.
\end{align}
In particular \eqref{assumption:local C product for nonvanishing of Whittaker} holds. In fact \eqref{eq:local assumption on tau implies on all roots} is more natural than \eqref{assumption:local C product for nonvanishing of Whittaker} because the latter only depends on $\mathbf{x}_{\alpha}$ for positive roots $\alpha$, while the Satake parameter is defined up to a permutation.
\begin{proposition}\label{proposition:rep induced from chi is irreducible}
If $\mathbf{x}_{\alpha}\ne 0,q$ for all $\alpha\in \Phi_k$, $\mathrm{I}_{\GL_k^{(m,r)}}(\vartheta,\chi)$ is irreducible.
\end{proposition}
\begin{proof}
The result follows from the analog of \cite[Proposition~3.5]{CS1} for covering groups ($\mathbf{x}_{\alpha}\ne 0$ for all $\alpha\in \Phi_k$ implies $\chi$ is regular); see the proof in \cite[Proposition~2.4]{KM}, which is applicable to $\GL_{k}^{(m,r)}$ as well. In more detail, $\mathrm{I}_{\GL_k^{(m,r)}}(\vartheta,\chi)$ is irreducible if and only if $M(J_k)$ is an isomorphism which by virtue of \eqref{eq:GK formula for GL} is equivalent to
\begin{align*}
\prod_{\alpha\in\Phi_d^+}\frac{1-q^{-1}\mathbf{x}_{\alpha}}{1-\mathbf{x}_{\alpha}}
\frac{1-q^{-1}\mathbf{x}_{-\alpha}}{1-\mathbf{x}_{-\alpha}}\ne0.
\end{align*}
This condition holds by our assumption.
\end{proof}
\begin{corollary}\label{corollary:functional on theta implies functional on rho c tau}
Any $(rk,c)$ functional on $\Theta_{rc,m,r,\vartheta}(\chi)$ is well defined and nonzero on $\rho_c(\tau)$.
Moreover, $\mathcal{W}(\rho_c(\tau))\subset \mathcal{W}(\Theta_{rc,m,r,\vartheta}(\chi))$.
\end{corollary}
\begin{proof}
By Corollary~\ref{corollary:local spaces are almost always good}, there are $0\subset V'\subset V\subset \Theta_{rc,m,r,\vartheta}(\chi)$ such that
$\rho_c(\tau)\cong V'\backslash V$. Let $\Lambda$ be an $(rk,c)$ functional on $\Theta_{rc,m,r,\vartheta}(\chi)$. Since
$\Theta_{rc,m,r,\vartheta}(\chi)$ is $(rk,c)$ and the Jacquet functor is exact, $\Lambda$ vanishes on $V'$ hence is well defined on $\rho_c(\tau)$.
If it vanishes on $V$, then again by exactness it must already vanish on $\Theta_{rc,m,r,\vartheta}(\chi)$. The second assertion follows because
with the present notation we can identify $\mathcal{W}(\Theta_{rc,m,r,\vartheta}(\chi))$ with $\mathcal{W}(V'\backslash \Theta_{rc,m,r,\vartheta}(\chi))$.
\end{proof}
This corollary does not imply $\mathcal{W}(\rho_c(\tau))=\mathcal{W}(\Theta_{rc,m,r,\vartheta}(\chi))$ (the latter might be reducible).
\begin{proposition}\label{proposition:local component is subrep for rk c}
Assume $\mathbf{x}_{\alpha}\ne q^a$ for all $\alpha\in \Phi_k^+$ and $0\leq a\leq rc$.
Then $\rho_c(\tau)\subset\Theta_{rc,m,r,\vartheta}(\chi)$.
\end{proposition}
\begin{proof}
By Proposition~\ref{proposition:rep induced from chi is irreducible}, $\mathrm{I}_{\GL_k^{(m,r)}}(\vartheta,\chi)$ is irreducible, in particular
$\tau\subset\mathrm{I}_{\GL_k^{(m,r)}}(\vartheta,\chi)$ so that \eqref{nontwisted subrep rho of tau} and thereby $\rho_c(\tau)$ are subrepresentations of
\begin{align}\label{unramified subrep with tau}
&\Ind_{\widetilde{P}_{(k^{rc})}}^{\GL^{(m,r)}_{rkc}}((\mathrm{I}_{\GL_k^{(m,r)}}(\vartheta,\chi)\otimes \ldots \otimes \mathrm{I}_{\GL_k^{(m,r)}}(\vartheta,\chi))\delta_{P_{(k^{rc})}}^{-1/(2rk)})
\\&=\notag
\Ind_{\widetilde{B}_{\GL_{rkc}}}^{\GL_{rkc}^{(m,r)}}\left(
\Ind_{\widetilde{T}_{\GL_{rkc},*}}^{\widetilde{T}_{\GL_{rkc}}}(
(\otimes_{i=1}^{k}\vartheta\chi_i|~|^{-(rc-1)/(2r)})\otimes \ldots \otimes
(\otimes_{i=1}^{k}\vartheta\chi_i|~|^{(rc-1)/(2r)}))\right).
\end{align}
This is an unramified principal series representation. Consider the standard intertwining operator $M(w)$ on this representation whose image is contained in \eqref{eq:Thete rc m as a subrep}.
We claim that $M(w)$ is well defined, and moreover, it is nonzero on
the normalized unramified vector $\xi^0$ of \eqref{unramified subrep with tau}.
Since $\xi^0$ also belongs to the space of $\rho_c(\tau)$, it will then follow that $M(w)$ restricts to a nonzero operator on
$\rho_c(\tau)$, which must then be an embedding because $\rho_c(\tau)$ is irreducible. Because
$\Theta_{rc,m,r,\vartheta}(\chi)$ is also an unramified subrepresentation of \eqref{eq:Thete rc m as a subrep}, it contains $\rho_c(\tau)$.

To show $M(w)$ is well defined and $M(w)\xi^0\ne0$ we argue as in the proof of
Theorem~\ref{theorem:CS formula for rk, c}. Decompose $M(w)$ into rank-$1$ intertwining operators
\begin{align*}
M(w_{\alpha}):&\mathrm{I}_{\GL_2^{(m,r)}}(\vartheta,
\chi_j|~|^{-(rc-2l'+1)/(2r)}\otimes\chi_i|~|^{-(rc-2l+1)/(2r)})
\\&\rightarrow
\mathrm{I}_{\GL_2^{(m,r)}}(\vartheta,
\chi_i|~|^{-(rc-2l+1)/(2r)}\otimes\chi_j|~|^{-(rc-2l'+1)/(2r)}),
\end{align*}
where $i<j$ and $1\leq l'<l\leq rc$. Since $q^{-l'+l}\mathbf{x}_{(j,i)}\ne1,q$ we deduce that
$M(w_{\alpha})$ is holomorphic and nonzero on the unramified vector (see \eqref{eq:GK formula for GL}).
\end{proof}
We mention that in the linear case, since $\tau$ is irreducible unramified and generic, it is a full induced unramified principal series; then since $\tau$ is also unitary, we have $q^{-1/2}<|\chi_i|<q^{1/2}$ for all $i$ by \cite[Corollary~2.5]{JS1};
consequently, the linear analog of $\Theta_{rc,m,r,\vartheta}(\chi)$ which is
$\Ind_{P_{(c^k)}}^{\GL_{kc}}(\chi_1\circ\det\otimes \ldots \otimes \chi_k\circ\det)$, is irreducible (\cite[Theorem~4.2]{Z3}), then this is $\rho_c(\tau)$. See \cite[\S~2.2]{CFGK2}.
In fact for any $m\geq1$, if we further take $\mathbf{x}$ in general position, we can replace $\rho_l(\tau)$ by $\Theta_{rl,m,r,\vartheta}(\chi)$ and the results of this section hold unconditionally (note that in this
case it is reasonable to expect $\Theta_{rl,m,r,\vartheta}(\chi)$ to be irreducible anyway). Indeed the only nontrivial result is now
Corollary~\ref{corollary:realization space for 0 < l < c}. To this end, because for the normalized Jacquet functor
\begin{align*}
J_{V_{(rl,r(c-l))}}(\Theta_{rc,m,r,\vartheta})=\delta_{P_{(rl,r(c-l))}}^{-1/(2r)}(\Theta_{rl,m,r,\vartheta}\otimes\Theta_{r(c-l),m,r,\vartheta})
\end{align*}
(see \cite[Theorem~5.1]{Kable}, the proof here is simpler), and since $\Theta_{rc,m,r,\vartheta}$ is irreducible,
\begin{align*}
\Theta_{rc,m,r,\vartheta}\subset
\Ind_{\widetilde{P}_{(rl,r(c-l))}}^{\GL_{rc}^{(m,r)}}((\Theta_{rl,m,r,\vartheta}\otimes\Theta_{r(c-l),m,r,\vartheta})\delta_{P_{(rl,r(c-l))}}^{-1/(2r)}).
\end{align*}
Denote the r.h.s.~ by $\Theta[rl,r(c-l)]$, then
\begin{align}\label{Theta_rc,m,r,vartheta rhs}
\Theta_{rc,m,r,\vartheta}(\chi)\subset
\Ind_{\widetilde{P}_{((rc)^k)}}^{\GL_{rkc}^{(m,r)}}(\chi_1\Theta[rl,r(c-l)]\otimes\ldots\otimes\chi_k\Theta[rl,r(c-l)]).
\end{align}
When we regard each $\Theta[rl,r(c-l)]$ as a subrepresentation of an unramified principal series representation, we can apply an intertwining operator $M(w)$ to permute the inducing data of the r.h.s.~ of \eqref{Theta_rc,m,r,vartheta rhs}, only involving blocks corresponding to $\chi_i$ and $\chi_j$ for $i\ne j$. This operator will be an isomorphism (see the proofs of Propositions~\ref{proposition:rep induced from chi is irreducible} and \ref{proposition:local component is subrep for rk c}), so that
\begin{align*}
\Theta_{rc,m,r,\vartheta}(\chi)\subset
\Ind_{\widetilde{P}_{(rkl,rk(c-l))}}^{\GL_{rkc}^{(m,r)}}((
\Theta_{rl,m,r,\vartheta}(\chi)\otimes \Theta_{r(c-l),m,r,\vartheta}(\chi))\delta_{P_{(rkl,rk(c-l))}}^{-1/(2rk)}).
\end{align*}

\subsection{Local decomposition of $(rk,c)$ functionals}\label{decomposition of functionals}
We proceed with the notation and assumptions of \S~\ref{Local components of rk c speh}.
In particular \eqref{eq:local assumption on tau} holds. In this section we construct an $(rk,c)$ functional on
$\rho_c(\tau)$ inductively. We use the $2$-cocycle $\sigma^{\diamondsuit}_{rkc}$.
We will repeatedly use the facts that $b\mapsto\langle b,1\rangle$ is a splitting for both $N_{\GL_{rkc}}$ and the subgroup of permutation matrices in $\GL_{rkc}$, and also recall that $\langle y,\eta_{rkc}^{\diamondsuit}(y)\rangle$ is the chosen splitting
for $K_{\GL_{rkc}}$ (see \S~\ref{covering of the Levi}).

Since \eqref{assumption:local C product for nonvanishing of Whittaker} holds,
Theorem~\ref{theorem:CS formula for rk, c} implies that the Jacquet integral $\Lambda$ (see \S~\ref{whittaker functionals}) is an $(rk,1)$ functional on $\Theta_{r,m,r,\vartheta}(\chi)$, thereby on $\rho_1(\tau)$ by Corollary~\ref{corollary:functional on theta implies functional on rho c tau}.

Fix $0<l<c$.
By Corollary~\ref{corollary:realization space for 0 < l < c}, $\rho_c(\tau)$ is a subrepresentation
of \eqref{rep:composition lemma output space before W}:
\begin{align*}
\rho_c(\tau)\subset \Ind_{\widetilde{P}_{(rkl,rk(c-l))}}^{\GL_{rkc}^{(m,r)}}((
\mathcal{W}(\rho_{l}(\tau))\otimes \mathcal{W}(\rho_{c-l}(\tau)))\delta_{P_{(rkl,rk(c-l))}}^{-1/(2rk)}).
\end{align*}

We construct an $(rk,c)$ functional on \eqref{rep:composition lemma output space before W}, and prove it does not vanish
on $\rho_c(\tau)$, hence can be used to realize $\mathcal{W}(\rho_c(\tau))$. Set
\begin{align*}
&\kappa=\kappa_{l,c-l}=\left(\begin{smallmatrix}I_l\\0&0&I_l\\0&0&0&0&I_l&\ddots\\&&&&&&I_l&0\\0&I_{c-l}\\0&0&0&I_{c-l}&&\ddots\\&&&&&&&I_{c-l}\end{smallmatrix}\right)\in\GL_{rkc}.
\end{align*}
For $v=(v_{i,j})_{1\leq i,j\leq rk}\in V_{(c^{rk})}$, write each block $v_{i,j}\in\Mat_c$ in the form
\begin{align}\label{eq:the blocks of V}
\left(\begin{smallmatrix}v_{i,j}^1&v_{i,j}^2\\v_{i,j}^3&v_{i,j}^4\end{smallmatrix}\right),\qquad
v_{i,j}^1\in\Mat_{l},\quad v_{i,j}^4\in\Mat_{c-l}.
\end{align}
Let $V_{(c^{rk})}^t<V_{(c^{rk})}$ be the subgroup obtained by deleting the blocks $v_{i,j}^{t'}$ for all $i<j$ and $t'\ne t$,
where $1\leq t\leq 4$. Put $V=V^3$. Consider the following integral for $\xi$ in the space of \eqref{rep:composition lemma output space before W}:
\begin{align}\label{eq:mnk functional using w_{n,m,k}}
\int\limits_{V}\xi(\langle \kappa,1\rangle\langle v,1\rangle)\,dv.
\end{align}
\begin{example}
When $c=r=k=2$ (then $l=1$),
\begin{align*}
&\kappa_{1,1}=\left(\begin{smallmatrix}1\\&&1\\&&&&1\\&&&&&&1\\&1\\&&&1\\&&&&&1\\&&&&&&&1\end{smallmatrix}\right),\qquad
V=\left\{\left(\begin{smallmatrix}1&&&&&&\\&1&v_{1,2}^3&&v_{1,3}^3&&v_{1,4}^3&
\\&&1&&&&\\&&&1&v_{2,3}^3&&v_{2,4}^3&\\
&&&&1&&&\\&&&&&1&v_{3,4}^3&\\&&&&&&1\\&&&&&&&1\end{smallmatrix}\right)\right\}.
\end{align*}
\end{example}

This defines an $(rk,c)$ functional, at least formally. To see this first note that
\begin{align*}
{}^{\kappa}V^1<\diag(V_{(l^{rk})},I_{rk(c-l)}),\quad, {}^{\kappa}V^2<V_{(rkl,rk(c-l))},\quad
{}^{\kappa}V<V_{(rkl,rk(c-l))}^-,\quad  {}^{\kappa}V^4<\diag(I_{rkl},V_{((c-l)^{rk})}).
\end{align*}
Hence $V^1$ and $V^4$ commute, both normalize $V^2$ and $V$, and the subgroups $V^2$, $V$ are abelian.
Also by \eqref{eq:sigma conjugate v by h}, for $v^t\in V^t$ with $t\ne 3$,
\begin{align*}
\langle\kappa,1\rangle\langle v^t,1\rangle=\langle{}^{\kappa}v^t,1\rangle\langle\kappa,1\rangle.
\end{align*}
It follows that for $u=u^1u^4u^2u^3\in V_{(c^{rk})}$ with $u^t\in V^t$,
\begin{align*}
\int\limits_{V}\xi(\langle \kappa,1\rangle\langle v,1\rangle\langle u,1\rangle)\,dv
=\int\limits_{V}\xi(\langle {}^{\kappa}u_v^1,1\rangle\langle {}^{\kappa}u_v^4,1\rangle\langle {}^{\kappa}u_v^2,1\rangle\langle \kappa,1\rangle\langle v,1\rangle)\,dv.
\end{align*}
Here we wrote $vu=(u_v^1u_v^4u_v^2)v_u$ where $u_v^t\in V^t$, $v_u\in V$, changed variables in $v_u$, and note that
since $\psi$ is a character of $V_{(c^{rk})}$ which is trivial on $V$, $\psi(u^1u^4u^2)=\psi(u_v^1u_v^4u_v^2)$.

The $(rk,c)$ character restricted to $V^1$ (resp., $V^4$) coincides with the $(rk,l)$ (resp., $(rk,c-l)$) character on
${}^{\kappa}V^1$ (resp., ${}^{\kappa}V^4$), and moreover, the $(rk,c)$ character is trivial on $V^2$.
Then by the definition of \eqref{rep:composition lemma output space before W}, vectors in the space of
\eqref{rep:composition lemma output space before W} transform on the left under
$\langle {}^{\kappa}V^t,1\rangle$ by the $(rk,l)$
character for $t=1$, $(rk,c-l)$ character for $t=4$ or trivially for $t=2$. Since the $(rk,c)$ character is also trivial on $V$,
we conclude that \eqref{eq:mnk functional using w_{n,m,k}} is an $(rk,c)$ functional on \eqref{rep:composition lemma output space before W}.

\begin{proposition}\label{proposition: rk c functional is nontrivial on nontwisted}
The integral \eqref{eq:mnk functional using w_{n,m,k}} is absolutely convergent and nonzero on any summand of
\eqref{rep:composition lemma output space before W}. In particular it is an $(rk,c)$ functional on $\rho_c(\tau)$. Moreover,
\eqref{eq:mnk functional using w_{n,m,k}} is nonzero on any
nonzero unramified vector $\xi$.
\end{proposition}
\begin{proof}
The proof is similar to the linear case (\cite[Lemma~10]{CFK}), where we argued using ``root elimination" (see e.g., \cite{G} and \cite[Proposition~6.1]{Soudry}). Here we focus on the computations that are different for the covering.
Let $X<{}^{\kappa^{-1}}V_{(rkl,rk(c-l))}$ be the subgroup of matrices $x$ such that the top right $rkl\times rk(c-l)$ block
of ${}^{\kappa}x$ is
\begin{align*}
&\left(\begin{smallmatrix}
0&0&\cdots&0\\
\vdots&x_{1,2}^3&\ddots&\vdots\\
\vdots&\vdots&\ddots&0\\
0&x_{1,rk}^3&\cdots&x_{rk-1,rk}^3\end{smallmatrix}\right),\qquad x_{i,j}^3\in\Mat_{l\times (c-l)}.
\end{align*}
For each pair $i<j$, let $y_{i,j}$ and $x_{i,j}$ be arbitrary elements of $V$ and $X$ (resp.), such that $v_{i',j'}^3=0$ and $x_{i',j'}^3=0$ for all $(i',j')\ne(i,j)$.

We fix the following ordering on the pairs $(i,j)$ with $i<j$. The first pair is $(rk-1,rk)$, next $(rk-2,rk-1),(rk-3,rk-2)\ldots,(1,2)$. Then $(rk-2,k),(rk-3,rk-1),\ldots,(1,3)$, etc. The last three pairs are $(2,rk)$, $(1,rk-1)$ and $(1,rk)$. Write $(i',j')<(i,j)$ if $(i',j')$ appears after $(i,j)$ with respect to this ordering. Assume we have reached the pair $(i,j)$. Then instead of an integral over $V$ we will integrate over the product of two subgroups: one generated by elements $y_{i.j}$, the other denoted $V'$ is generated by elements $y_{i',j'}$ for all $(i',j')<(i,j)$.

Let $\xi$ be an arbitrary function in the space of \eqref{rep:composition lemma output space before W}.
Observe that for $v'\in V'$,
\begin{align*}
{}^{x_{i,j}}(v'y_{i,j})=y_xl_xv_{x,y}'y_{i,j},
\end{align*}
where $y_x,l_x\in N_{\GL_{rkc}}$, $v_{x,y}'\in V'$ depends on $(v',x_{i,j},y_{i,j})$,
${}^{\kappa}y_x\in M_{(rkl,rk(c-l))}\cap N_{\GL_{rkc}}$, $\xi$ transforms on the left under
$\langle {}^{\kappa}y_x,1\rangle$ by $\psi(\mathrm{tr}(v_{i,j}^3x_{i,j}^3))$,
${}^{\kappa}l_x\in N_{\GL_{rkc}}$ and $\xi$ is left invariant on $\langle{}^{\kappa}l_x,1\rangle$.
Then by \eqref{eq:sigma on vh and h'v'} and \eqref{eq:sigma conjugate v by h},
\begin{align}\label{eq:conj kappa x i j 1}
\xi(\langle\kappa,1\rangle\,{}^{x_{i,j}}\langle v'y_{i,j},1\rangle)&=
\xi(\langle\kappa,1\rangle\langle y_x,1\rangle\langle l_x,1\rangle\langle v_{x,y}'y_{i,j},1\rangle)
\\&=\xi(\langle {}^{\kappa}y_x,1\rangle\langle {}^{\kappa}l_x,1\rangle\langle\kappa,1\rangle\langle v_{x,y}'y_{i,j},1\rangle)\nonumber
\\&=\psi(\mathrm{tr}(v_{i,j}^3x_{i,j}^3))\xi(\langle\kappa,1\rangle\langle v_{x,y}'y_{i,j},1\rangle).\nonumber
\end{align}

If $x_{i,j}^3\in\Mat_{l\times (c-l)}(\mathcal{O})$, then $x_{i,j}\in K_{\GL_{rkc}}$ and
${}^{\kappa}x_{i,j}\in V_{(rkl,rk(c-l))}(\mathcal{O})$. Since also $\kappa\in K_{\GL_{rkc}}$,
\begin{align*}
{}^{\kappa}\langle x_{i,j},\eta_{rkc}^{\diamondsuit}(x_{i,j})\rangle=
\langle {}^{\kappa}x_{i,j},\eta_{rkc}^{\diamondsuit}({}^{\kappa}x_{i,j})\rangle
=\langle {}^{\kappa}x_{i,j},1\rangle.
\end{align*}
Depending on $\xi$, if the coordinates of $x_{i,j}^3$ are small enough, $\xi$ is right-invariant on
$\langle x_{i,j},\eta_{rkc}^{\diamondsuit}(x_{i,j})\rangle$.
Combining this with \eqref{eq:conj kappa x i j 1} and the
left-invariance of $\xi$ on $\langle v,1\rangle$ for $v\in V_{(rkl,rk(c-l))}$, we obtain
\begin{align}\label{eq:functional xi psi invariance}
\xi(\langle\kappa,1\rangle \langle v'y_{i,j},1\rangle)&=
\xi(\langle\kappa,1\rangle \langle v'y_{i,j},1\rangle\langle x_{i,j},\eta_{rkc}^{\diamondsuit}(x_{i,j})\rangle)\\\nonumber
&=\xi(\langle {}^{\kappa}x_{i,j},1\rangle\langle\kappa,1\rangle {}^{x_{i,j}}\langle v'y_{i,j},1\rangle)\\\nonumber
&=\psi(\mathrm{tr}(v_{i,j}^3x_{i,j}^3))\xi(\langle\kappa,1\rangle\langle v_{x,y}'y_{i,j},1\rangle).\nonumber
\end{align}

Therefore using \eqref{eq:functional xi psi invariance} and a change of variables in $v'$,
\begin{align*}
\int\limits_{V'}\xi(\langle\kappa,1\rangle\langle v'y_{i,j},1\rangle)\,dv'
=\psi(\mathrm{tr}(v_{i,j}^3x_{i,j}^3))
\int\limits_{V'}\xi(\langle\kappa,1\rangle\langle v'y_{i,j},1\rangle)\,dv'.
\end{align*}
Now as in the linear case (\cite[Lemma~10]{CFK}), these computations first imply that the integrand is a Schwartz function
of $V$, hence \eqref{eq:mnk functional using w_{n,m,k}} is absolutely convergent. Second, it does not vanish on any summand
of \eqref{rep:composition lemma output space before W}, in particular on $\rho_c(\tau)$.
Third, when taking an unramified vector $\xi$, we can take $x_{i,j}^3$ with coordinates in $\mathcal{O}^*$ and deduce that the $dv'$-integral vanishes unless $v_{i,j}^3\in\Mat_{l\times (c-l)}(\mathcal{O})$, and since in this case $\xi$ is already right-invariant under $\langle y_{i,j},1\rangle$,
\begin{align*}
\int\limits_{V}\xi(\langle\kappa,1\rangle\langle v,1\rangle)\,dv=\mathrm{vol}(\mathcal{O})^{l(c-l)}\int\limits_{V'}\xi(\langle\kappa,1\rangle\langle v',1\rangle)\,dv'.
\end{align*}
Repeating this (in the order defined above) we obtain a nonzero constant multiplied by $\xi(\langle I_{rkc},1\rangle)$.

It remains to show that we can actually choose an unramified $\xi$ in the space of \eqref{rep:composition lemma output space before W} such that $\xi(\langle I_{rkc},1\rangle)\ne0$.
We argue using induction on $c$. For the base case $c=2$,
$\xi(\langle I_{rk2},1\rangle)=\Lambda(\xi_1)\Lambda(\xi_2)$,
where $\xi_i$ are unramified vectors in the space of $\rho_1(\tau)$ such that $\xi_i(\langle I_{rk},1\rangle)\ne0$.
The r.h.s.~ is nonzero by Theorem~\ref{theorem:CS formula for rk, c}, proving the base case. For $c>2$, we take
$l=1$ and obtain $\xi(\langle I_{rkc},1\rangle)=\Lambda(\xi_1)\Lambda'(\xi_2)$,
where now $\xi_2$ is an unramified vector in the space of $\rho_{c-1}(\tau)$
such that $\xi_2(\langle I_{rk(c-1)},1\rangle)\ne0$, and $\Lambda'$ is an $(rk,c-1)$ functional on $\rho_{c-1}(\tau)$.
By the inductive hypothesis $\Lambda'(\xi_2)\ne0$, so that $\xi(\langle I_{rkc},1\rangle)\ne0$.
\end{proof}

\subsection{Rankin--Selberg integrals}\label{RS integrals}
Assume $\mu_{2m}\subset F^*$. For this section we also assume $rk>1$ (the case $rk=1$ of \eqref{eq:Z integral GL 1 GL rk} below is known, see \cite{JPSS,Gan}).
Let $\tau$ be the genuine irreducible unramified constituent of $\mathrm{I}_{\GL_k^{(m)}}(\vartheta,\chi)$, and recall the representation $\Theta_{r,m,r,\vartheta_{\tau}}(\chi)$ of \S~\ref{local theta speh for c=1}.
By Proposition~\ref{proposition:local-semi}, it admits a unique Whittker model $\mathcal{W}(\Theta_{r,m,r,\vartheta}(\chi))$. Note that
$\Theta_{r,m,r,\vartheta_{\tau}}(\chi)$ might be reducible. We compute a Rankin--Selberg type integral, which will appear in the final part of the reduction used for the computation of the doubling integrals with unramified data (see \S~\ref{final reduction n = 1 linear groups}). The contents of this section, however, are independent of the doubling construction.

Let $\pi$ be a genuine irreducible unramified representation of $\GL_1^{(m,r)}$ and write $\pi=\mathrm{I}_{\GL_1^{(m,r)}}(\vartheta,\mu)$. The contragredient representation $\pi^{\vee}$ is induced from $\varepsilon^{-1}\otimes\vartheta^{-1}\mu^{-1}$. Let $\omega$ be a matrix
coefficient of $\pi^{\vee}$ and $W\in \mathcal{W}(\Theta_{r,m,r,\vartheta}(\chi))$. The integral
\begin{align}\label{eq:Z integral GL 1 GL rk}
Z(s,\omega,W)=\int\limits_{F^*}\omega(\langle a,1\rangle)W(\langle \diag(a,I_{rk-1}),1\rangle)|a|^{s-(rk-1)/2}\,d^*a
\end{align}
is formally well defined, and absolutely convergent in a right half-plane. This can be regarded as the covering analog of the Rankin--Selberg convolution for $\GL_1\times\GL_k$ of \cite[\S~2.4(3)]{JPSS} (with the parameter $j=0$). Now assume that
$\omega$ and $W$ are normalized and unramified. According to Proposition~\ref{proposition:support of unr is in T*}, the integrand vanishes unless $a\in F^{*r}$, so that
\begin{align*}
Z(s,\omega,W)=\int\limits_{F^{*r}}W(\langle \diag(a,I_{rk-1}),1\rangle)\vartheta^{-1}(a)\mu^{-1}(a)|a|^{s-(rk-1)/2}\,d^*a.
\end{align*}
Assume $\int_{\mathcal{O}^*}d^*a=1$. Since $W(\langle\diag(a,I_{rk-1}),1\rangle,s)=0$ unless $|a|\leq1$ (see e.g., \cite[\S~6]{CS2}), the $d^*a$-integral can be written as the infinite sum over $\varpi^{lr}$ where $l\geq0$.
Let $\mathbf{x}=(\chi_1(\varpi^r),\ldots,\chi_k(\varpi^r))$, which we identify with
$t_{\tau,\vartheta}=\diag(\chi_1(\varpi^r),\ldots,\chi_k(\varpi^r))$. By
Theorem~\ref{theorem:Whittaker on x,I},
\begin{align*}
W(\langle\diag(\varpi^{lr},I_{rk-1}),1\rangle)=\vartheta(\varpi^{lr})q^{(-lr(rk-1)/2)+l(r-1)/2}p_l(\mathbf{x}).
\end{align*}
Also $t_{\pi^{\vee},\vartheta}=t_{\pi^{\vee},\vartheta^{-1}}=\mu^{-1}(\varpi^{r})$ (if $m\equiv2\,(4)$, $\vartheta^{-1}=\vartheta$ when $\mu_{2m}\subset F^*$). Now
\begin{align}\label{RS integral for n=1}
Z(s,\omega,W)&=\sum_{l=0}^{\infty}p_l(\mathbf{x})\mu^{-1}(\varpi^{lr})q^{-l(rs-(r-1)/2)}\\&=\nonumber \prod_{1\leq i\leq k}(1-q^{-r(s-1/2)-1/2}\chi_i(\varpi^r)\mu^{-1}(\varpi^r))^{-1}=
L_{\vartheta}(r(s-1/2)+1/2,\pi^{\vee}\times\tau).
\end{align}
In particular $Z(s,\omega,W)$ is a rational function in $q^{-s}$.

\section{The construction}\label{global}
\subsection{The global integral}\label{global classical}
In this section we construct the global integral, following the linear case from \cite{CFGK2} with the necessary modifications and adjustments
for the covering. Let $F$ be a number field, and assume $\mu_{2m}\subset F^*$. We use the notation and definitions of \S~\ref{embedding}. In
particular $n$, $k$ and $m$ are integers, $c=2n$, $G=\Sp_c$ and $H=\Sp_{2rkc}$. We also have the Siegel parabolic subgroup $P=M_P\ltimes U_P$,
the parabolic subgroup $Q=M_Q\ltimes U$ and the character $\psi_U$ of $U$ given by \eqref{eq:character of U}. Then $G\times G$ is embedded in $H$ in the stabilizer of $\psi_U$ in $M_Q$.

As explained in \S~\ref{global covering}, we fix the global $2$-cocycle $\rho$ on $H^{(m)}(\A)$. This defines the $2$-cocycles on
each copy of $G^{(m)}(\A)$: $\rho_L$ on the left copy, $\rho_R$ on the right, and these subgroups commute in $H^{(m)}(\A)$. It also defines an identification of $H(F)$ in $H^{(m)}(\A)$, and of $G(F)$ in each copy. The notions of automorphic forms on these groups are now defined.
In addition, if $\varphi_1$ is an automorphic form on the right copy,
$\varphi_1\mapsto\varphi_1^{(\eta^{\times})^{-1}}$ is an automorphic form on the left copy, by
Corollary~\ref{corollary:eta times takes automorphic to automorphic}. Also recall the involution ${}^{\iota}$, lifted to $G^{(m)}(\A)$ in \S~\ref{extension of the involution}.

Let $\tau$ be a genuine unitary irreducible cuspidal automorphic representation of $\GL_{k}^{(m,r)}(\A)$.
Here and throughout \S~\ref{global}, we assume Conjectures~\ref{local Shimura conjecture} and \ref{Shimura conjecture} hold. Then let
$\mathcal{E}_{\tau}$ be the genuine irreducible automorphic $(rk,c)$ representation guaranteed by
Theorem~\ref{exthspeh1}. Consider the representation
\begin{align*}
\Ind_{\widetilde{P}({\A})}^{H^{(m)}({\A})}(\mathcal{E}_{\tau}\delta_P^s).
\end{align*}
For a standard $\widetilde{K}_H$-finite section $f$ in the space of this representation, the Eisenstein series $E(h;f,s)$ is defined by
\begin{align}\label{eq:ES for Hm}
E(h;f,s)=\sum_{\gamma\in P(F)\backslash H(F)}f(\langle\gamma,\eta^{-1}(\gamma)\rangle h,s),
\end{align}
which is absolutely convergent for $\Real(s)\gg0$ and defined for a general $s$ by meromorphic continuation.

Let $\pi_1$ and $\pi_2$ be genuine irreducible cuspidal automorphic representations of $G^{(m)}(\A)$, where $G^{(m)}(\A)$ is realized using $\rho_R$. Let $\varphi_i$ be a cusp form in the space of $\pi_i$, $i=1,2$.
The global integral is given by
\begin{align}\label{global1}
Z(s,\varphi_1,\varphi_2,f)=&\int\limits_{G(F)\times G(F)\backslash G({\A})\times G({\A})}\,
\int\limits_{U(F)\backslash U({\A})}\varphi_1^{(\eta^{\times})^{-1}}(\langle g_1,1\rangle)\,\overline{{}^{\iota}\varphi_2(\langle g_2,1\rangle)}
\\&\times E(\langle u,\eta^{-1}(u)\rangle\nonumber
\langle \mathfrak{e}_1(g_1),1\rangle\langle \mathfrak{e}_2(g_2),1\rangle;f,s)\,\psi_U(u)\,du\,dg_1\,dg_2.
\end{align}
\begin{theorem}\label{theorem:main theorem classical groups}
Integral \eqref{global1} is formally well defined, absolutely convergent away from the poles of
the series, and admits meromorphic continuation to the plane.
\end{theorem}
\begin{proof}
Since the image of $G(F)\times G(F)$ in $H(F)$ normalizes $U$ without changing the measure, and fixes $\psi_U$, and by Lemma~\ref{lemma:conjugation of N by H} we have
${}^{g}\langle u,\eta^{-1}(u)\rangle=\langle {}^gu,\eta^{-1}({}^gu)\rangle$ for $g=(g_1,g_2)$, the $du$-integral is well defined on the domain of the outer integral.
The outer integral is well defined by Proposition~\ref{proposition:global toy integral automorphic}.
Now convergence and continuation follow from the rapid decay of cusp forms and from the moderate growth and meromorphic continuation
of the Eisenstein series.
\end{proof}

\subsection{Obtaining the Euler product}\label{global symplectic}
To state the unfolding theorem, we introduce the following notation.
Let $L^2(G(F)\backslash G^{(m)}(\A))$ be the space of genuine square-integrable automorphic forms on $G^{(m)}(\A)$.
Denote the standard $G^{(m)}(\A)$-invariant inner product on $L^2(G(F)\backslash G^{(m)}(\A))$ by $\{\cdot,\cdot\}$,
\begin{align*}
\{\varphi_1,\varphi_2\}=
\int\limits_{G(F)\backslash G(\A)}\varphi_1(\langle g,1\rangle)\overline{\varphi_2(\langle g,1\rangle)}\,dg.
\end{align*}
This integral is well defined because both $\varphi_1$ and $\varphi_2$ are genuine functions.
While the definition itself does not depend on the $2$-cocycle realizing $G^{(m)}(\A)$, in the construction it will be $\rho_R$.

Let
\begin{align*}
&\delta=\delta_0\delta_1,\qquad \delta_0=\left(\begin{smallmatrix} &I_{rkc}\\ -I_{rkc}\end{smallmatrix}\right),\qquad
\delta_1=\left(\begin{smallmatrix}
I_{r(k-1)c}&&&&\\ &I_{c}&&I_c&\\ &&&I_{c}&\\
&&&&I_{r(k-1)c}\end{smallmatrix}\right),\\
&U_0=U\cap U_{P}=
\left\{\left(\begin{array}{cccc}I_{(k-1)c}&&X&Z\\&I_c&&X'\\&&I_c\\&&&I_{(k-1)c}\end{array}\right)\in H\right\}\qquad\left(\begin{array}{c}{}^tZJ_{(k-1)c}-J_{(k-1)c}Z=0\\ X'=J_{c}{}^tXJ_{(k-1)c}\end{array}\right).
\end{align*}
The character $\psi_U$ restricts to a character of $U_0$ and
\begin{align*}
&\psi_U(u_0)=\psi(\tr(\left(\begin{array}{cc}0&I_n\end{array}\right)X\left(\begin{array}c0\\I_n\end{array}\right))).
\end{align*}

Recall the embedding
\begin{align*}
\mathfrak{e}_2:G\hookrightarrow H,\qquad \mathfrak{e}_2(g)=\diag(I_{r(k-1)c+n},g,I_{n+r(k-1)c}).
\end{align*}
Also let $f_{\mathcal{W}(\mathcal{E}_{\tau})}$ denote the composition of the section $f$ with the $(rk,c)$ functional
\eqref{int:a c general Fourier coeff} attached to the $(rk,c)$ representation $\mathcal{E}_{\tau}$, namely
\begin{align}\label{eq:def of f W global}
f_{\mathcal{W}(\mathcal{E}_{\tau})}(h,s)=\int\limits_{V_{(c^{rk})}(F)\backslash V_{(c^{rk})}(\A)}f(\langle v,(\eta_{rkc}^{\diamondsuit})^{-1}(v)\rangle h,s)\psi^{-1}(v)\,dv.
\end{align}
\begin{theorem}\label{theorem:main gbl identity}
Integral \eqref{global1} is not identically zero only if
$\pi_1=\pi_2=\pi$. In this case for $\Real(s)\gg0$ it is equal to
\begin{align}\label{global2}
\int\limits_{G({\A})}\int\limits_{U_0({\A})}
\{\varphi_1,\pi(\langle g,1\rangle)\varphi_2\}f_{\mathcal{W}(\mathcal{E}_{\tau})}(\langle\delta u_0,\eta^{-1}(\delta u_0)\rangle
{}^{\iota}\langle\mathfrak{e}_2(g),1\rangle,s)
\,\psi_U(u_0)\,du_0\,dg.
\end{align}
\end{theorem}
\begin{proof}
Plugging the definition of the Eisenstein series into \eqref{global1},
integral~\eqref{global1} becomes
\begin{align*}
&\int\limits_{G(F)\times G(F)\backslash G({\A})\times G({\A})}\,
\int\limits_{U(F)\backslash U({\A})}\varphi_1^{(\eta^{\times})^{-1}}(\langle g_1,1\rangle)\,\overline{\varphi_2({}^{\iota}\langle g_2,1\rangle)}\,
\\&\times \sum\limits_{\gamma\in P(F)\backslash H(F)}f(\langle \gamma,\eta^{-1}(\gamma)\rangle\langle u,\eta^{-1}(u)\rangle\nonumber
\langle \mathfrak{e}_1(g_1),1\rangle\langle \mathfrak{e}_2(g_2),1\rangle,s)\,\psi_U(u)\,du\,dg_1\,dg_2.
\end{align*}
Let $R=(\mathfrak{e}_1(G)\times \mathfrak{e}_2(G))\ltimes U<Q$. The group $R(F)$ acts on the right on the homogenous space
$P(F)\backslash H(F)$. The stabilizer of $P(F)h$ is $R_{h}(F)={}^{h^{-1}}P(F)\cap R(F)$.
For $\Real(s)\gg0$, we can write the last integral as
\begin{align*}
&\sum\limits_{\gamma\in P(F)\backslash H(F)/R(F)}\quad\int\limits_{G(F)\times G(F)\backslash G({\A})\times G({\A})}
\quad\int\limits_{U(F)\backslash U({\A})}\varphi_1^{(\eta^{\times})^{-1}}(\langle g_1,1\rangle)\,\overline{\varphi_2({}^{\iota}\langle g_2,1\rangle)}\,
\\&\times\sum\limits_{y\in R_{\gamma}(F)\backslash R(F)} f(\langle \gamma,\eta^{-1}(\gamma)\rangle\langle y,\eta^{-1}(y)\rangle\langle u,\eta^{-1}(u)\rangle\nonumber
\langle \mathfrak{e}_1(g_1),1\rangle\langle \mathfrak{e}_2(g_2),1\rangle,s)\,\psi_U(u)\,du\,dg_1\,dg_2.
\end{align*}
Write $y=y_3\mathfrak{e}_1(y_1)\mathfrak{e}_1(y_2)$ for $y_3\in U(F)$ and $y_1,y_2\in G(F)$.
In $H^{(m)}(\A)$ we can write
\begin{align*}
\langle y,\eta^{-1}(y)\rangle=\langle y_3,\eta^{-1}(y_3)\rangle
\langle \mathfrak{e}_1(y_1),\eta^{-1}(\mathfrak{e}_1(y_1))\rangle
\langle \mathfrak{e}_2(y_2),\eta^{-1}(\mathfrak{e}_1(y_2))\rangle.
\end{align*}
The group $G\times G$ normalizes $U$ and stabilizes $\psi_U$, hence we can conjugate $\mathfrak{e}_1(y_1)$ and $\mathfrak{e}_1(y_2)$ to the right, and by Lemma~\ref{lemma:conjugation of N by H} and \eqref{gbl rho on left and right} we obtain
\begin{align*}
&\sum\limits_{\gamma\in P(F)\backslash H(F)/R(F)}\quad\int\limits_{G(F)\times G(F)\backslash G({\A})\times G({\A})}
\quad\int\limits_{U(F)\backslash U({\A})}\varphi_1^{(\eta^{\times})^{-1}}(\langle g_1,1\rangle)\,\overline{\varphi_2({}^{\iota}\langle g_2,1\rangle)}\,
\\&\sum\limits_{y\in R_{\gamma}(F)\backslash R(F)}\eta^{-1}(\mathfrak{e}_1(y_1))\eta^{-1}(\mathfrak{e}_2(y_2))\rho_L^{-1}(y_1,g_1)\rho_R(y_2,g_2)
\\& f(\langle \gamma,\eta^{-1}(\gamma)\rangle\langle y_3u,\eta^{-1}(y_3u)\rangle
\langle \mathfrak{e}_1(y_1g_1),1\rangle\langle \mathfrak{e}_2(y_2g_2),1\rangle,s)\psi_U(u)\,du\,dg_1\,dg_2.
\end{align*}
We also have
\begin{align*}
&\varphi_1^{(\eta^{\times})^{-1}}(\langle g_1,1\rangle)=\varphi_1^{(\eta^{\times})^{-1}}(\langle y_1,\eta(\mathfrak{e}_1(y_1))\rangle\langle g_1,1\rangle)=
\eta(\mathfrak{e}_1(y_1))\rho_L(y_1,g_1)\varphi_1^{(\eta^{\times})^{-1}}(\langle y_1g_1,1\rangle),
\end{align*}
and using \eqref{eq:gbl iota compatible with G(F)},
\begin{align*}
&\varphi_2({}^{\iota}\langle g_2,1\rangle)
=\varphi_2({}^{\iota}\langle y_2,\eta^{-1}(\mathfrak{e}_2(y_2))\rangle\,{}^{\iota}\langle g_2,1\rangle)=\eta(\mathfrak{e}_2(y_2))^{-1}\rho_R(y_2,g_2)\varphi_2({}^{\iota}\langle y_2g_2,1\rangle).
\end{align*}
Combining these computations, we obtain
\begin{align*}
&\sum\limits_{\gamma\in P(F)\backslash H(F)/R(F)}\quad\int\limits_{G(F)\times G(F)\backslash G({\A})\times G({\A})}
\quad\int\limits_{U(F)\backslash U({\A})}\sum\limits_{y\in R_{\gamma}(F)\backslash R(F)}\varphi_1^{(\eta^{\times})^{-1}}(\langle y_1g_1,1\rangle)\,\overline{\varphi_2({}^{\iota}\langle y_2g_2,1\rangle)}\,
\\& f(\langle \gamma,\eta^{-1}(\gamma)\rangle\langle y_3u,\eta^{-1}(y_3u)\rangle\nonumber
\langle \mathfrak{e}_1(y_1g_1),1\rangle\langle \mathfrak{e}_2(y_2g_2),1\rangle,s)\psi_U(u)\,du\,dg_1\,dg_2.
\end{align*}
Now we can collapse the summation into the integral, and we reach the sum
\begin{align*}
\sum_{\gamma\in P(F)\backslash H(F)/R(F)}\mathrm{I}(\gamma),
\end{align*}
where
\begin{align*}
\mathrm{I}(\gamma)=&\int\limits_{R_\gamma(F)\backslash R({\A})}\varphi_1^{(\eta^{\times})^{-1}}(\langle g_1,1\rangle)\,\overline{\varphi_2({}^{\iota}\langle g_2,1\rangle)}\\&
f(\langle \gamma,\eta^{-1}(\gamma)\rangle\langle u,\eta^{-1}(u)\rangle\langle \mathfrak{e}_1(g_1),1\rangle\langle \mathfrak{e}_2(g_2),1\rangle,s)\,\psi_U(u)\,du\,dg_1\,dg_2.
\end{align*}

As in the linear case, first we show that the summands $\mathrm{I}(\gamma)$ such that $\gamma\notin P(F)\delta R(F)$ vanish, then we
prove $\mathrm{I}(\delta)$ equals (after some modifications) \eqref{global2}.
We start with the vanishing. Three types of arguments were used for the proof in the linear case:
\begin{enumerate}[leftmargin=*]
  \item\label{it:lintype1} Using $\psi_U$: find $U'<U$ such that $\psi_U|_{U'}\ne1$ and ${}^{\gamma}U'<U_P$, then $\mathrm{I}(\gamma)=0$ since we obtain an inner integral $\int_{U'(F)\backslash U'(\A)}\psi_U(u')du'=0$.
  \item\label{it:lintype2} Using the cuspidality of $\pi_i$: obtain a unipotent radical $V$ of a parabolic subgroup of $G$ such that the $du$-integral of $f$ is invariant under $V$, then $\mathrm{I}(\gamma)=0$ because we have an inner integral
      $\int_{V(F)\backslash V(\A)}\varphi_i(v)dv=0$.
  \item\label{it:lintype3} Using the vanishing properties of the $(k,c)$ representation: construct as an inner integral,
    a Fourier coefficient of the $(k,c)$ representation attached to a unipotent orbit which is greater than or not comparable with $(k^{c})$.
\end{enumerate}
To extend the arguments from \cite[\S~2.3]{CFGK2} to the covering case, we argue as follows.
First, all occurrences of $k$ in \textit{loc. cit.} are replaced with $rk$.
Second, claims involving the structure of representatives $\gamma$ apply to the covering, because they only involve
multiplications in $H(F)$, and the covering is split over $H(F)$. Third, arguments where we introduced unipotent integrations
remain valid, as long as conjugations are between elements of $H(F)$ and $N_{rkc}(\A)$, such that the conjugation remains in
$N_{rkc}(\A)$, since then we can apply Lemma~\ref{lemma:conjugation of N by H}. This completes the handling of types
\eqref{it:lintype1} and \eqref{it:lintype3}; for type \eqref{it:lintype2} there are also cases where $V<N_{n}^-$, and
for such we have a splitting $v\mapsto\langle v,\eta'(v)\rangle$ of $V$
into the right copy of $G^{(m)}$ so that $\int_{V(F)\backslash V(\A)}\varphi_2(\langle v,\eta'(v)\rangle)dv=0$,
and ${}^{\gamma}V<N_{rkc}$. The arguments from the linear case extend to the covering and imply, using
${}^{\gamma}\langle v,\eta'(v)\rangle=\langle {}^{\gamma}v,\eta^{-1}({}^{\gamma}v)\rangle$ (by \eqref{eq:epsilon for conjugation between split subgroups}), that the integral of $f$ over $U$ is invariant with respect to $\{\langle v,\eta'(v)\rangle:v\in V(\A)\}$,
so that we can obtain the inner integral $\int_{V(F)\backslash V(\A)}\varphi_2(\langle v,\eta'(v)\rangle)dv$.
This completes the proof that $I(\gamma)=0$ for
$\gamma$ such that $\gamma\notin P(F)\delta R(F)$.

Finally consider $\mathrm{I}(\delta)$. Let $v^{\diamondsuit}=\diag(v,v^*)$ be the natural
embedding of $V_{(c^{rk})}$ in $M_P$ and put $V_{(c^{rk})}^{\diamondsuit}=\{v^{\diamondsuit}:v\in V_{(c^{rk})}\}$.
We compute the stabilizer and obtain
\begin{align*}
R_{\delta}=\{(g,{}^{\iota}g):g\in G\}
\ltimes{}^{\delta^{-1}}V_{(c^{rk})}^{\diamondsuit}.
\end{align*}
Write
\begin{align}\label{eq:decomposition of U with square}
U={}^{\delta^{-1}}V_{(c^{rk})}^{\diamondsuit}\ltimes(U\cap U_P)={}^{\delta^{-1}}V_{(c^{rk})}^{\diamondsuit}\ltimes U_0.
\end{align}
For $u\in{}^{\delta^{-1}}V_{(c^{rk})}^{\diamondsuit}$, if ${}^{\delta}u=v^{\diamondsuit}$ with $v\in V_{(c^{rk})}$,
then since both $V_{(c^{rk})}^{\diamondsuit}$ and ${}^{\delta^{-1}}V_{(c^{rk})}^{\diamondsuit}$ are subgroups of $N_{rkc}$, we can apply
Lemma~\ref{lemma:conjugation of N by H} and obtain
\begin{align}\label{eq:conjugation of v_g}
{}^{\delta}\langle u,\eta^{-1}(u)\rangle=\langle v^{\diamondsuit},\eta^{-1}(v^{\diamondsuit})\rangle=\langle v,(\eta_{rkc}^{\diamondsuit})^{-1}(v)\rangle.
\end{align}
Here the second equality follows from the definitions of $\GL_{rkc}^{(m,r)}(\A)$.
Also $\psi_U({}^{\delta^{-1}}v^{\diamondsuit})=\psi^{-1}(v)$, where $\psi$ is given by \eqref{eq:rk c character}. Thus
\begin{align*}
&\int\limits_{{}^{\delta^{-1}}V_{(c^{rk})}^{\diamondsuit}(F)\backslash U(\A)}
f(\langle\delta ,\eta^{-1}(\delta)\rangle\langle u,\eta^{-1}(u)\rangle,s)\,\psi_U(u)\,du
\\
&=\int\limits_{U_0(\A)}
\int\limits_{V_{(c^{rk})}(F)\backslash V_{(c^{rk})}(\A)}
f(\langle v,(\eta_{rkc}^{\diamondsuit})^{-1}(v)\rangle\langle \delta ,\eta^{-1}(\delta)\rangle\langle u_0,\eta^{-1}(u_0)\rangle
,s)\,\psi^{-1}(v)\psi_U(u_0)\,dv\,du_0
\\&=\int\limits_{U_0(\A)}
f_{\mathcal{W}(\mathcal{E}_{\tau})}(\langle\delta,\eta^{-1}(\delta)\rangle\langle u_0,\eta^{-1}(u_0)\rangle
,s)\,\psi_U(u_0)\,du_0.
\end{align*}
Here for the last equality we used \eqref{eq:def of f W global}.
We plug this into $\mathrm{I}(\delta)$ and obtain
\begin{align*}
&\int\limits_{\{(g,{}^{\iota}g):g\in G(F)\}\backslash G(\A)\times G(\A)}
\int\limits_{U_0(\A)}
\varphi_1^{(\eta^{\times})^{-1}}(\langle g_1,1\rangle)\,\overline{\varphi_2({}^{\iota}\langle g_2,1\rangle)}\,
\\&f_{\mathcal{W}(\mathcal{E}_{\tau})}(\langle\delta,\eta^{-1}(\delta)\rangle\langle u_0,\eta^{-1}(u_0)\rangle
\langle \mathfrak{e}_1(g_1),1\rangle\langle \mathfrak{e}_2(g_2),1\rangle,s)\,\psi_U(u_0)\,du_0\,dg_1\,dg_2.\nonumber
\end{align*}
The next step is to factor this integral through $\{(g,{}^{\iota}g):g\in G(\A)\}\cong G(\A)$.
Multiply $g_2\mapsto g_1^{\iota}g_2$, then $\{(g,{}^{\iota}g):g\in G(F)\}\mapsto\{(g,1):g\in G(F)\}=G(F)$ and the integral becomes
\begin{align}\label{int:I delta factoring unipotent}
&\int\limits_{G(F)\backslash G(\A)\times G(\A)}
\int\limits_{U_0(\A)}
\varphi_1^{(\eta^{\times})^{-1}}(\langle g_1,1\rangle)\,\overline{\varphi_2({}^{\iota}\langle {}^{\iota}g_1g_2,1\rangle)}\,
\\&f_{\mathcal{W}(\mathcal{E}_{\tau})}(\langle\delta,\eta^{-1}(\delta)\rangle\langle u_0,\eta^{-1}(u_0)\rangle
\langle \mathfrak{e}_1(g_1),1\rangle\langle \mathfrak{e}_2({}^{\iota}g_1g_2),1\rangle,s)\,\psi_U(u_0)\,du_0\,dg_1\,dg_2.\nonumber
\end{align}
Then
\begin{align*}
&\overline{\varphi_2({}^{\iota}\langle {}^{\iota}g_1g_2,1\rangle)}\mapsto
\rho_R({}^{\iota}g_1,g_2)\overline{\varphi_2({}^{\iota}\langle {}^{\iota}g_1,1\rangle\,{}^{\iota}\langle g_2,1\rangle)},
\\&f_{\mathcal{W}(\mathcal{E}_{\tau})}(\langle\delta,\eta^{-1}(\delta)\rangle\langle u_0,\eta^{-1}(u_0)\rangle
\langle \mathfrak{e}_1(g_1),1\rangle\langle \mathfrak{e}_2({}^{\iota}g_1g_2),1\rangle,s)
\\&\mapsto
\rho_R({}^{\iota}g_1,g_2)^{-1}f_{\mathcal{W}(\mathcal{E}_{\tau})}(\langle\delta,\eta^{-1}(\delta)\rangle\langle u_0,\eta^{-1}(u_0)\rangle
\langle \mathfrak{e}_1(g_1),1\rangle\langle \mathfrak{e}_2({}^{\iota}g_1),1\rangle\langle \mathfrak{e}_2(g_2),1\rangle,s).
\end{align*}
The conjugation ${}^{\iota}\langle {}^{\iota}g_1,1\rangle$ takes place in the right copy of $G^{(m)}$, so that
by \eqref{eq:iota on the gbl coverings} and because ${}^{\iota}$ is an involution,
\begin{align}\label{eq:iota inv using}
&{}^{\iota}\langle {}^{\iota}g_1,\eta_{\iota,R}^{-1}(g_1)\rangle=
{}^{\iota}({}^{\iota}\langle g_1,1\rangle)=\langle g_1,1\rangle,\\\nonumber
&\eta_{\iota,R}(g_1)\overline{\varphi_2({}^{\iota}\langle {}^{\iota}g_1,1\rangle\,{}^{\iota}\langle g_2,1\rangle)}
=\overline{\varphi_2(\langle g_1,1\rangle\,{}^{\iota}\langle g_2,1\rangle)}.
\end{align}
Consequently \eqref{int:I delta factoring unipotent} equals
\begin{align*}
&\int\limits_{G(F)\backslash G(\A)\times G(\A)}
\int\limits_{U_0(\A)}
\varphi_1^{(\eta^{\times})^{-1}}(\langle g_1,1\rangle)\,\overline{\varphi_2(\langle g_1,1\rangle\,{}^{\iota}\langle g_2,1\rangle)}\,
\\&f_{\mathcal{W}(\mathcal{E}_{\tau})}(\langle\delta,\eta^{-1}(\delta)\rangle\langle u_0,\eta^{-1}(u_0)\rangle
\langle \mathfrak{e}_1(g_1),1\rangle\langle \mathfrak{e}_2({}^{\iota}g_1),\eta_{\iota,R}^{-1}(g_1)\rangle\langle \mathfrak{e}_2(g_2),1\rangle,s)\,\psi_U(u_0)\,du_0\,dg_1\,dg_2.\nonumber
\end{align*}
Then by \eqref{eq:iota gbl image on product},
\begin{align*}
\langle \mathfrak{e}_1(g_1),1\rangle\langle \mathfrak{e}_2({}^{\iota}g_1),\eta_{\iota,R}^{-1}(g_1)\rangle
=\langle \mathfrak{e}_1(g_1),1\rangle\,{}^{\iota}\langle \mathfrak{e}_2(g_1),1\rangle
={}^{\iota}\langle (g_1,g_1),\rho(\mathfrak{e}_1(g_1),\mathfrak{e}_2(g_2))\rangle.
\end{align*}
Hence the last integral equals
\begin{align}\label{int:I delta factoring change g_1 g_2}
&\int\limits_{G(F)\backslash G(\A)\times G(\A)}
\int\limits_{U_0(\A)}
\varphi_1^{(\eta^{\times})^{-1}}(\langle g_1,1\rangle)\,\overline{\varphi_2(\langle g_1,1\rangle\,{}^{\iota}\langle g_2,1\rangle)}\,
\\&f_{\mathcal{W}(\mathcal{E}_{\tau})}(\langle \delta ,\eta^{-1}(\delta )\rangle
\langle u_0,\eta^{-1}(u_0)\rangle
{}^{\iota}\langle (g_1,g_1),\rho(\mathfrak{e}_1(g_1),\mathfrak{e}_2(g_1))\rangle\langle \mathfrak{e}_2(g_2),1\rangle,s)\,\psi_U(u_0)\,du_0\,dg_1\,dg_2.\nonumber
\end{align}
Next we see that ${}^{(g_1,{}^{\iota}g_1)^{-1}}u_0=v_{g_1}u_{g_1}$, where $v_{g_1}\in {}^{\delta^{-1}}V_{(c^{rk})}^{\diamondsuit}(\A)$ and $u_{g_1}\in U_0(\A)$.
Since $v_{g_1},u_{g_1}\in N_{rkc}(\A)$, by Lemma~\ref{lemma:conjugation of N by H}
\begin{align}\label{eq:conj gbl g1 g1 iota u_0}
{}^{(g_1,{}^{\iota}g_1)^{-1}}\langle u_0,\eta^{-1}(u_0)\rangle=
\langle {}^{(g_1,{}^{\iota}g_1)^{-1}}u_0,\eta^{-1}({}^{(g_1,{}^{\iota}g_1)^{-1}}u_0)\rangle=\langle v_{g_1},\eta^{-1}(v_{g_1})\rangle\langle u_{g_1},\eta^{-1}(u_{g_1})\rangle.
\end{align}
Thus the $du_0$-integral of \eqref{int:I delta factoring change g_1 g_2} becomes
\begin{align*}
\int\limits_{U_0(\A)}f_{\mathcal{W}(\mathcal{E}_{\tau})}
({}^{\delta\iota}\langle (g_1,g_1),\rho(\mathfrak{e}_1(g_1),\mathfrak{e}_2(g_1))\rangle\langle \delta ,\eta^{-1}(\delta)\rangle
\langle v_{g_1},\eta^{-1}(v_{g_1})\rangle\langle u_{g_1},\eta^{-1}(u_{g_1})\rangle
,s)\,\psi_U(u_0)\,du_0.
\end{align*}
Here ${}^{\delta\iota}\langle (g,g),\epsilon\rangle$ denotes the composition
${}^{\delta}({}^{\iota}\langle (g,g),\epsilon\rangle)$; in $H$, ${}^{\delta\iota}$ is the automorphism of $G\times G$ given by
${}^{\delta\iota}(g,g)={}^{\delta}(g,{}^{\iota}g)$ (recall $\iota\notin G$).
By Corollary~\ref{corollary:gbl splitting of (G,G)},
$(g,g)\mapsto\langle (g,g),(\eta^{\times})^{-1}(g)\rho(\mathfrak{e}_1(g),\mathfrak{e}_2(g))\rangle$ is the splitting
of the group $\{(g,g):g\in G(\A)\}$ in $H^{(m)}(\A)$.
\begin{claim}\label{claim:invariance of inner under G iota delta}
For any $g\in G(\A)$ and $h\in H^{(m)}(\A)$,
\begin{align*}
&f_{\mathcal{W}(\mathcal{E}_{\tau})}
({}^{\delta\iota}\langle (g,g),(\eta^{\times})^{-1}(g)\rho(\mathfrak{e}_1(g),\mathfrak{e}_2(g))\rangle h
,s)=f_{\mathcal{W}(\mathcal{E}_{\tau})}(h,s).
\end{align*}
\end{claim}
The claim is proved below.
It follows that \eqref{int:I delta factoring change g_1 g_2} equals
\begin{align*}
&\int\limits_{G(F)\backslash G(\A)\times G(\A)}
\int\limits_{U_0(\A)}
\varphi_1^{(\eta^{\times})^{-1}}(\langle g_1,\eta^{\times}(g_1)\rangle)\,\overline{\varphi_2(\langle g_1,1\rangle\,{}^{\iota}\langle g_2,1\rangle)}\,
\\&
f_{\mathcal{W}(\mathcal{E}_{\tau})}
(\langle \delta ,\eta^{-1}(\delta)\rangle\langle v_{g_1},\eta^{-1}(v_{g_1})\rangle\langle u_{g_1},\eta^{-1}(u_{g_1})\rangle
\langle \mathfrak{e}_2(g_2),1\rangle,s)\,\psi_U(u_0)\,du_0\,dg_1\,dg_2.\nonumber
\end{align*}
The character emitted from $\psi_U$ when we change variables $u_{g_1}\mapsto u_0$ is cancelled by the left equivariance property of
$f_{\mathcal{W}(\mathcal{E}_{\tau})}$ under $v_{g_1}$ (see \eqref{eq:conjugation of v_g}); this actually follows from the definition of the embedding in \S~\ref{embedding}. Also by definition
$\varphi_1^{(\eta^{\times})^{-1}}(\langle g_1,\eta^{\times}(g_1)\rangle)=
\varphi_1(\langle g_1,1\rangle)$.
We obtain
\begin{align*}
&\int\limits_{G(F)\backslash G(\A)\times G(\A)}
\int\limits_{U_0(\A)}
\varphi_1(\langle g_1,1\rangle)\,\overline{\varphi_2(\langle g_1,1\rangle\,{}^{\iota}\langle g_2,1\rangle)}\,
\\&
f_{\mathcal{W}(\mathcal{E}_{\tau})}
(\langle \delta u_0 ,\eta^{-1}(\delta u_0)\rangle
\langle \mathfrak{e}_2(g_2),1\rangle,s)\,\psi_U(u_0)\,du_0\,dg_1\,dg_2.\nonumber
\end{align*}
Here we also used Corollary~\ref{corollary:rho and eta on H and N without conjugation} to combine $\delta$ and $u_0$.
Now factoring through $\{(g,1):g\in G(\A)\}$, the integral
becomes
\begin{align}\label{int:I delta factoring 3}
&\int\limits_{G(\A)}\{\varphi_1,({}^{\iota}\langle g_2,1\rangle)\cdot\varphi_2\}
\int\limits_{U_0(\A)}
f_{\mathcal{W}(\mathcal{E}_{\tau})}(\langle \delta u_0,\eta^{-1}(\delta u_0)\rangle\langle \mathfrak{e}_2(g_2),1\rangle,s)\,\psi_U(u_0)\,du_0\,dg_2.
\end{align}
Using \eqref{eq:iota inv using} again and since
$\{\varphi_1,\langle g_2,\epsilon\rangle\cdot\varphi_2\}
=\epsilon^{-1}\{\varphi_1,\langle g_2,1\rangle\cdot\varphi_2\}$,
when we change $g_2\mapsto{}^{\iota}g_2$ in \eqref{int:I delta factoring 3} we obtain
\begin{align}\label{int:I delta factoring 4}
&\int\limits_{G(\A)}\{\varphi_1,\langle g_2,1\rangle\cdot\varphi_2\}
\int\limits_{U_0(\A)}
f_{\mathcal{W}(\mathcal{E}_{\tau})}(\langle \delta u_0,\eta^{-1}(\delta u_0)\rangle\langle \mathfrak{e}_2({}^{\iota}g_2),\eta_{\iota,R}^{-1}(g_2)\rangle,s)\,\psi_U(u_0)\,du_0\,dg_2.
\end{align}
Therefore we conclude
\begin{align}\label{int:I delta factoring 4}
&\mathrm{I}(\delta)=\int\limits_{G(\A)}\{\varphi_1,\langle g_2,1\rangle\cdot\varphi_2\}
\int\limits_{U_0(\A)}
f_{\mathcal{W}(\mathcal{E}_{\tau})}(\langle \delta u_0,\eta^{-1}(\delta u_0)\rangle\,{}^{\iota}\langle \mathfrak{e}_2(g_2),1\rangle,s)\,\psi_U(u_0)\,du_0\,dg_2.
\end{align}
This completes the proof that \eqref{global1} is equal to
\eqref{global2} for $\Real(s)\gg0$. Moreover, it is now clear that $\mathrm{I}(\delta)$ and thereby \eqref{global1} vanish, unless $\pi_1=\pi_2=\pi$.
\end{proof}
\begin{proof}[Proof of Claim~\ref{claim:invariance of inner under G iota delta}]
For $g\in G(\A)$, by matrix multiplication ${}^{\delta\iota}(g,g)=d_gl_g$ with
\begin{align*}
&d_g=\diag((g^*)^{\triangle},g^{\triangle}), \qquad g^{\triangle}=\diag(g,\ldots,g)\in\GL_{rkc},\\
&l_g=\left(\begin{smallmatrix}I_n&&&&-A_2\\&I_n&&-A_3\\&&I_{2(rk-1)c}\\&&&I_n\\&&&&I_n\end{smallmatrix}\right),\qquad g=\left(\begin{smallmatrix}A_1 & A_2 \\ A_3 & A_4  \end{smallmatrix}\right).
\end{align*}
Then $d_g\in M_P$ and $l_g,{}^{d_g}l_g\in U_P(\A)$. Hence locally, for $g,g'\in G(F_{\nu})$, by
\eqref{eq:sigma on vh and h'v'} and Proposition~\ref{proposition:sigma on diagonal embedding of SLc},
\begin{align}\label{eq:claim left inv f g giota local comp}
\sigma_{2rkc,\nu}(d_gl_g,d_{g'}l_{g'})=\sigma_{2rkc,\nu}(d_g,d_{g'})=
\sigma_{rkc,\nu}^{\diamondsuit}((g^*)^{\triangle},({g'}^*)^{\triangle})=\left(\frac{\varsigma_{*,c,\nu}(g^*)\varsigma_{*,c,\nu}({g'}^*)}{\varsigma_{*,c,\nu}((gg')^*)}\right)^{rk}.
\end{align}
Then by \eqref{eq:nu and sigma for covering of H}, $\rho_{\nu}(d_gl_g,d_{g'}l_{g'})$ equals
\begin{align*}
&{\eta_{\nu}(d_{g}l_{g}d_{g'}l_{g'})}\left(\frac{\varsigma_{*,c,\nu}(g^*)\varsigma_{*,c,\nu}({g'}^*)}{\varsigma_{*,c,\nu}((gg')^*)}\right)^{rk}
=\frac{(\eta_{\nu}(d_gl_g)\varsigma_{*,c,\nu}^{rk}(g^*))(\eta_{\nu}(d_{g'}l_{g'})\varsigma_{*,c,\nu}^{rk}({g'}^*))}
{\eta_{\nu}(d_{g}l_{g}d_{g'}l_{g'})\varsigma_{*,c,\nu}^{rk}((gg')^*)}.
\end{align*}
Since $\rho_{\nu}$ is $1$ on $K_{H,\nu}$ for almost all $\nu$, and $g\mapsto{}^{\delta\iota}(g,g)=d_gl_g$ is in particular a homomorphism of $K_{G,\nu}$ into $K_{H,\nu}$, we deduce that $g\mapsto\eta_{\nu}(d_gl_g)\varsigma_{*,c,\nu}^{rk}(g^*)=1$ on $K_{G,\nu}$ (see the proof of Corollary~\ref{corollary:gbl splitting for SL_c}). Therefore
\begin{align*}
[\eta\varsigma_{*,c}^{rk}](d_gl_g)=\prod_{\nu}\eta_{\nu}(d_gl_g)\varsigma_{*,c,\nu}^{rk}(g^*)
\end{align*}
is well defined on $\{{}^{\delta\iota}(g,g):g\in G(\A)\}$. Then by globalizing the identity for
$\rho_{\nu}(d_gl_g,d_{g'}l_{g'})$, it follows that
\begin{align*}
d_gl_g\mapsto\langle d_gl_g,[\eta\varsigma_{*,c}^{rk}]^{-1}(d_gl_g)\rangle
\end{align*}
is the unique splitting
of $\{{}^{\delta\iota}(g,g):g\in G(\A)\}$ (cf. \eqref{eq:local rho eta varsigma}).
Now by \eqref{eq:epsilon for conjugation between split subgroups}, with $\chi={}^{\delta\iota}$ and $Y=\{(g,g):g\in G(\A)\}$,
\begin{align}\label{eq:conj by delta globally}
{}^{\delta\iota}\langle (g,g),(\eta^{\times})^{-1}(g)\rho(\mathfrak{e}_1(g),\mathfrak{e}_2(g))\rangle=
\langle d_gl_g,
[\eta\varsigma_{*,c}^{rk}]^{-1}(d_gl_g)\rangle.
\end{align}
Furthermore, for $g\in G(F_{\nu})$ by \eqref{eq:nu and sigma for covering of H} we have
\begin{align*}
\rho_{\nu}(d_g,l_g)^{-1}\eta_{\nu}^{-1}(d_gl_g)\varsigma_{*,c,\nu}^{-rk}(g^*)=
\eta_{\nu}^{-1}(d_g)\eta_{\nu}^{-1}(l_g)\sigma_{2rkc,\nu}^{-1}(d_g,l_g)\varsigma_{*,c,\nu}^{-rk}(g^*)
=\eta_{\nu}^{-1}(d_g)\eta_{\nu}^{-1}(l_g)\varsigma_{*,c,\nu}^{-rk}(g^*).
\end{align*}
Here for the second equality we used \eqref{eq:sigma on h and v}. By definition $\eta_{\nu}(d_g)=\eta_{rkc,\nu}^{\diamondsuit}((g^*)^{\triangle})$ and then
$\eta_{\nu}(d_g)\varsigma_{*,c,\nu}^{rk}(g^*)=\eta_{rkc,\nu}^{\triangle}(g^*)$ (see Corollary~\ref{corollary:gbl splitting for SL_c}), hence
\begin{align*}
\rho_{\nu}(d_g,l_g)^{-1}\eta_{\nu}^{-1}(d_gl_g)\varsigma_{*,c,\nu}^{-rk}(g^*)=(\eta_{rkc,\nu}^{\triangle})^{-1}(g^*)\eta_{\nu}^{-1}(l_g).
\end{align*}
The l.h.s.~ globalizes because $\rho$ is well defined and so is $[\eta\varsigma_{*,c}^{rk}]$, and the r.h.s.~ globalizes because
by Corollary~\ref{corollary:gbl splitting for SL_c}, $\eta_{rkc}^{\triangle}$ is well defined (even on $\SL_c(\A)$ and $G(\A)<\SL_c(\A)$),
and $\eta$ is well defined (even on $N_{rkc}(\A)$). Consequently for any $g\in G(\A)$,
\begin{align*}
\rho(d_g,l_g)^{-1}[\eta\varsigma_{*,c}^{rk}]^{-1}(d_gl_g)=(\eta_{rkc}^{\triangle})^{-1}(g^*)\eta^{-1}(l_g).
\end{align*}
Thus
\begin{align*}
\langle d_gl_g,[\eta\varsigma_{*,c}^{rk}]^{-1}(d_gl_g)\rangle
=\langle d_g,\rho(d_g,l_g)^{-1}[\eta\varsigma_{*,c}^{rk}]^{-1}(d_gl_g)\rangle\langle l_g,1\rangle
=\langle d_g,(\eta_{rkc}^{\triangle})^{-1}(g^*)\rangle\langle l_g,\eta^{-1}(l_g)\rangle.
\end{align*}
Also note that $\langle d_g,(\eta_{rkc}^{\triangle})^{-1}(g^*)\rangle=\langle (g^*)^{\triangle},(\eta_{rkc}^{\triangle})^{-1}(g^*)\rangle$, when we regard $(g^*)^{\triangle}$ as an element of $M_P(\A)$ (which is also how $\GL_{rkc}^{(m,r)}(\A)$ was defined).
Plugging these results into $f_{\mathcal{W}(\mathcal{E}_{\tau})}$,
\begin{align*}
&f_{\mathcal{W}(\mathcal{E}_{\tau})}
({}^{\delta\iota}\langle (g,g),(\eta^{\times})^{-1}(g)\rho(\mathfrak{e}_1(g),\mathfrak{e}_2(g))\rangle h
,s)=f_{\mathcal{W}(\mathcal{E}_{\tau})}(\langle d_gl_g,
[\eta\varsigma_{*,c}^{rk}]^{-1}(d_gl_g)\rangle h,s)
\\&\quad=f_{\mathcal{W}(\mathcal{E}_{\tau})}(\langle (g^*)^{\triangle},(\eta_{rkc}^{\triangle})^{-1}(g^*)\rangle\langle l_g,\eta^{-1}(l_g)\rangle h,s)
=f_{\mathcal{W}(\mathcal{E}_{\tau})}(\langle l_g,\eta^{-1}(l_g)\rangle h,s)
=f_{\mathcal{W}(\mathcal{E}_{\tau})}(h,s).
\end{align*}
Here for the first equality we used \eqref{eq:conj by delta globally}, the third follows from
Proposition~\ref{proposition:extra invariance} and the last by the left-invariance property of $f_{\mathcal{W}(\mathcal{E}_{\tau})}$ under
$\langle u,\eta^{-1}(u)\rangle$ for $u\in U_P(\A)$.
\end{proof}

Now we write \eqref{global2} as an ``almost Euler product" (in the sense of \cite{Tk}),
and derive the structure of the local integrals.
Let $S$ be a finite set of places of $F$ such that for $\nu\notin S$, $F_{\nu}$, $\psi_{\nu}$, $\tau_{\nu}$ and $\pi_{\nu}$ are unramified.
We use the notation of \S~\ref{speh def}.
Identify $\mathcal{E}_{\tau}$ with $(\mathcal{E}_{\tau})_S\otimes\otimes'_{\nu\notin S}\rho_c(\tau_{\nu})$ ($\rho_c(\tau_{\nu})$ was defined in \S~\ref{Local components of rk c speh}). By \eqref{eq:partial decomp of Lambda}, for a decomposable $f$,
\begin{align*}
f_{\mathcal{W}(\mathcal{E}_{\tau})}(h,s)=f_{\mathcal{W}((\mathcal{E}_{\tau})_S)}(h_S,s)\prod_{\nu\notin S}f_{\mathcal{W}(\rho_c(\tau_{\nu}))}(h_{\nu},s),\qquad h\in H^{(m)}(\A).
\end{align*}
Here $f_{\mathcal{W}((\mathcal{E}_{\tau})_S)}$ (resp., $f_{\mathcal{W}(\rho_c(\tau_{\nu}))}$) is a standard section in the space of $\Ind_{\widetilde{P}(F_S)}^{H^{(m)}(F_S)}(\mathcal{W}((\mathcal{E}_{\tau})_S)\delta_P^s)$
(resp., $\Ind_{\widetilde{P}(F_{\nu})}^{H^{(m)}(F_{\nu})}(\mathcal{W}(\rho_c(\tau_{\nu}))\delta_P^s)$), regarded as a complex-valued function. Moreover, at almost all $\nu\notin S$, the sections $f_{\mathcal{W}(\rho_c(\tau_{\nu}))}$ are normalized
($f_{\mathcal{W}(\rho_c(\tau_{\nu}))}(\langle I_{2rkc},1\rangle,s)=1$) and unramified.

Next for decomposable $\varphi_1$ and $\varphi_2$, by the uniqueness of the $G^{(m)}(F_{\nu})$-invariant bilinear pairing on $\pi_{\nu}\times\pi_{\nu}^{\vee}$ at all places, we can write for any $g\in G^{(m)}(\A)$,
\begin{align*}
\{\varphi_1,\pi(g)\varphi_2\}=\prod_{\nu}\omega_{\nu}(g_{\nu}),
\end{align*}
where $\omega_{\nu}$ is a matrix coefficient of $\pi_{\nu}^{\vee}$, and at almost all places $\omega_{\nu}$ is normalized ($\omega_{\nu}(\langle I_c,1\rangle)=1$) and unramified.

Then by Theorem~\ref{theorem:main gbl identity} we have an almost Euler product:
\begin{align}\label{eq:almost Euler}
Z(s,\varphi_1,\varphi_2,f)=Z[S](s,\omega_S,f_
{\mathcal{W}((\mathcal{E}_{\tau})_S)})\prod_{\nu\notin S}Z_{\nu}(s,\omega_{\nu},
f_{\mathcal{W}(\rho_c(\tau_{\nu}))}),
\end{align}
where
\begin{align}\label{eq:local integral at one place}
&Z_{\nu}(s,\omega_{\nu},f_{\mathcal{W}(\rho_c(\tau_{\nu}))})\\&=
\int\limits_{G(F_{\nu})}\int\limits_{U_0(F_{\nu})}
\omega_{\nu}(\langle g,1\rangle)f_{\mathcal{W}(\rho_c(\tau_{\nu}))}(\langle\delta_{\nu} u_0,\eta_{\nu}^{-1}(\delta_{\nu} u_0)\rangle
{}^{\iota_{\nu}}\langle\mathfrak{e}_2(g),1\rangle,s)
\,(\psi_{\nu})_{U}(u_0)\,du_0\,dg,\nonumber
\end{align}
\begin{align}\label{eq:local integral at places}
&Z[S](s,\omega_S,f_{\mathcal{W}((\mathcal{E}_{\tau})_S)})\\&=
\int\limits_{G(F_{S})}\int\limits_{U_0(F_{S})}
\omega_{S}(\langle g,1\rangle)f_{\mathcal{W}((\mathcal{E}_{\tau})_S)}(\langle\delta_{S} u_0,\eta_{S}^{-1}(\delta_{S} u_0)\rangle
{}^{\iota_{S}}\langle\mathfrak{e}_2(g),1\rangle,s)
\,(\psi_{S})_{U}(u_0)\,du_0\,dg.\nonumber
\end{align}
The following is the main local result of this work: the computation of \eqref{eq:local integral at one place} with unramified data.
Its proof occupies \S~\ref{Computation of the local factors with unramified data}.
\begin{theorem}\label{theorem:unramified computation for Sp(2n),SO(2n)}
Let $\nu\notin S$ and assume $\omega_{\nu}$ and $f_{\mathcal{W}(\rho_c(\tau_{\nu}))}$ are normalized and unramified. Then
\begin{align*}
&Z_{\nu}(s,\omega_{\nu},f_{\mathcal{W}(\rho_c(\tau_{\nu}))})\\&=
\frac{L_{\vartheta_{\nu}}(r\alpha s+1/2,\pi_{\nu}\times\tau_{\nu})}
{[L_{\vartheta_{\nu}}(r\alpha s+rn+1/2,\tau_{\nu})]\prod\limits_{1\leq j\leq rn}L_{\vartheta_{\nu}}(2r\alpha s+2j,\tau_{\nu},\wedge^2)
L_{\vartheta_{\nu}}(2r\alpha s+2j-1,\tau_{\nu},\vee^2)}.
\end{align*}
Here $\alpha=rkc+1$, $L_{\vartheta_{\nu}}(r\alpha s+rn+1/2,\tau_{\nu})$ appears only for odd $m$ and $\vartheta_{\nu}$ is ignored when $m\not\equiv2\,(4)$.
\end{theorem}
\begin{remark}\label{remark:result unr known in some cases}
For $m=1$ this result was proved by \cite[Theorem~29]{CFGK2}, and for $m=2$ and $k=1$ by Li \cite[Proposition~4.6]{Li1992};
our proof will include these cases.
\end{remark}

We switch to local notation. Since \eqref{eq:almost Euler} also includes the integral \eqref{eq:local integral at places} over the places in $S$, we treat a single place and a finite product of places simultaneously. Let $S'$ be an arbitrary finite set of places of $F$,
and denote
\begin{align*}
&G=G(F_{S'}), \qquad H=H(F_{S'}),\qquad U_0=U_0(F_{S'}), \qquad \delta=\delta_{S'}, \qquad \iota=\iota_{S'},\\
&\psi_U=(\psi_{S'})_U,\qquad \sigma_{2rkc}=\prod_{\nu\in S'}\sigma_{2rkc,\nu},\qquad
\varsigma_{\iota,c}=\prod_{\nu\in S'}\varsigma_{\iota,c,\nu},
\qquad \varsigma_{*,c}=\prod_{\nu\in S'}\varsigma_{*,c,\nu}.
\end{align*}
Realize $H^{(m)}$ using $\sigma_{2rkc}$, and the right copy of $G^{(m)}$ by $\sigma_c=\prod_{\nu\in S'}\sigma_{c,\nu}$. The local properties from \S~\ref{local covering} are still applicable, because $S'$ is finite.
Let $\pi$ be a genuine irreducible representation of $G^{(m)}$, realized using $\sigma_c$. Assume
$\mathcal{E}$ is a genuine irreducible representation of $\GL_{rkc}^{(m,r)}$ (the cover obtained from $\widetilde{M}_P(F_{S'})$) which affords an $(rk,c)$ functional $\Lambda$, and
such that
\begin{align}\label{eq:invariance prop on SLc}
\Lambda(\mathcal{E}(\langle b^{\triangle},\varsigma_{*,c}^{-rk}(b)\rangle)\xi)=
\Lambda(\xi),\qquad \forall b\in G.
\end{align}
Even if $\Lambda$ is unique up to scaling, $\mathcal{E}$ might not be $(rk,c)$ because we do not assume the first condition in the definition ($\mathcal{O}(\mathcal{E},\beta',\psi)=0$ for any $\beta'\succsim((rk)^c)$) holds.
Let $\mathcal{W}(\mathcal{E})$ be the $(rk,c)$ model defined using $\Lambda$.

For the integrals \eqref{eq:local integral at one place}, the representations $\rho_c(\tau_{\nu})$ are $(rk,c)$ and
\eqref{eq:invariance prop on SLc} is satisfied by Corollary~\ref{corollary:invariance wrt SLc}. If $S'=S$, the functional chosen for the realization
of $\mathcal{W}((\mathcal{E}_{\tau})_S)$ in \eqref{eq:local integral at places} also satisfies \eqref{eq:invariance prop on SLc}, by
Corollary~\ref{corollary:extra invariance on Lambda S}.

For a matrix coefficient $\omega$ of $\pi^{\vee}$ and
a standard section $f_{\mathcal{W}(\mathcal{E})}$ of $\Ind_{\widetilde{P}}^{H^{(m)}}(\mathcal{W}(\mathcal{E})\delta_P^s)$,
\begin{align}\label{int:local classical integral abstract}
&Z(s,\omega,f_{\mathcal{W}(\mathcal{E})})=
\int\limits_{G}\int\limits_{U_0}
\omega(\langle g,1\rangle)f_{\mathcal{W}(\mathcal{E})}(\langle\delta u_0,1\rangle\,
{}^{\iota}\langle\mathfrak{e}_2(g),1\rangle,s)\,\psi_{U}(u_0)\,du_0\,dg.
\end{align}
Note that we omitted $\eta^{-1}(\delta u_0)$. This is because we can
use Corollary~\ref{corollary:rho and eta on H and N without conjugation} to separate $\delta$ from $u_0$ in the global integral; then
$\eta_{\nu}(\delta_{\nu})=1$ at almost all places; and locally when we use $\sigma_{2krc}$,
$v\mapsto\langle v,1\rangle$ is the splitting of $N_{rkc}$ and in particular of $U_0$, by \eqref{eq:sigma on h and v}, which also implies
$\langle\delta,1\rangle\langle u_0,1\rangle=\langle\delta u_0,1\rangle$.
The map $\delta\mapsto\langle\delta,1\rangle$ is a homomorphism when the local fields $\{F_{\nu}\}_{\nu\in S'}$ are all
unramified (see \S~\ref{local covering}); in general replacing
$\langle\delta,1\rangle$ by $\langle\delta,\epsilon\rangle$ simply multiplies the integral by $\epsilon$.

\begin{proposition}\label{proposition:equiv props}
The integral $Z(s,\omega,f_{\mathcal{W}(\mathcal{E})})$, at least formally, can be regarded as a morphism in the space
\begin{align}\label{eq:homspace G with W(E)}
\Hom_{G^{(m)}\times G^{(m)}}(J_{U,\psi_U^{-1}}(\Ind_{\widetilde{P}}^{H^{(m)}}(\mathcal{W}(\mathcal{E})\delta_P^s)),\pi^{\vee}\otimes\pi^{\iota}).
\end{align}
Here $J_{U,\psi_U^{-1}}(\cdots)$ is the Jacquet module with respect to $U$ and $\psi_U^{-1}$,
regarded as a representation of $G^{(m)}\times G^{(m)}$ by virtue of the embedding $(g_1,g_2)$.
\end{proposition}
\begin{proof}
First we introduce some notation. Given $\omega$, by definition there are vectors $\xi$ and $\xi^{\vee}$ in the spaces of $\pi$ and $\pi^{\vee}$, such that $\omega(g)=\pi^{\vee}(g)\xi^{\vee}(\xi)=\xi^{\vee}(\pi(g^{-1})\xi)$ for $g\in G^{(m)}$. Assume $\xi$ and $\xi^{\vee}$ are given, and for $g_1,g_2\in G^{(m)}$, let $\omega_{g_1,g_2}$ be the matrix coefficient of $\pi^{\vee}$ defined by
\begin{align*}
\omega_{g_1,g_2}(g)=\pi^{\vee}(g)((\pi^{\vee})^{\iota}(g_2)\xi^{\vee})(\pi(g_1)\xi)=
\pi^{\vee}(g({}^{\iota}g_2))\xi^{\vee}(\pi(g_1)\xi).
\end{align*}
Regarding the integral as a trilinear form on
\begin{align}\label{eq:tri}
\Ind_{\widetilde{P}}^{H^{(m)}}(\mathcal{W}(\mathcal{E})\delta_P^s)\times\pi\times(\pi^{\vee})^{\iota},
\end{align}
we claim that for $g_1,g_2\in G$ and $u\in U$,
\begin{align}\label{eq:inv both}
&Z(s,\omega_{\langle g_1,1\rangle,\langle g_2,1\rangle},(\langle \mathfrak{e}_1(g_1),\varsigma_{*,c}^{rk+1}(g_1)\rangle\langle \mathfrak{e}_2(g_2),1\rangle\langle u,1\rangle)\cdot f_{\mathcal{W}(\mathcal{E})})=
\psi_U^{-1}(u)Z(s,\omega,f_{\mathcal{W}(\mathcal{E})}).
\end{align}
Note that while \eqref{eq:tri} is a priori a representation of $G^{(m)}\times G^{(m)}$, it factors through $G\times G$ by
\eqref{eq:embeddings coverings G and G into H}. Also observe that since both $\pi$ and $\pi^{\vee}$ (thereby $(\pi^{\vee})^{\iota}$) are defined on $G^{(m)}$ which is realized using $\sigma_c$,
the image of $\langle g_1,1\rangle$ in $H^{(m)}$ is indeed $\langle \mathfrak{e}_1(g_1),\varsigma_{*,c}^{rk+1}(g_1)\rangle$, see
\eqref{eq:lcl embeding left copy using sigma c into H}.
Hence this trilinear form factors through $J_{U,\psi_U^{-1}}$, and since $(\pi^{\vee})^{\iota}=(\pi^{\iota})^{\vee}$, it can be identified with an element of \eqref{eq:homspace G with W(E)}.

First we prove (for any $\omega$)
\begin{align}\label{eq:inv wrt u}
&Z(s,\omega,\langle u,1\rangle\cdot f_{\mathcal{W}(\mathcal{E})})=
\psi_U^{-1}(u)Z(s,\omega,f_{\mathcal{W}(\mathcal{E})}).
\end{align}
Starting with the l.h.s.,
\begin{align*}
&Z(s,\omega,\langle u,1\rangle\cdot f_{\mathcal{W}(\mathcal{E})})=
\int\limits_{G}\int\limits_{U_0}
\omega(\langle g,1\rangle)f_{\mathcal{W}(\mathcal{E})}(\langle\delta u_0,1\rangle
{}^{\iota}\langle\mathfrak{e}_2(g),1\rangle\langle u,1\rangle,s)\,\psi_{U}(u_0)\,du_0\,dg.
\end{align*}
By the definition of the embedding in \S~\ref{embedding}, we can write ${}^{\mathfrak{e}_2({}^{\iota}g)}u=v_gu_g$ where
$v_g\in{}^{\delta^{-1}}V_{(c^{rk})}^{\diamondsuit}$ and $u_g\in U_0$ (also see \eqref{eq:decomposition of U with square}). Then
by \eqref{eq:sigma conjugate v by h} and \eqref{eq:sigma on vh and h'v'},
\begin{align*}
&{}^{\mathfrak{e}_2({}^{\iota}g)}\langle u,1\rangle=\langle v_g,1\rangle\langle u_g,1\rangle.
\end{align*}
The integral becomes
\begin{align*}
&\int\limits_{G}\int\limits_{U_0}
\omega(\langle g,1\rangle)f_{\mathcal{W}(\mathcal{E})}(\langle\delta u_0,1\rangle
\langle v_g,1\rangle\langle u_g,1\rangle
{}^{\iota}\langle\mathfrak{e}_2(g),1\rangle,s)\,\psi_{U}(u_0)\,du_0\,dg
\\&=\int\limits_{G}\int\limits_{U_0}
f_{\mathcal{W}(\mathcal{E})}(
\langle {}^{\delta}v_g,1\rangle
\langle\delta u_0u_g,1\rangle
{}^{\iota}\langle\mathfrak{e}_2(g),1\rangle,s)\,\psi_{U}(u_0)\,du_0\,dg,
\end{align*}
where the equality follows again from \eqref{eq:sigma conjugate v by h} and \eqref{eq:sigma on vh and h'v'}.
Note that if we write ${}^{\delta}v_g={v_g'}^{\diamondsuit}$ for $v_g'\in V_{(c^{rk})}$,
$\langle{}^{\delta}v_g,1\rangle=\langle v_g',1\rangle\in\GL_{rkc}^{(m,r)}$
(the local analog of \eqref{eq:conjugation of v_g}).
Now as in the linear case, the definition of the embedding implies that when we change variables $u_0\mapsto u_0u_g^{-1}$
and use the left invariance properties of $f_{\mathcal{W}(\mathcal{E})}$, we get
\begin{align*}
&\psi_U^{-1}(u)\int\limits_{G}\int\limits_{U_0}
\omega(\langle g,1\rangle)f_{\mathcal{W}(\mathcal{E})}(
\langle\delta u_0,1\rangle
{}^{\iota}\langle\mathfrak{e}_2(g),1\rangle,s)\,\psi_{U}(u_0)\,du_0\,dg,
\end{align*}
completing the proof of \eqref{eq:inv wrt u}.

It remains to prove \eqref{eq:inv both} where $u$ is omitted. Then the l.h.s.~ of \eqref{eq:inv both} equals
\begin{align*}
&\int\limits_{G}\int\limits_{U_0}
\pi^{\vee}(\langle g,1\rangle\,{}^{\iota}\langle g_2,1\rangle)
\xi^{\vee}(\pi(\langle g_1,1\rangle)\xi)
\\&f_{\mathcal{W}(\mathcal{E})}(\langle\delta u_0,1\rangle
{}^{\iota}\langle\mathfrak{e}_2(g),1\rangle\langle \mathfrak{e}_1(g_1),\varsigma_{*,c}^{rk+1}(g_1)\rangle\langle \mathfrak{e}_2(g_2),1\rangle,s)\,\psi_{U}(u_0)\,du_0\,dg.
\end{align*}
By \eqref{eq:iota on the local coverings} and because $\pi^{\vee}$ is anti-genuine,
\begin{align*}
\pi^{\vee}(\langle g,1\rangle\,{}^{\iota}\langle g_2,1\rangle)
=\pi^{\vee}(\langle g({}^{\iota}g_2),\sigma_c(g,{}^{\iota}g_2)\varsigma_{\iota,c}^{-1}(g_2)\rangle)
=\sigma_c^{-1}(g,{}^{\iota}g_2)\varsigma_{\iota,c}(g_2)\pi^{\vee}(\langle g({}^{\iota}g_2),1\rangle).
\end{align*}
Also by \eqref{eq:action of local iota on right image of G}, \eqref{eq:iota image commutes with G1} and Proposition~\ref{proposition:the $2$-cocycle on G times G},
\begin{align*}
{}^{\iota}\langle\mathfrak{e}_2(g),1\rangle\langle \mathfrak{e}_1(g_1),\varsigma_{*,c}^{rk+1}(g_1)\rangle\langle \mathfrak{e}_2(g_2),1\rangle
&={}^{\iota}\langle\mathfrak{e}_2(g),1\rangle\langle \mathfrak{e}_1(g_1),\varsigma_{*,c}^{rk+1}(g_1)\rangle
{}^{\iota}\langle\mathfrak{e}_2({}^{\iota}g_2),\varsigma_{\iota,c}({}^{\iota}g_2)\rangle
\\&=\langle\mathfrak{e}_1(g_1),\varsigma_{*,c}^{rk+1}(g_1)\rangle
{}^{\iota}(\langle\mathfrak{e}_2(g),1\rangle\langle \mathfrak{e}_2({}^{\iota}g_2),\varsigma_{\iota,c}({}^{\iota}g_2)\rangle)
\\&=\langle\mathfrak{e}_1(g_1),\varsigma_{*,c}^{rk+1}(g_1)\rangle
{}^{\iota}\langle\mathfrak{e}_2(g({}^{\iota}g_2)),\sigma_c(g,{}^{\iota}g_2)\varsigma_{\iota,c}({}^{\iota}g_2)\rangle.
\end{align*}
In addition by \eqref{eq:varsigma varsigme},
$\varsigma_{\iota,c}({}^{\iota}g_2)\varsigma_{\iota,c}(g_2)=1$ (!).
Hence the integral equals
\begin{align*}
&\int\limits_{G}\int\limits_{U_0}
\pi^{\vee}(\langle g({}^{\iota}g_2),1\rangle)
\xi^{\vee}(\pi(\langle g_1,1\rangle)\xi)
\\&f_{\mathcal{W}(\mathcal{E})}(\langle\delta u_0,1\rangle
\langle\mathfrak{e}_1(g_1),\varsigma_{*,c}^{rk+1}(g_1)\rangle
{}^{\iota}\langle\mathfrak{e}_2(g({}^{\iota}g_2)),1\rangle,s)\,\psi_{U}(u_0)\,du_0\,dg.
\end{align*}
Thus when we change $g\mapsto g({}^{\iota}g_2)^{-1}$ we obtain
\begin{align}\label{int:inv after getting rid of g_2}
&\int\limits_{G}\int\limits_{U_0}
\pi^{\vee}(\langle g,1\rangle)
\xi^{\vee}(\pi(\langle g_1,1\rangle)\xi)
f_{\mathcal{W}(\mathcal{E})}(\langle\delta u_0,1\rangle
\langle\mathfrak{e}_1(g_1),\varsigma_{*,c}^{rk+1}(g_1)\rangle
{}^{\iota}\langle\mathfrak{e}_2(g),1\rangle,s)\,\psi_{U}(u_0)\,du_0\,dg.
\end{align}

Next multiply $g\mapsto g_1g$. Then
\begin{align*}
&\pi^{\vee}(\langle g,1\rangle)\xi^{\vee}(\pi(\langle g_1,1\rangle)\xi)
\mapsto \pi^{\vee}(\langle g_1,\sigma_c(g_1,g)^{-1}\rangle\langle g,1\rangle)\xi^{\vee}(\pi(\langle g_1,1\rangle)\xi)
=\sigma_c(g_1,g)\pi^{\vee}(\langle g,1\rangle)\xi^{\vee}(\xi),\\
&{}^{\iota}\langle\mathfrak{e}_2(g),1\rangle\mapsto
{}^{\iota}\langle\mathfrak{e}_2(g_1),\sigma_c(g_1,g)^{-1}\rangle\,{}^{\iota}\langle\mathfrak{e}_2(g),1\rangle.
\end{align*}
Note that $\omega(\langle g,1\rangle)=\pi^{\vee}(\langle g,1\rangle)\xi^{\vee}(\xi)$. Then \eqref{int:inv after getting rid of g_2} becomes
\begin{align*}
&\int\limits_{G}\int\limits_{U_0}
\omega(\langle g,1\rangle)
f_{\mathcal{W}(\mathcal{E})}(\langle\delta u_0,1\rangle
\langle\mathfrak{e}_1(g_1),\varsigma_{*,c}^{rk+1}(g_1)\rangle\,{}^{\iota}\langle\mathfrak{e}_2(g_1),1\rangle
\,{}^{\iota}\langle\mathfrak{e}_2(g),1\rangle,s)\,\psi_{U}(u_0)\,du_0\,dg.
\end{align*}
By \eqref{eq:the $2$-cocycle on G times G formula} and \eqref{eq:iota image on product},
\begin{align*}
\langle\mathfrak{e}_1(g_1),\varsigma_{*,c}^{rk+1}(g_1)\rangle\,{}^{\iota}\langle\mathfrak{e}_2(g_1),1\rangle
=\langle (g_1,{}^{\iota}g_1),\varsigma_{*,c}^{rk+1}(g_1)\varsigma_{\iota,c}^{-1}(g_1)\rangle={}^{\iota}\langle(g_1,g_1),\varsigma_{*,c}^{rk+1}(g_1)\rangle.
\end{align*}
Thus the last integral equals
\begin{align}\label{int:to show inv under g iota g}
&\int\limits_{G}\int\limits_{U_0}
\omega(\langle g,1\rangle)
f_{\mathcal{W}(\mathcal{E})}(\langle\delta u_0,1\rangle
{}^{\iota}\langle(g_1,g_1),\varsigma_{*,c}^{rk+1}(g_1)\rangle\,
{}^{\iota}\langle\mathfrak{e}_2(g),1\rangle,s)\,\psi_{U}(u_0)\,du_0\,dg.
\end{align}

We have the local analog of \eqref{eq:conj gbl g1 g1 iota u_0}: ${}^{(g_1,{}^{\iota}g_1)}u_0=v_{g_1}u_{g_1}$ with
$v_{g_1}\in {}^{\delta^{-1}}V_{(c^{rk})}^{\diamondsuit}$ and $u_{g_1}\in U_0$, and using \eqref{eq:sigma conjugate v by h},
\begin{align*}
{}^{(g_1,{}^{\iota}g_1)^{-1}}\langle u_0,1\rangle=\langle v_{g_1},1\rangle\langle u_{g_1},1\rangle.
\end{align*}
Then using \eqref{eq:sigma on h and v} and \eqref{eq:sigma conjugate v by h},
the integral~\eqref{int:to show inv under g iota g} equals
\begin{align}\label{int:inv before g giota inv}
&\int\limits_{G}\int\limits_{U_0}
\omega(\langle g,1\rangle)
f_{\mathcal{W}(\mathcal{E})}(
{}^{\delta\iota}\langle(g_1,g_1),\varsigma_{*,c}^{rk+1}(g_1)\rangle
\langle\delta ,1\rangle
\langle v_{g_1} ,1\rangle\langle u_{g_1} ,1\rangle
{}^{\iota}\langle\mathfrak{e}_2(g),1\rangle,s)\,\psi_{U}(u_0)\,du_0\,dg,
\end{align}
where ${}^{\delta\iota}\langle(g_1,g_1),\epsilon\rangle={}^{\delta}({}^{\iota}\langle (g_1,g_1),\epsilon\rangle)$ and
${}^{\delta\iota}(g_1,g_1)={}^{\delta}(g_1,{}^{\iota}g_1)$.

By Corollary~\ref{corollary:lcl splitting of (G,G)},
$(g,g)\mapsto\langle (g,g),\varsigma_{*,c}^{rk+1}(g)\rangle$ is the splitting
of $\{(g,g):g\in G\}$ in $H^{(m)}$.
Now we describe the local version of Claim~\ref{claim:invariance of inner under G iota delta} to prove that
for any $g\in G$ and $h\in H^{(m)}$,
\begin{align}\label{eq:local f left inv by g iota g}
f_{\mathcal{W}(\mathcal{E})}({}^{\delta\iota}\langle(g,g),\varsigma_{*,c}^{rk+1}(g)\rangle h,s)=f_{\mathcal{W}(\mathcal{E})}(h,s).
\end{align}
Using the notation of the proof of Claim~\ref{claim:invariance of inner under G iota delta},
${}^{\delta\iota}(g,g)=d_gl_g$ with $d_g$ and $l_g$ described there. By
\eqref{eq:claim left inv f g giota local comp},
\begin{align*}
\sigma_{2rkc}(d_gl_g,d_{g'}l_{g'})=\sigma_{2rkc}(d_g,d_{g'})=
\sigma_{rkc}^{\diamondsuit}((g^*)^{\triangle},({g'}^*)^{\triangle})=\left(\frac{\varsigma_{*,c}(g^*)\varsigma_{*,c}({g'}^*)}{\varsigma_{*,c}((gg')^*)}\right)^{rk}.
\end{align*}
Hence $d_gl_g\mapsto\langle d_gl_g,\varsigma_{*,c}^{-rk}(g^*)\rangle$
is the unique splitting of $\{{}^{\delta\iota}(g,g):g\in G\}$, then
Corollary~\ref{corollary:lcl splitting of (G,G)} and
\eqref{eq:epsilon for conjugation between split subgroups} imply
\begin{align*}
{}^{\delta\iota}\langle (g,g),\varsigma_{*,c}^{rk+1}(g)\rangle=
\langle d_gl_g,
\varsigma_{*,c}^{-rk}(g^*)\rangle
\end{align*}
(cf. \eqref{eq:conj by delta globally}). Now by \eqref{eq:sigma on h and v} (recall $l_g\in U_P$),
\begin{align*}
\langle d_gl_g,
\varsigma_{*,c}^{-rk}(g^*)\rangle
=\langle d_g,\varsigma_{*,c}^{-rk}(g^*)\rangle\langle l_g,1\rangle
=\langle (g^*)^{\triangle},\varsigma_{*,c}^{-rk}(g^*)\rangle\langle l_g,1\rangle,
\end{align*}
where in the last equality we regard $(g^*)^{\triangle}$ as an element of $M_P$.
Consequently
\begin{align*}
f_{\mathcal{W}(\mathcal{E})}({}^{\delta\iota}\langle(g,g),\varsigma_{*,c}^{rk+1}(g)\rangle h,s)
&=f_{\mathcal{W}(\mathcal{E}_{\tau})}(\langle (g^*)^{\triangle},\varsigma_{*,c}^{-rk}(g^*)\rangle\langle l_g,1\rangle h,s)\\&
=f_{\mathcal{W}(\mathcal{E}_{\tau})}(\langle l_g,1\rangle h,s)=f_{\mathcal{W}(\mathcal{E}_{\tau})}(h,s),
\end{align*}
where the second equality follows from \eqref{eq:invariance prop on SLc}. This proves \eqref{eq:local f left inv by g iota g}.
Returning to \eqref{int:inv before g giota inv} we have
\begin{align*}
&\int\limits_{G}\int\limits_{U_0}
\omega(\langle g,1\rangle)
f_{\mathcal{W}(\mathcal{E})}(
\langle\delta ,1\rangle
\langle v_{g_1} ,1\rangle\langle u_{g_1} ,1\rangle
{}^{\iota}\langle\mathfrak{e}_2(g),1\rangle,s)\,\psi_{U}(u_0)\,du_0\,dg.
\end{align*}
Finally when we conjugate
$\langle v_{g_1} ,1\rangle$ to the left, use the equivariance properties of $\mathcal{W}(\mathcal{E})$ and change variables in $u_{g_1}$,
we obtain
\begin{align}\label{int:inv after g giota inv}
&\int\limits_{G}\int\limits_{U_0}
\omega(\langle g,1\rangle)
f_{\mathcal{W}(\mathcal{E})}(
\langle\delta u_0,1\rangle
{}^{\iota}\langle\mathfrak{e}_2(g),1\rangle,s)\,\psi_{U}(u_0)\,du_0\,dg.
\end{align}
This completes the proof of \eqref{eq:inv both} in all cases.
\end{proof}
\begin{corollary}\label{corollary:du integral of f invariant under g iota g}
For any section $f_{\mathcal{W}(\mathcal{E})}$ and $g_1\in G$,
\begin{align*}
&\int\limits_{U_0}
f_{\mathcal{W}(\mathcal{E})}(\langle\delta u_0 ,1\rangle\,
{}^{\iota}\langle(g_1,g_1),\varsigma_{*,c}^{rk+1}(g_1)\rangle,s)\,\psi_{U}(u_0)\,du_0
=\int\limits_{U_0}
f_{\mathcal{W}(\mathcal{E})}(\langle\delta u_0 ,1\rangle,s)\,\psi_{U}(u_0)\,du_0.
\end{align*}
\end{corollary}
\begin{proof}
Immediate from the passage
\eqref{int:to show inv under g iota g}--\eqref{int:inv after g giota inv} in the proof.
\end{proof}
Let $S_{\infty}$ be the set of archimedean places. Since we are assuming (throughout \S~\ref{global}) that
the number field contains $\mu_{2m}$, this is also true locally, then archimedean places are
complex (when $m>1$), whence the covering is trivial. We can therefore regard these places in a manner similar
to linear groups (for $m=1$ real and complex places were treated in \cite{CFK}).
\begin{proposition}\label{proposition:local props}
The integral $Z(s,\omega,f_{\mathcal{W}(\mathcal{E})})$ enjoys the following properties.
\begin{enumerate}[leftmargin=*]
\item\label{it:conv} It is absolutely convergent in a right half-plane $\Real(s)\gg0$ depending only on the representations.
\item\label{it:p-adic} If $S'\cap S_{\infty}=\emptyset$, there is data $(\omega,f_{\mathcal{W}(\mathcal{E})})$ for which $Z(s,\omega,f_{\mathcal{W}(\mathcal{E})})$ is absolutely convergent and equals $1$, for all $s$.
\item\label{it:arch} If $S'\subset S_{\infty}$, then for a given $s$, there is $(\omega,f_{\mathcal{W}(\mathcal{E})})$ where
$f_{\mathcal{W}(\mathcal{E})}$ is a smooth section, such that
$Z(s,\omega,f_{\mathcal{W}(\mathcal{E})})$ is holomorphic and nonzero in a neighborhood of $s$.
\item\label{it:all} Assume $\mathcal{E}$ is a local component, or a finite tensor of local components, of
the global $(rk,c)$ representation obtained by Theorem~\ref{exthspeh1}, and for $\nu\in S'\cap S_{\infty}$, $\mathcal{E}_{\nu}$ is $(rk,c)$.
For any $s$ there is $(\omega,f_{\mathcal{W}(\mathcal{E})})$ where $f_{\mathcal{W}(\mathcal{E})}$ is $\prod_{\nu\in S'}\widetilde{K}_{H,\nu}$-finite
such that $Z(s,\omega,f_{\mathcal{W}(\mathcal{E})})$ is holomorphic and nonzero in a neighborhood of $s$.
\end{enumerate}
\end{proposition}
\begin{proof}
Parts~\eqref{it:conv}--\eqref{it:arch} were proved for linear groups in \cite[Propositions~20, 21]{CFK}, when $|S'|=1$. The arguments extend to
covering groups and to finite products of integrals (i.e., to $|S'|>1$), and only depend on the equivariance
property of the $(k,c)$ functional under $V_{(c^k)}$ and invariance under $\SL_c^{\triangle}$; here we have an $(rk,c)$
functional with an equivariance property under $V_{(c^{rk})}$ and the invariance property is guaranteed by assumption
\eqref{eq:invariance prop on SLc}. In particular for \eqref{it:p-adic} note that if we take a sufficiently small compact
open neighborhood of the identity $\mathcal{N}$ in $H$, then $H^{(m)}$ is split over $\mathcal{N}$.

For the last assertion, first note that under the assumption on the archimedean places, we can write
$Z(s,\omega,f_{\mathcal{W}(\mathcal{E})})$ as the product of an integral \eqref{eq:local integral at places}
with respect to $S''=S'-S_{\infty}$, and $|S'\cap S_{\infty}|$ local integrals over the places of $S'\cap S_{\infty}$. The integral over $S''$
can be made constant by \eqref{it:p-adic}. Since we also have part~\eqref{it:arch}, it remains to prove that
the integrals over the archimedean places admit meromorphic continuation to $\C$ as functions of $s$, and this continuation is continuous in $\omega_{\nu}$ and $f_{\mathcal{W}(\mathcal{E}_{\nu})}$.

We can now use local notation and consider a single archimedean place. Put
\begin{align*}
\tau^{r}=\Ind_{P_{(k^r)}}^{\GL_{rk}}((\tau\otimes\ldots\otimes\tau)\delta_{P_{(k^r)}}^{-1/(2rk)}).
\end{align*}
This representation admits a unique Whittaker model, because $\tau$ is irreducible and generic. Then by
Theorem~\ref{theorem:prop1},
\begin{align*}
\mathcal{E}\subset\Ind_{P_{((rk)^{c})}}^{\GL_{rkc}}((\tau^r\otimes \ldots \otimes \tau^r)\delta_{P_{((rk)^{c})}}^{-1/(2rk)}).
\end{align*}
By our assumption, $\mathcal{E}$ admits a unique $(rk,c)$ functional, hence we can realize it using a functional on the full induced representation, granted it does not vanish on $\mathcal{E}$. Such can be obtained inductively using the integral from \S~\ref{Local components of rk c speh}: first construct an integral on
\begin{align*}
V_1=\Ind_{P_{((rk)^{2})}}^{\GL_{2rk}}((\tau^r\otimes \tau^r)\delta_{P_{((rk)^{2})}}^{-1/(2rk)}),
\end{align*}
then compose it with an integral on
\begin{align*}
\Ind_{P_{(2rk,rk)}}^{\GL_{3rk}}((V_1\otimes \tau^r)\delta_{P_{(2rk,rk)}}^{-1/(2rk)}),
\end{align*}
etc. With this realization we can follow the arguments from
\cite[\S~6.13]{CFK}: using the multiplicativity arguments from \cite[\S~6.4.1]{CFK}
we reduce these continuation and continuity statements to the case of doubling integrals for $\GL_1\times\GL_{rk}$ (whose covering version will be described here in \S~\ref{integarls for GL}) and $\Sp_{c-2}\times\GL_{rk}$, then repeat the arguments on the $\Sp_{c-2}\times\GL_{rk}$ integral. We obtain $n$ doubling integrals for $\GL_1\times\GL_{rk}$. Note that in \textit{loc. cit.} the multiplicativity arguments were formulated under the assumption that $\tau$ is essentially tempered, but this was only needed in order to have the integral realization of the model \cite[\S~2.4]{CFK}, which is similar to \S~\ref{Local components of rk c speh} here. The continuation and continuity for a
$\GL_1\times\GL_{rk}$ integral with a quasi-character of $F^*$, and a representation of $\GL_{rk}$ which admits a unique Whittaker model, was then proved directly in \cite[\S~6.13]{CFK}; here the representation of $\GL_{rk}$ is $\tau^r$, in \textit{loc. cit.} it was $\tau$. Now nonvanishing for a $K_{H,\nu}$-finite section follows from the continuity.
\end{proof}
\begin{remark}
Over archimedean places the argument is greatly simplified if we assume $|\det|^{s'}\tau^r$ is unitary for some $s'\in\R$. Then $\mathcal{E}$ is precisely $\rho_c(\tau^r)$ of \cite{CFK}, $\mathcal{E}$ is $(rk,c)$ by \cite[Theorem~5]{CFK}, and
all the assertions on the integral already follow from \cite{CFK}. For example if $\tau$ is
a unitary irreducible representation of $\GL_k(\C)$ and $|\chi_i|=|~|^{\alpha_i}$ with $-1/2<r\alpha_i<1/2$ for each $\chi_i$ in the inducing data of $\tau$, $\tau^r$ is unitary; so $|\det|^{s'}\tau^r$ is unitary if $\tau$ is essentially tempered.
\end{remark}
\begin{remark}
Meromorphic continuation of the integrals over non-archimedean places can be proved using
Bernstein's continuation principal (in \cite{Banks}), by combining
Proposition~\ref{proposition:equiv props} with Proposition~\ref{proposition:local props}, and granted that the dimension of \eqref{eq:homspace G with W(E)} is at most $1$ outside a finite set of values of $q^{-s}$. This uniqueness result requires the uniqueness of the $(rk,c)$ model, and hence we may not use it here at the ramified places. We will however prove the meromorphic continuation of
$Z(s,\omega,f_{\mathcal{W}(\mathcal{E})})$ with unramified data and when the model is unique, as part of the computation.
\end{remark}
As a corollary of our results, we deduce the meromorphic continuation of the global partial $L$-function, but under an additional assumption on the archimedean places. Also recall that we assume Conjectures~\ref{local Shimura conjecture} and \ref{Shimura conjecture}, but for $r=1$ or $k=1$ our only assumption is $\mu_{2m}\subset F^*$.
\begin{theorem}\label{theorem:mero}
Let $\pi$ and $\tau$ be genuine irreducible cuspidal automorphic representations of $G^{(m)}(\A)$ and $\GL_{k}^{(m,r)}(\A)$, and let $S$ be
a finite set of places such that for $\nu\notin S$, $F_{\nu}$, $\psi_{\nu}$, $\tau_{\nu}$ and $\pi_{\nu}$ are unramified. Assume
$\mathcal{E}_{\nu}$ is $(rk,c)$ for all $\nu\in S_{\infty}$, e.g., $r=1$ or $k=1$ (or $\tau_{\nu}$ is essentially tempered). Then
$L_{\vartheta}^S(s,\pi\times\tau)$ admits meromorphic continuation to $\C$.
\end{theorem}
\begin{proof}
According to Theorem~\ref{theorem:main gbl identity}, $Z(s,\varphi_1,\varphi_2,f)$ coincides with \eqref{global2} for $\Real(s)\gg0$. Given $S$, we can choose the global data $\varphi_1,\varphi_2$ and $f$ such that $\omega_{\nu}$ and $f_{\mathcal{W}(\rho_c(\tau_{\nu}))}$ are normalized and unramified for all $\nu\notin S$. Then for $\Real(s)\gg0$, by \eqref{eq:almost Euler} and Theorem~\ref{theorem:unramified computation for Sp(2n),SO(2n)},
\begin{align*}
Z(s,\varphi_1,\varphi_2,f)=&Z[S](s,\omega_S,f_{\mathcal{W}((\mathcal{E}_{\tau})_S)})
\\&\times\frac{L_{\vartheta}^S(r\alpha s+1/2,\pi\times\tau)}
{[L_{\vartheta}^S(r\alpha s+rn+1/2,\tau)]\prod\limits_{1\leq j\leq rn}L_{\vartheta}^S(2r\alpha s+2j,\tau,\wedge^2)
L_{\vartheta}^S(2r\alpha s+2j-1,\tau,\vee^2)}.
\end{align*}
($\alpha=rkc+1$ and the factor in brackets appears only when $m$ is odd.)

The l.h.s.~ admits meromorphic continuation to $\C$, by Theorem~\ref{theorem:main theorem classical groups}.
The partial $L$-functions appearing in the denominator on the r.h.s.~ admit meromorphic continuation by \cite{Gao2018}.

It is therefore sufficient to prove that for any $s_0\in\C$, one can further choose the global data such that
$Z[S](s,\omega_S,f_{\mathcal{W}((\mathcal{E}_{\tau})_S)})$ is nonzero and holomorphic in a neighborhood of $s_0$. This is possible by
Proposition~\ref{proposition:local props} \eqref{it:all}.
Therefore $L_{\vartheta}^S(r\alpha s+1/2,\pi\times\tau)$ admits meromorphic continuation to $\C$.
\end{proof}
\begin{remark}
Gao \cite{Gao2018} already implies the assertion of the theorem unconditionally (even without $\mu_{2m}\subset F^*$).
\end{remark}

\subsection{The local $\GL_n^{(m,r)}\times\GL_k^{(m,r)}$ integrals}\label{integarls for GL}
The proof of Theorem~\ref{theorem:unramified computation for Sp(2n),SO(2n)} in \S~\ref{Computation of the local factors with unramified data} involves the computation of $\GL_n^{(m,r)}\times\GL_k^{(m,r)}$ integrals, which we introduce here. The global theory of these integrals is developed in Appendix~\ref{gbl GL}, because for the present work they are only needed locally and when the field is unramified (for the linear case see \cite{CFK}). For the purpose of their definition, we take any local field.
As in the symplectic case, we start with the description of the groups and embeddings,
proceed with the coverings, then consider the representations and integral.

Let $n,k$ and $m$ be positive integers, and $r$ be either $m$ when $m$ is odd, or $r=m/2$ otherwise.
Put $c=n$. Denote $G=\GL_c$, $H=\GL_{2rkc}$ and $P=P_{(rkc,rkc)}$.
For the embedding $G\times G\hookrightarrow H$, set $Q=P_{(c^{rk-1},2c,c^{rk-1})}=M_Q\ltimes U$, then $U=V_{(c^{rk-1},2c,c^{rk-1})}$, and define
\begin{align}\label{eq:GL character of U}
\psi_U(\left(\begin{smallmatrix}v&x&y\\&I_{2c}&z\\&&v'\end{smallmatrix}\right))=\psi^{-1}(v)\psi(-\tr(X_1)+\tr(Z_1))\psi^{-1}(v'), \qquad v,v'\in V_{(c^{rk-1})},
\end{align}
where $\psi(v)$ and $\psi(v')$ are defined by \eqref{eq:psi on V rk-1},
$X_1$ is the bottom left $c\times c$ block of $x$ and
$Z_1$ is the top left $c\times c$ block of $z$. Then $G\times G$ is embedded in the stabilizer of $\psi_U$ in $M_Q$ by
\begin{align*}
(g_1,g_2)=\diag(g_1,\ldots,g_1,g_1,g_2,g_1,\ldots,g_1),\qquad g_1,g_2\in\GL_c,
\end{align*}
where $g_1$ appears $rk$ times on the left of $g_2$ and $rk-1$ times on the right.
Again we denote $\mathfrak{e}_1(g)=(g,1)$ and $\mathfrak{e}_2(g)=(1,g)$.

Recall the covering $\GL_d^{(m,r)}$ defined in \S~\ref{covering of the Levi}, realized using the $2$-cocycle
\begin{align*}
\sigma^{\diamondsuit}_{d}(b,b')=\sigma_{2d}(\diag(b,b^*),\diag(b',{b'}^*)).
\end{align*}
We realize $H^{(m,r)}$ using $\sigma_{2rkc}^{\diamondsuit}$, and both copies of $G^{(m)}$ are realized using $\sigma_c^{\diamondsuit}$. By
\eqref{eq:block compatibility on Levi of P} we immediately deduce that
the images of $\mathfrak{e}_1(G)$ and $\mathfrak{e}_2(G)$ commute in $H^{(m,r)}$, and moreover
for all $g_i,g_i'\in G$,
\begin{align}\label{eq:the $2$-cocycle on G times G formula GL}
\sigma^{\diamondsuit}_{2rkc}((g_1,g_2),(g_1',g_2'))&=\sigma^{\diamondsuit}_{c}(g_1,g_1')^{2rkc-1}\sigma_{c}^{\diamondsuit}(g_2,g_2')
=\sigma^{\diamondsuit}_{c}(g_1,g_1')^{-1}\sigma_{c}^{\diamondsuit}(g_2,g_2').
\end{align}
(Cf. \eqref{eq:the $2$-cocycle on G times G formula}).

Thus we can lift the embedding $G\times G \hookrightarrow H$ to an embedding of
\begin{align*}
\{(\epsilon_1,\epsilon_2)\in\mu_m^2:\epsilon_1=\epsilon_2\}\backslash G^{(m,r)}\times G^{(m,r)} \hookrightarrow H^{(m,r)}
\end{align*}
via
\begin{align}\label{eq:GL embeddings coverings G and G into H}
\langle g,\epsilon\rangle\mapsto\langle \mathfrak{e}_1(g),\epsilon^{-1}\rangle,\qquad
\langle g,\epsilon\rangle\mapsto\langle \mathfrak{e}_2(g),\epsilon\rangle.
\end{align}
Equality~\eqref{eq:the $2$-cocycle on G times G formula GL} also implies an identity similar to \eqref{eq:g_1 and g_2 product in H}:
\begin{align}\label{eq:g_1 and g_2 product in H GL}
\langle\mathfrak{e}_1(g_1),\epsilon_1^{-1}\rangle\langle \mathfrak{e}_2(g_2),\epsilon_2\rangle=\langle (g_1,g_2),\epsilon_1^{-1}\epsilon_2\rangle.
\end{align}

To define the integral we introduce the following notation. Let
\begin{align*}
&\delta=\delta_0\delta_1,\qquad \delta_0=\left(\begin{smallmatrix}&I_{rkc}\\I_{rkc}\end{smallmatrix}\right),\qquad
\delta_1=\left(\begin{smallmatrix}I_{(rk-1)c}\\&I_c&I_c\\&&I_c\\&&&I_{(rk-1)c}\end{smallmatrix}\right),\\
&U_0=U\cap U_P, \qquad
\psi_U(u_0)=\psi(\tr(Z_1)).
\end{align*}

Recall that $G^{\triangle}$ is the stabilizer in $M_{(c^{rk})}$ of the character $\psi$ of $V_{(c^{rk})}$ given by \eqref{eq:rk c character} (now $G=\GL_c)$.
By \eqref{eq:block compatibility on Levi of P},
$\sigma^{\diamondsuit}_{rkc}(g^{\triangle},{g'}^{\triangle})=\sigma_c^{\diamondsuit}(g,g')^{rk}$.
Hence $\sigma^{\diamondsuit}_{2rkc}$ is trivial on $\{(g,g):g\in G\}$, and note that $(g,g)=\diag(g^{\triangle},g^{\triangle})$,
i.e., the diagonal embedding of $2rk$ copies of $G$ in $H$. Thus
$H^{(m,r)}$ is split over $\{(g,g):g\in G\}$, and the trivial section is a splitting (not necessarily unique, $G$ is not perfect).

Let $\pi$ be a genuine irreducible representation of $G^{(m,r)}$, and let
$\mathcal{E}$ and $\mathcal{E'}$ be two genuine irreducible representations of $\GL_{rkc}^{(m,r)}$. Assume
$\mathcal{E}$ (resp., $\mathcal{E}'$) affords an $(rk,c)$ functional $\Lambda$ (resp., $\Lambda'$),
and let $\mathcal{W}(\mathcal{E})$ (resp., $\mathcal{W}(\mathcal{E}')$) be the $(rk,c)$ model defined using $\Lambda$ (resp., $\Lambda'$).
We do not assume this model is unique. The additional assumption we need for the definition of the integral is that for any
$\xi$ (resp., $\xi'$) in the space of $\mathcal{E}$ (resp., $\mathcal{E}'$),
\begin{align}\label{eq:invariance prop on GL or SL}
\Lambda(\mathcal{E}(\langle g^{\triangle},1\rangle)\xi)\Lambda'(\mathcal{E}'(\langle g^{\triangle},1\rangle)\xi')=
\Lambda(\xi)\Lambda'(\xi'),\qquad \forall g\in G.
\end{align}

Let $\omega$ be a matrix coefficient of $\pi^{\vee}$, and
$f_{\mathcal{W}(\mathcal{E})\otimes\mathcal{W}(\mathcal{E}')}$ be a standard section of
\begin{align*}
\Ind_{\widetilde{P}}^{H^{(m,r)}}((\mathcal{W}(\mathcal{E})\otimes \mathcal{W}(\mathcal{E}'))\delta_P^s).
\end{align*}
The integral is defined by
\begin{align}\label{eq:local GL GL integral}
&Z(s,\omega,f_{\mathcal{W}(\mathcal{E})\otimes\mathcal{W}(\mathcal{E}')})=
\int\limits_{G}\int\limits_{U_0}
\omega(\langle g,1\rangle)f_{\mathcal{W}(\mathcal{E})\otimes\mathcal{W}(\mathcal{E}')}(\langle\delta u_0,1\rangle
\langle\mathfrak{e}_2(g),1\rangle,s)\,\psi_{U}(u_0)\,du_0\,dg.
\end{align}
Note that we do not have an involution here ($\iota=I_c$). We prove the analog of Proposition~\ref{proposition:equiv props}:
\begin{proposition}\label{proposition:equiv props GL}
The integral $Z(s,\omega,f_{\mathcal{W}(\mathcal{E})\otimes\mathcal{W}(\mathcal{E}')})$ can be regarded (at least formally) as a morphism in the space
\begin{align}\label{eq:homspace G with W(E) GL}
\Hom_{G^{(m,r)}\times G^{(m,r)}}(J_{U,\psi_U^{-1}}(\Ind_{\widetilde{P}}^{H^{(m)}}((\mathcal{W}(\mathcal{E})\otimes\mathcal{W}(\mathcal{E}'))\delta_P^s)),\pi^{\vee}\otimes\pi).
\end{align}
\end{proposition}
\begin{proof}
Let $\omega_{g_1,g_2}$ be defined as in Proposition~\ref{proposition:equiv props} (with $\iota=I_c$). We need to prove that for $g_1,g_2\in G$ and $u\in U$,
\begin{align}\label{eq:inv both GL}
&Z(s,\omega_{\langle g_1,1\rangle,\langle g_2,1\rangle},(\langle \mathfrak{e}_1(g_1),1\rangle\langle \mathfrak{e}_2(g_2),1\rangle\langle u,1\rangle)\cdot f_{\mathcal{W}(\mathcal{E})\otimes\mathcal{W}(\mathcal{E}')})=
\psi_U^{-1}(u)Z(s,\omega,f_{\mathcal{W}(\mathcal{E})\otimes\mathcal{W}(\mathcal{E}')}).
\end{align}
(Note that the analog of \eqref{eq:tri} factors through $G\times G$ by \eqref{eq:GL embeddings coverings G and G into H}.)

Regarding $u$, the proof is very similar to the proof in Proposition~\ref{proposition:equiv props}. The definition of the embedding implies that we can write ${}^{\mathfrak{e}_2(g)}u=v_gu_g$, with $u_g\in U_0$ but
$v_g\in{}^{\delta^{-1}}\diag(V_{(c^{rk})},V_{(c^{rk})})$. Then we apply
\eqref{eq:sigma conjugate v by h} and \eqref{eq:sigma on vh and h'v'} (as mentioned in \S~\ref{covering of the Levi}, these are still applicable with $\sigma_{2rkc}^{\diamondsuit}$ instead of $\sigma_{2rkc}$), and check that the equivariance properties of both
$\mathcal{W}(\mathcal{E})$ (resp., $\mathcal{W}(\mathcal{E}')$) under $\diag(V_{(c^{rk})},I_{rkc})$ (resp., $\diag(I_{rkc},V_{(c^{rk})})$)
together with $\psi_U(u_g)$ combine to emit $\psi_U^{-1}(u)$.

Now assume $u$ is trivial. The l.h.s.~ of \eqref{eq:inv both GL} equals
\begin{align*}
&\int\limits_{G}\int\limits_{U_0}
\pi^{\vee}(\langle g,1\rangle\langle g_2,1\rangle)
\xi^{\vee}(\pi(\langle g_1,1\rangle)\xi)
\\&f_{\mathcal{W}(\mathcal{E})\otimes\mathcal{W}(\mathcal{E}')}(\langle\delta u_0,1\rangle
\langle\mathfrak{e}_2(g),1\rangle\langle \mathfrak{e}_1(g_1),1\rangle\langle \mathfrak{e}_2(g_2),1\rangle,s)\,\psi_{U}(u_0)\,du_0\,dg.
\end{align*}
Since $\langle \mathfrak{e}_1(g_1),1\rangle$ and $\langle \mathfrak{e}_2(g_2),1\rangle$ commute,
and using \eqref{eq:the $2$-cocycle on G times G formula GL} (with $g_1=g_1'=I_c$) and the fact that $\pi^{\vee}$ is anti-genuine, when we change $g\mapsto gg_2^{-1}$ we obtain
\begin{align}\label{int:inv after getting rid of g_2 GL}
&\int\limits_{G}\int\limits_{U_0}
\pi^{\vee}(\langle g,1\rangle)
\xi^{\vee}(\pi(\langle g_1,1\rangle)\xi)
f_{\mathcal{W}(\mathcal{E})\otimes\mathcal{W}(\mathcal{E}')}(\langle\delta u_0,1\rangle
\langle\mathfrak{e}_1(g_1),1\rangle
\langle\mathfrak{e}_2(g),1\rangle,s)\,\psi_{U}(u_0)\,du_0\,dg.
\end{align}
(The absence of $\iota$ simplifies some of the computations!)
Then when we multiply $g\mapsto g_1g$,
\begin{align*}
&\pi^{\vee}(\langle g,1\rangle)\xi^{\vee}(\pi(\langle g_1,1\rangle)\xi)
\mapsto \sigma_c^{\diamondsuit}(g_1,g)\pi^{\vee}(\langle g,1\rangle)\xi^{\vee}(\xi),\\
&\langle\mathfrak{e}_2(g),1\rangle\mapsto
\langle\mathfrak{e}_2(g_1),\sigma_c^{\diamondsuit}(g_1,g)^{-1}\rangle\langle\mathfrak{e}_2(g),1\rangle,
\end{align*}
and \eqref{int:inv after getting rid of g_2 GL} equals
\begin{align}\label{int:GL GL the U_0 integral before g g inv}
&\int\limits_{G}\int\limits_{U_0}
\omega(\langle g,1\rangle)
f_{\mathcal{W}(\mathcal{E})\otimes\mathcal{W}(\mathcal{E}')}(\langle\delta u_0,1\rangle
\langle(g_1,g_1),1\rangle
\langle\mathfrak{e}_2(g),1\rangle,s)\,\psi_{U}(u_0)\,du_0\,dg,
\end{align}
where we also used \eqref{eq:g_1 and g_2 product in H GL}.

As in the proof of Proposition~\ref{proposition:equiv props}, ${}^{(g_1,g_1)}u_0=v_{g_1}u_{g_1}$ with
$v_{g_1}\in{}^{\delta^{-1}}\diag(V_{(c^{rk})},V_{(c^{rk})})$ and $u_{g_1}\in U_0$, and using \eqref{eq:sigma conjugate v by h}
and \eqref{eq:sigma on h and v},
the integral becomes
\begin{align}\label{int:inv before g giota inv GL}
&\int\limits_{G}\int\limits_{U_0}
\omega(\langle g,1\rangle)
f_{\mathcal{W}(\mathcal{E})\otimes\mathcal{W}(\mathcal{E}')}(
{}^{\delta}\langle(g_1,g_1),1\rangle
\langle\delta ,1\rangle
\langle v_{g_1} ,1\rangle\langle u_{g_1} ,1\rangle
\langle\mathfrak{e}_2(g),1\rangle,s)\,\psi_{U}(u_0)\,du_0\,dg.
\end{align}

We proceed by first proving
\begin{align}\label{eq:local f left inv by g iota g GL}
f_{\mathcal{W}(\mathcal{E})\otimes\mathcal{W}(\mathcal{E}')}({}^{\delta}\langle(g,g),1\rangle h,s)=f_{\mathcal{W}(\mathcal{E})\otimes\mathcal{W}(\mathcal{E}')}(h,s),\qquad\forall g\in G,h\in H^{(m,r)}.
\end{align}
Since ${}^{\delta}(g,g)=(g,g)$ ($\delta_1$ commutes with $(g,g)$), we can write
${}^{\delta}\langle (g,g),\epsilon\rangle=\langle (g,g),\epsilon_g\epsilon\rangle$ where $\epsilon_g\in\mu_m$.
Then
because $\sigma_{2rkc}^{\diamondsuit}$ is trivial on $\{(g,g):g\in G\}$, the map $g\mapsto\epsilon_g$ is a character of $G$, hence trivial on $\SL_c$. Also since $\delta^2=I_{2rkc}$, $\epsilon_g^2=1$. We claim $\epsilon_g=1$ for all $g\in G$, and it remains to show this
for $g=t\in T_{\GL_c}$. This follows from Proposition~\ref{proposition:action of W on torus is trivial on Sp} under our assumption
$\mu_{2m}\subset F^*$, but we provide a proof without this assumption.

We compute $\epsilon_t$ directly in $\Sp_{4rkc}^{(m)}$, using the $2$-cocycle $\sigma_{4rkc}$ of $\GL_{4rkc}$. Put
\begin{align*}
d_t=\diag(t,\ldots,t,t^*,\ldots,t^*)\in\Sp_{4rkc},
\end{align*}
where $t$ and $t^*$ each occurs $2rkc$ times. By definition
\begin{align*}
\epsilon_g=
\sigma_{4rkc}(\diag(\delta,\delta^*),d_t)\sigma_{4rkc}(\diag(\delta,\delta^*)d_t,\diag(\delta,\delta^*)^{-1})
\sigma_{4rkc}(\diag(\delta,\delta^*),\diag(\delta,\delta^*)^{-1})^{-1},
\end{align*}
so that in fact
\begin{align*}
{}^{\diag(\delta,\delta^*)}\langle d_t,1\rangle=\langle d_t,\epsilon_g\rangle.
\end{align*}
We write $\diag(\delta,\delta^*)$ in the form $t_0w'$ where $w'\in\mathfrak{W}_{4rkc}$ using Example~\ref{example:delta decomposition t_0 and w'} twice, once for $\delta$, the other for
$\delta^*(=\delta)$. We see that $w'=\diag(w'',w'')$ where $w''\in\mathfrak{W}_{2rkc}$ and
\begin{align*}
t_0=\diag((-1)^{rkc}I_{rkc},I_{rkc},(-1)^{rkc}I_{rkc},I_{rkc})\in\GL_{4rkc}.
\end{align*}
Also the set of roots $\alpha\in\Phi_{4rkc}^+$ such that $\diag(\delta,\delta^*)\alpha<0$ is the set
\begin{align*}
\{(i,j):1\leq i\leq rkc<j\leq 2rkc\}\coprod\{(i,j):2rkc+1\leq i\leq 3rkc<j\leq 4rkc\}.
\end{align*}
Then by \eqref{eq:conj mathcal t in GLd},
\begin{align*}
{}^{w'}\langle d_t,1\rangle=\langle {}^{w'}d_t,\prod_{1\leq i,j\leq c}(-t_j,t_i)_m^{(rk)^2}\prod_{1\leq i,j\leq c}(-t_j^{-1},t_i^{-1})_m^{(rk)^2}\rangle=\langle d_t,1\rangle.
\end{align*}
Hence
\begin{align*}
{}^{\diag(\delta,\delta^*)}\langle d_t,1\rangle=
{}^{t_0}\langle d_t,1\rangle.
\end{align*}
When $rkc$ is even, $t_0=I_{4rkc}$ and we deduce $\epsilon_t=1$. For odd $rkc$, using \eqref{eq:sigma on torus of GL} we see that
\begin{align*}
{}^{t_0}\langle d_t,1\rangle=\langle d_t,\sigma_{rkc}(-I_{rkc},t^{\triangle})\sigma_{rkc}(-I_{rkc},(t^*)^{\triangle})
\sigma_{rkc}(t^{\triangle},-I_{rkc})
\sigma_{rkc}((t^*)^{\triangle},-I_{rkc})
\rangle,
\end{align*}
and again by \eqref{eq:sigma on torus of GL},
\begin{align*}
&\sigma_{rkc}(-I_{rkc},t^{\triangle})=\prod_{l=0}^{rk-1}(\prod_{i=2}^c\prod_{j=i}^c(-1,t_j)_m)(-1,\det t)_m^{lc},\\
&\sigma_{rkc}(-I_{rkc},(t^*)^{\triangle})=\prod_{l=0}^{rk-1}(\prod_{i=1}^{c-1}\prod_{j=1}^i(-1,t_j^{-1})_m)(-1,\det t)_m^{lc},\\
&\sigma_{rkc}(t^{\triangle},-I_{rkc})=\prod_{l=0}^{rk-1}(\prod_{i=1}^{c-1}\prod_{j=1}^i(t_j,-1)_m)(-1,\det t)_m^{lc},\\
&\sigma_{rkc}((t^*)^{\triangle},-I_{rkc})=\prod_{l=0}^{rk-1}(\prod_{i=2}^{c}\prod_{j=i}^c(t_j^{-1},-1)_m)(-1,\det t)_m^{lc}.
\end{align*}
Since $(-1,x)_m(x^{-1},-1)_m=1$, the product of first and last $2$-cocycles is $1$, and so is the product of the second and third. We deduce $\epsilon_t=1$ (thereby $\epsilon_g=1$) in all cases.

We complete the proof of \eqref{eq:local f left inv by g iota g GL}. Since $\epsilon_g=1$ for all $g\in G$, we can use
\eqref{eq:block compatibility on Levi of P} and obtain
\begin{align*}
{}^{\delta}\langle (g,g),1\rangle=\langle \diag(g^{\triangle},I_{rkc}),1\rangle\langle \diag(I_{rkc},g^{\triangle}),1\rangle.
\end{align*}
Also $\delta_{P}((g,g))=1$. This together with \eqref{eq:invariance prop on GL or SL} proves \eqref{eq:local f left inv by g iota g GL}.

Returning to \eqref{int:inv before g giota inv GL} we have
\begin{align*}
&\int\limits_{G}\int\limits_{U_0}
\omega(\langle g,1\rangle)
f_{\mathcal{W}(\mathcal{E})\otimes\mathcal{W}(\mathcal{E}')}(
\langle\delta ,1\rangle
\langle v_{g_1} ,1\rangle\langle u_{g_1} ,1\rangle
\langle\mathfrak{e}_2(g),1\rangle,s)\,\psi_{U}(u_0)\,du_0\,dg.
\end{align*}
Using the equivariance properties of $\mathcal{W}(\mathcal{E})$ and $\mathcal{W}(\mathcal{E}')$ with respect to $\{\langle v,1\rangle:v\in V_{(c^{rk})}\}$
and changing variables in $u_{g_1}$,
we reach the r.h.s.~ of \eqref{eq:inv both GL}, as required.
\end{proof}
As a corollary of the proof, we obtain the analog of Corollary~\ref{corollary:du integral of f invariant under g iota g}.
\begin{corollary}\label{corollary:GL GL du integral of f invariant under g iota g}
For any section $f_{\mathcal{W}(\mathcal{E})\otimes\mathcal{W}(\mathcal{E}')}$ and $g_1\in G$,
\begin{align*}
&\int\limits_{U_0}
f_{\mathcal{W}(\mathcal{E})\otimes\mathcal{W}(\mathcal{E}')}(\langle\delta u_0 ,1\rangle\,
\langle(g_1,g_1),1\rangle,s)\,\psi_{U}(u_0)\,du_0
=\int\limits_{U_0}
f_{\mathcal{W}(\mathcal{E})\otimes\mathcal{W}(\mathcal{E}')}(\langle\delta u_0 ,1\rangle,s)\,\psi_{U}(u_0)\,du_0.
\end{align*}
\end{corollary}
\begin{proof}
See \eqref{int:GL GL the U_0 integral before g g inv} and the remaining part of the proof.
\end{proof}

\begin{proposition}\label{proposition:applicability of the GL GL construction}
Let $F$ be a local non-archimedean field. Assume $\mathcal{E}$ and $\mathcal{E}'$ are $(rk,c)$ representations.
Then \eqref{eq:invariance prop on GL or SL} holds for $g\in \SL_c$. Further assume $\mathcal{E}'=\mathcal{E}^*$ where
${}^*$ is given by \eqref{eq:involution b*0}, and if $m$ does not divide $rk$ (i.e., $r=m/2$ and $k$ is odd) we also
assume $\mu_{2m}\subset F^*$. Then \eqref{eq:invariance prop on GL or SL} holds for $g\in G$.
\end{proposition}
\begin{proof}
By definition $J_{V_{(c^{rk})},\psi}(\mathcal{E})$ and $J_{V_{(c^{rk})},\psi}(\mathcal{E}')$ are one-dimensional, hence so is their tensor.
The action of the subgroup $\{\langle (g,g),1\rangle:g\in G\}$ on $J_{V_{(c^{rk})},\psi}(\mathcal{E})\otimes J_{V_{(c^{rk})},\psi}(\mathcal{E}')$ is by a character, which must then be trivial when $g\in\SL_c$. Since $(\Lambda\otimes\Lambda')(\xi\otimes\xi')=\Lambda(\xi)\otimes\Lambda'(\xi')$ factors through the Jacquet modules
(on both $\mathcal{E}$ and $\mathcal{E}'$), and by \eqref{eq:block compatibility on Levi of P} for any $g\in G$,
\begin{align*}
\langle (g,g),1\rangle=\langle \diag(g^{\triangle},I_{rkc}),1\rangle\langle \diag(I_{rkc},g^{\triangle}),1\rangle,
\end{align*}
the first assertion holds.

Assume $m|rk$, then $\sigma^{\diamondsuit}_{rkc}$ is already trivial on $G^{\triangle}$, hence
$\GL_{rkc}^{(m,r)}$ is split over $G^{\triangle}$. Thus $G^{\triangle}$ acts by a non-genuine character
on $J_{V_{(c^{rk})},\psi}(\mathcal{E})$ (resp., $J_{V_{(c^{rk})},\psi}(\mathcal{E}')$), and we can write
\begin{align*}
\Lambda(\mathcal{E}(\langle g^{\triangle},1\rangle)\xi)=\varrho(\det{g})\Lambda(\xi),\qquad
\Lambda'(\mathcal{E}'(\langle g^{\triangle},1\rangle)\xi')=
\varrho'(\det{g})\Lambda'(\xi'),
\end{align*}
for some quasi-characters $\varrho,\varrho'$ of $F^*$.
When $\mathcal{E}'=\mathcal{E}^*$, by \eqref{eq:involution b*0} we have
$\mathcal{E}'(\langle g^{\triangle},1\rangle)=\mathcal{E}({}^*\langle g^{\triangle},1\rangle)=\mathcal{E}(\langle (g^*)^{\triangle},1\rangle)$,
whence
\begin{align*}
\Lambda'(\mathcal{E}'(\langle g^{\triangle},1\rangle)\xi')=
\Lambda'(\mathcal{E}(\langle (g^*)^{\triangle},1\rangle)\xi')=
\varrho(\det{g}^*)\Lambda'(\xi').
\end{align*}
This proves \eqref{eq:invariance prop on GL or SL} for all $g\in G$.

Now assume $m$ does not divide $rk$. Since $H^{(m,r)}$ is split over $\{(g,g):g\in G\}$ and we already have \eqref{eq:invariance prop on GL or SL} on $\SL_c$, it suffices to prove \eqref{eq:invariance prop on GL or SL} for the subgroup $A_1=\{\diag(a,I_{c-1}):a\in F^*\}$. For any
$t=\diag(a,I_{c-1})\in A_1$ and $t'=\diag(a',I_{c-1})\in A_1$,
by \eqref{eq:block compatibility on Levi of P} and \eqref{eq:Nice GL $2$-cocycle on torus},
\begin{align}\label{eq:(a,a)}
\sigma_{rkc}^{\diamondsuit}(\diag(a,I_{c-1})^{\triangle},\diag(a',I_{c-1})^{\triangle})=(a,a')_m^{rk}
=(a,a')_2.
\end{align}
Hence if we fix some nontrivial additive character $\psi'$ of $F$, there is a quasi-character $\varrho$ of $F^*$ such that
$\Lambda(\mathcal{E}(\langle \diag(a,I_{c-1})^{\triangle},1\rangle)\xi')=\gamma_{\psi'}(a)\varrho(a)\Lambda(\xi')$.

Furthermore, taking $w=(-1)^{c(c-1)/2}J_c\in\SL_c$, Proposition~\ref{proposition:action of W on torus is trivial on Sp}
(now $\mu_{2m}\subset F^{*m}$) and the left invariance of $\Lambda$ under $\langle w^{\triangle},1\rangle$ imply
\begin{align*}
\Lambda(\mathcal{E}(\langle \diag(a,I_{c-1})^{\triangle},1\rangle)\xi)&=
\Lambda(\mathcal{E}(\langle w^{\triangle},1\rangle\langle \diag(a,I_{c-1})^{\triangle},1\rangle)\xi)\\&=
\Lambda(\mathcal{E}(\langle ({}^w\diag(a,I_{c-1}))^{\triangle},1\rangle\langle w^{\triangle},1\rangle)\xi)\\&=
\Lambda(\mathcal{E}(\langle\diag(I_{c-1},a)^{\triangle},1\rangle\langle w^{\triangle},1\rangle)\xi).
\end{align*}
Since $\langle \diag(I_{c-1},a)^{\triangle},1\rangle$ must transform on the left by some character,
using the left invariance under $\langle w^{\triangle},1\rangle$ again we obtain for all $a\in F^*$,
\begin{align*}
\Lambda(\mathcal{E}(\langle \diag(I_{c-1},a)^{\triangle},1\rangle)\xi')=\gamma_{\psi'}(a)\varrho(a)\Lambda(\xi').
\end{align*}
In addition by \eqref{eq:involution b*0}, ${}^*\langle \diag(a,I_{c-1})^{\triangle},1\rangle=\langle \diag(I_{c-1},a^{-1})^{\triangle},1\rangle$.
Therefore when $\mathcal{E}'=\mathcal{E}^*$,
\begin{align*}
\Lambda'(\mathcal{E}'(\langle \diag(a,I_{c-1})^{\triangle},1\rangle)\xi')=
\Lambda'(\mathcal{E}(\langle \diag(I_{c-1},a^{-1})^{\triangle},1\rangle)\xi')=
\gamma_{\psi'}(a^{-1})\varrho(a^{-1})\Lambda'(\xi').
\end{align*}
Since $\gamma_{\psi'}(a)\gamma_{\psi'}(a^{-1})=1$ (see \eqref{eq:some props of gamma psi'}, $(a,a)_2=1$) and $\varrho(a)\varrho(a^{-1})=1$, again \eqref{eq:invariance prop on GL or SL} holds.
\end{proof}

The analog of Proposition~\ref{proposition:local props} \eqref{it:conv}--\eqref{it:arch} is valid here as well.
E.g., the integrals are absolutely convergent in a right half plane, and for a given $s$, there is data
$(\omega,f_{\mathcal{W}(\mathcal{E})\otimes\mathcal{W}(\mathcal{E}')})$ for which the integral is nonzero and holomorphic
in a neighborhood of $s$.

\section{Computation of the local factors with unramified data}\label{Computation of the local factors with unramified data}
In this section we prove Theorem~\ref{theorem:unramified computation for Sp(2n),SO(2n)}, namely we compute \eqref{eq:local integral at one place} with unramified data. Throughout, $F$ is unramified and $\mu_{2m}\subset F^*$.
Also fix an unramified character $\psi$ of $F$, and the Haar measure on $F$ is the one assigning $1$ to
$\mathcal{O}$. The measure of the subgroups $K_{H}$, $K_G$ and $K_{\GL_n}$ is also normalized to be $1$.

In
\S~\ref{Outline of the computation} we reduce the $G^{(m)}\times \GL_k^{(m,r)}$ integral to the
$\GL_{n}^{(m,r)}\times \GL_k^{(m,r)}$ integral \eqref{eq:local GL GL integral}. In
\S~\ref{proof of lemma:reduction from GLn to GLa GLb} we reduce the latter to the product of integrals
$\GL_{a}^{(m,r)}\times \GL_k^{(m,r)}$ and $\GL_{b}^{(m,r)}\times \GL_k^{(m,r)}$, where $n=a+b$.
The computation of the $\GL_{1}^{(m,r)}\times \GL_k^{(m,r)}$ integrals with unramified data is carried out in
\S~\ref{final reduction n = 1 linear groups}. Finally in
\S~\ref{putting it together} we collect the previous results to conclude the proof.

\subsection{The reduction from $\Sp_{c}^{(m)}$ to $\GL_{n}^{(m,r)}$}\label{Outline of the computation}
Let $n,k$ and $m$ be positive integers and put $c=2n$. Set $r=m/2$ if $m$ is even, otherwise $r=m$.
Let $G=\Sp_{c}$ and $H=\Sp_{2rkc}$. The covering $H^{(m)}$ is realized using $\sigma_{2rkc}$, the embedding
$G\times G \hookrightarrow H$ given in \S~\ref{embedding} is lifted to the coverings via \eqref{eq:embeddings coverings G and G into H}, and we fixed the $2$-cocycles on $G$ to be $\sigma_c^{*,rk}$ for the left copy and
$\sigma_c$ for the right copy. For the description of the local $G^{(m)}\times \GL_k^{(m,r)}$ integral see \S~\ref{global symplectic}.
Also recall that the splitting of $N_H$ is given by $v\mapsto\langle v,1\rangle$; the canonical splitting of $K_H$ is $y\mapsto\langle y,\eta_{2rkc}(y)\rangle$; and $\eta_{2rkc}$ is trivial on $N_{rkc}\cap K_H$, $T_{rkc}\cap K_H$ and on permutation matrices
(see \S~\ref{local covering}).

Let $\pi$ be a genuine irreducible unramified representation of $G^{(m)}$, where $G^{(m)}$ is realized using $\sigma_c$, and
$\tau$ be a genuine irreducible unramified representation of $\GL_k^{(m,r)}$, which is the unramified subrepresentation of
$\mathrm{I}_{\GL_k^{(m,r)}}(\vartheta,\chi)$.

We assume $\tau$ satisfies the following properties, which are plainly the restatement of several properties
from \S~\ref{Local components of rk c speh}.
Assume $\tau$ satisfies \eqref{eq:local assumption on tau}, and also assume that for all $0<l\leq c$, the representation
\begin{align*}
\Ind_{\widetilde{P}_{(k^{rl})}}^{\GL^{(m,r)}_{rkl}}((\tau\otimes \ldots \otimes \tau)\delta_{P_{(k^{rl})}}^{-1/(2rk)})
\end{align*}
contains a unique irreducible unramified $(rk,l)$ subrepresentation $\rho_l(\tau)$. Denote the $(rk,l)$ model of
$\rho_l(\tau)$ by $\mathcal{W}(\rho_l(\tau))$. It follows that
Corollary~\ref{corollary:realization space for 0 < l < c}
and Corollary~\ref{corollary:functional on theta implies functional on rho c tau} are applicable. Hence
we have the realization of the $(rk,c)$ functional from \S~\ref{decomposition of functionals}, using $(rk,l)$ and $(rk,c-l)$ functionals.

As proved in \S~\ref{Local components of rk c speh}, these properties
are all satisfied in the context of \eqref{eq:local integral at one place}, i.e., when $\tau$ is a local component at an unramified place of a genuine unitary irreducible cuspidal automorphic representation of $\GL_{k}^{(m,r)}(\A)$, and under
Conjectures~\ref{local Shimura conjecture} and \ref{Shimura conjecture}, and the global assumption that the field contains $\mu_{2m}$.
See Proposition~\ref{proposition:local L function condition claim}.

Let $\omega$ be the normalized unramified matrix coefficient of $\pi^{\vee}$, and
$f_{\mathcal{W}(\rho_c(\tau))}$ be the normalized unramified standard section of $\Ind_{\widetilde{P}}^{H^{(m)}}(\mathcal{W}(\rho_c(\tau))\delta_P^s)$,
as in the statement of Theorem~\ref{theorem:unramified computation for Sp(2n),SO(2n)}.

We can always assume $\pi$ is a quotient of $\Ind_{\widetilde{R}}^{G^{(m)}}(\pi_n)$, where
$R=M_R\ltimes U_R$ is the standard Siegel parabolic subgroup of $G$ (i.e., $M_R=\GL_n$), and $\pi_n$ is a genuine irreducible unramified
representation of $\GL_n^{(m,r)}$. Let ${}^*$ be defined by \eqref{eq:involution b*0}.
By definition $\tau^*$ enjoys the properties of $\tau$ stated above (see also Proposition~\ref{proposition:vartheta of *}), so that \eqref{eq:local assumption on tau} holds and
$\rho_l(\tau^*)$ is defined for all $0<l\leq c$. Also by \eqref{eq:block compatibility on Levi of P},
\begin{align*}
(\Ind_{\widetilde{P}_{(k^{rl})}}^{\GL^{(m,r)}_{rkl}}((\tau\otimes \ldots \otimes \tau)\delta_{P_{(k^{rl})}}^{-1/(2rk)}))^*
=\Ind_{\widetilde{P}_{(k^{rl})}}^{\GL^{(m,r)}_{rkl}}((\tau^*\otimes \ldots \otimes \tau^*)\delta_{P_{(k^{rl})}}^{-1/(2rk)}),
\end{align*}
hence $\rho_l(\tau^*)=\rho_l(\tau)^*$.

Let $\omega_n$ be the normalized unramified  matrix coefficient of $\pi_n^{\vee}$, and
$f_{\mathcal{W}(\rho_c(\tau))\otimes \mathcal{W}(\rho_c(\tau^*))}$ be the normalized unramified standard section of
\begin{align*}
\Ind_{\widetilde{P}_{(rkn,rkn)}}^{\GL_{2rkn}^{(m,r)}}((\mathcal{W}(\rho_n(\tau))\otimes \mathcal{W}(\rho_n(\tau^*)))\delta_{P_{(rkn,rkn)}}^s).
\end{align*}

\begin{remark}\label{remark:tau * vs dual}
In the linear setting, whenever $\tau$ is irreducible, $\tau^*=\tau^{\vee}$. Here it is important we use $\tau^*$,
since $\tau^{\vee}$ is anti-genuine, then $\tau\otimes\tau^{\vee}$ is non-genuine (rendering \eqref{eq:local GL GL integral} undefined).
\end{remark}

Put $\alpha=2rkc+1$ and define
\begin{align}\label{eq:d tau(s)}
d_{\tau,\vartheta}(s)=&
\left[\frac{L_{\vartheta}(r\alpha s+1/2,\tau)}{L_{\vartheta}(r\alpha s+rn+1/2,\tau)}\right]
\prod_{1\leq j\leq \lfloor rn/2\rfloor}
\frac{L_{\vartheta}(2r\alpha s+2j,\tau,\vee^2)}{L_{\vartheta}(2r\alpha s+2j+2rn-2\lfloor rn/2\rfloor-1,\tau,\vee^2)}\\\nonumber&\times
\prod_{1\leq j\leq \lceil rn/2\rceil}
\frac{L_{\vartheta}(2r\alpha s+2j-1,\tau,\wedge^2)}{L_{\vartheta}(2r\alpha s+2j+2rn-2\lceil rn/2\rceil,\tau,\wedge^2)},
\end{align}
where the factor in square brackets is included only when $m$ is odd, and for $x\in \R$,
$\lfloor x \rfloor$ denotes the largest integer smaller than or equal to $x$.

\begin{lemma}\label{lemma:reduction from classical to GLn}
For $\Real(s)\gg0$,
$Z(s,\omega,f_{\mathcal{W}(\rho_c(\tau))})=d_{\tau,\vartheta}(s)Z(\alpha s/(rkn),\omega_{n},f_{\mathcal{W}(\rho_c(\tau))\otimes \mathcal{W}(\rho_c(\tau^*))})$.
\end{lemma}
\begin{proof}
We adapt the proof of the linear case \cite[Lemma~27]{CFGK2}.
For the purpose of computing the integral with unramified data, we may assume
$\pi=\Ind_{\widetilde{R}}^{G^{(m)}}(\pi_n)$, then
$\pi^{\vee}=\Ind_{\widetilde{R}}^{G^{(m)}}(\pi_n^{\vee})$. Let
$\{,\}_{M_R}$ be the $\widetilde{M}_R$-invariant pairing on $\pi_n\times\pi_n^{\vee}$ and choose
a semi-invariant measure on $R\backslash G$ (see \cite[1.21]{BZ1}). We can then realize the $G^{(m)}$-invariant pairing on
$\pi\times\pi^{\vee}$ by the integral
\begin{align*}
\{\varphi,\varphi^{\vee}\}=\int\limits_{R\backslash G}\{\varphi(\langle g,1\rangle),\varphi^{\vee}(\langle g,1\rangle)\}_{M_R}\,dg,
\end{align*}
where $\varphi$ and $\varphi^{\vee}$ belong to the spaces of $\pi$ and $\pi^{\vee}$, respectively.
Using the Iwasawa decomposition $G=RK_G$, and since $\varphi$ is genuine and $\varphi^{\vee}$ is anti-genuine,
\begin{align}\label{eq:classical G pairing using M}
\{\varphi,\varphi^{\vee}\}=\int\limits_{K_G}\{\varphi(\langle y,\eta_c(y)\rangle),\varphi^{\vee}(\langle y,\eta_c(y)\rangle)\}_{M_R}\,dy.
\end{align}
If $\varphi$ is unramified,
$\varphi(\langle y,\eta_c(y)\rangle)=\varphi(e)$ where $e=\langle I_{c},1\rangle$.
Since $\omega$ is unramified, so is $\varphi$ (and $\varphi^{\vee}$) and
\begin{align*}
\omega(\langle g,1\rangle)=\int\limits_{K_G}\{\varphi(e),\varphi^{\vee}(\langle y,\eta_c(y)\rangle\langle g,1\rangle)\}_{M_R}\,dy.
\end{align*}

Plugging this into $Z(s,\omega,f_{\mathcal{W}(\rho_c(\tau))})$ we obtain
\begin{align}\label{int:Z with realization of coefficient}
\int\limits_{G}
\int\limits_{K_G}\{\varphi(e),\varphi^{\vee}(\langle y,\eta_c(y)\rangle\langle g,1\rangle)\}_{M_R}\,dy
\int\limits_{U_0}f_{\mathcal{W}(\rho_c(\tau))}(\langle\delta u_0,1\rangle\,{}^{\iota}\langle\mathfrak{e}_2(g),1\rangle,s)\,\psi_U(u_0)\,du_0\,dg.
\end{align}
Multiplying $g\mapsto y^{-1}g$ (we can change the order of integration when $\Real(s)\gg0$),
\begin{align*}
\langle g,1\rangle
\mapsto \langle y^{-1}, \sigma_{c}(y^{-1},g)^{-1}\rangle\langle g,1\rangle,
\end{align*}
and because $\langle y,\eta_c(y)\rangle\langle y^{-1},\eta_c(y^{-1})\rangle=e$ and $\varphi^{\vee}$ is anti-genuine,
\begin{align*}
\varphi^{\vee}(\langle y,\eta_c(y)\rangle\langle y^{-1}, \sigma_{c}(y^{-1},g)^{-1}\rangle\langle g,1\rangle)
=\eta_c(y^{-1})\sigma_{c}(y^{-1},g)\varphi^{\vee}(\langle g,1\rangle).
\end{align*}
Also
\begin{align*}
f_{\mathcal{W}(\rho_c(\tau))}(\langle\delta u_0,1\rangle\,{}^{\iota}\langle\mathfrak{e}_2(g),1\rangle,s)
&\mapsto f_{\mathcal{W}(\rho_c(\tau))}(\langle\delta u_0,1\rangle\,{}^{\iota}
\langle\mathfrak{e}_2(y^{-1}g),1\rangle,s)
\\&=f_{\mathcal{W}(\rho_c(\tau))}(\langle\delta u_0,1\rangle\,{}^{\iota}
\langle\mathfrak{e}_2(y^{-1}),\sigma_{2rkc}(\mathfrak{e}_2(y^{-1}),\mathfrak{e}_2(g))^{-1}\rangle\,{}^{\iota}\langle\mathfrak{e}_2(g),1\rangle,s).
\end{align*}
Since $\sigma_{c}(y^{-1},g)=\sigma_{2rkc}(\mathfrak{e}_2(y^{-1}),\mathfrak{e}_2(g))$
by Proposition~\ref{proposition:the $2$-cocycle on G times G}, and $\{\varphi,\varphi^{\vee}\}_{M_R}$ is bilinear, \eqref{int:Z with realization of coefficient} becomes
\begin{align}\label{int:Z with realization of coefficient 2}
&\int\limits_{G}
\int\limits_{K_G}\{\varphi(e),\varphi^{\vee}(\langle g,1\rangle)\}_{M_R}
\int\limits_{U_0}f_{\mathcal{W}(\rho_c(\tau))}(\langle\delta u_0,1\rangle
\,{}^{\iota}\langle\mathfrak{e}_2(y^{-1}),\eta_c(y^{-1})\rangle
\,{}^{\iota}\langle\mathfrak{e}_2(g),1\rangle,s)\,\psi_U(u_0)\,du_0\,dy\,dg.
\end{align}
The map $\mathfrak{e}_1(y)\mapsto \langle\mathfrak{e}_1(y),\eta_{2rkc}(\mathfrak{e}_1(y))\rangle$ is a splitting of $\mathfrak{e}_1(K_G)$, and
by \eqref{eq:lcl embeding left copy using sigma c into H} so is the map
\begin{align*}
\mathfrak{e}_1(y)\mapsto y\mapsto\langle y,\eta_c(y)\rangle\mapsto\langle\mathfrak{e}_1(y),\varsigma_{*,c}^{rk+1}(y)\eta_c^{-1}(y)\rangle.
\end{align*}
Hence because $K_G$ is perfect,
\begin{align}\label{eq:classical eta_H 1 compatible}
\eta_{2rkc}(\mathfrak{e}_1(y))=\varsigma_{*,c}^{rk+1}(y)\eta_c^{-1}(y).
\end{align}
Then since $f_{\mathcal{W}(\rho_c(\tau))}$ is right invariant on
$\langle \mathfrak{e}_1(y^{-1}),\eta_{2rkc}(\mathfrak{e}_1(y^{-1}))\rangle$ and using
Proposition~\ref{proposition:the $2$-cocycle on G times G} and \eqref{eq:iota image on product},
\begin{align*}
&f_{\mathcal{W}(\rho_c(\tau))}(\langle\delta u_0,1\rangle
\,{}^{\iota}\langle\mathfrak{e}_2(y^{-1}),\eta_c(y^{-1})\rangle
\,{}^{\iota}\langle\mathfrak{e}_2(g),1\rangle,s)
\\&=f_{\mathcal{W}(\rho_c(\tau))}(\langle\delta u_0,1\rangle
\,{}^{\iota}\langle(y^{-1},y^{-1}),\varsigma_{*,c}^{rk+1}(y^{-1})\rangle
\,{}^{\iota}\langle\mathfrak{e}_2(g),1\rangle,s).
\end{align*}
The integral over $U_0$ becomes
\begin{align}\label{eq:classical U_0 alone}
\int\limits_{U_0}f_{\mathcal{W}(\rho_c(\tau))}(\langle\delta u_0,1\rangle
\,{}^{\iota}\langle(y^{-1},y^{-1}),\varsigma_{*,c}^{rk+1}(y^{-1})\rangle
\,{}^{\iota}\langle\mathfrak{e}_2(g),1\rangle,s)\,\psi_U(u_0)\,du_0.
\end{align}
Now by Corollary~\ref{corollary:du integral of f invariant under g iota g} we can omit
${}^{\iota}\langle(y^{-1},y^{-1}),\varsigma_{*,c}^{rk+1}(y^{-1})\rangle$ from this integral, and since the measure of $K_G$ is $1$,
\eqref{int:Z with realization of coefficient 2} equals
\begin{align}\label{int:Z with section}
\int\limits_{G}
\{\varphi(e),\varphi^{\vee}(\langle g,1\rangle)\}_{M_R}
\int\limits_{U_0}f_{\mathcal{W}(\rho_c(\tau))}(\langle\delta u_0,1\rangle\,{}^{\iota}\langle\mathfrak{e}_2(g),1\rangle,s)\,\psi_U(u_0)\,du_0\,dg.
\end{align}

The function $a\mapsto \delta_P^{-1/2-s}(a)f_{\mathcal{W}(\rho_c(\tau))}(ah,s)$, $a\in \GL_n^{(m,r)}$, belongs to
the $(rk,c)$ model $\mathcal{W}(\rho_c(\tau))$ of $\rho_c(\tau)$. By Corollary~\ref{corollary:realization space for 0 < l < c} with $l=n=c/2$,
\begin{align*}
\rho_c(\tau)\subset\Ind_{\widetilde{P}_{(rkn,rkn)}}^{\GL_{rkc}^{(m,r)}}((
\mathcal{W}(\rho_{n}(\tau))\otimes \mathcal{W}(\rho_{n}(\tau)))\delta_{P_{(rkn,rkn)}}^{-1/(2rk)}).
\end{align*}
Let $L=M_L\ltimes U_L$ be the standard parabolic subgroup of $H$ with $M_L=\GL_{rkn}\times\GL_{rkn}=M_{(rkn,rkn)}$.
According to Proposition~\ref{proposition: rk c functional is nontrivial on nontwisted},
we can realize $\mathcal{W}(\rho_c(\tau))$ using \eqref{eq:mnk functional using w_{n,m,k}}. Then by transitivity of induction
there is a standard section $f_{\mathcal{W}(\rho_n(\tau))\otimes \mathcal{W}(\rho_n(\tau))}$ on
\begin{align}\label{rep:induced f before M(s)}
\Ind_{\widetilde{L}}^{H^{(m)}}
(|\det|^{-n/2+\alpha s}\mathcal{W}(\rho_n(\tau))\otimes |\det|^{n/2+\alpha s}\mathcal{W}(\rho_n(\tau)))
\end{align}
such that for all $h\in H^{(m)}$,
\begin{align}\label{eq:new formula for f with realization}
f_{\mathcal{W}(\rho_c(\tau))}(h,s)=\int\limits_Vf_{\mathcal{W}(\rho_n(\tau))\otimes \mathcal{W}(\rho_n(\tau))}(\langle\kappa v,1\rangle h,s)\,dv.
\end{align}
Here $V$ and $\kappa$ are given in \S~\ref{decomposition of functionals}, and implicitly identified with their images in $M_P$. Moreover, $f_{\mathcal{W}(\rho_n(\tau))\otimes \mathcal{W}(\rho_n(\tau))}$ is normalized and unramified. Indeed it is clearly unramified, and by the proof of Proposition~\ref{proposition: rk c functional is nontrivial on nontwisted}
it is also normalized, because the volume of $V\cap K_{\GL_{rkc}}$ is $1$ and
\begin{align*}
f_{\mathcal{W}(\rho_n(\tau))\otimes \mathcal{W}(\rho_n(\tau))}(\langle I_{2rkc},1\rangle,s)=f_{\mathcal{W}(\rho_c(\tau))}(\langle I_{2rkc},1\rangle,s)=1.
\end{align*}

With this modification, \eqref{int:Z with section} takes the form
\begin{align}\label{int:Z with section 2}
\int\limits_{G}
\{\varphi(e),\varphi^{\vee}(\langle g,1\rangle)\}_{M_R}
\int\limits_{U_0}
\int\limits_V
f_{\mathcal{W}(\rho_n(\tau))\otimes \mathcal{W}(\rho_n(\tau))}(\langle \kappa v,1\rangle\langle\delta u_0,1\rangle\,{}^{\iota}\langle\mathfrak{e}_2(g),1\rangle,s)\,\psi_U(u_0)\,dv\,du_0\,dg.
\end{align}
This integral is absolutely convergent for $\Real(s)\gg0$ as a triple integral.

Next we shift $v$ to the right. By \eqref{eq:sigma on h and v} and since $\sigma_{2rkc}$ is trivial $\mathfrak{W}^+_{2rkc}$ (see \S~\ref{local covering}),
\begin{align*}
\langle \kappa v,1\rangle\langle\delta u_0,1\rangle=
\langle \kappa ,1\rangle\langle v,1\rangle\langle\delta_0,1\rangle\langle\delta_1,1\rangle\langle u_0,1\rangle.
\end{align*}
Observe the following properties.
\begin{enumerate}[leftmargin=*]
\item\label{it:GL observe 1} The element $\delta_0$ normalizes $V$, in particular ${}^{\delta_0^{-1}}\langle v,1\rangle=\langle {}^{\delta_0^{-1}}v,1\rangle$ by \eqref{eq:sigma conjugate v by h}.
\item\label{it:GL observe 2} For $v\in V$, ${}^v\delta_1=\delta_1u'$ with $u'\in U_0$ such that $\psi_U(u')=1$;
since $\delta_1u'\in N_{rkc}$, by \eqref{eq:sigma conjugate v by h} and \eqref{eq:sigma on h and v} we have
${}^v\langle\delta_1,1\rangle=\langle\delta_1u',1\rangle=\langle\delta_1,1\rangle\langle u',1\rangle$.
\item\label{it:GL observe 3} The elements of $V$ normalize $U_0$ and fix $\psi_U|_{U_0}$.
\item\label{it:GL observe 4} Since $V$ is also a subgroup of $M_{(rkc-n,n)}$ and the image of $v\in V$ in $M_P$ takes the form
$\diag(v',I_c,{v'}^*)$ with $v'\in M_{(rkc-n,n)}$, $V$ and $\mathfrak{e}_2(G)$ commute; by \eqref{eq:BLS block compatible} this extends to $H^{(m)}$ ($\det v=1$ for any $v\in V$).
\item\label{it:GL observe 5} Since $\delta_0$ and $\kappa$ commute and $\sigma_{2rkc}$ is trivial on $\mathfrak{W}^+_{2rkc}$, $\langle\delta_0,1\rangle$ and $\langle\kappa,1\rangle$ commute.
\item\label{it:5.5} Because $\delta_1,{}^{\kappa}\delta_1\in N_{rkc}$, \eqref{eq:sigma conjugate v by h} implies ${}^{\kappa}\langle\delta_1,1\rangle=\langle{}^{\kappa}\delta_1,1\rangle$.
\item\label{it:GL observe 6} The image of $\kappa$ in $M_P$ takes the form
$\diag(\kappa_0,I_c,\kappa_0^*)$, where $\kappa_0\in\GL_{rkc-n}$ is a permutation matrix, and also $\det g=1$, whence \eqref{eq:BLS block compatible} implies that $\langle\kappa,1\rangle$ and $\langle\mathfrak{e}_2(g),1\rangle$ commute.
\end{enumerate}
A computation shows
\begin{align}\label{eq:U_0' classical}
U_0'={}^{\kappa}U_0=\left\{\left(\begin{smallmatrix}I_{rkn}&&U_1&U_2\\&I_{rkn}&U_3&U_4\\&&I_{rkn}\\&&&I_{rkn}\end{smallmatrix}\right)\in H\right\},
\end{align}
where $U_1=\left(\begin{smallmatrix}*&*\\0&*\end{smallmatrix}\right)$ with $0\in\Mat_{n}$, so that
$\left\{\left(\begin{smallmatrix}I_{rkn}&U_1\\&I_{rkn}\end{smallmatrix}\right)\right\}$
is the unipotent subgroup appearing in the $\GL_n^{(m,r)}\times \GL_k^{(m,r)}$ integral; $\psi_U$ is trivial on $U_2$ and $U_3$, and
its restriction to the coordinates of $U_1$ is the character $\psi_U$ for this integral;
$U_2$ (resp., $U_3$) takes the form $\left(\begin{smallmatrix}*&*\\0&*\end{smallmatrix}\right)$ with $0\in\Mat_n$; and $U_4$ is already determined by $U_1$ and the symplectic form defining $H$.
Regard $\psi_U$ as a character of $U_0'$ by $\psi_U(u_0')=\psi_U({}^{\kappa^{-1}}u_0')$, for $u_0'\in U_0'$.

We can now use \eqref{it:GL observe 1}--\eqref{it:GL observe 6} to rewrite \eqref{int:Z with section 2} in the form
\begin{align}\label{int:classical after props}
&\int\limits_{G}
\{\varphi(e),\varphi^{\vee}(\langle g,1\rangle)\}_{M_R}
\\&\int\limits_{U_0'}
\int\limits_V
f_{\mathcal{W}(\rho_n(\tau))\otimes \mathcal{W}(\rho_n(\tau))}(\langle \delta_0({}^{\kappa}\delta_1),1\rangle
\langle u_0',1\rangle \,{}^{\iota}\langle\mathfrak{e}_2(g),1
\rangle\langle \kappa v,1\rangle ,s)\,\psi_U(u_0')\,dv\,du_0'\,dg.\notag
\end{align}
Next, factoring this integral through $U_R$ we obtain
\begin{align*}
&\int\limits_{U_R\backslash G}
\int\limits_{U_R}
\{\varphi(e),\varphi^{\vee}(\langle z,1\rangle\langle g,1\rangle)\}_{M_R}
\\&\int\limits_{U_0'}
\int\limits_V
f_{\mathcal{W}(\rho_n(\tau))\otimes \mathcal{W}(\rho_n(\tau))}(\langle \delta_0({}^{\kappa}\delta_1),1\rangle
\langle u_0',1\rangle \,{}^{\iota}\langle\mathfrak{e}_2(z),1
\rangle\,{}^{\iota}
\langle\mathfrak{e}_2(g),1
\rangle\langle \kappa v,1\rangle ,s)\,\psi_U(u_0')\,dv\,du_0'\,dz\,dg.
\end{align*}

Let $U_0^{\bullet}$ be the subgroup obtained from $U_0'$ by replacing the $0$ block of $U_2$ with arbitrary coordinates such that
the elements still belong in $H$.
For $z\in U_R$ and $u_0'\in U_0'$,
\begin{align*}
{}^{\mathfrak{e}_2({}^{\iota}z^{-1})}(({}^{\kappa}\delta_1) u_0')=m_z({}^{\kappa}\delta_1)u_z,
\end{align*}
where $m_z$ belongs to the subgroup $V_{((2rk-1)n,n)}$ of $\GL_{rkc}$ and when
$z$ and $u_0'$ vary over $U_R$ and $U_0'$, $u_z$ varies over $U_0^{\bullet}$. Moreover
${}^{\kappa}\delta_1,u_0',m_z,u_z\in N_{rkc}$, so that by \eqref{eq:sigma conjugate v by h} and \eqref{eq:sigma on vh and h'v'},
\begin{align*}
{}^{\mathfrak{e}_2({}^{\iota}z^{-1})}\langle({}^{\kappa}\delta_1) u_0',1\rangle=\langle m_z,1\rangle\langle({}^{\kappa}\delta_1)u_z,1\rangle.
\end{align*}
After this conjugation and since $\varphi^{\vee}(\langle z,1\rangle\langle g,1\rangle)=\varphi^{\vee}(\langle g,1\rangle)$, the integral equals
\begin{align*}
&\int\limits_{U_R\backslash G}
\{\varphi(e),\varphi^{\vee}(\langle g,1\rangle)\}_{M_R}
\int\limits_{U_R}
\int\limits_{U_0'}
\int\limits_V
f_{\mathcal{W}(\rho_n(\tau))\otimes \mathcal{W}(\rho_n(\tau))}
\\&(
\langle \delta_0,1\rangle
\,{}^{\iota}\langle\mathfrak{e}_2(z),1\rangle
\langle m_z,1\rangle\langle({}^{\kappa}\delta_1)u_z,1\rangle
\,{}^{\iota}
\langle\mathfrak{e}_2(g),1
\rangle\langle \kappa v,1\rangle ,s)\,\psi_U(u_0')\,dv\,du_0'\,dz\,dg.
\end{align*}
Since ${}^{\delta_0}(^{\iota}\langle\mathfrak{e}_2(z),1\rangle)\in N_{rkc}$, by
\eqref{eq:epsilon for conjugation between split subgroups} we have
${}^{\delta_0}({}^{\iota}\langle\mathfrak{e}_2(z),1\rangle)=\langle{}^{\delta_0}\mathfrak{e}_2({}^{\iota}z),1\rangle$.
Also because ${}^{\delta_0}m_z\in V_{(n,(2rk-1)n)}<N_{rkc}$, by \eqref{eq:sigma conjugate v by h} we have
${}^{\delta_0}\langle m_z,1\rangle=\langle{}^{\delta_0}m_z,1\rangle$. Thus we can further write
\begin{align*}
&\int\limits_{U_R\backslash G}
\{\varphi(e),\varphi^{\vee}(\langle g,1\rangle)\}_{M_R}
\int\limits_{U_R}
\int\limits_{U_0'}
\int\limits_V
f_{\mathcal{W}(\rho_n(\tau))\otimes \mathcal{W}(\rho_n(\tau))}
\\&(\langle{}^{\delta}\mathfrak{e}_2({}^{\iota}z),1\rangle
\langle {}^{\delta_0}m_z,1\rangle
\langle \delta_0({}^{\kappa}\delta_1),1\rangle\langle u_z,1\rangle
\,{}^{\iota}
\langle\mathfrak{e}_2(g),1
\rangle\langle \kappa v,1\rangle ,s)\,\psi_U(u_0')\,dv\,du_0'\,dz\,dg.
\end{align*}
Since ${}^{\delta}\mathfrak{e}_2({}^{\iota}z)\in U_P<U_L$, the section is invariant on the left on $\langle{}^{\delta}\mathfrak{e}_2({}^{\iota}z),1\rangle$.
Now as in the linear case (\cite[(3.21)--(3.22)]{CFGK2}),
we can change variables in $u_z$ to remove the dependency on $z$, this affects $\psi_U$ but the change is cancelled when we
use the equivariance properties of the inducing data of $f_{\mathcal{W}(\rho_n(\tau))\otimes \mathcal{W}(\rho_n(\tau))}$ under
$\langle {}^{\delta_0}m_z,1\rangle$ (the first model emits a character).
Therefore we can combine the integrations over $U_R$ and $U_0'$ into an integration over $U_0^{\bullet}$ and obtain
\begin{align}\label{int:Z with section 4}
&\int\limits_{U_R\backslash G}
\{\varphi(e),\varphi^{\vee}(\langle g,1\rangle)\}_{M_R}
\int\limits_{U_0^{\bullet}}
\int\limits_V
f_{\mathcal{W}(\rho_n(\tau))\otimes \mathcal{W}(\rho_n(\tau))}\\&\nonumber(\langle \delta_0({}^{\kappa}\delta_1),1\rangle
\langle u_0^{\bullet},1\rangle \,{}^{\iota}\langle\mathfrak{e}_2(g),1
\rangle\langle \kappa v,1\rangle ,s)\,\psi_U(u_0^{\bullet})\,dv\,du_0^{\bullet}\,dg.
\end{align}
Here $\psi_U$ was lifted to a character of $U_0^{\bullet}$ trivially on the coordinates outside $U_0'$.

Now we decompose $\delta_0$ and write $du_0^{\bullet}$ as an iterated integral. Let $\delta'=\delta_0'\delta_1'$ be the elements corresponding to the $\GL_n^{(m,r)}\times\GL_{k}^{(m,r)}$ integral (defined in \S~\ref{integarls for GL}), regarded as elements of $M_P$. For example, $\delta_0'$ is the image of
$\left(\begin{smallmatrix}&I_{rkn}\\I_{rkn}\end{smallmatrix}\right)$ in $M_P$. We can then write
\begin{align*}
\delta_0=w\delta_0'w,\qquad w=\left(\begin{smallmatrix}I_{rkn}\\&&I_{rkn}\\&-I_{rkn}\\&&&I_{rkn}\end{smallmatrix}\right).
\end{align*}

Observe that because $U_0^{\bullet}$ differs from $U_0'$ by only one block (the bottom left $n\times n$ block of $U_2$), one may still use
\eqref{eq:U_0' classical} to describe the elements of $U_0^{\bullet}$.
For $u_0^{\bullet}\in U_0^{\bullet}$ and $i\in\{2,3\}$, let $u^i$ denote the element obtained from $u_0^{\bullet}$ by zeroing out the coordinates in the blocks $U_j$ with $j\ne i$, and $u^{1,4}$ be similarly obtained by removing the coordinates in $U_2$ and $U_3$. Then we have
\begin{align}
&\delta_0u_0^{\bullet}=w\cdot{}^{(\delta_{0}'w)}u^2\cdot\delta_0'\cdot {}^{w}(u^{1,4})\cdot wu^3,\nonumber\\\label{eq:classical decomp delta_0 1}&\delta_0({}^{\kappa}\delta_1)u_0^{\bullet}=w\cdot{}^{(\delta_{0}'w)}u^2\cdot\delta'\cdot {}^{w}(u^{1,4})\cdot wu^3.
\end{align}
Note that ${}^{\kappa}\delta_1\in U_P$, hence commutes with $u_0^{\bullet}$. We can extend \eqref{eq:classical decomp delta_0 1} to
$H^{(m)}$: indeed $\sigma_{2rkc}$ is trivial on $\mathfrak{W}^+_{2rkc}$; and also $u_0^{\bullet},{}^{(\delta_{0}'w)}u^2,{}^{w}(u^{1,4}),{}^{\kappa}\delta_1\in N_{rkc}$, so that
\eqref{eq:sigma conjugate v by h} is applicable and the elements ${}^{\kappa}\delta_1$ and $u_0^{\bullet}$ commute in $H^{(m)}$. Hence
\begin{align}\label{eq:decomp of delta in covering}
\langle \delta_0({}^{\kappa}\delta_1),1\rangle
\langle u_0^{\bullet},1\rangle=
\langle w\,{}^{(\delta_{0}'w)}u^2,1\rangle\langle \delta'\,{}^{w}(u^{1,4}),1\rangle\langle wu^3,1\rangle.
\end{align}

Let $U^2$ be the subgroup of elements $^{(\delta_{0}'w')}u^2$:
\begin{align*}
U^2=\left\{\left(\begin{smallmatrix}I_{rkn}\\&I_{rkn}&Z\\&&I_{rkn}\\&&&I_{rkn}\end{smallmatrix}\right)\in H\right\}.
\end{align*}
This subgroup will be used below to define an intertwining operator.
Let $U^{1,4}$ be the subgroup of elements
${}^{w}(u^{1,4})$, this is the subgroup $U_0$ of the $\GL_n^{(m,r)}\times\GL_{k}^{(m,r)}$ integral,
when we identify $\GL_{2rkn}$ with $M_P$ by $b\mapsto\diag(b,b^*)$. Additionally
denote the subgroup of elements $u^3$ by $U^3$.

Let $M(s)$ be the standard intertwining operator from the space of \eqref{rep:induced f before M(s)} to the space of
\begin{align}\label{eq:image of M(s)}
\Ind_{\widetilde{L}}^{H^{(m)}}
(|\det|^{-n/2+\alpha s}\mathcal{W}(\rho_n(\tau))\otimes |\det|^{-n/2-\alpha s}\mathcal{W}(\rho_n(\tau^*))),
\end{align}
defined by
\begin{align*}
M(s)f_{\mathcal{W}(\rho_n(\tau))\otimes \mathcal{W}(\rho_n(\tau))}(h,s)=
\int\limits_{U^2}f_{\mathcal{W}(\rho_n(\tau))\otimes \mathcal{W}(\rho_n(\tau))}(\langle w,1\rangle\langle u^2,1\rangle h,s)\,du^2.
\end{align*}
Note that the image of $M(s)$ is indeed in this space, where ${}^*$ is defined by \eqref{eq:involution b*0}.
This is because $\mathcal{W}(\rho_n(\tau))^*=\mathcal{W}(\rho_n(\tau)^*)=\mathcal{W}(\rho_n(\tau^*))$, which is true for any extension of
${}^*$ to $\GL_{rkn}^{(m,r)}$ (see \S~\ref{Outline of the computation}); for $a\in\GL_{rkn}$,
${}^{w}\diag(I_{rkn},a,a^*,I_{rkn})=\diag(I_{rkn},a^*,a,I_{rkn})$, hence when we regard $\langle a,\epsilon\rangle$ in
$\GL_{rkn}^{(m,r)}$, we obtain a lift of ${}^*$; to check which one, consider $a=\diag(a_1,I_{rkn-1})$ and now
\begin{align*}
{}^{w}\langle\diag(I_{rkn},a,a^*,I_{rkn}),\epsilon\rangle=\langle\diag(I_{rkn},a^*,a,I_{rkn}),\epsilon\rangle
\end{align*}
by Proposition~\ref{proposition:action of W on torus is trivial on Sp}.
The image of $M(s)f_{\mathcal{W}(\rho_n(\tau))\otimes \mathcal{W}(\rho_n(\tau))}$ is an unramified vector in the space of \eqref{eq:image of M(s)}, and we denote
$d_{\tau,\vartheta}(s)=M(s)f_{\mathcal{W}(\rho_n(\tau))\otimes \mathcal{W}(\rho_n(\tau))}(\langle I_{2rkc},1\rangle)$. We will prove below that $d_{\tau,\vartheta}(s)$ is equal to \eqref{eq:d tau(s)}.

When we plug \eqref{eq:decomp of delta in covering} into \eqref{int:Z with section 4}, it becomes
\begin{align}\label{int:Z with section 5}
&\int\limits_{U_R\backslash G}
\{\varphi(e),\varphi^{\vee}(\langle g,1\rangle)\}_{M_R}
\int\limits_{U^3}
\int\limits_{U^{1,4}}
\int\limits_V
M(s)f_{\mathcal{W}(\rho_n(\tau))\otimes \mathcal{W}(\rho_n(\tau))}\\&\nonumber(
\langle \delta'u',1\rangle\langle wu^3,1\rangle
\,{}^{\iota}\langle\mathfrak{e}_2(g),1
\rangle\langle \kappa v,1\rangle ,s)\,\psi_U(u')\,dv\,du'\,du^3\,dg.
\end{align}

We proceed using the Iwasawa decomposition $G=RK_G$. Since $\sigma_{c}(y,y')=\sigma_{2rkc}(\mathfrak{e}_2(y),\mathfrak{e}_2(y'))$ for any $y,y'\in K_G$ and $K_G$ is perfect, $\eta_{c}(y)=\eta_{2rkc}(\mathfrak{e}_2(y))$. Then by \eqref{eq:iota image on K},
\begin{align*}
{}^{\iota}\langle\mathfrak{e}_2(y),\eta_{c}(y)\rangle
=\langle\mathfrak{e}_2({}^{\iota}y),\eta_{c}({}^{\iota}y)\rangle
=\langle\mathfrak{e}_2({}^{\iota}y),\eta_{2rkc}(\mathfrak{e}_2({}^{\iota}y))\rangle.
\end{align*}
As explained above (\eqref{it:GL observe 4} and \eqref{it:GL observe 6}),
$\mathfrak{e}_2(G)$ commutes with $\kappa$ and $v$ in $H^{(m)}$. Hence
\begin{align*}
&M(s)f_{\mathcal{W}(\rho_n(\tau))\otimes \mathcal{W}(\rho_n(\tau))}\nonumber(
\langle \delta'u',1\rangle\langle wu^3,1\rangle
\,{}^{\iota}\langle\mathfrak{e}_2(g),1
\rangle
\,{}^{\iota}\langle\mathfrak{e}_2(y),\eta_{c}(y)\rangle
\langle \kappa v,1\rangle ,s)
\\&=M(s)f_{\mathcal{W}(\rho_n(\tau))\otimes \mathcal{W}(\rho_n(\tau))}(
\langle \delta'u',1\rangle\langle wu^3,1\rangle
\,{}^{\iota}\langle\mathfrak{e}_2(g),1
\rangle
\langle \kappa v,1\rangle\langle\mathfrak{e}_2({}^{\iota}y),\eta_{2rkc}(\mathfrak{e}_2({}^{\iota}y))\rangle,s)
\\&=M(s)f_{\mathcal{W}(\rho_n(\tau))\otimes \mathcal{W}(\rho_n(\tau))}(
\langle \delta'u',1\rangle\langle wu^3,1\rangle
\,{}^{\iota}\langle\mathfrak{e}_2(g),1
\rangle
\langle \kappa v,1\rangle,s).
\end{align*}
In addition the vector $\varphi^{\vee}$ is right invariant on $\langle y,\eta_c(y)\rangle$.
It follows that when we write a representative $g$ in the form $\diag(g',{g'}^*)y$ with $g'\in\GL_n$ and $y\in K_G$, the integrand is right invariant on $y$. Since the measure of $K_G$ is $1$, \eqref{int:Z with section 5} equals
\begin{align}\label{int:Z with section 6}
&\int\limits_{\GL_n}
\{\varphi(e),\varphi^{\vee}(\langle \diag(g',{g'}^*),1\rangle)\}_{M_R}
\int\limits_{U^3}
\int\limits_{U^{1,4}}
\int\limits_V
M(s)f_{\mathcal{W}(\rho_n(\tau))\otimes \mathcal{W}(\rho_n(\tau))}\\&\nonumber(
\langle \delta'u',1\rangle\langle wu^3,1\rangle
\,{}^{\iota}\langle\mathfrak{e}_2(\diag(g',{g'}^*)),1
\rangle\langle \kappa v,1\rangle ,s)\,\psi_U(u')\delta_R^{-1}(\diag(g',{g'}^*))\,dv\,du'\,du^3\,dg'.
\end{align}

The function
\begin{align*}
\langle g',\epsilon\rangle\mapsto \delta_R^{-1/2}(\diag(g',{g'}^*))\{\varphi(e),\varphi^{\vee}(\langle \diag(g',{g'}^*),\epsilon\rangle)\}_{M_R}
\end{align*}
is the normalized unramified matrix coefficient of $\pi_n^{\vee}$, which we denoted $\omega_n$. When we conjugate
$u^3$ by ${}^{\iota}\mathfrak{e}_2(\diag(g',{g'}^*))$, the measure is multiplied by $|\det g'|^{(-rk+1)n}$. Using
\eqref{eq:sigma on h and v} and \eqref{eq:sigma conjugate v by h}, we have
\begin{align}\label{int:Z with section 7}
&\int\limits_{U^3}
\int\limits_V
\int\limits_{\GL_n}
\omega_n(\langle g',1\rangle)
\int\limits_{U^{1,4}}
|\det g'|^{(n-\alpha)/2}M(s)f_{\mathcal{W}(\rho_n(\tau))\otimes \mathcal{W}(\rho_n(\tau))}\\&\nonumber(
\langle \delta'u',1\rangle
\,{}^{w\iota}\langle\mathfrak{e}_2(\diag(g',{g'}^*)),1\rangle
\langle wu^3,1\rangle\langle \kappa v,1\rangle ,s)\,\psi_U(u')\,du'\,dg'\,du^3\,dv.
\end{align}
Here we also changed the integration order, which is allowed for $\Real(s)\gg0$, and for any $g\in G$, the notation
${}^{w\iota}\langle\mathfrak{e}_2(g),1\rangle$ stands for ${}^{w}({}^{\iota}\langle\mathfrak{e}_2(g),1\rangle)$.

Let $(,)':\GL_n\times\GL_n\hookrightarrow \GL_{2rkn}$ denote the embedding in the construction of the
$\GL_n^{(m,r)}\times\GL_{k}^{(m,r)}$ integral, when $\GL_{2rkn}$ is identified with $M_P$ as above.
Now we see that
\begin{align*}
{}^{w\iota}\mathfrak{e}_2(\diag(g',{g'}^*))=\diag(I_{rkn},g',I_{2(rk-1)n},{g'}^*,I_{rkn})=(I_n,g')'.
\end{align*}
We claim that this conjugation extends to the covering, namely
\begin{align}\label{eq:conj of a using delta}
{}^{w\iota}\langle\mathfrak{e}_2(\diag(g',{g'}^*)),1\rangle=\langle (I_n,g')',(\det g',\det g')_m^{(rk+1)n+1}\rangle.
\end{align}
In fact, since we only need to apply this identity to the integral \eqref{int:Z with section 7}, it is sufficient to
establish \eqref{eq:conj of a using delta} on a dense subset of $\GL_n$.
We provide two proofs of this, both valid also when $F$ is not unramified. The first argument shows
\eqref{eq:conj of a using delta} on a dense subset without using $\mu_{2m}\subset F^*$. The second proves
\eqref{eq:conj of a using delta} for any $g'\in\GL_n$, under the assumption $\mu_{2m}\subset F^*$ (which we have throughout \S~\ref{Computation of the local factors with unramified data}), in which case
$(\det g',\det g')_m^{(rk+1)n+1}=1$.

The first proof: we prove \eqref{eq:conj of a using delta} on the subset $N_{\GL_n}^-B_{\GL_n}$, which is dense in $\GL_n$. Thus
$g'=u_g^-t_gv_g$ where $u_g^-\in N_{\GL_n}^-$, $t_g\in T_{\GL_n}$ and $v_g\in N_{\GL_n}$. Fix the splitting
$u^-\mapsto\langle u^-,\varsigma_n(u^-)\rangle$ of $N_{n}^-$ in $G^{(m)}$ and similarly fix $\varsigma_{rkc}$ for $N_{rkc}^-$ in $H^{(m)}$. As explained in \S~\ref{extension of the involution}, we can compute $\varsigma_{\iota,c}$ by regarding $\iota$ as an element of $\GL_c$, then $w\iota$ is the element in $\GL_{2rkc}$ given by $w\diag(I_{2rkc-n},\iota,I_{2rkc-n})$. By the analog of \eqref{eq:sigma conjugate v by h} for $N_{rkc}^-$ and $\varsigma_{rkc}$ (use \eqref{eq:epsilon for conjugation between split subgroups}),
\begin{align*}
&{}^{w\iota}\langle \mathfrak{e}_2(\diag(u_g^-,(u_g^-)^*)),\varsigma_{rkc}(\mathfrak{e}_2(\diag(u_g^-,(u_g^-)^*)))\rangle=
\langle (I_n,u_g^-)',\varsigma_{rkc}((I_n,u_g^-)')\rangle.
\end{align*}
By \eqref{eq:the $2$-cocycle on G times G formula} and \eqref{eq:BLS block compatible}, for any $z_1,z_2\in N_{\GL_n}^-$,
\begin{align*}
\sigma_{c}(\diag(z_1,z_1^*),\diag(z_2,z_2^*))=\sigma_{2rkc}(\mathfrak{e}_2(\diag(z_1,z_1^*)),\mathfrak{e}_2(\diag(z_2,{z_2}^*)))=
\sigma_{2rkc}((I_n,z_1)',(I_n,z_2)'),
\end{align*}
so that because $\varsigma_n$ is canonical,
\begin{align*}
\varsigma_{n}(\diag(u_g^-,(u_g^-)^*))=\varsigma_{rkc}(\mathfrak{e}_2(\diag(u_g^-,(u_g^-)^*)))=\varsigma_{rkc}((I_n,u_g^-)').
\end{align*}
Hence \eqref{eq:conj of a using delta} holds for $u_g^-$.

Consider $t_g$ and set $d_g=\diag(t_g,t_g^*)\in G$. Write $\iota=t_0w'$ where $w'\in\mathfrak{W}_{c}$ and $t_0\in T_{\GL_{c}}$ is such that its diagonal coordinates are $\pm1$. Then by \eqref{eq:conj mathcal t in GLd} and since
the set of roots $\alpha\in\Phi_{c}^+$ such that $\iota\alpha<0$ is the set
$\{(i,j):1\leq i\leq n<j\leq c\}$,
\begin{align*}
{}^{\iota}\langle d_g,1\rangle&={}^{t_0}({}^{w'}\langle d_g,1\rangle)=
{}^{t_0}(\langle {}^{w'}d_g,\prod_{1\leq i,j\leq n}(-t_j^{-1},t_i)_m\rangle)
={}^{t_0}(\langle {}^{w'}d_g,\prod_{1\leq i\ne j\leq n}(-1,t_it_j)_m\rangle).
\end{align*}
Here to show the second equality note that each pair $(i,j)$ with $i<j$ appears twice in the product over $1\leq i,j\leq n$, and
$(-t_j^{-1},t_i)_m(-t_i^{-1},t_j)_m=(-1,t_it_j)_m$ (also $(-t_i^{-1},t_i)_m=1$).

We proceed similarly with $w$; write $w=t_1w''$ with $t_1\in T_{\GL_{2rkc}}$ and $w''\in\mathfrak{W}_{2rkc}$ and denote $t_2=t_1{}^{w''}\mathfrak{e}_2(t_0)$. By \eqref{eq:conj mathcal t in GLd} and since the roots to consider are now
\begin{align*}
\{(i,j):rkn+1\leq i\leq rkc<j\leq rkc+rkn)\},
\end{align*}
\begin{align*}
{}^{w\iota}\langle \mathfrak{e}_2(d_g),1\rangle&={}^{t_2}({}^{w''}\langle \mathfrak{e}_2({}^{w'}d_g),\prod_{1\leq i\ne j\leq n}(-1,t_it_j)_m\rangle)
\\&={}^{t_2}(\langle {}^{w''}\mathfrak{e}_2({}^{w'}d_g),
\prod_{1\leq i\ne j\leq n}(-1,t_it_j)_m^2\rangle)={}^{t_2}\langle(I_n,t_g)',1\rangle.
\end{align*}
Now by Example~\ref{example:delta decomposition t_0 and w'},
\begin{align*}
&t_0=\diag((-1)^nI_n,I_n),\qquad
t_1=\diag(I_{rkn},(-1)^{rkn}I_{rkn},-I_{rkn},I_{rkn}), \\
&t_2=\diag(I_{rkn},(-1)^{rkn}I_{rkn},-I_{(rk-1)n},(-1)^{n+1}I_n,I_{rkn}).
\end{align*}
By \eqref{eq:sigma on torus of GL} ($t_2\in T_{\GL_{2rkc}}$) we see that
\begin{align*}
{}^{t_2}\langle(I_n,t_g)',1\rangle=\langle(I_n,t_g)',(\det t_g,\det t_g)_m^{n+1+rkn}\rangle.
\end{align*}
Also \eqref{eq:conj of a using delta} holds for
$v_g$ by \eqref{eq:sigma conjugate v by h}, because ${}^{w\iota}\mathfrak{e}_2(\diag(v_g,v_g^*))\in N_{rkc}$.

Since $\sigma_c(u_g^-t_g,v_g)=1$ (by \eqref{eq:sigma on h and v}),
we have shown that for $g'=u_g^-t_gv_g$,
\begin{align*}
&{}^{w\iota}\langle\mathfrak{e}_2(\diag(g',{g'}^*)),1\rangle\\&=
\sigma_{2rkc}(\mathfrak{e}_2(\diag(u_g^-,{u_g^-}^*)),\mathfrak{e}_2(\diag(t_g,t_g^*)))^{-1}
\,{}^{w\iota}\langle\mathfrak{e}_2(\diag(u_g^-,{u_g^-}^*)),1\rangle
\,{}^{w\iota}\langle\mathfrak{e}_2(\diag(t_g,t_g^*)),1\rangle\\&
\quad{}^{w\iota}\langle\mathfrak{e}_2(\diag(v_g,v_g^*)),1\rangle\\&=
\sigma_{c}(\diag(u_g^-,{u_g^-}^*),\diag(t_g,t_g^*))^{-1}
\langle(I_n,u_g^-)',1\rangle
\langle(I_n,t_g)',(\det t_g,\det t_g)_m^{(rk+1)n+1}\rangle
\langle(I_n,v_g)',1\rangle
\\&=\langle(I_n,g')',(\det g',\det g')_m^{(rk+1)n+1}\rangle.
\end{align*}
Here for the last equality we used \eqref{eq:BLS block compatible}. This completes the proof of \eqref{eq:conj of a using delta} for
$g'\in N_{\GL_n}^-B_{\GL_n}$.

The second proof: consider any $g'\in\GL_n$.
According to the Bruhat decomposition in $\GL_n$, we can write
$g'=u_gt_gw_gv_g$ with $u_g,v_g\in N_{\GL_n}$, $t_g\in T_{\GL_n}$ and
$w_g$ is a permutation matrix. If $\diag(w_g,w_g^*)=t_0w'$ where $t_0\in T_{\GL_c}$ and $w'\in\mathfrak{W}_c$,
then by \eqref{eq:2-cocycle} (with $g=\diag(t_g,t_g^*)$, $g'=t_0$ and $g''=w'$) and \eqref{eq:sigma w mathcal and t},
\begin{align}\label{eq:tg and wg}
\sigma_c(\diag(t_g,t_g^*),t_0w')=\sigma_c(\diag(t_g,t_g^*),t_0)=1
\end{align}
($\mu_{2m}\subset F^*$). Hence
by \eqref{eq:sigma on h and v},
\begin{align}\label{int:integrand Z with section 8}
&{}^{w\iota}\langle\mathfrak{e}_2(\diag(g',{g'}^*)),1\rangle
\\&=\nonumber
{}^{w\iota}\langle\mathfrak{e}_2(\diag(u_g,u_g^*)),1\rangle
\,{}^{w\iota}\langle\mathfrak{e}_2(\diag(t_g,t_g^*)),1\rangle
\,{}^{w\iota}\langle\mathfrak{e}_2(\diag(w_g,w_g^*)),1\rangle
\,{}^{w\iota}\langle\mathfrak{e}_2(\diag(v_g,{v_g}^*)),1\rangle.
\end{align}
As above \eqref{eq:conj of a using delta} holds for $u_g$ and $v_g$ by \eqref{eq:sigma conjugate v by h}
(${}^{w\iota}\mathfrak{e}_2(\diag(u_g,u_g^*))\in N_{rkc}$ and similarly for $v_g$); and for $w_g$
because $\sigma_{2rkc}$ is trivial on $\mathfrak{W}^+_{2rkc}$. Regarding $t_g$, let $t_0=\diag(I_{2rkc-n},-I_n,I_{2rkc})\in T_{\GL_{2rkc}}$, then $w\iota t_0\in H\cap \mathfrak{W}^+_{2rkc}$. Since $t_0$ commutes with any element of $T_{\GL_{2rkc}}$, we can replace
$w\iota$ with $w\iota t_0$ in \eqref{eq:conj of a using delta} for $t_g$, then it holds
by Proposition~\ref{proposition:action of W on torus is trivial on Sp}.
Therefore \eqref{int:integrand Z with section 8} becomes
\begin{align*}
&{}^{w\iota}\langle\mathfrak{e}_2(\diag(g',{g'}^*)),1\rangle
=
\langle(I_n,u_g)',1\rangle
\langle(I_n,t_g)',1\rangle
\langle(I_n,w_g)',1\rangle
\langle(I_n,v_g)',1\rangle=\langle(I_n,g')',1\rangle.
\end{align*}
Here for the second equality note that
$\sigma_{2rkc}((I_n,t_g)',(I_n,w_g)')=1$ which follows, e.g., from \eqref{eq:tg and wg} and \eqref{eq:BLS block compatible}, or simply
by applying the arguments used for the proof of \eqref{eq:tg and wg} to $\sigma_{2rkc}$.
Again we conclude \eqref{eq:conj of a using delta}.

Now we plug \eqref{eq:conj of a using delta} into \eqref{int:Z with section 7} and obtain
\begin{align}\label{int:Z with section 8}
&\int\limits_{U^3}
\int\limits_V
\int\limits_{\GL_n}
\omega_n(\langle g',1\rangle)
\int\limits_{U^{1,4}}
|\det g'|^{(n-\alpha)/2}M(s)f_{\mathcal{W}(\rho_n(\tau))\otimes \mathcal{W}(\rho_n(\tau))}\\&\nonumber(
\langle \delta'u',1\rangle\langle (I_n,g')',1\rangle
\langle wu^3,1\rangle\langle \kappa v,1\rangle ,s)\,\psi_U(u')\,du'\,dg'\,du^3\,dv.
\end{align}
Looking at \eqref{eq:image of M(s)} we see that when we restrict
$M(s)f_{\mathcal{W}(\rho_n(\tau))\otimes \mathcal{W}(\rho_n(\tau))}$ to $\widetilde{M}_P=\GL_{2rkn}^{(m,r)}$, we obtain a rational section of
\begin{align}\label{eq:classical image restrict Siegel}
|\det|^{(\alpha-n)/2}
\Ind_{\widetilde{P}_{(rkn,rkn)}}^{\GL_{2rkn}^{(m,r)}}
((\mathcal{W}(\rho_n(\tau))\otimes\mathcal{W}(\rho_n(\tau^*)))\delta_{P_{(rkn,rkn)}}^{\alpha s/(rkn)}).
\end{align}
Therefore the inner $du'dg'$-integral of \eqref{int:Z with section 8} is the $\GL_n^{(m,r)}\times\GL_k^{(m,r)}$ integral, and
we can rewrite \eqref{int:Z with section 8} in the form
\begin{align}\label{int:Z with section 9}
\int\limits_{U^3}
\int\limits_V
Z(\alpha s/rkn,\omega_n,(\langle wu^3,1\rangle\langle \kappa v,1\rangle)\cdot M(s)f_{\mathcal{W}(\rho_n(\tau))\otimes \mathcal{W}(\rho_n(\tau))})\,du^3\,dv.
\end{align}

As in the linear case, the next step is to show that the integral over $U^3$ vanishes unless $v,u^3\in K_H$.
Fix the splitting $v^-\mapsto\langle v^-,\varsigma_{rkc}(v^-)\rangle$ for $N_{rkc}^-$ in $H^{(m)}$ (see above).
We can write
\begin{align*}
\langle wu^3,1\rangle\langle \kappa v,1\rangle&=
{}^{w}(\langle u^3,1\rangle{}^{\kappa}\langle v,1\rangle)\langle w\kappa,1\rangle=
{}^{w}(\langle u^3,1\rangle\langle {}^{\kappa}v,\varsigma_{rkc}({}^{\kappa}v)\rangle)\langle w\kappa,1\rangle
\\&=\langle {}^{w}u^3,\varsigma_{rkc}({}^{w}u^3)\rangle\langle {}^{w\kappa}v,\varsigma_{rkc}({}^{w\kappa}v)\rangle)\langle w\kappa,1\rangle=
\langle {}^{w}(u^3({}^{\kappa}v)),\varsigma_{rkc}({}^{w}(u^3({}^{\kappa}v)))\rangle\langle w\kappa,1\rangle.
\end{align*}
Here we used the fact that ${}^{\kappa}V,{}^{w}U^3,{}^{w\kappa}V<N_{rkc}^-$ and \eqref{eq:epsilon for conjugation between split subgroups}.
The section $M(s)f_{\mathcal{W}(\rho_n(\tau))\otimes \mathcal{W}(\rho_n(\tau))}$ is unramified, thus \eqref{int:Z with section 9} becomes
\begin{align*}
\int\limits_{U^3}
\int\limits_V
Z(\alpha s/rkn,\omega_n,\langle {}^{w}(u^3({}^{\kappa}v)),\varsigma_{rkc}({}^{w}(u^3({}^{\kappa}v)))\rangle\cdot M(s)f_{\mathcal{W}(\rho_n(\tau))\otimes \mathcal{W}(\rho_n(\tau))})\,du^3\,dv.
\end{align*}
Now we use conjugations by elements from $N_{rkc}(\mathcal{O})$. We briefly explain how to extend the argument from the linear setting;
see the proof of \cite[Lemma~27]{CFGK2} for more details, and note that all occurrences of $k$ from the argument in
\cite[pp.~1037--1040]{CFGK2} should here be replaced with $rk$. If $x\in N_{rkc}(\mathcal{O})$, then
$M(s)f_{\mathcal{W}(\rho_n(\tau))\otimes \mathcal{W}(\rho_n(\tau))}$ is right-invariant on $\langle x,1\rangle$. Assume $z\in {}^{w}(U^3({}^{\kappa}V))$
is such that ${}^{x^{-1}}z=u_xz_x$ where $u_x\in P\cap N_{rkc}$, the projection of $u_x$ to $M_P$ belongs to the subgroup $U$ appearing in the
$\GL_n^{(m,r)}\times\GL_k^{(m,r)}$ integral, and $z_x\in {}^{w}(U^3({}^{\kappa}V))$. Then by \eqref{eq:sigma on h and v},
\begin{align*}
{}^{x^{-1}}\langle z,\varsigma_{rkc}(z)\rangle&=
\langle {}^{x^{-1}}z,\sigma_{2rkc}(x^{-1},z)\sigma_{2rkc}(x^{-1}z,x)\sigma_{2rkc}(x^{-1},x)^{-1}\varsigma_{rkc}(z)\rangle=
\langle u_xz_x,\varsigma_{rkc}(z)\rangle\\&=\langle u_x,1\rangle\langle z_x,\varsigma_{rkc}(z_x)\rangle.
\end{align*}
Now we can change variables $z_x\mapsto z$, this change only involves $U^3$, and use the equivariance properties of the inner integral $Z(\cdots)$ (with respect to the subgroup $U$ of that integral). It then follows that when we select the elements $x$ in the correct order, the
$du^3$ integral vanishes unless $v\in K_H$. The coordinates of $U^3$ are also handled by choosing $x\in N_{rkc}(\mathcal{O})$ in a certain order, such that ${}^{x^{-1}}z=u_xz_x$ where $u_x$ and $z_x$ are as above. The change of variables involves blocks of $U^3$ which have not been handled, and the vanishing follows by producing an inner integral $\int_{\mathcal{X}}\psi(\mathrm{tr}(xu))dx$, where $\mathcal{X}\subset\Mat_{n}(\mathcal{O})$ and $u$ is a certain block of $U^3$. This integral vanishes unless
$u\in\Mat_{n}(\mathcal{O})$.

Turning back to \eqref{int:Z with section 9}, since the measures of $U^3\cap K_H$ and $V\cap K_H$ are $1$, we obtain
\begin{align*}
Z(\alpha s/rkn,\omega_n,M(s)f_{\mathcal{W}(\rho_n(\tau))\otimes \mathcal{W}(\rho_n(\tau))}).
\end{align*}
The section $M(s)f_{\mathcal{W}(\rho_n(\tau))\otimes \mathcal{W}(\rho_n(\tau))}$ restricted to $\GL_{2rkn}^{(m,r)}$ is
equal to $d_{\tau,\vartheta}(s)f_{\mathcal{W}(\rho_c(\tau))\otimes \mathcal{W}(\rho_c(\tau^*))}$ (by the definition of the latter section),
hence we proved (in $\Real(s)\gg0$)
\begin{align*}
Z(s,\omega,f_{\mathcal{W}(\rho_c(\tau))})=d_{\tau,\vartheta}(s)Z(\alpha s/(rkn),\omega_{n},f_{\mathcal{W}(\rho_c(\tau))\otimes \mathcal{W}(\rho_c(\tau^*))}).
\end{align*}
To complete the proof of the lemma, it remains to show that $d_{\tau,\vartheta}(s)$ is equal to \eqref{eq:d tau(s)}.

Recall $\tau\subset I_{\GL_k^{(m,r)}}(\vartheta,\chi)$. Identify $\Sp_{2rkn}$ with its image in the standard Levi subgroup of $H$ isomorphic to $\GL_{rkn}\times \Sp_{2rkn}$. By
\eqref{eq:BLS block compatible} and the definition of $\Sp_{2rkn}^{(m)}$ (using restriction, see \S~\ref{local covering}),
$\widetilde{\Sp}_{2rkn}=\Sp_{2rkn}^{(m)}$ (alternatively use \cite[\S~2, Theorem~7]{BLS}).
Looking at \eqref{rep:induced f before M(s)} we see that the restriction of $f_{\mathcal{W}(\rho_c(\tau))}$ to $\Sp_{2rkn}^{(m)}$ is the normalized unramified element in the space of
\begin{align*}
\mathrm{I}_{\Sp_{2rkn}^{(m)}}(\vartheta,\otimes_{1\leq i\leq k,1\leq j\leq rn}\chi_i|~|^{\alpha s+j/r-1/(2r)}).
\end{align*}
To compute $d_{\tau,\vartheta}(s)$ we deploy the
Gindikin--Karpelevich formula \cite[Theorem~7.10]{Gao2018}, which includes the interpretation of the constant in terms of local Langlands--Shahidi $L$-functions (for unramified data), i.e., using a decomposition of the adjoint action.

The adjoint action of $\GL_{rkn}(\C)$ on the Lie algebra of the $L$-group of $U^2$ is now either $\text{st}\oplus\wedge^2$ for odd $m$,
because then ${\Sp_{2rkn}^{(m)}}^{\vee}=\SO_{2rkn+1}(\C)$, or $\wedge^2$ if $m$ is even, since then
${\Sp_{2rkn}^{(m)}}^{\vee}=\Sp_{2rkn}(\C)$.

When $m$ is odd, the standard representation contributes
\begin{align}\label{eq:constant term st}
\prod_{1\leq i\leq k}\prod_{1\leq j\leq rn}\frac{L_{\vartheta}(r\alpha s +j-1/2,\chi_i)}{L_{\vartheta}(r\alpha s
+j+1/2,\chi_i)}
=\frac{L_{\vartheta}(r\alpha s +1/2,\tau)}{L_{\vartheta}(r\alpha s +rn+1/2,\tau)}.
\end{align}
For all $m$, the exterior square contributes for each pair $1\leq i\ne i'\leq k$,
\begin{align*}
\prod_{1\leq j,j'\leq rn}\frac{L_{\vartheta}(2\alpha s +j+j'-1,\chi_i\chi_{i'})}{L_{\vartheta}(2\alpha s +j+j',\chi_i\chi_{i'})}=
\prod_{1\leq j\leq rn}\frac{L_{\vartheta}(2\alpha s +j,\chi_i\chi_{i'})}{L_{\vartheta}(2\alpha s +j+rn,\chi_i\chi_{i'})},
\end{align*}
and for $1\leq i\leq k$,
\begin{align*}
&\prod_{1\leq j_1<rn}\,
\prod_{j_1<j_2\leq rn}\frac{L_{\vartheta}(2\alpha s +j_1+j_2-1,\chi_i^2)}{L_{\vartheta}(2\alpha s +j_1+j_2,\chi_i^2)}
=\prod_{1\leq j<rn}
\frac{L_{\vartheta}(2\alpha s +2j,\chi_i^2)}{L_{\vartheta}(2\alpha s +j+rn,\chi_i^2)}.
\end{align*}
(This product is by definition $1$ for $r=n=1$.)
Thus when $rn$ is odd we obtain
\begin{align}\label{eq:odd rn}
\prod_{1\leq j\leq (rn-1)/2}
\frac{L_{\vartheta}(2\alpha s +2j,\tau,\vee^2)}{L_{\vartheta}(2\alpha s +2j+rn,\tau,\vee^2)}
\prod_{1\leq j\leq (rn+1)/2}
\frac{L_{\vartheta}(2\alpha s +2j-1,\tau,\wedge^2)}{L_{\vartheta}(2\alpha s +2j+rn-1,\tau,\wedge^2)},
\end{align}
and for even $rn$,
\begin{align}\label{eq:even rn}
\prod_{1\leq j\leq rn/2}
\frac{L_{\vartheta}(2\alpha s +2j,\tau,\vee^2)}{L_{\vartheta}(2\alpha s +2j+rn-1,\tau,\vee^2)}
\prod_{1\leq j\leq rn/2}
\frac{L_{\vartheta}(2\alpha s +2j-1,\tau,\wedge^2)}{L_{\vartheta}(2\alpha s +2j+rn,\tau,\wedge^2)}.
\end{align}
We conclude that for even $m$, $d_{\tau,\vartheta}(s)$ is either \eqref{eq:odd rn} or \eqref{eq:even rn} (depending on $rn$),
and if $m$ is odd, it is also multiplied by \eqref{eq:constant term st}. In both cases we obtain \eqref{eq:d tau(s)}.
\end{proof}

\subsection{The reduction from $\GL_{n}^{(m,r)}$ to $\GL_{1}^{(m,r)}$}\label{proof of lemma:reduction from GLn to GLa GLb}
The set-up and local integral were defined in \S~\ref{integarls for GL} and we use the same notation and definitions. In particular $c=n$,
$G=\GL_c$, $H=\GL_{2rkc}$ and $P=P_{(rkc,rkc)}$. The covering $H^{(m,r)}$ is realized using $\sigma_{2rkc}^{\diamondsuit}$ and $G$ is realized using $\sigma_{c}^{\diamondsuit}$. We have a splitting of $K_H$ given by $y\mapsto\langle y,\eta_{2rkc}^{\diamondsuit}(y)\rangle$, where
$\eta_{2rkc}^{\diamondsuit}(y)=\eta_{4rkc}(\diag(y,y^*))$.

Let $\pi$ be a genuine irreducible unramified representation of $G$, and let $\tau$ and $\tau'$ be genuine irreducible unramified
representations of $\GL_k^{(m,r)}$. We assume $\tau$ and $\tau'$ both satisfy the properties from \S~\ref{Outline of the computation},
in particular they each satisfy \eqref{eq:local assumption on tau},
$\tau\subset \mathrm{I}_{\GL_k^{(m,r)}}(\vartheta,\chi)$, $\tau'\subset \mathrm{I}_{\GL_k^{(m,r)}}(\vartheta,\chi')$ and
$\mathcal{W}(\rho_l(\tau))$ and $\mathcal{W}(\rho_l(\tau'))$ are the $(rk,l)$ models for $0<l\leq c$.

We further assume that the pair of representations $(\rho_l(\tau),\rho_l(\tau'))$ satisfies \eqref{eq:invariance prop on GL or SL}
for all $l$. Since we are already assuming the $(rk,l)$ models are unique, this condition does not depend on the functionals realizing each model. Recall that \eqref{eq:invariance prop on GL or SL} was necessary for the definition of the $\GL_{l}^{(m,r)}\times\GL_{k}^{(m,r)}$ integrals. Note that since
$\tau'=\tau^*$ when the integrals are produced by Lemma~\ref{lemma:reduction from classical to GLn}, Proposition~\ref{proposition:applicability of the GL GL construction} guarantees this condition holds, in the context where we will be applying
the results of this section.

Let $\omega$ be the normalized unramified matrix coefficient of $\pi^{\vee}$ and
$f_{\mathcal{W}(\rho_c(\tau))\otimes \mathcal{W}(\rho_c(\tau'))}$ be the normalized unramified standard section of
\begin{align*}
\Ind_{\widetilde{P}}^{H^{(m,r)}}(\mathcal{W}(\rho_c(\tau))\otimes \mathcal{W}(\rho_c(\tau'))\delta_P^s).
\end{align*}

Write $c=a+b$ for some $0<a<c$. We can now assume $\pi$ is a quotient of $\Ind_{\widetilde{P}_{(a,b)}}^{ G^{(m,r)}}(\pi_a\otimes\pi_b)$, where $\pi_a$ and $\pi_b$ are genuine irreducible unramified representations of $\GL_a^{(m,r)}$ and $\GL_b^{(m,r)}$. Let
$\omega_a$ and $\omega_b$ be the normalized unramified matrix coefficients of $\pi_a^{\vee}$ and $\pi_b^{\vee}$. Let
$f_{\mathcal{W}(\rho_a(\tau))\otimes \mathcal{W}(\rho_a(\tau'))}$ and $f_{\mathcal{W}(\rho_b(\tau))\otimes \mathcal{W}(\rho_b(\tau'))}$ be the normalized unramified standard sections of
\begin{align*}
&\Ind_{\widetilde{P}_{(rka,rka)}}^{\GL_{2rka}^{(m,r)}}((\mathcal{W}(\rho_a(\tau))\otimes \mathcal{W}(\rho_a(\tau')))\delta_{P_{(rka,rka)}}^s)\qquad\text{and}\\
&\Ind_{\widetilde{P}_{(rkb,rkb)}}^{\GL_{2rkb}^{(m,r)}}((\mathcal{W}(\rho_b(\tau))\otimes \mathcal{W}(\rho_b(\tau')))\delta_{P_{(rkb,rkb)}}^s).
\end{align*}

Denote $\alpha=rkn$ and
\begin{align}\label{eq:GL d tau tau}
&d_{\tau,\tau',\vartheta,a,b}(s)=\prod_{1\leq j\leq rb}\frac{L_{\vartheta,\vartheta^{-1}}(2r\alpha s+j,\tau\times{\tau'}^{\vee})}{L_{\vartheta,\vartheta^{-1}}(2r\alpha s+ra+j,\tau\times{\tau'}^{\vee})}
=\prod_{1\leq j\leq rb}\frac{L_{\vartheta}(2r\alpha s+j,\tau\times{\tau'}^{\vee})}{L_{\vartheta}(2r\alpha s+ra+j,\tau\times{\tau'}^{\vee})}.
\end{align}
Note that by Proposition~\ref{proposition:vartheta of *}, ${\tau'}^*\subset\mathrm{I}_{\GL_k^{(m,r)}}(\vartheta,{\chi'}^{-1})$ hence
$L_{\vartheta,\vartheta^{-1}}(s,\tau\times{\tau'}^{\vee})=L_{\vartheta}(s,\tau\times{\tau'}^{*})$; the second equality holds because $\mu_{2m}\subset F^*$.

The following is the analog of Lemma~\ref{lemma:reduction from classical to GLn}, but while there we only considered the Siegel parabolic subgroup, here we handle any maximal parabolic subgroup.
\begin{lemma}\label{lemma:reduction from GLn to GLa GLb}
In a right half-plane $\Real(s)\gg0$,
\begin{align*}
&Z(s,\omega,f_{\mathcal{W}(\rho_c(\tau))\otimes \mathcal{W}(\rho_c(\tau'))})\\&=
d_{\tau,\tau',\vartheta,a,b}(s)Z(\alpha s/(rka),\omega_{a},f_{\mathcal{W}(\rho_a(\tau))\otimes \mathcal{W}(\rho_a(\tau'))})Z(\alpha s/(rkb),\omega_{b},f_{\mathcal{W}(\rho_b(\tau))\otimes \mathcal{W}(\rho_b(\tau'))}).
\end{align*}
\end{lemma}
\begin{proof}
This proof is the adaptation of \cite[Lemma~33]{CFGK2} to the covering case; the arguments similar to those used for the proof of
Lemma~\ref{lemma:reduction from classical to GLn} will be described briefly.

Assume $\pi=\Ind_{\widetilde{R}}^{G^{(m,r)}}(\pi_a\otimes\pi_b)$ and
$\pi^{\vee}=\Ind_{\widetilde{R}}^{G^{(m,r)}}(\pi_a^{\vee}\otimes\pi_b^{\vee})$.
Let $\{,\}_{M_R}$ be the $\widetilde{M}_R$-invariant pairing on $(\pi_a\otimes\pi_b)\times(\pi_a^{\vee}\otimes\pi_b^{\vee})$.
Then we can write (cf. \eqref{eq:classical G pairing using M})
\begin{align*}
&\{\varphi,\varphi^{\vee}\}=\int\limits_{K_G}\{\varphi(\langle y,\eta_c^{\diamondsuit}(y)\rangle),\varphi^{\vee}(\langle y,\eta_c^{\diamondsuit}(y)\rangle)\}_{M_R}\,dy,\\
&\omega(\langle g,1\rangle)=\int\limits_{K_G}\{\varphi(e),\varphi^{\vee}(\langle y,\eta_c^{\diamondsuit}(y)\rangle\langle g,1\rangle)\}_{M_R}\,dy.
\end{align*}
Here $e=\langle I_c,1\rangle$. Thus $Z(s,\omega,f_{\mathcal{W}(\rho_c(\tau))\otimes \mathcal{W}(\rho_c(\tau'))})$ equals
\begin{align}\label{int:GL GL Z with realization of coefficient}
\int\limits_{G}
\int\limits_{K_G}\{\varphi(e),\varphi^{\vee}(\langle y,\eta_c^{\diamondsuit}(y)\rangle\langle g,1\rangle)\}_{M_R}\,dy
\int\limits_{U_0}f_{\mathcal{W}(\rho_c(\tau))\otimes \mathcal{W}(\rho_c(\tau'))}(\langle\delta u_0,1\rangle\langle\mathfrak{e}_2(g),1\rangle,s)\,\psi_U(u_0)\,du_0\,dg.
\end{align}
(Cf. \eqref{int:Z with realization of coefficient}.) We repeat the arguments
\eqref{int:Z with realization of coefficient}--\eqref{int:Z with realization of coefficient 2}, i.e.,
multiply $g\mapsto y^{-1}g$, but use the fact that
$\sigma^{\diamondsuit}_{2rkc}(\mathfrak{e}_2(g),\mathfrak{e}_2(g))=\sigma_{c}^{\diamondsuit}(g,g')$
(see \eqref{eq:the $2$-cocycle on G times G formula GL}) instead of Proposition~\ref{proposition:the $2$-cocycle on G times G}, to obtain
\begin{align}\label{int:GL GL Z with realization of coefficient 2}
&\int\limits_{G}
\int\limits_{K_G}\{\varphi(e),\varphi^{\vee}(\langle g,1\rangle)\}_{M_R}
\\&\nonumber\int\limits_{U_0}f_{\mathcal{W}(\rho_c(\tau))\otimes \mathcal{W}(\rho_c(\tau'))}(\langle\delta u_0,1\rangle
\langle\mathfrak{e}_2(y^{-1}),\eta_c^{\diamondsuit}(y^{-1})\rangle
\langle\mathfrak{e}_2(g),1\rangle,s)\,\psi_U(u_0)\,du_0\,dy\,dg.
\end{align}
To proceed we need to show $(\eta_c^{\diamondsuit})^{-1}(y)=\eta_{2rkc}^{\diamondsuit}(\mathfrak{e}_1(y))$ (cf. \eqref{eq:classical eta_H 1 compatible}).
First we claim
\begin{align}\label{ea 2rkc in GL}
\eta_c^{\diamondsuit}(y)=\eta_{2rkc}^{\diamondsuit}(\mathfrak{e}_2(y)),\qquad\forall y\in K_G.
\end{align}
This follows as in
the proof of Proposition~\ref{proposition:gbl rho beta and rho are cohomologous}:
consider the embedding $\Sp_{2c}\hookrightarrow\Sp_{4rkc}$,
\begin{align*}
x^{\blacksquare,rk}=
\diag(I_{rkc},\left(\begin{smallmatrix}x_1&&x_2\\&I_{2rkc-2c}\\x_3&&x_4\end{smallmatrix}\right),I_{rkc}),\qquad
x=\left(\begin{smallmatrix}x_1&x_2\\x_3&x_4\end{smallmatrix}\right)
,\qquad x_i\in\Mat_c.
\end{align*}
By \cite[\S~2, Lemma~5]{BLS}, since $K_{\Sp_{2c}}$ is perfect,
$\eta_{2c}(x)=
\eta_{4rkc}(x^{\blacksquare,rk})$ for all $x\in K_{\Sp_{2c}}$.
This proves \eqref{ea 2rkc in GL}, because by definition $\eta_c^{\diamondsuit}(y)=\eta_{2c}(\diag(y,y^*))$ and
\begin{align*}
\eta_{2rkc}^{\diamondsuit}(\mathfrak{e}_2(y))=\eta_{2rkc}^{\diamondsuit}(\diag(I_{rkc},y,I_{rkc-c}))
=\eta_{4rkc}(\diag(y,y^*)^{\blacksquare,rk}).
\end{align*}

There are $2rk-1$ additional standard embeddings of $\Sp_{2c}$ in $\Sp_{4rkc}$, given by
\begin{align*}
x^{\blacksquare,l}=
\diag(I_{lc},\left(\begin{smallmatrix}x_1&&x_2\\&I_{4rkc-2(l+1)c}\\x_3&&x_4\end{smallmatrix}\right),I_{lc}),\qquad 0\leq l<2rk, \quad l\ne rk.
\end{align*}
By \cite[\S~2, Theorem~7]{BLS}, all $2rk$ images $\Sp_{2c}^{\blacksquare,l}$ ($0\leq l<2rk$) are commuting in $\Sp_{4rkc}^{(m)}$ and
\begin{align*}
\sigma_{4rkc}(\prod_{0\leq l<2rk}(x_l)^{\blacksquare,l},\prod_{0\leq l<2rk}(x'_l)^{\blacksquare,l})=\prod_{0\leq l<2rk}\sigma_{2c}(x_l,x'_l),\qquad x_l,x'_l\in \Sp_{2c}.
\end{align*}
Therefore $x^{\blacksquare,l}\mapsto \langle x^{\blacksquare,l},\eta_{2c}(x)\rangle$ is the canonical splitting of
$K_{\Sp_{2c}}^{\blacksquare,l}$ in $\Sp_{4rkc}^{(m)}$. Consider the diagonal embedding $x\mapsto x^{\times}$ of $K_{\Sp_{2c}}$ in
the direct product of $K_{\Sp_{2c}}^{\blacksquare,l}$ over $l\ne rk$, i.e.,
\begin{align*}
x\mapsto x^{\times}=\prod_{0\leq l<2rk,l\ne rk}x^{\blacksquare,l}.
\end{align*}
The canonical splitting of $K_{\Sp_{2c}}^{\times}=\{x^{\times}:x\in K_{\Sp_{2c}}\}$ in $\Sp_{4rkc}^{(m)}$ is
$x^{\times}\mapsto \langle x^{\times},\eta_{2c}^{2rk-1}(x)\rangle$.
Hence for any $x\in K_{\Sp_{2c}}$,
$\eta_{4rkc}(x^{\times})=\eta_{2c}^{2rk-1}(x)=\eta_{2c}^{-1}(x)$.
In particular for $y\in K_{G}$, since
\begin{align*}
\diag(\mathfrak{e}_1(y),\mathfrak{e}_1(y)^*)=\diag(y,y^*)^{\times}
\end{align*}
($\mathfrak{e}_1(y)=\diag(y,\ldots,y,I_c,y,\ldots,y)$ where $y$ appears $rk$ times on the left of $I_c$, and $rk-1$ times on the right),
\begin{align*}
(\eta_{c}^{\diamondsuit})^{-1}(y)=\eta_{2c}^{-1}(\diag(y,y^*))=
\eta_{4rkc}(\diag(y,y^*)^{\times})=\eta_{4rkc}(\diag(\mathfrak{e}_1(y),\mathfrak{e}_1(y)^*))=
\eta_{2rkc}^{\diamondsuit}(\mathfrak{e}_1(y)).
\end{align*}

Now using the right-invariance of $f_{\mathcal{W}(\rho_c(\tau)\otimes \mathcal{W}(\rho_c(\tau'))}$ on
\begin{align*}
\langle\mathfrak{e}_1(y^{-1}),\eta_{2rkc}^{\diamondsuit}(\mathfrak{e}_1(y^{-1}))\rangle=
\langle\mathfrak{e}_1(y^{-1}),(\eta_{c}^{\diamondsuit})^{-1}(y^{-1})\rangle
\end{align*}
and \eqref{eq:the $2$-cocycle on G times G formula GL},
the $du_0$-integral equals (cf. \eqref{eq:classical U_0 alone})
\begin{align*}
\int\limits_{U_0}f_{\mathcal{W}(\rho_c(\tau))\otimes \mathcal{W}(\rho_c(\tau'))}(\langle\delta u_0,1\rangle
\langle(y^{-1},y^{-1}),1\rangle\langle\mathfrak{e}_2(g),1\rangle,s)\,\psi_U(u_0)\,du_0.
\end{align*}
By Corollary~\ref{corollary:GL GL du integral of f invariant under g iota g} we can remove $\langle(y^{-1},y^{-1}),1\rangle$ from this integral and \eqref{int:GL GL Z with realization of coefficient 2} equals (cf. \eqref{int:Z with section})
\begin{align}\label{int:GL GL Z with realization of coefficient 3}
&\int\limits_{G}
\{\varphi(e),\varphi^{\vee}(\langle g,1\rangle)\}_{M_R}
\int\limits_{U_0}f_{\mathcal{W}(\rho_c(\tau)\otimes \mathcal{W}(\rho_c(\tau'))}(\langle\delta u_0,1\rangle
\langle\mathfrak{e}_2(g),1\rangle,s)\,\psi_U(u_0)\,du_0\,dg.
\end{align}

By Corollary~\ref{corollary:realization space for 0 < l < c} with $l=a=c-b$,
\begin{align*}
\rho_c(\tau)\subset\Ind_{\widetilde{P}_{(rka,rkb)}}^{\GL_{rkc}^{(m,r)}}((
\mathcal{W}(\rho_{a}(\tau))\otimes \mathcal{W}(\rho_{b}(\tau)))\delta_{P_{(rka,rkb)}}^{-1/(2rk)}),
\end{align*}
and similarly for $\rho_c(\tau')$.
Set $L=P_{(rka,rkb,rka,rkb)}$.
We realize $\mathcal{W}(\rho_c(\tau))$ and $\mathcal{W}(\rho_c(\tau'))$ using \eqref{eq:mnk functional using w_{n,m,k}}, and there is a standard section $f_{\mathcal{W}(\rho_a(\tau))\otimes \mathcal{W}(\rho_b(\tau))
\otimes\mathcal{W}(\rho_a(\tau'))\otimes\mathcal{W}(\rho_b(\tau'))}$ on
\begin{align}\label{rep:GL GL induced f before M(s) GL}
\Ind_{\widetilde{L}}^{H^{(m,r)}}(&
|\det|^{-b/2+\alpha s}\mathcal{W}(\rho_a(\tau))\otimes |\det|^{a/2+\alpha s}\mathcal{W}(\rho_b(\tau))\\&
\otimes |\det|^{-b/2-\alpha s}\mathcal{W}(\rho_a(\tau'))\otimes |\det|^{a/2-\alpha s}\mathcal{W}(\rho_b(\tau')))\nonumber
\end{align}
such that for all $h\in H^{(m,r)}$,
\begin{align*}
&f_{\mathcal{W}(\rho_c(\tau))\otimes\mathcal{W}(\rho_c(\tau'))}(h,s)\\&=\int\limits_V\int\limits_V
f_{\mathcal{W}(\rho_a(\tau))\otimes \mathcal{W}(\rho_b(\tau))
\otimes\mathcal{W}(\rho_a(\tau'))\otimes\mathcal{W}(\rho_b(\tau'))}(\langle\diag(\kappa,\kappa)\diag(v,v'),1\rangle h,s)\,dv\,dv'.
\end{align*}
(Cf. \eqref{rep:induced f before M(s)}, \eqref{eq:new formula for f with realization}.) Put $\kappa^{\bullet}=\diag(\kappa,\kappa)$. Note that we implicitly computed
\begin{align*}
\langle\diag(\kappa,I_{rkc}),1\rangle\langle\diag(v,I_{rkc}),1\rangle
\langle\diag(I_{rkc},\kappa),1\rangle\langle\diag(I_{rkc},v'),1\rangle=
\langle\kappa^{\bullet}\diag(v,v'),1\rangle,
\end{align*}
which is valid by \eqref{eq:block compatibility on Levi of P} and
\eqref{eq:sigma on h and v} (recall \eqref{eq:sigma on h and v}--\eqref{eq:sigma conjugate v by h} are valid for $\sigma^{\diamondsuit}_{d}$).

Since $f_{\mathcal{W}(\rho_c(\tau))\otimes\mathcal{W}(\rho_c(\tau'))}$ is normalized and unramified, and by the
proof of Proposition~\ref{proposition: rk c functional is nontrivial on nontwisted}, the section
$f_{\ldots}=f_{\mathcal{W}(\rho_a(\tau))\otimes \mathcal{W}(\rho_b(\tau))
\otimes\mathcal{W}(\rho_a(\tau'))\otimes\mathcal{W}(\rho_b(\tau'))}$ is also normalized and unramified.

Integral~\eqref{int:GL GL Z with realization of coefficient 3} becomes (cf. \eqref{int:Z with section 2})
\begin{align}\label{int:GL GL Z with realization of coefficient 4}
&\int\limits_{G}
\{\varphi(e),\varphi^{\vee}(\langle g,1\rangle)\}_{M_R}
\\&\int\limits_{U_0}\int\limits_{V}\int\limits_{V}f_{\ldots}(\langle\kappa^{\bullet}\diag(v,v'),1\rangle
\langle\delta u_0,1\rangle\langle\mathfrak{e}_2(g),1\rangle,s)\psi_U(u_0)\,dv\,dv'\,du_0\,dg,\nonumber
\end{align}
which is absolutely convergent for $\Real(s)\gg0$ as a quadruple integral.

Also since $\sigma_{2rkc}^{\diamondsuit}$ is trivial on $\mathfrak{W}^+_{2rkc}$ ($\sigma_{4rkc}$ is trivial on $\mathfrak{W}^+_{4rkc}$)
and by \eqref{eq:sigma on h and v},
\begin{align*}
&\langle\delta u_0,1\rangle=\langle\delta_0,1\rangle\langle\delta_1,1\rangle\langle u_0,1\rangle.
\end{align*}

Properties \eqref{it:GL observe 1}--\eqref{it:GL observe 6} from the proof of Lemma~\ref{lemma:reduction from classical to GLn} now take the following form:
\begin{enumerate}[leftmargin=*]
\item\label{it:GL observe 1 in GL} By \eqref{eq:sigma conjugate v by h}, ${}^{\delta_0^{-1}}\langle\diag(v,v'),1\rangle=\langle\diag(v',v),1\rangle$.
\item\label{it:GL observe 2 in GL} By \eqref{eq:sigma on h and v} and \eqref{eq:sigma conjugate v by h}, ${}^{\diag(v',v)}\langle\delta_1,1\rangle=
\langle\delta_1,1\rangle\langle u',1\rangle$ where $u'\in U_0$ and $\psi_U(u')=1$.
\item\label{it:GL observe 3 in GL} The elements of both copies of $V$ normalize $U_0$ and fix $\psi_U|_{U_0}$.
\item\label{it:GL observe 4 in GL} Since $\mathfrak{e}_2(G)<\diag(I_{rkc},\GL_{rkc})$, by \eqref{eq:block compatibility on Levi of P} the subgroups $\diag(V,I_{rkc})$ and $\mathfrak{e}_2(G)$ commute in $H^{(m,r)}$.
\item\label{it:GL observe 5 in GL} $\langle\delta_0,1\rangle$ commutes with $\langle\kappa^{\bullet},1\rangle$.
\item\label{it:GL observe 51 in GL} Since $\delta_1,{}^{\kappa^{\bullet}}\delta_1\in N_{\GL_{2rkc}}$, ${}^{\kappa^{\bullet}}\langle\delta_1,1\rangle=\langle{}^{\kappa^{\bullet}}\delta_1,1\rangle$.
\item\label{it:GL observe 6 in GL} The element $\langle \diag(\kappa,I_{rkc}),1\rangle$ commutes with $\langle\mathfrak{e}_2(G),1\rangle$, again by \eqref{eq:block compatibility on Levi of P}.
\end{enumerate}
Define
\begin{align*}
U_0'={}^{\kappa^{\bullet}}U_0=\left\{\left(\begin{smallmatrix}I_{rka}&&U_1&U_2\\&I_{rkb}&U_3&U_4\\&&I_{rka}\\&&&I_{rkb}\end{smallmatrix}\right)\right\}.
\end{align*}
As opposed to \eqref{eq:U_0' classical}, here $U^4$ does not depend on $U_1$. In fact
$\left\{\left(\begin{smallmatrix}I_{rka}&U_1\\&I_{rka}\end{smallmatrix}\right)\right\}$ and $\left\{\left(\begin{smallmatrix}I_{rkb}&U_4\\&I_{rkb}\end{smallmatrix}\right)\right\}$
are the unipotent subgroups corresponding to the $\GL_a^{(m,r)}\times \GL_k^{(m,r)}$ and $\GL_b^{(m,r)}\times \GL_k^{(m,r)}$ integrals, and
restriction of $\psi_U$ to the coordinates of $U_1$ and $U_4$ gives the similar character for these integrals.
In addition, the bottom left $a\times b$ (resp., $b\times a$) block of $U_2$ (resp., $U_3$) is $0$.

Utilizing properties \eqref{it:GL observe 1 in GL}--\eqref{it:GL observe 6 in GL},
\eqref{int:GL GL Z with realization of coefficient 4} equals (cf. \eqref{int:classical after props})
\begin{align}\label{int:GL after 7 props}
&\int\limits_{G}
\{\varphi(e),\varphi^{\vee}(\langle g,1\rangle)\}_{M_R}\int\limits_{U_0'}
\int\limits_{V}\int\limits_{V}\\&f_{\ldots}(\langle\delta_0({}^{\kappa^{\bullet}}\delta_1) ,1\rangle
\langle u_0',1\rangle
\langle\diag(I_{rkc},\kappa v),1\rangle
\langle\mathfrak{e}_2(g),1\rangle\langle\diag(\kappa v',I_{rkc}),1\rangle,s)\psi_U(u_0')\,dv\,dv'\,du_0'\,dg.\notag
\end{align}
We can factor \eqref{int:GL after 7 props} through $U_R$ and it becomes
\begin{align}\label{int:GL after 7 props 2}
&\int\limits_{U_R\backslash G}
\int\limits_{U_R}
\{\varphi(e),\varphi^{\vee}(\langle z,1\rangle\langle g,1\rangle)\}_{M_R}\int\limits_{U_0'}
\int\limits_{V}\int\limits_{V}\\&f_{\ldots}(\langle\delta_0({}^{\kappa^{\bullet}}\delta_1) ,1\rangle
\langle u_0',1\rangle
\langle\diag(I_{rkc},\kappa v),1\rangle
\langle\mathfrak{e}_2(z),1\rangle
\langle\mathfrak{e}_2(g),1\rangle\langle\diag(\kappa v',I_{rkc}),1\rangle,s)\notag\\&\psi_U(u_0')\,dv\,dv'\,du_0'\,dz\,dg.\notag
\end{align}
Form the group $U_0^{\bullet}$ by letting the $0$ block of $U_2$ in $U_0'$ take arbitrary coordinates. For $z\in U_R$,
\begin{align*}
&{}^{\mathfrak{e}_2(z)^{-1}}\langle\diag(I_{rkc},v),1\rangle=
\langle\diag(I_{rkc},v_z),1\rangle\langle\diag(I_{rkc},v),1\rangle,\qquad v_z\in V_{(a,rkc-a)},\\
&({}^{\diag(I_{rkc},\kappa)}\mathfrak{e}_2(z))^{-1}\,\langle({}^{\kappa^{\bullet}}\delta_1)u_0',1\rangle\,
({}^{\diag(I_{rkc},\kappa)}\mathfrak{e}_2(z))
=\langle({}^{\kappa^{\bullet}}\delta_1)u_z^{\bullet},1\rangle,\\
&{}^{\delta_0}\langle{}^{\diag(I_{rkc},\kappa)}\mathfrak{e}_2(z),1\rangle
=\langle {}^{\delta_0\diag(I_{rkc},\kappa)}\mathfrak{e}_2(z),1\rangle,
\end{align*}
where we only need to use \eqref{eq:sigma on h and v} and \eqref{eq:sigma conjugate v by h} throughout
($\mathfrak{e}_2(z),{}^{\diag(I_{rkc},\kappa)}\mathfrak{e}_2(z)\in N_{\GL_{2rkc}}$), and $u_z^{\bullet}\in U_0^{\bullet}$ depends on $u_0$ and $z$.
Using the equivariance properties of the top left $(rk,a)$ functional in the inducing data of $f_{\ldots}$ we see that
$\langle {}^{\delta_0\diag(I_{rkc},\kappa)}\mathfrak{e}_2(z),1\rangle$ vanishes, without a character.
Regarding $v_z$,
\begin{align*}
{}^{\diag(I_{rkc},\kappa)}\diag(I_{rkc},v_z)\in\diag(I_{rkc},V_{(a,rka-a)},I_{rkb})
\end{align*}
and
\begin{align*}
({}^{\diag(I_{rkc},\kappa)}\diag(I_{rkc},v_z))^{-1}\,({}^{\kappa^{\bullet}}\delta_1)u_0^{\bullet}\,({}^{\diag(I_{rkc},\kappa)}\diag(I_{rkc},v_z))
=({}^{\kappa^{\bullet}}\delta_1)u_{v_z}^{\bullet},
\end{align*}
where $u_{v_z}^{\bullet}\in U_0^{\bullet}$. When we change variables to remove the dependence on $v_z$, there is a character emitted from $\psi_U(u_0')$. Then
\begin{align*}
{}^{\delta_0\diag(I_{rkc},\kappa)}\diag(I_{rkc},v_z)\in \diag(V_{(a,rka-a)},I_{rkb},I_{rkc}),
\end{align*}
which belongs to the subgroup $V_{(a^{rk})}$ appearing in the definition of the $(rk,a)$ functional.
Again using the equivariance properties of the top left $(rk,a)$ functional in the inducing data of $f_{\ldots}$, the character (from $\psi_U(u_0')$) is cancelled.  This argument extends to the covering by \eqref{eq:sigma conjugate v by h}.

Altogether \eqref{int:GL after 7 props 2} is equal to (cf. \eqref{int:Z with section 4})
\begin{align}\label{int:GL after 7 props 3}
&\int\limits_{U_R\backslash G}
\{\varphi(e),\varphi^{\vee}(\langle g,1\rangle)\}_{M_R}\int\limits_{U_0^{\bullet}}
\int\limits_{V}\int\limits_{V}\\&f_{\ldots}(\langle\delta_0({}^{\kappa^{\bullet}}\delta_1) ,1\rangle
\langle u_0^{\bullet},1\rangle
\langle\diag(I_{rkc},\kappa v),1\rangle
\langle\mathfrak{e}_2(g),1\rangle\langle\diag(\kappa v',I_{rkc}),1\rangle,s)\psi_U(u_0^{\bullet})\,dv\,dv'\,du_0^{\bullet}\,dg.\notag
\end{align}

Next for any $l>0$ we let $\delta_{l,0}$ and $\delta_{l,1}$ be the elements $\delta_{0}$ and $\delta_{1}$ corresponding to the
$\GL_l^{(m,r)}\times\GL_k^{(m,r)}$ integral, and denote $\delta'_l=\delta_{l,0}\delta_{l,1}$. Then
\begin{align*}
&\delta_0=w\diag(\delta_{a,0},\delta_{b,0})w^{-1},\qquad
w=\left(\begin{array}{cccc}I_{rka}\\&&I_{rkb}\\&I_{rka}\\&&&I_{rkb}\end{array}\right).
\end{align*}
Then (cf. \eqref{eq:classical decomp delta_0 1})
\begin{align*}
\delta_0({}^{\kappa^{\bullet}}\delta_1)u_0^{\bullet}=w\cdot{}^{(\diag(\delta_{a,0},\delta_{b,0})w^{-1})}u^2\cdot
\diag(\delta'_{a},\delta'_{b})\cdot {}^{w^{-1}}(u^1u^4)\cdot w^{-1}u^3.
\end{align*}
Since $u_0^{\bullet},{}^{(\diag(\delta_{a,0},\delta_{b,0})w^{-1})}u^2,{}^{w^{-1}}(u^1u^4),{}^{\kappa^{\bullet}}\delta_1\in N_{\GL_{2rkc}}$,
we can apply \eqref{eq:sigma on h and v}--\eqref{eq:sigma conjugate v by h} to write a similar identity in $H^{(m,r)}$ (cf., \eqref{eq:decomp of delta in covering}):
\begin{align}\label{eq:GL decomp of delta in covering}
\langle \delta_0({}^{\kappa^{\bullet}}\delta_1),1\rangle
\langle u_0^{\bullet},1\rangle=
\langle w\,{}^{(\diag(\delta_{0,a},\delta_{0,b})w^{-1})}u^2,1\rangle\langle \diag(\delta'_{a},\delta'_{b})\, {}^{w^{-1}}(u^1u^4),1\rangle\langle w^{-1}u^3,1\rangle.
\end{align}
Let $U^2$ be the subgroup of elements ${}^{(\diag(\delta_{a,0},\delta_{b,0})w^{-1})}u^2$; a quick computation shows
\begin{align*}
U^2=\diag(I_{rka},V_{(rka,rkb)},I_{rkb}).
\end{align*}
Also $U^1$ (resp., $U^4$) denotes the subgroup of elements ${}^{w^{-1}}u^1$ (resp., ${}^{w^{-1}}u^4$) and let $U^3$ be the subgroup of elements $u^3$.
These subgroups will play the same role as in the proof of Lemma~\ref{lemma:reduction from classical to GLn}: the integration over $U^2$ constitutes an intertwining operator, $U^1$ and $U^2$ appear in the $\GL_a^{(m,r)}\times\GL_k^{(m,r)}$ and $\GL_b^{(m,r)}\times\GL_k^{(m,r)}$ integrals, and the integral over $U^3$ evaluates to a constant.

Let $M(s)$ be the standard intertwining operator from the space of \eqref{rep:GL GL induced f before M(s) GL} to the space of
\begin{align}\label{eq:GL image of M(s)}
\Ind_{\widetilde{P}_{(rka,rka,rkb,rkb)}}^{H^{(m,r)}}(&
|\det|^{-b/2+\alpha s}\mathcal{W}(\rho_a(\tau))\otimes |\det|^{-b/2-\alpha s}
\mathcal{W}(\rho_a(\tau'))\\&\notag
\otimes |\det|^{a/2+\alpha s}\mathcal{W}(\rho_b(\tau))
\otimes |\det|^{a/2-\alpha s}\mathcal{W}(\rho_b(\tau')))\nonumber
\end{align}
(cf. \eqref{eq:image of M(s)}), defined by
\begin{align*}
M(s)f_{\cdots}(h,s)=
\int\limits_{U^2}f_{\cdots}(\langle w,1\rangle\langle u^2,1\rangle h,s)\,du^2.
\end{align*}
To see that the image of $M(s)$ is indeed in \eqref{eq:GL image of M(s)} we must verify
\begin{align*}
{}^{w}\langle\diag(I_{rka},x,y,I_{rkb}),\epsilon\rangle=\langle\diag(I_{rka},y,x,I_{rkb}),\epsilon\rangle,\qquad\forall \diag(x,y)\in M_{(rka,rkb)}.
\end{align*}
We can consider $x$ and $y$ separately, and for each, argue using the Bruhat decomposition
(see the proof of \eqref{eq:conj of a using delta} in \S~\ref{Outline of the computation}). Since the conjugation by $w$ preserves
upper (resp., lower) unipotent elements, and $w$ commutes with the elements of $\mathfrak{W}^+_{rkc}$, it remains to check torus elements,
where we apply Proposition~\ref{proposition:action of W on torus is trivial on Sp}.

Put $d_{\tau,\tau',\vartheta,a,b}(s)=M(s)f_{\cdots}(\langle I_{2rkc},1\rangle,s)$.
When we apply \eqref{eq:GL decomp of delta in covering} to \eqref{int:GL after 7 props 3} we obtain (cf. \eqref{int:Z with section 5})
\begin{align}\label{int:GL after 7 props 4}
&\int\limits_{U_R\backslash G}
\{\varphi(e),\varphi^{\vee}(\langle g,1\rangle)\}_{M_R}
\int\limits_{U^3}\int\limits_{U^4}\int\limits_{U^1}
\int\limits_{V}\int\limits_{V}
\\&M(s)f_{\ldots}(
\langle \diag(\delta'_{a},\delta'_{b})u^1u^4,1\rangle\langle w^{-1}u^3,1\rangle
\langle\diag(I_{rkc},\kappa v),1\rangle
\langle\mathfrak{e}_2(g),1\rangle\langle\diag(\kappa v',I_{rkc}),1\rangle,s)\nonumber\\&\psi_{U}(u^1)\psi_{U}(u^4)
\,dv\,dv'\,du^1\,du^4\,du^3\,dg.\notag
\end{align}
Here $\psi_{U}(u^1)$ (resp., $\psi_{U}(u^4)$) is the character defined for the subgroup $U$ appearing in the $\GL_a^{(m,r)}\times\GL_k^{(m,r)}$
(resp., $\GL_b^{(m,r)}\times\GL_k^{(m,r)}$) integral, evaluated on an element in the subgroup $U_0$ of that integral.

Next we write the $dg$-integral using the Iwasawa decomposition $G=RK_G$.
As in the passage \eqref{int:Z with section 5}--\eqref{int:Z with section 6} but using \eqref{ea 2rkc in GL}, we see
that the integrand is invariant on the right with respect to the map $g\mapsto gy$, where $y\in K_G$. Then \eqref{int:GL after 7 props 4} equals
\begin{align*}
&\int\limits_{\GL_b}\int\limits_{\GL_a}
\{\varphi(e),\varphi^{\vee}(\langle\diag(x,y),1\rangle)\}_{M_R}
\int\limits_{U^3}\int\limits_{U^4}\int\limits_{U^1}
\int\limits_{V}\int\limits_{V}
\\&M(s)f_{\ldots}(
\langle \diag(\delta'_{a},\delta'_{b})u^1u^4,1\rangle\langle w^{-1}u^3,1\rangle
\langle\diag(I_{rkc},\kappa v),1\rangle
\langle\mathfrak{e}_2(\diag(x,y)),1\rangle\langle\diag(\kappa v',I_{rkc}),1\rangle,s)\nonumber\\&\psi_U(u^1)\psi_U(u^4)
\delta_R^{-1}(\diag(x,y))\,dv\,dv'\,du^1\,du^4\,du^3\,dx\,dy.\notag
\end{align*}

The conjugation of $\diag(I_{rkc},V)$ by $\mathfrak{e}_2(\diag(x,y))$ multiplies $dv$ by $|\det y|^{(rk-1)a}$;
conjugating $U^3$ by ${}^{\diag(I_{rkc},\kappa)}\mathfrak{e}_2(\diag(x,y))$ multiplies $du^3$ by $|\det x|^{(1-rk)b}$;
and
\begin{align*}
\langle\diag(x,y),\epsilon\rangle&\mapsto \delta_R^{-1/2}(\diag(x,y))\{\varphi(e),\varphi^{\vee}(\langle \diag(x,y),\epsilon\rangle)\}_{M_R}
\\&=\epsilon|\det x|^{-b/2}|\det y|^{a/2}\omega_a(\langle x,1\rangle)\omega_b(\langle y,1\rangle).
\end{align*}
Therefore when we shift $\mathfrak{e}_2(\diag(x,y))$ to the left the integral becomes
\begin{align}\label{int:GL after 7 props 6}
&\int\limits_{V}\int\limits_{V}\int\limits_{U^3}
\int\limits_{\GL_b}\int\limits_{\GL_a}
\int\limits_{U^4}\int\limits_{U^1}
\omega_a(\langle x,1\rangle)\omega_b(\langle y,1\rangle
|\det{x}|^{b/2-rkb}|\det{y}|^{rka-a/2}\\&\notag M(s)f_{\ldots}(
\langle \diag(\delta'_{a},\delta'_{b})u^1u^4,1\rangle
\,{}^{w^{-1}\diag(I_{rkc},\kappa)}\langle\mathfrak{e}_2(\diag(x,y)),1\rangle
\\&\notag\langle w^{-1}u^3,1\rangle
\langle\diag(I_{rkc},\kappa v),1\rangle
\langle\diag(\kappa v',I_{rkc}),1\rangle,s)\nonumber\\&\psi_{U}(u^1)\psi_{U}(u^4)
\,du^1\,du^4\,dx\,dy\,du^3\,dv\,dv'.\notag
\end{align}
(Cf. \eqref{int:Z with section 7}.)

Denote the embedding $\GL_l\times\GL_l\hookrightarrow \GL_{2rkl}$ in the construction of the
$\GL_l^{(m,r)}\times\GL_{k}^{(m,r)}$ integral by $(,)_l'$, and identify
$\GL_{2rka}\times \GL_{2rkb}$ with $M_{(2rka,2rkb)}$. Then $(x_1,x_2)_a'$ belongs to the top left $\GL_{2rka}$ block
of $M_{(2rka,2rkb)}$, and $(y_1,y_2)_b'$ to the bottom right $\GL_{2rkb}$ block. We have
\begin{align*}
{}^{w^{-1}\diag(I_{rkc},\kappa)}\mathfrak{e}_2(\diag(x,y))=\diag(I_{rka},x,I_{(rk-1)a},I_{kb},y,I_{(rk-1)b})=(I_a,x)_a'(I_b,y)_b'.
\end{align*}
We claim the analog of \eqref{eq:conj of a using delta}, which we state and prove only under the assumption $\mu_{2m}\subset F^*$:
\begin{align}\label{eq:GL conj of x y using delta}
{}^{w^{-1}\diag(I_{rkc},\kappa)}\langle\mathfrak{e}_2(\diag(x,y)),1\rangle=\langle (I_a,x)_a'(I_b,y)_b',1\rangle.
\end{align}
To this end first observe that by \eqref{eq:block compatibility on Levi of P},
\begin{align*}
\langle\mathfrak{e}_2(\diag(x,y)),1\rangle
=\langle\mathfrak{e}_2(\diag(x,I_b)),1\rangle\langle\mathfrak{e}_2(\diag(I_a,y)),1\rangle.
\end{align*}
Therefore we can consider $x$ and $y$ separately.
Consider $x$ first.
Since
\begin{align*}
\diag(I_{rkc},\kappa)\in\diag(I_{rkc+a},\GL_{rkc-a})\qquad \mathrm{and}\qquad \mathfrak{e}_2(\diag(x,I_b))\in \diag(I_{rkc},\GL_a,I_{rkc-a}),
\end{align*}
these elements commute in $H^{(m,r)}$ by \eqref{eq:block compatibility on Levi of P}. As in the proof
of \eqref{eq:conj of a using delta} in \S~\ref{Outline of the computation},
we can argue either by considering a dense subset $N_{\GL_a}^-B_{\GL_a}$ or the Bruhat decomposition, where the latter approach is not difficult granted $\mu_{2m}\subset F^{*m}$. The arguments are similar and omitted. Either way, it remains to consider $x\in T_{\GL_a}$, and we deduce \eqref{eq:GL conj of x y using delta} from Proposition~\ref{proposition:action of W on torus is trivial on Sp}.

Regarding $y$, we can argue similarly (e.g., using $N_{\GL_b}^-B_{\GL_b}$). Here we have to consider
$\diag(I_{rkc},\kappa)$, but because
\begin{align*}
{}^{\diag(I_{rkc},\kappa)}\mathfrak{e}_2(\diag(I_a,y))\in\diag(I_{rk(a+c)},\GL_b,I_{(rk-1)b})\qquad \mathrm{and}\qquad w^{-1}\in\diag(I_{rka},\GL_{rkc},I_{rkb}),
\end{align*}
$w^{-1}$ commutes ${}^{\diag(I_{rkc},\kappa)}\mathfrak{e}_2(\diag(I_a,y))$. Note that for $u_y^-\in N_{\GL_b}^-$ and $v_y\in N_{\GL_b}$, we have
${}^{\diag(I_{rkc},\kappa)}\mathfrak{e}_2(u_y^-)\in N_{\GL_{2rkc}}^-$ and
${}^{\diag(I_{rkc},\kappa)}\mathfrak{e}_2(v_y)\in N_{\GL_{2rkc}}$. Finally for $t\in T_{\GL_b}$ we use
Proposition~\ref{proposition:action of W on torus is trivial on Sp}. This completes the proof of \eqref{eq:GL conj of x y using delta}.

Now by \eqref{eq:GL conj of x y using delta} and \eqref{eq:block compatibility on Levi of P}, \eqref{int:GL after 7 props 6} equals
(cf. \eqref{int:Z with section 8})
\begin{align}\label{int:GL after 7 props 6 1}
&\int\limits_{V}\int\limits_{V}\int\limits_{U^3}
\int\limits_{\GL_b}
\int\limits_{U^4}
\int\limits_{\GL_a}
\int\limits_{U^1}
\omega_a(\langle x,1\rangle)\omega_b(\langle y,1\rangle
|\det{x}|^{b/2-rkb}|\det{y}|^{rka-a/2}\\&\notag M(s)f_{\ldots}(
\langle \delta'_au^1,1\rangle\langle (I_a,x)_a',1\rangle
\langle \delta'_bu^4,1\rangle\langle (I_b,y)_b',1\rangle
\langle w^{-1}u^3,1\rangle
\langle k^{\bullet}\diag(v',v),1\rangle,s)\nonumber\\&\psi_{U}(u^1)\psi_{U}(u^4)
\,du^1\,dx\,du^4\,dy\,du^3\,dv\,dv'.\notag
\end{align}
The covering $\widetilde{M}_{(2rka,2rkb)}$ is isomorphic to the quotient of the direct product $\GL_{2rka}^{(m,r)}\times \GL_{2rkb}^{(m,r)}$ by $\{(\epsilon_1,\epsilon_2)\in\mu_m^2:\epsilon_1\epsilon_2=1\}$.
According to \eqref{eq:GL image of M(s)} the restriction of $M(s)f_{\ldots}$ to $\widetilde{M}_{(2rka,2rkb)}$ is a rational section of (see cf. \eqref{eq:classical image restrict Siegel})
\begin{align*}
&|\det|^{rkb-b/2}\Ind_{\widetilde{P}_{(rka,rka)}}^{\GL_{2rka}^{(m,r)}}
((\mathcal{W}(\rho_a(\tau))\otimes\mathcal{W}(\rho_a(\tau')))\delta_{P_{(rka,rka)}}^{\alpha s/(rka)})
\\&\otimes|\det|^{-rka+a/2}\Ind_{\widetilde{P}_{(rkb,rkb)}}^{\GL_{2rkb}^{(m,r)}}
((\mathcal{W}(\rho_b(\tau))\otimes\mathcal{W}(\rho_b(\tau')))\delta_{P_{(rkb,rkb)}}^{\alpha s/(rkb)}).
\end{align*}
The inner $du^1dx$ and $du^4dy$ integrals are the
$\GL_a^{(m,r)}\times\GL_{k}^{(m,r)}$ and $\GL_b^{(m,r)}\times\GL_{k}^{(m,r)}$ integrals, respectively.
The $du^3$-integral is first seen to vanish unless $v,v'\in K_G$, then since the section is unramified, the $dvdv'$-integral
equals the constant $1$; then the $du^3$-integral vanishes outside $U^3\cap K_H$, whence the whole outer integral can be separated from the integral $du^1dxdu^4dy$, and it equals $1$. Finally \eqref{int:GL after 7 props 6 1} becomes the product
\begin{align}\label{int:GL after 7 props 8}
d_{\tau,\tau',\vartheta,a,b}(s)Z(\alpha s/(rka),\omega_a,f_{\mathcal{W}(\rho_a(\tau))\otimes \mathcal{W}(\rho_a(\tau'))})
Z(\alpha s/(rkb),\omega_b,f_{\mathcal{W}(\rho_b(\tau))\otimes \mathcal{W}(\rho_b(\tau'))}),
\end{align}
where $f_{\mathcal{W}(\rho_a(\tau))\otimes \mathcal{W}(\rho_a(\tau'))}$ and
$f_{\mathcal{W}(\rho_b(\tau))\otimes \mathcal{W}(\rho_b(\tau'))}$ are the normalized unramified sections appearing in the statement of the lemma. The proof is complete up to the computation of $d_{\tau,\tau',\vartheta,a,b}(s)$.

Put $H'=\diag(I_{rka},\GL_{rkc},I_{rkb})$. The restriction of $f_{\ldots}$ (which is a section of \eqref{rep:GL GL induced f before M(s) GL}) to $\widetilde{H'}=\GL_{rkc}^{(m,r)}$ (see \eqref{eq:block compatibility on Levi of P}) is the normalized unramified section of
the unramified representation of $\GL_{rkc}^{(m,r)}$ induced from $\widetilde{P}_{(rkb,rka)}$ and
\begin{align*}
\mathrm{I}_{\GL_{rkb}^{(m,r)}}(\vartheta,\otimes_{1\leq i\leq k,1\leq j\leq rb}\chi_i|~|^{\alpha s+(a-b)/2+j/r-1/(2r)})
\otimes
\mathrm{I}_{\GL_{rka}^{(m,r)}}(\vartheta,\otimes_{1\leq i\leq k,1\leq j\leq ra}\chi_i'|~|^{-\alpha s-c/2+j/r-1/(2r)}).
\end{align*}
The adjoint action of $\GL_{rkb}(\C)\times\GL_{rka}(\C)$ on the Lie algebra of the $L$-group of $U^2$ is $\text{st}\otimes\widetilde{\text{st}}$. As in the proof of Lemma~\ref{lemma:reduction from classical to GLn}, the value of $d_{\tau,\tau',\vartheta,a,b}(s)$ can be computed using \cite[Theorem~7.10]{Gao2018}, and we obtain \eqref{eq:GL d tau tau}.
\end{proof}

\subsection{The $\GL_1^{(m,r)}\times \GL_k^{(m,r)}$ integral}\label{final reduction n = 1 linear groups}
In this section we compute the $\GL_1^{(m,r)}\times\GL_k^{(m,r)}$ integral with unramified data.
We proceed with the set-up and notation of \S~\ref{proof of lemma:reduction from GLn to GLa GLb}, now
with $c=n=1$. In particular $G=\GL_1$, $H=\GL_{2rk}$, $P=P_{(rk,rk)}$, $U_P=V_{(rk,rk)}$, $U=V_{(1^{(rk-1)},2,1^{(rk-1)})}$ and the character $\psi_U$ given by \eqref{eq:GL character of U} becomes
\begin{align*}
\psi_U(\left(\begin{smallmatrix}v&x&y\\&I_{2}&z\\&&v'\end{smallmatrix}\right))=\psi(-\sum_{i=1}^{rk-2}v_{i,i+1}-x_{rk-1,1}+z_{1,1}
-\sum_{i=1}^{rk-2}v'_{i,i+1}), \qquad v,v'\in N_{\GL_{rk-1}}.
\end{align*}
Also
$\tau\subset\mathrm{I}_{\GL_k^{(m,r)}}(\vartheta,\chi)$,
$\mathcal{W}(\rho_1(\tau))$ is the $(rk,1)$ model (Whittaker model) of $\rho_1(\tau)$ and similarly for $\tau'$,
e.g., $\tau'\subset\mathrm{I}_{\GL_k^{(m,r)}}(\vartheta,\chi')$. The representations $\rho_1(\tau)$ and $\rho_1(\tau')$ are assumed to satisfy \eqref{eq:invariance prop on GL or SL}.

Let $f=f_{\mathcal{W}(\rho_1(\tau))\otimes \mathcal{W}(\rho_1(\tau'))}$ be a standard section of
\begin{align}\label{eq:ind sections}
\mathrm{I}(\mathcal{W}(\rho_1(\tau)),\mathcal{W}(\rho_1(\tau')),s)=\Ind_{\widetilde{P}}^{H^{(m,r)}}((\mathcal{W}(\rho_1(\tau))\otimes \mathcal{W}(\rho_1(\tau')))\delta_P^s).
\end{align}
Since $c=1$, $\pi=\Ind_{\widetilde{A}}^{G^{(m,r)}}(\varepsilon\otimes\vartheta\mu)$ (see \S~\ref{unramified reps}), and $\omega$ is a matrix coefficient of $\pi^{\vee}$.

In this section we compute the integral $Z(s,\omega,f)$,
when $\omega$ and $f$ are normalized and unramified,
by relating it to another integral (\eqref{int:after functional equation to compare} below) which is computed using
the Rankin--Selberg integral of \S~\ref{RS integrals}.

First consider the $\GL_1^{(m,r)}\times \GL_k^{(m,r)}$ integral \eqref{eq:local GL GL integral}, now with arbitrary $\omega$ and standard section $f$. It is absolutely convergent, in a right half-plane independent of the data, and by Proposition~\ref{proposition:equiv props GL} can be regarded as a morphism of \eqref{eq:homspace G with W(E) GL} (with $c=1$), namely of
\begin{align}\label{eq:homspace G with W(rho1tau) GL}
\Hom_{G\times G}(J_{U,\psi_U^{-1}}(\mathrm{I}(\mathcal{W}(\rho_1(\tau)),\mathcal{W}(\rho_1(\tau)),s)),\pi^{\vee}\otimes\pi).
\end{align}
\begin{lemma}\label{lemma:uniqueness}
For all but a finite set of values of $q^{-s}$, the space
\eqref{eq:homspace G with W(rho1tau) GL} is at most one dimensional.
\end{lemma}
\begin{proof}
The proof in the linear case was given in \cite[Lemma~35]{CFGK2}, and the situation here is similar.
It suffices to prove the statement for
\begin{align*}
\Hom_{\mathfrak{e}_2(G)}(J_{U,\psi_U^{-1}}(\mathrm{I}(\mathcal{W}(\rho_1(\tau)),\mathcal{W}(\rho_1(\tau)),s)),\pi).
\end{align*}
Note that both the Jacquet module and $\pi$ are genuine (otherwise this space would be identically zero).
By \cite[1.9]{BZ2}, this space is isomorphic to
\begin{align}\label{bil}
\Bil_{H}(\ind_{G^{(m,r)}U}^{H^{(m,r)}}{(\pi^{\vee}\otimes\psi_U)},
\mathrm{I}(\mathcal{W}(\rho_1(\tau)),\mathcal{W}(\rho_1(\tau')),s)),
\end{align}
where $\Bil_H(\cdots)$ is the space of $H$-equivariant bilinear forms and $\ind(\cdots)$ is the compact
induction. The double coset space $P\backslash H/ GU$ is finite. For $h,h'\in H$, set $h\sim h'$ if
$PhGU=Ph'GU$, and $h\not\sim h'$ otherwise. For each representative $h\in P\backslash H/ GU$, let
\begin{align}\label{eq:hom h}
\mathcal{H}(h)=\Hom_{(GU)^h}({}^h(\pi^{\vee}\otimes\psi_U)\otimes
\left( (\mathcal{W}(\rho_1(\tau))\otimes \mathcal{W}(\rho_1(\tau')))\delta_{P}^{s} \right),\theta).
\end{align}
Here $(GU)^h={}^h(GU)\cap P$; for a representation $\varrho$ of
$G^{(m,r)}U=G^{(m,r)}\ltimes U$, ${}^h\varrho$ is the
representation of $(G^{(m,r)}U)^h$ on the space of $\varrho$ given by
${}^h\varrho(x)=\varrho({}^{h^{-1}}x)$; and
$\theta$ is a certain non-genuine modulus character, which is independent of $s$. Being a modulus character, $\theta$ is trivial on unipotent elements. According to the Bruhat theory (see e.g., \cite[Theorems~1.9.4 and 1.9.5]{Silb}), \eqref{bil} is embedded in the
semi-simplification $\bigoplus_{h}\mathcal{H}(h)$.

The representatives $h$ can be taken in the form $h=w$ or $h=w\delta_1$, for a permutation matrix $w$. Set $\kappa=\diag(I_{rk-1},\left(\begin{smallmatrix}&1\\1\end{smallmatrix}\right),I_{rk-1})$.

First assume $rk>1$, then $U$ is nontrivial. The arguments
in \cite[Lemma~35]{CFGK2} show that if $w\not\sim\delta_0$ and $w\not\sim\delta_0\kappa$, $\psi_U|_{{}^{h^{-1}}V_{(rk,rk)}\cap U}\ne1$, and then
\eqref{eq:hom h} vanishes because we can take $u\in {}^{h^{-1}}V_{(rk,rk)}\cap U$ such that ${}^hu$ acts by $\psi_U(u)\ne1$
on the left and trivially on the right ($\theta({}^hu)=1$). To extend the
argument to $m>1$ one simply has to change $k$ to $rk$ throughout, and since both $u$ and ${}^hu\in N_{H}$,
${}^h\langle u,1\rangle=\langle {}^hu,1\rangle$ by \eqref{eq:sigma conjugate v by h}, so that the argument extends to the covering.

It remains to consider $\delta_0,\delta_0\kappa$ and $\delta$ ($\delta_0\kappa\delta_1\sim\delta_0\kappa$). These are also the only representatives to consider
when $rk=1$, so we proceed for any $rk\geq1$.
Regarding $h=\delta_0$, as in \textit{loc. cit.} we see that any $\mathcal{L}\in\Hom(\delta_0)$ factors through the Jacquet module
$J_{V_{(1,rk-1)}}(\mathcal{W}(\rho_1(\tau)))$ (change $k$ to $rk$ in \cite[(3.57)]{CFGK2}). Then when we change $k$ to $rk$ and repeat the computation from \textit{loc. cit.}, we see that for $x\in G=F^*$ and any
pure tensor $\xi\otimes\xi'$ in the space of $\mathcal{W}(\rho_1(\tau))\otimes\mathcal{W}(\rho_1(\tau'))$,
\begin{align*}
\mathcal{L}(\mathcal{W}(\rho_1(\tau))(\langle \diag(x,I_{rk-1}),1\rangle)\xi\otimes \xi')= \epsilon_x\pi(x)|x|^{-rks}\theta(x)\mathcal{L}(\xi\otimes\xi')
\end{align*}
(\cite[(3.58)]{CFGK2}). Here $\epsilon_x$ is defined by
\begin{align*}
{}^{\delta_0}\langle \diag(x,I_{2k-1}),1\rangle=\langle\mathfrak{e}_2(x),\epsilon_x\rangle.
\end{align*}
In fact $\epsilon_x=1$ for any $x\in F^*$ by Proposition~\ref{proposition:action of W on torus is trivial on Sp} ($\mu_{2m}\subset F^{*m}$),
but for the argument here this is not needed.
The genuine representation $J_{V_{(1,rk-1)}}(\mathcal{W}(\rho_1(\tau)))$ of $\GL_1^{(m,r)}\times\GL_{rk-1}^{(m,r)}$ is of finite length and
$\mathcal{L}$ must factor through one of its composition factors. The subgroup
\begin{align*}
\{\langle\diag(x,I_{rk-1}),\epsilon\rangle:x\in F^*,\epsilon\in\mu_m\}
\end{align*}
acts by an irreducible representation of $G^{(m,r)}$ on each of these factors, say $\Ind_{\widetilde{A}}^{G^{(m,r)}}(\varepsilon\otimes\vartheta\beta)$ ($\beta$ depends on the factor). Hence when we take $x\in F^{*m}$ (thereby $\epsilon_x=1$) we obtain
\begin{align*}
\mathcal{L}(\xi\otimes \xi')= \beta^{-1}(x)\pi(x)|x|^{-rks}\theta(x)\mathcal{L}(\xi\otimes\xi').
\end{align*}
Then on the one hand if $\mathcal{L}\ne0$, it is nonzero on some $\xi\otimes \xi'$, which may depend on $s$, but on the other hand
\begin{align*}
|x|^{rks}=\beta^{-1}(x)\pi(x)\theta(x),\qquad\forall x\in F^{*m},
\end{align*}
which can only hold for finitely many values of $q^{-s}$. Thus $\mathcal{L}=0$ and
\eqref{eq:hom h}  is zero for $h=\delta_0$ outside finitely many values of $q^{-s}$.
The case $h=\delta_0\kappa$ is handled similarly, now $\mathcal{L}$ must factor through $J_{V_{(rk-1,1)}}(\mathcal{W}(\rho_1(\tau')))$ (see
\cite[(3.59)]{CFGK2} and also \cite[Remark~38]{CFGK2} for $rk=1$).

Finally consider $h=\delta$, we need to show $\Hom(\delta)$ is at most one-dimensional. Assume $rk>1$. Then $\mathcal{L}$ is in particular a Whittaker functional on the tensor $\mathcal{W}(\rho_1(\tau))\otimes\mathcal{W}(\rho_1(\tau'))$, which is unique up to scaling because
$\mathcal{W}(\rho_1(\tau))\otimes\mathcal{W}(\rho_1(\tau'))$ is a quotient of $\rho_1(\tau)\otimes\rho_1(\tau')$
(in fact isomorphic to, the latter is irreducible), and both
$\rho_1(\tau)$ and $\rho_1(\tau')$ are $(rk,1)$ representations.
(Recall that the tensor here is the standard one, up to the quotient by $\{(\epsilon_1,\epsilon_2):\epsilon_1\epsilon_2=1\}$,
as opposed to the tensor for coverings of \cite{KP} studied in \cite{Kable,Mezo,Tk2}.)

If $rk=1$, \eqref{eq:hom h} is a priori at most one-dimensional (for all $h$ and $s$). This is because
now $\mathcal{W}(\rho_1(\tau))$ and $\mathcal{W}(\rho_1(\tau'))$ are both irreducible representations of $\GL_1^{(m,1)}$, where $m\leq2$, and as such are
one-dimensional. Indeed for $m=1$, $\mathcal{W}(\rho_1(\tau))$ is a one-dimensional representation of $F^*$, namely the quasi-character $\tau$.
When $m=2$, $\mathcal{W}(\rho_1(\tau))$ is an irreducible representation of $\GL_1^{(2,1)}$, but since in this case
$\sigma_1^{\diamondsuit}(a,b)=(a,b)_2$, $\mathcal{W}(\rho_1(\tau))$ is still one-dimensional and takes the form
$\vartheta\otimes\beta$ ($\vartheta=\gamma_{\psi'}$) where $\beta$ is a quasi-character of $F^*$, and
$\vartheta\otimes\beta(\langle x,1\rangle)=\vartheta(x)\beta(x)$ for all $x\in F^*$. (Of course for a general $m$, an irreducible representation of $\GL_1^{(m,r)}$ is finite dimensional and not one-dimensional, but then $r>1$.)
\end{proof}
\begin{corollary}\label{corollary:mero cont of GL1 integral}
For any rational section $f$, $Z(s,\omega,f)$ admits meromorphic continuation to a rational function in $q^{-s}$.
\end{corollary}
\begin{proof}
This follows from Bernstein's continuation principle (in \cite{Banks}), because Lemma~\ref{lemma:uniqueness} provides uniqueness and one can choose data for which the integral becomes a constant.
\end{proof}
We denote a general element of $V_{(rk,rk)}$ by
\begin{align*}
[\begin{smallmatrix}y&z\\u&x\end{smallmatrix}]=\left(\begin{smallmatrix}I_{rk-1}&&y&z\\&1&u&x\\&&1\\&&&I_{rk-1}\end{smallmatrix}\right).
\end{align*}
Then $\delta_1=[\begin{smallmatrix}0&0\\1&0\end{smallmatrix}]$, $U_0$ is the subgroup $\{[\begin{smallmatrix}y&z\\0&x\end{smallmatrix}]\}$
where $y,z$ and $x$ are arbitrary, and if $x=(x_1,\ldots,x_{rk-1})$, $\psi_U([\begin{smallmatrix}y&z\\0&x\end{smallmatrix}])=\psi(x_1)$.

We claim that
\eqref{eq:ind sections} admits a unique Whittaker model. This follows from
the Geometric Lemma \cite[Theorem~5.2]{BZ2}, because $\mathcal{W}(\rho_1(\tau))\otimes \mathcal{W}(\rho_1(\tau'))$ admits a unique Whittaker model (see \cite{Banks1998}).
In fact, consider the filtration of \eqref{eq:ind sections} according to the space $P\backslash H/N_H$.
The arguments in \cite[Lemma~35]{CFGK2} used above, namely that $\psi_U|_{{}^{h^{-1}}V_{(rk,rk)}\cap U}\ne1$ unless
$w\not\sim\delta_0,\delta_0\kappa$ (with the above notation), can be repeated
when $U$ is replaced by $N_{H}$ and $\psi_U$ is replaced by a generic character $\psi'$ of $N_{H}$. Then
$\psi'|_{{}^{h^{-1}}V_{(rk,rk)}\cap N_H}\ne1$ for any $w$ such that
$PwN_H\ne P\delta_0N_H$. This verifies condition $(\bigstar)$ of
\cite[\S~5]{BZ2} for all but one representative of $P\backslash H/N_H$, and it remains to consider the representative $\delta_0$. As in the proof above we reduce to the uniqueness of the Whittaker model for the tensor. Note that the Bruhat theory method (used in the proof above and in \cite[Lemma~35]{CFGK2}) will only produce
a bound, namely \eqref{eq:ind sections} carries at most one Whittaker model, the proof of
\cite[Theorem~5.2]{BZ2} considers the Jacquet module $J_{N_{H},\psi'}$ hence also implies existence (by exactness). Either way, existence follows immediately from the following Jacquet integral.

We have the standard Jacquet integral on the space of \eqref{eq:ind sections}, defined
for a holomorphic section $f$ by
\begin{align}\label{int:standard Jacquet on I W W}
\int\limits_{V_{(rk,rk)}} f(\langle\delta_0[\begin{smallmatrix}y&z\\u&x\end{smallmatrix}],1\rangle ,s)\psi(u)\,dx\,dy\,dz\,du.
\end{align}
This integral is absolutely convergent for $\Real(s)\gg0$, and one can choose $f$ such that
\eqref{int:standard Jacquet on I W W} is absolutely convergent and equals $1$, for all $s$.
Since there is (in particular) at most one Whittaker model on \eqref{eq:ind sections}, Bernstein's continuation principal (\cite{Banks}) implies that \eqref{int:standard Jacquet on I W W} admits analytic continuation to a polynomial function in $q^{-s}$ and $q^s$, i.e.,
\eqref{int:standard Jacquet on I W W} belongs to $\C[q^{-s},q^s]$. This continuation is not identically zero for any $s$.
Consequently it is a realization of the unique (up to scaling) Whittaker functional on \eqref{eq:ind sections},
with respect to the character
\begin{align}\label{eq:character for Ind GL 1}
\psi(\diag(d,d')[\begin{smallmatrix}y&z\\u&x\end{smallmatrix}])=\psi(\sum_{i=1}^{rk-1}d_{i,i+1}-u+\sum_{i=1}^{rk-1}d'_{i,i+1}),\qquad d,d'\in N_{\GL_{rk}}.
\end{align}
The corresponding Whittaker model of \eqref{eq:ind sections} is spanned by the functions
\begin{align}\label{int:standard Whittaker on I W W}
W_f(h,s)=\int\limits_{V_{(rk,rk)}} f(\langle\delta_0[\begin{smallmatrix}y&z\\u&x\end{smallmatrix}],1\rangle h,s)\psi(u)\,dx\,dy\,dz\,du.
\end{align}
Then by definition the Whittaker model of $\mathrm{I}(\mathcal{W}(\rho_1(\tau)),\mathcal{W}(\rho_1(\tau')),s)^*$ with respect to the inverse of
\eqref{eq:character for Ind GL 1} is spanned by functions
\begin{align}\label{int:standard Whittaker on I W W *}
W_f^*(h,s)=W_f({}^*h,s)=\int\limits_{V_{(rk,rk)}} f(\langle\delta_0[\begin{smallmatrix}y&z\\u&x\end{smallmatrix}],1\rangle\,{}^*h,s)\psi(u)\,dx\,dy\,dz\,du.
\end{align}

If $rk>1$, denote
\begin{align*}
[t,v]=\diag(I_{rk},\left(\begin{smallmatrix}1&&\\&I_{rk-2}&\\-t&v&1\end{smallmatrix}\right)),\qquad
w'=\left(\begin{smallmatrix} I_{rk} \\ & & I_{rk-1} \\ & 1 \end{smallmatrix}\right).
\end{align*}
For $rk=1$ we take $[t,v]=w'=I_2$.
Since $[t,v]\in N_{H}^-$, we fix the splitting $\langle v^-,\varsigma(v^-)\rangle$ of $N_{H}^-$.
Then $[t,v]\mapsto \langle[t,v],\varsigma([t,v])\rangle$ is the splitting of the subgroup of elements $[t,v]$, which is $\diag(I_{rk},V_{(rk-1,1)}^-)$. Also let $\zeta\in\C$.

For any matrix coefficient $\omega$ of $\pi^{\vee}$ and
a holomorphic section $f$ of \eqref{eq:ind sections}, consider the integral
\begin{align}\label{int:after functional equation to compare}
&\Psi(\zeta,s,\omega,f)\\&\nonumber=\int\limits_{F^*}\int\limits_{F^{rk-2}}\int\limits_{F} W_f(\langle\diag(I_{2rk-1},a),1\rangle\langle [t,v],\varsigma([t,v])\rangle
\langle w',1\rangle,s)\omega(\langle a,1\rangle)|a|^{\zeta+rk-1}\,dt\,dv\,d^*a.
\end{align}
In the linear case this is \cite[(3.42)]{CFGK2} (with $r=1$), and note that in \textit{loc. cit.} $\omega$ was replaced by $\pi^{-1}$, because there $\pi^{\vee}$ is one-dimensional. Also if $rk=1$, the $dtdv$ integration is omitted.

\begin{proposition}\label{proposition:functional to compare basic props}
Integral~\eqref{int:after functional equation to compare} is well defined, absolutely convergent for
$\Real(\zeta)\leq B\Real(s)+D$, where $B$ and $D$ are real constants depending only on $\pi$ and $\tau$, and in this domain belongs to \eqref{eq:homspace G with W(rho1tau) GL} where $\pi$ is replaced by
$|~|^{-\zeta}\pi$. Consequently, it admits meromorphic continuation which belongs to $\C(q^{-\zeta},q^{-s})$.
Moreover, outside finitely many values of $q^{-s}$, the continuation with $\zeta=0$ belongs to \eqref{eq:homspace G with W(rho1tau) GL}.
\end{proposition}
\begin{proof}
Since $W_f$ is genuine and $\omega$ is anti-genuine, the integrand is a non-genuine function on $F^*$. Also
by \eqref{eq:block compatibility on Levi of P},
\begin{align*}
\langle\diag(I_{2rk-1},aa'),1\rangle=\langle\diag(I_{2rk-1},a),\sigma_1^{\diamondsuit}(a,a')^{-1}\rangle\langle\diag(I_{2rk-1},a'),1\rangle,
\end{align*}
and because
\begin{align*}
\omega(\langle aa',1\rangle)=\omega(\langle a,\sigma_1^{\diamondsuit}(a,a')^{-1}\rangle\langle a',1\rangle)
=\sigma_1^{\diamondsuit}(a,a')\omega(\langle a,1\rangle\langle a',1\rangle),
\end{align*}
the $d^*a$-integral is a right-invariant functional on $F^*$.

It is straightforward to prove convergence of this integral in a domain of the form $\Real(\zeta)\leq B\Real(s)+D$:
the integrand vanishes unless $v$ and $t$ belong in (large) compact subsets, and since the Whittaker function is smooth,
we reduce to an integral over $F^*$. The integrand vanishes unless $a$ is bounded from below, then the usual bounds on the exponents of $W_f$ can be used. See \cite{GJ,JPSS} for
similar arguments.

Next we prove that for $\Real(\zeta)\ll0$, the integral belongs to
\begin{align}\label{eq:space with s and zeta}
\Hom_{G\times G}(J_{U,\psi_U^{-1}}(
\mathrm{I}(\mathcal{W}(\rho_1(\tau)),\mathcal{W}(\rho_1(\tau)),s)),|~|^{\zeta}\pi^{\vee}\otimes|~|^{-\zeta}\pi).
\end{align}

First we prove
\begin{align}\label{eq:GL 1 equiv for other int U}
\Psi(\zeta,s,\omega,\langle u,1\rangle\cdot f)=\psi_U^{-1}(u)\Psi(\zeta,s,\omega,f).
\end{align}
Write $u\in U$ in the form $u=\diag(d,d')[\begin{smallmatrix}y&z\\0&x\end{smallmatrix}]$ where $d,d'\in N_{\GL_{rk}}$. By \eqref{eq:block compatibility on Levi of P} the elements $w'$, $[t,v]$ and $\diag(I_{2rk-1},a)$
commute with $\diag(d,I_{rk})$ in $H^{(m,r)}$, hence the equivariance properties under $\diag(d,I_{rk})$ are satisfied, because
of the equivariance properties of $W_f$ with respect to \eqref{eq:character for Ind GL 1}. Since
the conjugation of $[\begin{smallmatrix}y&z\\0&x\end{smallmatrix}]$ by the elements $w'$, $[t,v]$ and $\diag(I_{2rk-1},a)$ still belongs to $N_{H}$ (even to $V_{(rk,rk)}$), we can use \eqref{eq:sigma conjugate v by h} to extend this conjugation to $H^{(m,r)}$, and since
\begin{align*}
{}^{w'}[\begin{smallmatrix}y&z\\0&x\end{smallmatrix}]=\left(\begin{smallmatrix}I_{rk-1}&&z&y\\&1&x&0\\&&I_{rk-1}\\&&&1\end{smallmatrix}\right),
\end{align*}
the equivariance properties under $[\begin{smallmatrix}y&z\\0&x\end{smallmatrix}]$ are preserved as well, by definition \eqref{eq:character for Ind GL 1}. Regarding $d'$, first write
\begin{align*}
\langle \diag(I_{rk},d'),1\rangle=\langle \diag(I_{rk},\left(\begin{smallmatrix}1&t'&v'\\&1\\&&I_{rk-2}\end{smallmatrix}\right)),1\rangle\langle \diag(I_{rk+1},d''),1\rangle,\qquad d''\in N_{\GL_{rk-1}}.
\end{align*}
Then by \eqref{eq:sigma conjugate v- to v by h},
\begin{align*}
{}^{w'}\langle \diag(I_{rk},\left(\begin{smallmatrix}1&t'&v'\\&1\\&&I_{rk-2}\end{smallmatrix}\right)),1\rangle=
\langle [t',v'],\varsigma([t',v']),1\rangle.
\end{align*}
Therefore we can change variables in $t$ and $v$ to obtain invariance under the first row of $d'$, as required. Regarding $d''$,
here ${}^{w'}\diag(I_{rk+1},d'')\in N_{H}$, so that we can conjugate by $w'$ without introducing a root of unity. Then
${}^{w'}\diag(I_{rk+1},d'')=\diag(I_{rk},d'',1)$ normalizes $\diag(I_{rk},V_{(rk-1,1)}^-)$ hence by the analog of
\eqref{eq:sigma conjugate v by h} (use \eqref{eq:epsilon for conjugation between split subgroups} with $Y=N_H^-$),
\begin{align*}
{}^{\diag(I_{rk},d'',1)}\langle [t,v],\varsigma([t,v])\rangle
=\langle [t_{d''},v_{d''}],\varsigma([t_{d''},v_{d''}])\rangle.
\end{align*}
Here $t_{d''}$ and $v_{d''}$ depend on $d''$, and we can change variables. We see that the l.h.s.~ of
\eqref{eq:GL 1 equiv for other int U} with $u=\diag(I_{rk},d')$ becomes
\begin{align*}
&\int\limits_{F^*}\int\limits_{F^{rk-2}}\int\limits_{F} W_f(
\langle\diag(I_{rk},d'',1),1\rangle\langle\diag(I_{2rk-1},a),1\rangle\langle [t,v],\varsigma([t,v])\rangle
\langle w',1\rangle,s)\\&\omega(\langle a,1\rangle)|a|^{\zeta+rk-1}\,dt\,dv\,d^*a.
\end{align*}
Again we obtain $\psi_U^{-1}(\diag(I_{rk},d'))$ by \eqref{eq:character for Ind GL 1}.
This completes the verification of \eqref{eq:GL 1 equiv for other int U}.

Write $\omega(g)=\pi^{\vee}(g)\xi^{\vee}(\xi)$, where $\xi^{\vee}$ (resp., $\xi$) belongs to the space of $\pi^{\vee}$ (resp., $\pi$).
Let $g_1,g_2\in G$ and denote $\omega_{g_1,g_2}(g)=\pi^{\vee}(gg_2)\xi^{\vee}(\pi(g_1)\xi)$.
We turn to prove
\begin{align}\label{eq:eq:GL 1 equiv for other int G}
&\Psi(\zeta,s,\omega_{\langle g_1,1\rangle,\langle g_2,1\rangle},(\langle \mathfrak{e}_1(g_1),1\rangle\langle \mathfrak{e}_2(g_2),1\rangle)\cdot f)=|g_1|^{\zeta}|g_2|^{-\zeta}\Psi(\zeta,s,\omega,f).
\end{align}
Since $\mu_{2m}\subset F^*$, we can use Proposition~\ref{proposition:action of W on torus is trivial on Sp} to write
\begin{align*}
&\Psi(\zeta,s,\omega_{\langle g_1,1\rangle,\langle g_2,1\rangle},(\langle \mathfrak{e}_1(g_1),1\rangle\langle \mathfrak{e}_2(g_2),1\rangle)\cdot f)\\&=
\int\limits_{F^*}\int\limits_{F^{rk-2}}\int\limits_{F} W_f(\langle [t,v],\varsigma([t,v])\rangle
\langle w',1\rangle\langle \mathfrak{e}_2(a),1\rangle\langle \mathfrak{e}_1(g_1),1\rangle\langle \mathfrak{e}_2(g_2),1\rangle,s)
\\&\pi^{\vee}(\langle a,1\rangle\langle g_2,1\rangle)\xi^{\vee}(\pi(\langle g_1,1\rangle)\xi)
|a|^{\zeta}\,dt\,dv\,d^*a.
\end{align*}
Now changing $a\mapsto ag_2^{-1}$ then $a\mapsto g_1a$ (as in the proof of Proposition~\ref{proposition:equiv props GL}) and using
\eqref{eq:the $2$-cocycle on G times G formula GL} we obtain
\begin{align*}
&|g_1|^{\zeta}|g_2|^{-\zeta}\int\limits_{F^*}\int\limits_{F^{rk-2}}\int\limits_{F} W_f(\langle [t,v],\varsigma([t,v])\rangle
\langle w',1\rangle\langle (g_1,g_1),1\rangle\langle \mathfrak{e}_2(a),1\rangle,s)
\omega(\langle a,1\rangle)
|a|^{\zeta}\,dt\,dv\,d^*a.
\end{align*}
Since $(g_1,g_1)\in T_{H}$ and belongs to the center of $H$, it commutes with $w'$ and $[t,v]$. This also holds for $H^{(m,r)}$:
regarding $w'$ this follows from Proposition~\ref{proposition:action of W on torus is trivial on Sp}, and for $[v,t]$ by
\eqref{eq:sigma conjugate v by h}. Then by virtue of
\eqref{int:standard Whittaker on I W W}, \eqref{eq:sigma conjugate v by h}, Proposition~\ref{proposition:action of W on torus is trivial on Sp}
and assumption \eqref{eq:invariance prop on GL or SL} we see that
\begin{align*}
W_f(\langle(g_1,g_1),1\rangle h, s)=W_f(h, s),\qquad\forall h\in H^{(m,r)}.
\end{align*}
This completes the proof of \eqref{eq:eq:GL 1 equiv for other int G}, and we have thus shown that
$\Psi(\zeta,s,\omega,f)$ belongs to \eqref{eq:homspace G with W(rho1tau) GL} where $\pi$ is replaced by
$|~|^{-\zeta}\pi$.

We deduce the meromorphic continuation using Bernstein's continuation principle.
In general consider the group $\GL_l$, the generic character of $N_{\GL_l}$ given by $\psi(v)=\psi(\sum_{i=1}^{l-1}v_{i,i+1})$,
and the subgroup $P_l<P_{(l-1,1)}$ of matrices with the last row $(0,\ldots,0,1)$, i.e., the so-called mirabolic subgroup. By
\cite[5.15]{BZ1}, if $\vartheta$ is a smooth representation of $P_l$ affording at least one Whittaker model
(a non-degenerate representation in the terminology of \textit{loc. cit.}), then $\vartheta$ admits at least one subrepresentation $\ind_{N_{\GL_l}}^{P_l}(\psi)$. This result extends to $\GL_l^{(m,r)}$: the main points to note are that
the functors $\Psi^-=J_{V_{(l-1,1)}}$, $\Phi^-=J_{V_{(l-1,1)},\psi}$,
$\Psi^+(\cdot)=\ind_{\GL_{l-1}V_{(l-1,1)}}^{P_l}(\cdot\otimes1)$ and $\Phi^+(\cdot)=\ind_{P_{l-1}V_{(l-1,1)}}^{P_l}(\cdot\otimes\psi)$ carry genuine representations into genuine representations, are still exact, and $\widetilde{P}_l$ acts on the group of characters of $V_{(l-1,1)}$ with two orbits. We deduce that $\ind_{N_{H}}^{\widetilde{P}_{2rk}}(\psi)$ is a subrepresentation of $\mathrm{I}(\mathcal{W}(\rho_1(\tau)),\mathcal{W}(\rho_1(\tau')),s)$.
Furthermore, we can redefine $P_l$ to be the subgroup of $P_{(1,l-1)}$ with the first row $(1,0,\ldots,0)$, and obtain similar results.

Therefore we can take $f$ such that $W_f|_{\widetilde{P}_{2rk}}$ is a genuine smooth function, whose support is contained in a small compact open neighborhood of the identity (with $P_{2rk}<P_{(1,2rk-1)}$). Since
$\diag(I_{2rk-1},a)[t,v]w'$ belong to $P_{2rk}$, there is a choice of data $f$ and $\omega$ such that $\Psi(\zeta,s,\omega,f)$ is absolutely convergent and equals $1$, for all $\zeta$ and $s$. Now
by Lemma~\ref{lemma:uniqueness}, the space \eqref{eq:space with s and zeta} is at most one-dimensional outside finitely many values of $q^{-\zeta}$ and $q^{-s}$ (in the statement we had $\zeta$ fixed, but the proof implies this more general result), then by Bernstein's continuation principle (\cite{Banks}), $\Psi(\zeta,s,\omega,f)$ admits meromorphic continuation which belongs to $\C(q^{-\zeta},q^{-s})$.

For the last assertion, consider now $\Psi(0,s,\omega,f)$. It belongs to \eqref{eq:homspace G with W(rho1tau) GL} as long as
$\Psi(\zeta,s,\omega,f)$ is well defined for any $\omega$ and $f$ when we take $\zeta=0$. This is true by \cite{Banks}, since $\Psi(\zeta,s,\omega,f)$ is holomorphic whenever there is a unique solution (in the terminology of \textit{loc. cit.}), and
\eqref{eq:homspace G with W(rho1tau) GL} is one-dimensional outside finitely many values of $q^{-s}$.
\end{proof}
The proposition implies that both $Z(s,\omega,f)$ and
$\Psi(0,s,\omega,f)$, defined by meromorphic continuation, belong to \eqref{eq:homspace G with W(rho1tau) GL}, hence
are proportional. We can thus compute the former by first computing the latter, then finding the proportionality factor. This technique is due to Soudry \cite{Soudry} and has since been used in several similar settings including \cite{Soudry3,GRS4,Soudry2} (see also \cite{me2,me3,me6}).

\begin{proposition}\label{proposition:unramified computation for GL(1) other integral}
Let $\omega$ and $f$ be the normalized unramified vectors. Then
\begin{align*}
&\Psi(\zeta,s,\omega,f)=\frac{L_{\vartheta}(-r\zeta+r\alpha s+1/2,\pi\times\tau^{\vee})L_{\vartheta}(-r\zeta-r\alpha s+1/2,\pi\times{\tau'}^{\vee})}{\prod_{j=1}^rL_{\vartheta}(2r\alpha s+j,\tau\times{\tau'}^{\vee})}\qquad(\alpha=rk).
\end{align*}
\end{proposition}
\begin{proof}
Since $f$ is unramified, so is $W_f$. As in the linear case
\cite[(3.42)--(3.44)]{CFGK2}, we can use conjugations to deduce that the integrand is zero unless $t,v\in\mathcal{O}$, then the $dvdt$-integral evaluates to the constant $1$. Specifically, for any $b\in F$,
\begin{align}\label{eq:t v e conj}
[t,v]w'[\begin{smallmatrix}0&0\\b&0\end{smallmatrix}]=
[t,v]\left(\begin{smallmatrix}I_{rk-1}\\&1&&b\\&&I_{rk-1}&\\&&&1\end{smallmatrix}\right)w'
={}^{[t,v]}\left(\begin{smallmatrix}I_{rk-1}\\&1&&b\\&&I_{rk-1}&\\&&&1\end{smallmatrix}\right)[t,v]w'
=\left(\begin{smallmatrix}I_{rk-1}\\&1&b t&-b v&b\\&&1&\\&&&I_{rk-2}\\&&&&1\end{smallmatrix}\right)[t,v]w'.
\end{align}
Since $[\begin{smallmatrix}0&0\\b&0\end{smallmatrix}],{}^{w'}[\begin{smallmatrix}0&0\\b&0\end{smallmatrix}],{}^{[t,v]w'}[\begin{smallmatrix}0&0\\b&0\end{smallmatrix}]\in N_{H}$, by \eqref{eq:sigma conjugate v by h} we can extend \eqref{eq:t v e conj} to $H^{(m,r)}$. We also have
\begin{align}\label{eq:e epsilon resuse 2}
W_f(\langle \left(\begin{smallmatrix}I_{rk-1}\\&1&b t&-b v&a^{-1}b\\&&1&\\&&&I_{rk-2}\\&&&&1\end{smallmatrix}\right),1\rangle h, s)=\psi^{-1}(b t)W_f(h,s)\qquad\forall h\in H^{(m,r)},
\end{align}
and in particular independent of $a$. While this applies to any $b$, when we take $b\in\mathcal{O}^*$,
$W_f$ is right-invariant on $\langle [\begin{smallmatrix}0&0\\b&0\end{smallmatrix}],1\rangle$ and then since $\psi$ is unramified,
\eqref{eq:e epsilon resuse 2} implies that the integrand vanishes unless $t\in\mathcal{O}$. Then by \eqref{eq:sigma conjugate v- to v by h},
${}^{{w'}^{-1}}\langle [t,v],\varsigma([t,v])\rangle=\langle{}^{{w'}^{-1}}[t,v],1)\rangle$ (this holds whether $t$ is in $\mathcal{O}$ or not), and since $W_f$ is unramified,
\begin{align}\label{eq:resuse conj t to the right}
&W_f(\langle\diag(I_{2rk-1},a),1\rangle\langle [t,v],\varsigma([t,v])\rangle
\langle w',1\rangle,s)\\&=\nonumber
W_f(\langle\diag(I_{2rk-1},a),1\rangle\langle [0,v],\varsigma([0,v])\rangle
\langle [t,0],\varsigma([t,0])\rangle
\langle w',1\rangle,s)\\&=\nonumber
W_f(\langle\diag(I_{2rk-1},a),1\rangle\langle [0,v],\varsigma([0,v])\rangle
\langle w',1\rangle,s).
\end{align}
We proceed similarly, now with
\begin{align*}
\mathrm{e}(o_1,\ldots,o_{rk-2})={}^{{w'}^{-1}}\diag(I_{rk},\left(\begin{smallmatrix}1&&&&o_1\\&\ddots&&&\vdots\\&&1&0&o_{rk-2}\\&&&1&0\\&&&&1\end{smallmatrix}\right)),\qquad o_1,\ldots,o_{rk}\in\mathcal{O}^*.
\end{align*}
Starting with $i=1$ and up to $i=rk-2$, we show that each coordinate of $v$ belongs in $\mathcal{O}$ then remove it from the integrand; the integral over this coordinate equals the volume of $\mathcal{O}$, which is $1$. Eventually we remain with
\begin{align*}
W_f(\langle\diag(I_{2rk-1},a),1\rangle\langle w',1\rangle,s)=W_f(\langle\diag(I_{2rk-1},a),1\rangle,s).
\end{align*}
Thus
\begin{align*}
&\Psi(\zeta,s,\omega,f)=\int\limits_{F^*}W_f(\langle\diag(I_{2rk-1},a),1\rangle,s)\omega(\langle a,1\rangle)|a|^{\zeta+rk-1}\,d^*a.
\end{align*}
Then by \eqref{int:standard Whittaker on I W W *} and changing $a\mapsto a^{-1}$, the r.h.s.~ equals
\begin{align}\label{eq:last GL1 other integral}
\int\limits_{F^*}W_f^*(\langle\diag(a,I_{2rk-1}),1\rangle,s)\omega(\langle a^{-1},1\rangle)|a|^{-\zeta+1/2-(2rk-1)/2}\,d^*a.
\end{align}

This integral resembles the integral \eqref{eq:Z integral GL 1 GL rk} from \S~\ref{RS integrals}.
To complete the proof we relate $\omega(\langle a^{-1},1\rangle)$ to the matrix coefficient appearing in \eqref{eq:Z integral GL 1 GL rk},
and $W_f^*$ to the normalized unramified Whittaker function from \S~\ref{RS integrals}.

Let $\omega{}^*$ be the matrix coefficient of $(\pi^{\vee})^*$ given by $\omega{}^*(\langle a,1\rangle)=\omega({}^*\langle a,1\rangle)$, where ${}^*$ is defined by \eqref{eq:involution b*0}. Since ${}^*\langle a,1\rangle=\langle a^{-1},1\rangle$,
$\omega(\langle a^{-1},1\rangle)=\omega{}^*(\langle a,1\rangle)$. Moreover by
Proposition~\ref{proposition:vartheta of *} the representation $(\pi^{\vee})^*$ is induced from
$\varepsilon^{-1}\otimes\vartheta^{-1}\mu$, thus equal to $(\pi^*)^{\vee}$.

Now consider $W_f^*$. Let
\begin{align*}
\chi^{\bullet}=|~|^{rks}\chi_1\otimes\ldots\otimes|~|^{rks}\chi_k\otimes|~|^{-rks}\chi'_1\otimes\ldots\otimes|~|^{-rks}\chi'_k,
\end{align*}
which is an unramified character of $T_{\GL_{2k}}$,
\begin{align*}
&\mathbf{x}=(x_1,\ldots,x_k), \qquad x_i=(|~|^{rks}\chi_i)(\varpi^r),\qquad
\mathbf{x}'=(x'_1,\ldots,x'_k),\qquad x'_i=(|~|^{-rks}\chi'_i)(\varpi^r),\\
&\mathbf{x}^{\bullet}=(x_1,\ldots,x_k,x'_1,\ldots,x'_k).
\end{align*}

At this point we add the assumption that $\chi$ and $\chi'$ are regular, which means $\mathbf{x}_{\alpha}\ne 1$ and $\mathbf{x}'_{\alpha}\ne 1$
for all $\alpha\in \Phi_k^+$; this together with \eqref{eq:local assumption on tau} implies the condition of Proposition~\ref{proposition:local component is subrep for rk c}, hence
$\rho_1(\tau)\subset\Theta_{r,m,r,\vartheta}(\chi)$ and $\rho_1(\tau')\subset\Theta_{r,m,r,\vartheta}(\chi')$. Therefore
\begin{align*}
\mathrm{I}(\mathcal{W}(\rho_1(\tau)),\mathcal{W}(\rho_1(\tau')),s)\subset
\Ind_{\widetilde{P}}^{H^{(m,r)}}((\Theta_{r,m,r,\vartheta}(\chi)\otimes
\Theta_{r,m,r,\vartheta}(\chi'))\delta_P^s)
=\Theta_{2r,m,r,\vartheta}(\chi^{\bullet}).
\end{align*}
If $f'$ is the normalized unramified function in the space of $\Theta_{2r,m,r,\vartheta}(\chi^{\bullet})$, we have
\begin{align*}
f(\langle I_{2rk},1\rangle,s)=W_{\mathbf{0}}(\mathbf{0},\vartheta,\chi_{\Theta})^{-1}
W_{\mathbf{0}}(\mathbf{0},\vartheta,\chi'_{\Theta})^{-1}f'(\langle I_{2rk},1\rangle,s).
\end{align*}
Since by Theorem~\ref{theorem:CS formula for rk, c},
\begin{align*}
&W_{\mathbf{0}}(\mathbf{0},\vartheta,\chi_{\Theta})=\prod_{j=1}^{r}\prod_{1\leq i<j\leq k}(1-q^{-j}\mathbf{x}_{(i,j)}),\qquad
W_{\mathbf{0}}(\mathbf{0},\vartheta,\chi'_{\Theta})=\prod_{j=1}^{r}\prod_{1\leq i<j\leq k}(1-q^{-j}\mathbf{x}'_{(i,j)}),\\
&W_{\mathbf{0}}(\mathbf{0},\vartheta,(\chi^{\bullet})_{\Theta})=\prod_{j=1}^{r}\prod_{1\leq i<j\leq 2k}(1-q^{-j}\mathbf{x}^{\bullet}_{(i,j)}),
\end{align*}
we obtain
\begin{align*}
W_f(\langle I_{2rk},1\rangle,s)=\frac{W_{\mathbf{0}}(\mathbf{0},\vartheta,(\chi^{\bullet})_{\Theta})}{
W_{\mathbf{0}}(\mathbf{0},\vartheta,\chi_{\Theta})
W_{\mathbf{0}}(\mathbf{0},\vartheta,\chi'_{\Theta})}
=\prod_{j=1}^rL_{\vartheta}(2r\alpha s+j,\tau\times{\tau'}^{\vee})^{-1}.
\end{align*}
For the last equality also note that $\vartheta=\vartheta^{-1}$.
This is also the value of $W_f^*(\langle I_{2rk},1\rangle,s)$, and by Proposition~\ref{proposition:* of Theta chi}, $\Theta_{2r,m,r,\vartheta}(\chi^{\bullet})^*=\Theta_{2r,m,r,\vartheta}((\chi^{\bullet})^*)$.
Thus if $W^0$ is the normalized unramified Whittaker function of $\Theta_{2r,m,r,\vartheta}((\chi^{\bullet})^*)$,
\begin{align*}
W_f^*(g,s)=\prod_{j=1}^rL_{\vartheta}(2r\alpha s+j,\tau\times{\tau'}^{\vee})^{-1}W^0(g)
\end{align*}
and \eqref{eq:last GL1 other integral} equals
\begin{align*}
\prod_{j=1}^rL_{\vartheta}(2r\alpha s+j,\tau\times{\tau'}^{\vee})^{-1}
Z(-\zeta+1/2,\omega^*,W^0).
\end{align*}
Now by \eqref{RS integral for n=1} ($2rk>1$) and since $t_{(\pi^*)^{\vee},\vartheta}=t_{\pi,\vartheta}$, and still assuming $\chi$ and $\chi'$ are regular,
\begin{align*}
&\Psi(\zeta,s,\omega,f)=\frac{L_{\vartheta}(-r\zeta+1/2-r\alpha s,\pi\times\tau^{\vee})L_{\vartheta}(-r\zeta+1/2+r\alpha s,\pi\times{\tau'}^{\vee})}{\prod_{j=1}^rL_{\vartheta}(2r\alpha s+j,\tau\times{\tau'}^{\vee})}.
\end{align*}
Finally because $W_f$ is an analytic function of $\mathbf{x}$ and $\mathbf{x}'$, we can remove the regularity assumption on the l.h.s.,
and the r.h.s.~ is also well defined without it. Therefore we can drop this assumption and the proof is complete.
\end{proof}

We digress momentarily to the case $rk=1$ and compute the integral $Z(s,\omega,f)$ with unramified data. Assume $m=2$, the case $m=1$ was proved in \cite[\S~6.1]{PSR} (for any $n$, and using \cite{GJ}; here $n=1$). By Proposition~\ref{proposition:sigma * and sigma on GLd} we can assume $H^{(2,1)}$ is realized using $(\det,\det)_2$ (this is immediate for $\GL_1^{(2,1)}$).
Now we can write $\pi=\vartheta\otimes\mu$ ($\mu$ is an unramified quasi-character of $F^*$) and
$\pi(\langle a,\epsilon\rangle)=\vartheta(a)\mu(a)$ (see \S~\ref{unramified reps}). Similarly
$\tau=\vartheta\otimes\chi$ and $\tau'=\vartheta\otimes\chi'$.
For a section $f$ of the space of $\Ind_{P}^{H}((\chi\otimes\chi')\delta_P^s)$, the section $f_{\vartheta}(\langle h,\epsilon\rangle,s)=\epsilon\vartheta(\det h)f(h,s)$ belongs to the space of \eqref{eq:ind sections}. Then if $\omega$ is normalized,
\begin{align}\label{eq:covering to linear minimal GL}
Z(s,\omega,f_{\vartheta})=\int\limits_{F^*}\vartheta(-a)f(\left(\begin{smallmatrix}&1\\1\end{smallmatrix}\right)
\left(\begin{smallmatrix}1&1\\&1\end{smallmatrix}\right)\left(\begin{smallmatrix}1&\\&a\end{smallmatrix}\right),s)\vartheta^{-1}(a)\mu^{-1}(a)\,d^*a.
\end{align}
(Cf. \cite[pp.~81--83]{Gan}, the constant $2$ there was needed because the integration was over the cover.)
This is the linear $\GL_1\times\GL_1$ doubling integral for the representations $(-1,\cdot)_2\mu$ and $\chi\otimes\chi'$, and the section $f$.
For the normalized unramified $f$, by \cite[\S~6.1]{PSR} (see also \cite[Proposition~34]{CFGK2}), the r.h.s.~ of \eqref{eq:covering to linear minimal GL} equals
\begin{align*}
&\frac{(1-(-1,\varpi)_2\mu^{-1}(\varpi)\chi(\varpi)q^{-s-1/2})^{-1}
(1-(-1,\varpi)_2\mu(\varpi){\chi'}^{-1}(\varpi)q^{-s-1/2})^{-1}}
{(1-\chi(\varpi){\chi'}^{-1}(\varpi)q^{-2s-1})^{-1}}.
\end{align*}
Since $(-1,\varpi)_2=1$ and $\vartheta=\vartheta^{-1}$, we obtain
\begin{align}\label{eq:covering minimal GL}
Z(s,\omega,f)=
\frac{L_{\vartheta}(s+1/2,\pi^{\vee}\times\tau)
L_{\vartheta}(s+1/2,\pi\times{\tau'}^{\vee})
}{L_{\vartheta}(2s+1,\tau\times{\tau'}^{\vee})}.
\end{align}

We proceed with the general case $rk\geq1$.
Denote
\begin{align*}
L_{\pi,\tau,\vartheta}(s)
=\frac{L_{\vartheta}(1/2-r\alpha s,\pi\times\tau^{\vee})}
{L_{\vartheta}(1/2+r\alpha s,\pi^{\vee}\times\tau)}.
\end{align*}
\begin{lemma}\label{lemma:substitution}
As meromorphic continuations $L_{\pi,\tau,\vartheta}(s)Z(s,\omega,f)=\Psi(0,s,\omega,f)$, for any $\omega$ and rational section $f$.
\end{lemma}
\begin{proof}
For $rk=1$ the result already follows from Proposition~\ref{proposition:unramified computation for GL(1) other integral} and \eqref{eq:covering minimal GL} (and \cite[\S~6.1]{PSR} for $m=1$), so that throughout the proof we assume $rk>1$.
We adapt the proof of \cite[Claim~36]{CFGK2} which, as mentioned above, was based on the method of Soudry \cite[p.~70]{Soudry}. Our use of the Fourier inversion formula (see below) was adapted from \cite{Soudry3}.

First we define an action of (the ring of) Schwartz--Bruhat functions on sections, this will be used repeatedly in the proof. For a section $f$ and a Schwartz--Bruhat function $\phi$ on $F$,
\begin{align*}
\phi(f)(h,s)=\int\limits_Ff(h\langle[\begin{smallmatrix}0&0\\b&0\end{smallmatrix}],1\rangle,s)\phi(b)\,db.
\end{align*}
This integral is well defined because $[\begin{smallmatrix}0&0\\b&0\end{smallmatrix}]\mapsto\langle[\begin{smallmatrix}0&0\\b&0\end{smallmatrix}],1\rangle$ is the splitting of $\{[\begin{smallmatrix}0&0\\b&0\end{smallmatrix}]:b\in F\}$. Also observe that for each $f$ there is some $\phi$ such that $\phi(f)=f$:
for any $l>0$, $H^{(m,r)}$ is split over the subgroup $\mathcal{N}_l=\{[\begin{smallmatrix}0&0\\b&0\end{smallmatrix}]:b\in\mathcal{P}^l\}$.
One splitting is trivial:
$v\mapsto \langle v,1\rangle$. In fact this is the only splitting, for if $v\mapsto \langle v,\eta'(v)\rangle$ is another splitting, then $\eta':\mathcal{P}^l\rightarrow\mu_m$ is a homomorphism and since the exponent of $\mu_m$ is $m$ and $|m|=1$ ($F$ is unramified), $\eta'=1$. Now since $f$ is smooth, there is $l\gg0$ such that $f$ is right-invariant on $\{\langle v,1\rangle: v\in\mathcal{N}_l\}$. Thus the characteristic function $\phi$ of $\mathcal{P}^l$ satisfies $\phi(f)=f$.
In addition define the Fourier transform of $\phi$ by $\widehat{\phi}(t)=\int_F\phi(b)\psi^{-1}(bt)db$.

For $t\in F$, let $\jmath(t)\in P\cap N_{H}$ be given by
\begin{align*}
\jmath(t)=\diag(I_{rk},\left(\begin{smallmatrix}1&-t\\&1\end{smallmatrix}\right),I_{rk-2}).
\end{align*}
Since ${}^{\delta_0}\jmath(t)\in N_{H}$, ${}^{\delta_0}\langle \jmath(t),1\rangle=
\langle{}^{\delta_0}\jmath(t),1\rangle$ by \eqref{eq:sigma conjugate v by h} (also ${}^{\delta_0}\jmath(t)\in P$). We can also use
\eqref{eq:sigma conjugate v by h} to conjugate elements of $U_P$ by $\jmath(t)$ without introducing a root of unity. Also
by \eqref{eq:sigma on vh and h'v'}, for any $h\in H$, $\langle h,1\rangle\langle\jmath(t),1\rangle=\langle h\jmath(t),1\rangle$ and
$\langle{}^{\delta_0}\jmath(t)h,1\rangle=\langle{}^{\delta_0}\jmath(t),1\rangle\langle h,1\rangle$.
Then for fixed $u,t\in F$,
\begin{align*}
&\int f(\langle \delta_0[\begin{smallmatrix}y&z\\u&x\end{smallmatrix}]\jmath(t),1\rangle,s)\psi(x_1)\,dx\,dy\,dz=\psi((u-1)t)\int f(\langle \delta_0[\begin{smallmatrix}y&z\\u&x\end{smallmatrix}],1\rangle,s)\psi(x_1)\,dx\,dy\,dz.
\end{align*}
Since $\int_F\psi((u-1)t)dt=0$ unless $u=1$,
\begin{align}\label{int:justification 1}
Z(s,\omega,f)=
\int\limits_{F^*}\omega(\langle a,1\rangle)
\int f(\langle\delta_0[\begin{smallmatrix}y&z\\u&x\end{smallmatrix}]\jmath(t),1\rangle\langle\mathfrak{e}_2(a),1\rangle,s)\psi(x_1)\,dx\,dy\,dz\,dt\,du\,d^*a.
\end{align}
Here $x\in F^{rk-1}$ is a row, $y\in F^{rk-1}$ is a column, $z\in\Mat_{rk-1}$ (i.e., $du_0=dxdydz$) and $u,t\in F$.
Integral~\eqref{int:justification 1} is defined in the domain of definition of $Z(s,\omega,f)$, but is not absolutely convergent as a multiple integral.

However, we can still consider the integral obtained by formally changing $dtdu$ to $dudt$:
\begin{align*}
Z'(s,\omega,f)=
\int\limits_{F^*}\omega(\langle a,1\rangle)
\int f(\langle\delta_0[\begin{smallmatrix}y&z\\u&x\end{smallmatrix}]\jmath(t),1\rangle\langle\mathfrak{e}_2(a),1\rangle,s)\psi(x_1)\,dx\,dy\,dz\,du\,dt\,d^*a.
\end{align*}
As we proved in \cite[Claim~36, (3.48)]{CFGK2} for the linear case, for $\Real(s)\gg0$,
\begin{align}\label{int:justification 2}
\int\limits_{F^*}\int\limits_F\left|
\int \omega(\langle a,1\rangle)f(\langle\delta_0[\begin{smallmatrix}y&z\\u&x\end{smallmatrix}]\jmath(t),1\rangle\langle\mathfrak{e}_2(a),1\rangle,s)\psi(x_1)\,dx\,dy\,dz\,du\right|\,dt\,d^*a<\infty.
\end{align}
To extend the proof to the covering, first note that we can find $\phi$ such that $\phi(f)=f$ (as explained above). Second,
we can conjugate $[\begin{smallmatrix}0&0\\b&0\end{smallmatrix}]$ ($\jmath'(b)$ in the notation of \textit{loc. cit.}) by
$\mathfrak{e}_2(a)$ and $\jmath(t)$, and also $\jmath(t)$ by $\mathfrak{e}_2(a)$, without introducing a root of unity.

Now $Z'(s,\omega,f)$ is convergent in the sense of \eqref{int:justification 2} in a right half-plane, where it also belongs to
\eqref{eq:homspace G with W(rho1tau) GL}. Thus by Lemma~\ref{lemma:uniqueness} the integrals $Z'(s,\omega,f)$ and
$Z(s,\omega,f)$ are proportional in $\Real(s)\gg0$. The proportionality factor is $1$: this was proved
in \textit{loc. cit.} using a section of the form $\phi(f)$, conjugations of elements $\jmath(t)$
and $[\begin{smallmatrix}0&0\\b&0\end{smallmatrix}]$ which extend to the covering, and the
Fourier inversion formula $\int_F\widehat{\phi}(t)\psi(aut)dt=\phi(au)$. We deduce $Z'(s,\omega,f)=Z(s,\omega,f)$. In particular
by Corollary~\ref{corollary:mero cont of GL1 integral}, $Z'(s,\omega,f)$ admits meromorphic continuation which belongs to $\C(q^{-s})$. In the remaining part of the proof we show
\begin{align}\label{eq:identity with Z'}
L_{\pi,\tau,\vartheta}(s)Z'(s,\omega,f)=\Psi(0,s,\omega,f),
\end{align}
using a specific substitution.

Let $\omega$ be the normalized unramified matrix coefficient. Let $W\in \mathcal{W}(\rho_1(\tau))$ be the normalized unramified function and take a small compact open neighborhood $\mathcal{N}<K_{H}$ of the identity. Choose $f'$ in the space of
\eqref{eq:ind sections} such that
$\langle\delta_0,1\rangle\cdot f'$ is right invariant by $\{\langle y,\eta^{\diamondsuit}_{2rk}(y)\rangle:y\in\mathcal{N}\}$,
supported in the preimage of $P\mathcal{N}$ in $H^{(m,r)}$, and for all $a\in\GL_{rk}$,
\begin{align}\label{eq:cond on W of f'}
\langle\delta_0,1\rangle\cdot f'(\langle\diag(a,I_{rk}),1\rangle,s)=|\det a|^{rk(s+1/2)}W(\langle a,1\rangle).
\end{align}
Then take $f=\phi(f')$ such that
\begin{align}\label{eq:cond on FT of phi}
\int\limits_Ff'(h\langle\jmath(t),1\rangle,s)\widehat{\phi}(t)\,dt=f'(h,s),\qquad\forall h,s.
\end{align}
The arguments used above to show that we can find $\phi$ such that $\phi(f)=f$ can be repeated with $\jmath(t)$ instead of $[\begin{smallmatrix}0&0\\b&0\end{smallmatrix}]$ ($\jmath(t)\in N_{H}$), hence the condition \eqref{eq:cond on FT of phi} is satisfied for a suitable $\phi$.

Starting with the l.h.s.~ of \eqref{eq:identity with Z'}, $Z'(s,\omega,f)$ equals
\begin{align*}
&
\int\limits_{F^*}\omega(\langle a,1\rangle)
\int f'(\langle\delta_0[\begin{smallmatrix}y&z\\u&x\end{smallmatrix}]\jmath(t),1\rangle\langle\mathfrak{e}_2(a),1\rangle\langle[\begin{smallmatrix}0&0\\b&0\end{smallmatrix}]
,1\rangle,s)\phi(b)\psi(x_1)\,db\,dx\,dy\,dz\,du\,dt\,d^*a.
\end{align*}

As in \cite[Claim~36]{CFGK2}, after several manipulations including a change in the order of integration, conjugating
$\langle\jmath(t),1\rangle$ by $\mathfrak{e}_2(a)^{-1}$, then conjugating
$[\begin{smallmatrix}0&0\\b&0\end{smallmatrix}]$ by $\jmath(t)$ and $\mathfrak{e}_2(a)$ and changing variables in $x_1$ and $u$, we obtain $\psi^{-1}(bt)$. We then integrate first over $b$ and obtain
\begin{align*}
&\int\limits_{F^*}\omega(\langle a,1\rangle)
\int f'(\langle\delta_0[\begin{smallmatrix}y&z\\u&x\end{smallmatrix}],1\rangle\langle\mathfrak{e}_2(a),1\rangle\langle\jmath(t),1\rangle,s)\widehat{\phi}(t)\psi(x_1)|a|\,dx\,dy\,dz\,du\,dt\,d^*a
\\&=\int\limits_{F^*}\omega(\langle a,1\rangle)
\int f'(\langle\delta_0[\begin{smallmatrix}y&z\\u&x\end{smallmatrix}],1\rangle\langle\mathfrak{e}_2(a),1\rangle,s)\psi(x_1)|a|\,dx\,dy\,dz\,du\,d^*a.
\end{align*}
Here for the equality we first integrated over $t$ then used \eqref{eq:cond on FT of phi}.
Now conjugate $[\begin{smallmatrix}y&z\\u&x\end{smallmatrix}]$ by $\mathfrak{e}_2(a)^{-1}$ (multiplying the measure by $|a|^{-rk}$),
and $\langle\mathfrak{e}_2(a),1\rangle$ by $\delta_0$. By Proposition~\ref{proposition:action of W on torus is trivial on Sp},
${}^{\delta_0}\langle\mathfrak{e}_2(a),1\rangle=\langle{}^{\delta_0}\mathfrak{e}_2(a),1\rangle$. Then as in the linear case, the assumption on the support of $\langle\delta_0,1\rangle\cdot f'$ implies that the integrand vanishes unless the coordinates
of $x,y,z,u$ are small (independently of $a$). Let $D$ be the volume assigned to $V_{(rk,rk)}\cap \mathcal{N}$ by the measure
$dxdydzdu$. Thus using \eqref{eq:cond on W of f'} we obtain
\begin{align*}
&D\int\limits_{F^*}\omega(\langle a,1\rangle)W(\langle\diag(a,I_{rk-1}),1\rangle)|a|^{rks+1/2-(rk-1)/2}\,d^*a=DZ(rks+1/2,\omega,W),
\end{align*}
where $Z(rks+1/2,\omega,W)$ is defined by \eqref{eq:Z integral GL 1 GL rk} and note that
$\mathcal{W}(\rho_1(\tau))\subset \mathcal{W}(\Theta_{r,m,r,\vartheta}(\chi))$
by Corollary~\ref{corollary:functional on theta implies functional on rho c tau} (with $c=1$). Since $\omega$ and $W$ are the normalized unramified vectors, by \eqref{RS integral for n=1},
\begin{align}\label{eq:lhs of compare GL}
Z'(s,\omega,f)=DL_{\vartheta}(r^2ks+1/2,\pi^{\vee}\times\tau)
=DL_{\vartheta}(r\alpha s+1/2,\pi^{\vee}\times\tau).
\end{align}

Turning to the r.h.s.~ of \eqref{eq:identity with Z'}, we first compute $\Psi(\zeta,s,\omega,f)$ for $\zeta\ll0$. In this
case adapting the proof from \cite[Claim~36]{CFGK2} takes a bit more care, for the following reasons. First, we use a lower triangular unipotent subgroup (the subgroup of $[v,t]$). Second, while it was sufficient in \textit{loc. cit.} to obtain an arbitrary Rankin--Selberg integral of $\GL_1\times\GL_k$ (since the $\gamma$-factor of $\pi^{-1}\times\tau$ of \cite{JPSS} could be used), here we must evaluate
the integral explicitly (as above, using \eqref{RS integral for n=1}), which means we must also compute the integral over $[v,t]$.
To start,
\begin{align*}
\Psi(\zeta,s,\omega,f)&=
\int\limits_{F^*}\int W_{f'}(\langle\diag(I_{2rk-1},a),1\rangle\langle [t,v],\varsigma([t,v])\rangle
\langle w',1\rangle\langle[\begin{smallmatrix}0&0\\b&0\end{smallmatrix}],1\rangle,s)\phi(b)\\&\quad\omega(\langle a,1\rangle)|a|^{\zeta+rk-1}\,db\,dt\,dv\,d^*a.
\end{align*}
Here $b\in F$, $t\in F$ and $v\in F^{rk-2}$.
As we have seen in the proof of Proposition~\ref{proposition:unramified computation for GL(1) other integral} (\eqref{eq:t v e conj}--\eqref{eq:e epsilon resuse 2}),
we can conjugate $[\begin{smallmatrix}0&0\\b&0\end{smallmatrix}]$ to the left and obtain
the character $\psi^{-1}(b t)$, then as in \eqref{eq:resuse conj t to the right} (though \eqref{eq:e epsilon resuse 2} no longer implies
$t\in\mathcal{O}$, because $f'$ is not unramified) and using ${}^{{w'}^{-1}}[t,0]=\jmath(t)$,
the integral becomes
\begin{align*}
&\int\limits_{F^*}\int W_{f'}(\langle\diag(I_{2rk-1},a),1\rangle\langle [0,v],\varsigma([0,v])\rangle
\langle w',1\rangle\langle\jmath(t),1\rangle,s)\psi^{-1}(bt)\phi(b)\omega(\langle a,1\rangle)|a|^{\zeta+rk-1}\\&\quad\,db\,dt\,dv\,d^*a\\&=
\int\limits_{F^*}\int W_{f'}(\langle\diag(I_{2rk-1},a),1\rangle\langle [0,v],\varsigma([0,v])\rangle
\langle w',1\rangle\langle\jmath(t),1\rangle,s)\widehat{\phi}(t)\omega(\langle a,1\rangle)|a|^{\zeta+rk-1}\,dt\,dv\,d^*a
\\&=
\int\limits_{F^*}\int W_{f'}(\langle\diag(I_{2rk-1},a),1\rangle\langle [0,v],\varsigma([0,v])\rangle
\langle w',1\rangle,s)\omega(\langle a,1\rangle)|a|^{\zeta+rk-1}\,dv\,d^*a,
\end{align*}
where the second equality follows from \eqref{eq:cond on FT of phi}.
Plugging \eqref{int:standard Whittaker on I W W} into the last integral, we obtain
\begin{align}\label{eq:to justify}
&\int\limits_{F^*}\int\limits_{F^{rk-2}}
\int\limits_{V_{(rk,rk)}}
f'(\langle\delta_0[\begin{smallmatrix}y&z\\u&x\end{smallmatrix}],1\rangle\langle\diag(I_{2rk-1},a),1\rangle\langle [0,v],\varsigma([0,v])\rangle
\langle w',1\rangle,s)\omega(\langle a,1\rangle)\psi(u)|a|^{\zeta+rk-1}\\&\,dx\,dy\,dz\,du\,dv\,d^*a.\nonumber
\end{align}
This formal step will be justified by the argument below.
Since $\diag(I_{2rk-1},a)$, $[0,v]$ and $w'$ all normalize $V_{(rk,rk)}$, and also
\begin{align*}
{}^{\delta_0}\langle \diag(I_{2rk-1},a),1\rangle=\langle {}^{\delta_0}\diag(I_{2rk-1},a),1\rangle, \qquad
{}^{\delta_0}[0,v]\in N_H^-, \qquad
{}^{\delta_0}\langle w',1\rangle=\langle {}^{\delta_0}w',1\rangle,
\end{align*}
we can shift $\diag(I_{2rk-1},a)$, $[0,v]$ and $w'$ to the left and obtain
\begin{align*}
&\int\limits_{F^*}\int\limits_{F^{rk-2}}
\int\limits_{V_{(rk,rk)}}
f'(\langle{}^{\delta_0}\diag(a,I_{2rk-1}),1\rangle
\langle {}^{\delta_0}[0,v],\varsigma({}^{\delta_0}[0,v])\rangle
\langle{}^{\delta_0}w',1\rangle\langle\delta_0[\begin{smallmatrix}y&z\\u&x\end{smallmatrix}],1\rangle
,s)\omega(\langle a,1\rangle)\psi(u)\\&|a|^{\zeta-1}\,dx\,dy\,dz\,du\,dv\,d^*a.
\end{align*}
Now as with $Z'(s,\omega,f)$, the coordinates $x,y,z,u$ are all small (due to the choice of support for $\langle\delta_0,1\rangle\cdot f'$), and since $\langle\delta_0,1\rangle\cdot f'$ is right-invariant on $\{\langle y,1\rangle:y\in \mathcal{N}\cap N_H\}$ ($\eta^{\diamondsuit}_{2rk}$ is trivial on $K_H\cap N_H$), we can use \eqref{eq:cond on W of f'} to obtain
\begin{align}\label{int:right hand side after FT}
&D\int\limits_{F^*}\int\limits_{F^{rk-2}}
W(\langle\diag(I_{rk-1},a),1\rangle
\langle\left(\begin{smallmatrix}1\\&I_{rk-2}\\&v&1\end{smallmatrix}\right),\varsigma'(\left(\begin{smallmatrix}1\\&I_{rk-2}\\&v&1\end{smallmatrix}\right))\rangle
\langle\left(\begin{smallmatrix}&I_{rk-1}\\1&\end{smallmatrix}\right),1\rangle)\\&\omega(\langle a,1\rangle)|a|^{\zeta+rks-3/2+rk-(rk-1)/2}\,dv\,d^*a.\nonumber
\end{align}
Here $v\mapsto \langle v,\varsigma'(v)\rangle$ is the splitting of $N_{\GL_{rk}}$ in $\GL_{rk}^{(m,r)}$. This also justifies \eqref{eq:to justify}, because the argument on the support can be applied to $|f'|$, then \eqref{eq:to justify} is absolutely convergent as
a multiple integral for $\Real(\zeta)\ll0$, since this is true for \eqref{int:right hand side after FT} (see \cite[\S~2]{JPSS}).

Next we show that the coordinates of $v$ belong in $\mathcal{O}$ (assuming $rk>2$). Put $v=(v_1,\ldots,v_{rk-2})$.
For simplicity, re-denote $[0,v]=\diag(1,\left(\begin{smallmatrix}I_{rk-2}\\v&1\end{smallmatrix}\right))$ and $w'=\left(\begin{smallmatrix}&I_{rk-1}\\1&\end{smallmatrix}\right)$.
Let $\mathrm{e}(o_1)=\diag(\left(\begin{smallmatrix}1\\o_1&1\end{smallmatrix}\right),I_{rk-2})$, for $o_1\in\mathcal{O}$.
Then $W$ is right-invariant on the elements $\langle \mathrm{e}(o_1),\eta^{\diamondsuit}_{rk}(\mathrm{e}(o_1))\rangle$. Since
$y\mapsto\langle y,\eta^{\diamondsuit}_{rk}(y)\rangle$ is a splitting of $K_{\GL_{rk}}$ in $\GL_{rk}^{(m,r)}$;
there is a unique splitting of $\mathrm{e}(\mathcal{O})$ in $\GL_{rk}^{(m,r)}$ (the exponent of $\mu_m$ is $m$ and $|m|=1$);
and ${}^{w'}\mathrm{e}(\mathcal{O})\in N_{\GL_{rk}}$, \eqref{eq:epsilon for conjugation between split subgroups} implies
\begin{align*}
{}^{w'}\langle \mathrm{e}(o_1),\eta^{\diamondsuit}_{rk}(\mathrm{e}(o_1))\rangle=
\langle {}^{w'}\mathrm{e}(o_1),1\rangle.
\end{align*}
We also have ${}^{[0,v]w'}\mathrm{e}(o_1)\in N_{\GL_{rk}}$, therefore
\begin{align*}
&W(\langle\diag(I_{rk-1},a),1\rangle\langle[0,v],\varsigma'([0,v])\rangle\langle w',1\rangle)\\&=
W(\langle\diag(I_{rk-1},a),1\rangle\langle[0,v],\varsigma'([0,v])\rangle\langle w',1\rangle\langle\mathrm{e}(o_1),\eta^{\diamondsuit}_{rk}(\mathrm{e}(o_1))\rangle)
\\&=
W(\langle\diag(I_{rk-1},a),1\rangle\langle\left(\begin{smallmatrix}1&-o_1v&o_1\\&I_{rk-2}\\&&1\end{smallmatrix}\right),1\rangle
\langle[0,v],\varsigma'([0,v])\rangle\langle w',1\rangle)
\\&=\psi^{-1}(o_1v_1)W(\langle\diag(I_{rk-1},a),1\rangle
\langle[0,v],\varsigma'([0,v])\rangle\langle w',1\rangle).
\end{align*}
Taking $o_1\in\mathcal{O}^*$, we deduce the integrand of \eqref{int:right hand side after FT} vanishes unless $v_1\in\mathcal{O}$. Then since ${}^{{w'}^{-1}}[0,v]\in N_{\GL_{rk}}$,
we can argue as in \eqref{eq:resuse conj t to the right} to show that $v_1$ can be omitted from the integrand, and we remain with the measure of $\mathcal{O}$ which equals $1$. We proceed similarly (assuming $rk>3$), with
\begin{align*}
\mathrm{e}(o_2,\ldots,o_{rk-2})=\left(\begin{smallmatrix}1\\0&1\\o_2&&1\\\vdots&&&\ddots\\o_{rk-2}&&&&1\\0&&&&&1\end{smallmatrix}\right),
\end{align*}
in increasing order ($o_2\in\mathcal{O}^*$ and the remaining coordinates zero to show $v_2\in\mathcal{O}$, etc.).
Hence \eqref{int:right hand side after FT} equals
\begin{align*}
&D\int\limits_{F^*}W(\langle\diag(I_{rk-1},a),1\rangle)\omega(\langle a,1\rangle)|a|^{\zeta+rks-3/2+rk-(rk-1)/2}\,d^*a.
\end{align*}
As in the proof of
Proposition~\ref{proposition:unramified computation for GL(1) other integral} and by \eqref{RS integral for n=1}, this integral equals
\begin{align*}
&D\int\limits_{F^*}W^*(\langle\diag(a,I_{rk-1}),1\rangle)\omega^*(\langle a,1\rangle)|a|^{-\zeta-rks+1/2-(rk-1)/2}\,d^*a\\&
=DZ(-\zeta-rks+1/2,\omega^*,W^*)
=DL_{\vartheta}(-r\zeta-r\alpha s+1/2,\pi\times\tau^{\vee}).
\end{align*}
Note that $W^*$ is indeed normalized, unramified and belongs to $\mathcal{W}(\Theta_{r,m,r,\vartheta}(\chi))^*=
\mathcal{W}(\Theta_{r,m,r,\vartheta}(\chi)^*)$, which is
$\mathcal{W}(\Theta_{r,m,r,\vartheta}(\chi^*))$ by Proposition~\ref{proposition:* of Theta chi}.
We obtain
\begin{align*}
\Psi(\zeta,s,\omega,f)=DL_{\vartheta}(-r\zeta-r\alpha s+1/2,\pi\times\tau^{\vee}),
\end{align*}
in $\C(q^{-\zeta},q^{-s})$. Now taking $\zeta=0$, we deduce
\begin{align*}
\Psi(0,s,\omega,f)=DL_{\vartheta}(-r\alpha s+1/2,\pi\times\tau^{\vee}).
\end{align*}
Comparing this to \eqref{eq:lhs of compare GL} and using the definition of
$L_{\pi,\tau,\vartheta}(s)$ we conclude \eqref{eq:identity with Z'}.
\end{proof}
As a corollary we finally obtain the computation of the integral
with unramified data:
\begin{corollary}\label{corollary:unramified computation for GL(1)}
Let $\omega$ and $f$ be the normalized unramified vectors. Then
\begin{align*}
Z(s,\omega,f)
=\frac{L_{\vartheta}(r\alpha s+1/2,\pi^{\vee}\times\tau)L_{\vartheta}(r\alpha s+1/2,\pi\times{\tau'}^{\vee})}{\prod_{j=1}^{r}L_{\vartheta}(2r\alpha s+j,\tau\times{\tau'}^{\vee})}.
\end{align*}
\end{corollary}
\begin{proof}
Immediate from Lemma~\ref{lemma:substitution}, Proposition~\ref{proposition:unramified computation for GL(1) other integral} and the definition of $L_{\pi,\tau,\vartheta}(s)$.
\end{proof}

\subsection{Concluding the computation}\label{putting it together}
\begin{theorem}\label{theorem:unr GL}
Let $\pi$ be a genuine irreducible unramified representation of $\GL_n$, and $\tau$ and $\tau'$ be genuine irreducible unramified representations of $\GL_k^{(m,r)}$. Assume the conditions on $\tau$ and $\tau'$ from \S~\ref{proof of lemma:reduction from GLn to GLa GLb}. Let
$\omega$ and $f_{\mathcal{W}(\rho_n(\tau))\otimes \mathcal{W}(\rho_n(\tau'))}$ be normalized and unramified. Put $\alpha=rkn$. Then
\begin{align*}
&Z(s,\omega,f_{\mathcal{W}(\rho_n(\tau))\otimes \mathcal{W}(\rho_n(\tau'))})=
\frac{L_{\vartheta}(r\alpha s+1/2,\pi^{\vee}\times\tau)L_{\vartheta}(r\alpha s+1/2,\pi\times{\tau'}^{\vee})}{\prod_{j=1}^{rn}L_{\vartheta}(2r\alpha s+j,\tau\times{\tau'}^{\vee})}.
\end{align*}
\end{theorem}
\begin{proof}
The proof is by induction on $n$. The base case $n=1$ is Corollary~\ref{corollary:unramified computation for GL(1)}. Assuming the result for $n-1$, we deduce the result for $n$ by applying Lemma~\ref{lemma:reduction from GLn to GLa GLb} with $a=1$ and $b=n-1$, using the base case for $n=1$ and the induction hypothesis on $n-1$, and combining this with \eqref{eq:GL d tau tau}. For more details see \cite[Theorem~28]{CFGK2}, proved in exactly the same way.
\end{proof}
\begin{proof}[Proof of Theorem~\ref{theorem:unramified computation for Sp(2n),SO(2n)}]
Using local notation, let $\omega$ and $f_{\mathcal{W}(\rho_c(\tau))}$ be normalized and unramified. Also recall $\alpha=rkc+1$. We need to show
\begin{align*}
&Z(s,\omega,f_{\mathcal{W}(\rho_c(\tau))})=
\frac{L_{\vartheta}(r\alpha s+1/2,\pi\times\tau)}
{[L_{\vartheta}(r\alpha s+rn+1/2,\tau)]\prod\limits_{1\leq j\leq rn}L_{\vartheta}(2r\alpha s+2j,\tau,\wedge^2)
L_{\vartheta}(2r\alpha s+2j-1,\tau,\vee^2)}.
\end{align*}
Here $L_{\vartheta}(r\alpha s+rn+1/2,\tau)$ appears only for odd $m$. The first step is to apply
Lemma~\ref{lemma:reduction from classical to GLn}:
\begin{align*}
Z(s,\omega,f_{\mathcal{W}(\rho_c(\tau))})=d_{\tau,\vartheta}(s)Z(\alpha s/(rkn),\omega_{n},f_{\mathcal{W}(\rho_c(\tau))\otimes \mathcal{W}(\rho_c(\tau^*))}).
\end{align*}
For the $\GL_n^{(m,r)}\times\GL_k^{(m,r)}$ integral on the r.h.s., the representation $\tau'$ is $\tau^*$, hence as noted in
\S~\ref{Outline of the computation} and \S~\ref{proof of lemma:reduction from GLn to GLa GLb}, $\tau$ and $\tau'$ satisfy the additional assumptions
of \S~\ref{proof of lemma:reduction from GLn to GLa GLb}, in particular \eqref{eq:invariance prop on GL or SL} which was
needed for this integral. Hence we can proceed to apply Theorem~\ref{theorem:unr GL}.
We also have $L_{\vartheta}(s,\tau\times{\tau^*}^{\vee})=
L_{\vartheta}(s,\tau\times\tau)$, because by
Proposition~\ref{proposition:vartheta of *} and since $\vartheta=\vartheta^{-1}$, $t_{(\tau^*)^{\vee},\vartheta}=
t_{(\tau^*)^{*},\vartheta}=t_{\tau,\vartheta}$.
The result follows when we combine this with \eqref{eq:d tau(s)}, \eqref{eq:standard sym and ext} and
\eqref{eq:standard L of pi and tau identity}.
\end{proof}
\begin{remark}\label{remark:unconditionally local}
As mentioned in the introduction, the local theory can be stated independently of the conjectures of \S~\ref{speh gbl}.
In this case define $\rho_l(\tau)=\Theta_{rl,m,r,\vartheta}(\chi)$, and the computation of the integrals with unramified data remains valid in the sense of meromorphic continuation in $\mathbf{x}$.
\end{remark}

\appendix
\section{The global $\GL_c^{(m,r)}(\A)\times\GL_k^{(m,r)}(\A)$ integral}\label{gbl GL}
Let $F$ be a number field containing $\mu_{2m}$, but note that the assumption $\mu_{2m}<F^*$ will only be used in Lemma~\ref{AppLemma:invariance of inner under G iota delta} below. We use the definitions and notation of \S~\ref{global covering} and \S~\ref{covering of the Levi}.
Recall $\rho_{2d}$ is the global $2$-cocycle for $\Sp_{2d}^{(m)}(\A)$ and $\GL_d^{(m,r)}(\A)$ is by definition realized using $\rho^{\diamondsuit}_{d}$. Then $y\mapsto\langle y,(\eta^{\diamondsuit}_{d})^{-1}(y)\rangle$ is the splitting of $N_{\GL_d}(\A)$ and a splitting of $\GL_d(F)$. Automorphic forms on $\GL_d^{(m,r)}(\A)$ are defined with respect to this fixed splitting of $\GL_d(F)$.
We also defined a global block-compatible cocycle $\rho_{\beta}$ and by Propositions~\ref{proposition:rho beta and rho are cohomologous locally} and \ref{proposition:gbl rho beta and rho are cohomologous}, $\eta_{\beta}=\prod_{\nu}\eta_{\beta,\nu}$ is well defined on $M_{\beta}(\A)$ and
\begin{align}\label{Appendixeq:rho beta and rho square globally}
\rho_{\beta}(m,m')=\frac{\eta_{\beta}(m)\eta_{\beta}(m')}{\eta_{\beta}(mm')}\rho^{\diamondsuit}_{d}(m,m').
\end{align}

Let $G=\GL_c$ and $H=\GL_{2rkc}$. We use the notation and definitions of \S~\ref{integarls for GL} in a global setup. Define additional $2$-cocycles for $G^{(m,r)}(\A)$ as follows:
\begin{align}\label{Appendixgbl rho on left and right}
\rho_L(g,g')=(\rho^{\diamondsuit}_{2rkc})^{-1}(\mathfrak{e}_1(g),\mathfrak{e}_1(g')),\qquad\rho_R(g,g')=\rho^{\diamondsuit}_{2rkc}(\mathfrak{e}_2(g),\mathfrak{e}_2(g')).
\end{align}
Since \eqref{eq:the $2$-cocycle on G times G formula GL} implies
$\sigma^{\diamondsuit}_{2rkc,\nu}(\mathfrak{e}_1(g),\mathfrak{e}_1(g'))^{-1}=\sigma^{\diamondsuit}_{c,\nu}(g,g')$,
$\rho_{L,\nu}=\sigma_{c,\nu}^{\diamondsuit}=\rho_{R,\nu}$ in $\mathrm{H}^2(G(F_{\nu}),\mu_m)$, whence the induced coverings on both copies of
$G(\A)$ are isomorphic to the covering $G^{(m,r)}(\A)$ (realized as explained above, with $\rho_c^{\diamondsuit}$). Let $G^{(m,r)}(\A)[\rho_L]$ and $G^{(m,r)}(\A)[\rho_R]$ denote the realizations of $G^{(m,r)}(\A)$ using $\rho_L$ and $\rho_R$ (resp.). We have global embeddings analogous to \eqref{eq:GL embeddings coverings G and G into H}:
\begin{align}\label{Appendixeq:gbl embedding G left in H}
&G^{(m,r)}(\A)[\rho_L]\rightarrow H^{(m,r)}[\A],\qquad \langle g,\epsilon\rangle\mapsto \langle \mathfrak{e}_1(g),\epsilon^{-1}\rangle,\\
\label{Appendixeq:gbl embedding G right in H}
&G^{(m,r)}(\A)[\rho_R]\rightarrow H^{(m,r)}[\A],\qquad \langle g,\epsilon\rangle\mapsto \langle \mathfrak{e}_2(g),\epsilon\rangle.
\end{align}
According to the local relation \eqref{eq:block compatibility on Levi of P}, globally
\begin{align}\label{Appendixeq:gbl commuting G and G}
\langle \mathfrak{e}_1(g_1),1\rangle \langle \mathfrak{e}_2(g_2),1\rangle =
\langle \mathfrak{e}_2(g_2),1\rangle \langle \mathfrak{e}_1(g_1),1\rangle,\qquad\forall g_1,g_2\in G(\A).
\end{align}
Thus as in the local case (see \S~\ref{integarls for GL}), the embedding $G(\A)\times G(\A) \rightarrow H(\A)$ lifts to an embedding
\begin{align*}
\{(\epsilon_1,\epsilon_2)\in\mu_m^2:\epsilon_1=\epsilon_2\}\backslash G^{(m,r)}(\A)[\rho_L]\times G^{(m,r)}(\A)[\rho_R] \rightarrow H^{(m,r)}(\A).
\end{align*}

To relate between $G^{(m,r)}(\A)[\rho_L]$, $G^{(m,r)}(\A)[\rho_R]$ and $G^{(m,r)}(\A)=G^{(m,r)}(\A)[\rho^{\diamondsuit}_{c}]$, observe that because $\rho_{(c^{2rk})}=\rho^{\diamondsuit}_{2rkc}$ in $\mathrm{H}^2(M_{(c^{2rk})}(\A),\mu_m)$ and also
\begin{align*}
\rho_{(c^{2rk})}(\mathfrak{e}_1(g),\mathfrak{e}_1(g'))=(\rho^{\diamondsuit}_{c})^{2rk-1}(g,g')=(\rho^{\diamondsuit}_{c})^{-1}(g,g'),\qquad
\rho_{(c^{2rk})}(\mathfrak{e}_2(g),\mathfrak{e}_2(g'))=\rho^{\diamondsuit}_{c}(g,g'),
\end{align*}
when we use
\eqref{Appendixeq:rho beta and rho square globally} and \eqref{Appendixgbl rho on left and right}
we obtain
\begin{align*}
&G^{(m,r)}(\A)[\rho_L]\rightarrow G^{(m,r)}(\A),\qquad \langle g,\epsilon\rangle\mapsto \langle g,\eta_{(c^{2rk})}(\mathfrak{e}_1(g))\epsilon\rangle,\\
&G^{(m,r)}(\A)[\rho_R]\rightarrow G^{(m,r)}(\A),\qquad \langle g,\epsilon\rangle\mapsto \langle g,\eta_{(c^{2rk})}^{-1}(\mathfrak{e}_2(g))\epsilon\rangle.
\end{align*}
Dualizing, for a function $\varphi$ on $G^{(m,r)}(\A)$
we have the following functions $\varphi^{L}$ and $\varphi^{R}$ on $G^{(m)}(\A)[\rho_L]$ and $G^{(m)}(\A)[\rho_R]$,
\begin{align}\label{Appendixeq:gbl iso rhoR rhoL for functions}
\varphi^{L}(\langle g,\epsilon\rangle)=
\varphi(\langle g,\eta_{(c^{2rk})}(\mathfrak{e}_1(g))\epsilon\rangle),\qquad
\varphi^{R}(\langle g,\epsilon\rangle)=
\varphi(\langle g,\eta_{(c^{2rk})}^{-1}(\mathfrak{e}_2(g))\epsilon\rangle).
\end{align}

\begin{appproposition}\label{Appendixproposition:global toy integral automorphic}
Let $\varphi_1,\varphi_2$ be continuous functions on $G(F)\backslash G^{(m,r)}(\A)$, such that $\varphi_1$ is genuine and $\varphi_2$ is anti-genuine, and $f$ be a
continuous genuine function on the image of
\begin{align*}G(F)\times G(F)\backslash(G^{(m,r)}(\A)[\rho_L]\times G^{(m,r)}(\A)[\rho_R])
\end{align*}
in $H(F)\backslash H^{(m,r)}(\A)$ (e.g., $f$ on $H(F)\backslash H^{(m,r)}(\A)$). Then
\begin{align*}
\int\limits_{G(F)\times G(F)\backslash G(\A)\times G(\A)}\varphi_1^{L}(\langle g_1,1\rangle)\varphi_2^{R}(\langle g_2,1\rangle)f(\langle\mathfrak{e}_1(g_1),1\rangle\langle\mathfrak{e}_2(g_2),1\rangle)\,dg_1\,dg_2
\end{align*}
is well defined, provided it is absolutely convergent. (To ensure convergence one typically deals with the center of $G(\A)$.)
\end{appproposition}
\begin{proof}
The integrand is well defined on $G(\A)\times G(\A)$ by
\eqref{Appendixeq:gbl embedding G left in H}--\eqref{Appendixeq:gbl embedding G right in H}. Specifically, if we
replace the section $g_1\mapsto\langle g_1,1\rangle$ of $G(\A)\rightarrow G^{(m,r)}(\A)[\rho_L]$ with
$g_1\mapsto \langle g_1,\epsilon_{g_1}\rangle$, and similarly take a section $g_2\mapsto \langle g_2,\epsilon'_{g_2}\rangle$ of $G(\A)\rightarrow G^{(m,r)}(\A)[\rho_R]$, the integrand does not change:
\begin{align*}
&\varphi_1^{L}(\langle g_1,\epsilon_{g_1}\rangle)\varphi_2^{R}(\langle g_2,\epsilon'_{g_2}\rangle)f(\langle\mathfrak{e}_1(g_1),\epsilon_{g_1}^{-1}\rangle\langle\mathfrak{e}'_2(g_2),\epsilon'_{g_2}\rangle)
\\&=\varphi_1^{L}(\langle g_1,1\rangle)\varphi_2^{R}(\langle g_2,1\rangle)f(\langle\mathfrak{e}_1(g_1),1\rangle\langle\mathfrak{e}_2(g_2),1\rangle)
\end{align*}
Then \eqref{Appendixgbl rho on left and right} and \eqref{Appendixeq:gbl commuting G and G} imply the integral is formally a right-invariant functional on $G(\A)\times G(\A)$:
\begin{align*}
&\varphi_1^{L}(\langle g_1,1\rangle\langle h_1,1\rangle)=\rho_L(g_1,h_1)\varphi_1^{L}(\langle g_1h_1,1\rangle)=(\rho_{2rkc}^{\diamondsuit})^{-1}(\mathfrak{e}_1(g_1),\mathfrak{e}_1(h_1))\varphi_1^{L}(\langle g_1h_1,1\rangle),\\
&\varphi_2^{R}(\langle g_2,1\rangle\langle h_2,1\rangle)=\rho_R(g_2,h_2)\varphi_2^{R}(\langle g_2h_2,1\rangle)
=\rho_{2rkc}^{\diamondsuit}(\mathfrak{e}_2(g_2),\mathfrak{e}_2(h_2))\varphi_2^{R}(\langle 1,g_2h_2\rangle),\\
&f(\langle \mathfrak{e}_1(g_1),1\rangle\langle \mathfrak{e}_2(g_2),1\rangle
\langle \mathfrak{e}_1(h_1),1\rangle\langle \mathfrak{e}_2(h_2),1\rangle)\\&\quad=
\rho_{2rkc}^{\diamondsuit}(\mathfrak{e}_1(g_1),\mathfrak{e}_1(h_1))\rho_{2rkc}^{\diamondsuit}(\mathfrak{e}_2(g_2),\mathfrak{e}_2(h_2))f(\langle
\mathfrak{e}_1(g_1h_1),1\rangle\langle\mathfrak{e}_2(g_2h_2),1\rangle).
\end{align*}
Finally let $y_1,y_2\in G(F)$, $\epsilon_i=\eta_{2rkc}^{\diamondsuit}(\mathfrak{e}_i(y_i))$ and $g_1,g_2\in G(\A)$. The crucial point here is that $\eta_c^{\diamondsuit}$ is well defined on
$G(F)$, hence \eqref{eq:eta beta nu} globalizes and $\eta_{2rkc}^{\diamondsuit}(b)\eta_{(c^{2rk})}(b)=\prod_{i=1}^{2rk}
\eta_{c}^{\diamondsuit}(b_i)$, so that
\begin{align*}
\epsilon_1\eta_{(c^{2rk})}(\mathfrak{e}_1(y_1))=\prod_{i=1}^{2rk}
(\eta_{c}^{\diamondsuit})(y_1)=(\eta_{c}^{\diamondsuit})^{-1}(y_1),\quad
\epsilon_2\eta_{(c^{2rk})}(\mathfrak{e}_2(y_2))=\eta_{c}^{\diamondsuit}(y_2).
\end{align*}
Combining this with \eqref{Appendixeq:gbl iso rhoR rhoL for functions}, we have
\begin{align*}
&\varphi_1^{L}(\langle y_1,\epsilon_1\rangle\langle g_1,1\rangle)=
\varphi_1(\langle y_1,\epsilon_1\eta_{(c^{2rk})}(\mathfrak{e}_1(y_1))\rangle\langle g_1,1\rangle)
=\varphi_1(\langle y_1,(\eta_{c}^{\diamondsuit})^{-1}(y_1)\rangle\langle g_1,1\rangle)
=\varphi_1(\langle g_1,1\rangle),\\
&\varphi_2^{R}(\langle y_2,\epsilon_2^{-1}\rangle\langle g_2,1\rangle)=
\varphi_2(\langle y_2,\epsilon_2^{-1}\eta_{(c^{2rk})}^{-1}(\mathfrak{e}_2(y_2)\rangle\langle g_2,1\rangle)
=\varphi_2(\langle y_2,(\eta_{c}^{\diamondsuit})^{-1}(y_2)\rangle\langle g_2,1\rangle)
=\varphi_2(\langle g_2,1\rangle).
\end{align*}
Therefore
\begin{align}\label{Appendixeq:y_1 y_2}
&\varphi_1^{L}(\langle y_1g_1,1\rangle)\varphi_2^R(\langle y_2g_2,1\rangle)f(
\langle\mathfrak{e}_1(y_1g_1),1\rangle\langle\mathfrak{e}_2(y_2g_2),1\rangle)
\\&=\varphi_1^{L}(\langle y_1,1\rangle \langle g_1,1\rangle )\varphi_2^{R}(\langle y_2,1\rangle \langle g_2,1\rangle ) f(\nonumber
\langle\mathfrak{e}_1(y_1),1\rangle\langle\mathfrak{e}_2(y_2),1\rangle\langle\mathfrak{e}_1(g_1),1\rangle\langle\mathfrak{e}_2(g_2),1\rangle)\nonumber
\\&=\varphi_1^{L}(\langle y_1,\epsilon_1\rangle \langle g_1,1\rangle )\varphi_2^{R}(\langle y_2,\epsilon_2^{-1}\rangle \langle g_2,1\rangle ) f(\nonumber
\langle\mathfrak{e}_1(y_1),\epsilon_1^{-1}\rangle\langle\mathfrak{e}_2(y_2),\epsilon_2^{-1}\rangle\langle\mathfrak{e}_1(g_1),1\rangle\langle\mathfrak{e}_2(g_2),1\rangle)\nonumber\\
&=\varphi_1^{L}(\langle g_1,1\rangle)\nonumber
\varphi_2(\langle g_2,1\rangle)f(\langle\mathfrak{e}_1(g_1),1\rangle\langle\mathfrak{e}_2(g_2),1\rangle)\nonumber.
\end{align}
This completes the proof.
\end{proof}

We turn to construct the integral, following the linear case from \cite{CFK}. Recall $P=M_P\ltimes U_P$ ($P=P_{(rkc,rkc)}$).
Let $\tau_0$ be a genuine cuspidal representation of $\GL_{k}^{(m,r)}(\A)$, and
$\mathcal{E}_{\tau_0}$ be the representation defined in \S~\ref{speh gbl}.
Realize the covering $\widetilde{M}_P(\A)$ using $\rho_{(rkc,rkc)}$, then we can define the genuine automorphic representation $\mathcal{E}_{\tau_0}\otimes\mathcal{E}_{\tau_0^*}$ of
$\widetilde{M}_P(\A)$. The space of the induced representation
\begin{align}\label{Appendixrep:gbl parabolic induction on GL}
\Ind_{\widetilde{P}({\A})}^{H^{(m,r)}({\A})}(|\det|^{r^{-1}(s-1/2)}\mathcal{E}_{\tau_0}\otimes\mathcal{E}_{\tau_0^*})
\end{align}
is the space of genuine functions $f$ on $H^{(m,r)}(\A)$ taking values in the space of $\mathcal{E}_{\tau_0}\otimes\mathcal{E}_{\tau_0^*}$, satisfying
\begin{align}\label{Appendixrep:gbl parabolic induction on GL properties}
f(s,\langle b,1\rangle \langle u,(\eta_{2rkc}^{\diamondsuit})^{-1}(u)\rangle g)=
\eta_{(rkc,rkc)}^{-1}(b)|\det{b}|^{r^{-1}(s-1/2)}\delta_P^{1/2}(b)(\mathcal{E}_{\tau_0}\otimes\mathcal{E}_{\tau_0^*})(\langle b,1\rangle)f(s,g),
\end{align}
for all $b=\diag(b_1,b_2)\in M_{P}(\A)$, $u\in U_P(\A)$ and $g\in \GL_{d}^{(m,r)}(\A)$. Here $\det{b}=\det(b_1b_2^{-1})$. We regard these functions as complex-valued by evaluating at the identity.

According to \eqref{eq:splitting of K H} and \eqref{eq:nu and sigma for covering of H}, $y\mapsto\langle y,1\rangle$ is a splitting of $K_{G,\nu}$ in $G^{(m,r)}(\A)$ and $K_{H,\nu}$ in $H^{(m,r)}(\A)$, which is compatible with our local choices (see the paragraph following \eqref{eq:M beta as a quotient}). Hence $y\mapsto\langle y,1\rangle$ is a splitting of $\prod_{\nu\notin S}K_{H,\nu}$ for any set $S$ such that $F_{\nu}$ is unramified for all $\nu\notin S$.

Let $f$ be a standard
$\widetilde{K}_H$-finite section in the space of \eqref{Appendixrep:gbl parabolic induction on GL}. Define for $\Real(s)\gg0$,
\begin{align}\label{Appendixeq:ES for Hm}
E(h;s,f)=\sum_{y\in P(F)\backslash H(F)}f(s,\langle y,(\eta_{2rkc}^{\diamondsuit})^{-1}(y)\rangle h),\qquad h\in H^{(m,r)}(\A)
\end{align}
and in general by meromorphic continuation. Let
\begin{align*}
E^{U,\psi_U}(h;s,f)=
\int\limits_{U(F)\backslash U({\A})}
E(\langle u,(\eta_{2rkc}^{\diamondsuit})^{-1}(u)\rangle h;s,f)\,\psi_U(u)\,du.
\end{align*}

Let $\pi$ be a genuine cuspidal representation of $G^{(m,r)}(\A)$, $\varphi_1$ be a cusp form
in the space of $\pi$ and $\varphi_2$ be a cusp form in the space of $\pi^{\vee}$ (if $\pi$ is unitary, $\pi^{\vee}=\overline{\pi}$).
Since $G\times G$ normalizes $U$ and stabilizes $\psi_U$ and $du$, by Lemma~\ref{lemma:conjugation of N by H}, $E^{U,\psi_U}(\cdot;s,f)$ is an automorphic function on the image of $G(F)\times G(F)\backslash(G^{(m,r)}(\A)[\rho_L]\times G^{(m,r)}(\A)[\rho_R])$ in
$H(F)\backslash H^{(m,r)}(\A)$, thus we can define the global integral as in
Proposition~\ref{Appendixproposition:global toy integral automorphic}, once we take care of convergence.

Let $C_{r,c}(\A)=\{xI_{c}:x\in \A^{*r}\}$, it is the center $C_{G^{(m,r)}(\A)}$ of $G^{(m,r)}(\A)$.
The product
\begin{align}\label{Appendixintegrand}
\varphi_1^{L}(\langle g_1,1\rangle)\varphi_2^R(\langle g_2,1\rangle)
E^{U,\psi_U}(\langle \mathfrak{e}_1(g_1),1\rangle\langle \mathfrak{e}_2(g_2),1\rangle;s,f)
\end{align}
is invariant under multiplication $(g_1,g_2)\mapsto(zg_1,zg_2)$ for $z\in \widetilde{C}_{r,c}(\A)$. Since the index of $C_{r,c}(\A)$ in $C_{G(\A)}$ is infinite, dividing by $\widetilde{C}_{r,c}(\A)$ is not enough, one must consider a larger subgroup $C'$. We argue as in \cite[pp.~159--160]{BG}. Let $S'$ be a finite set of places of $F$ such that for $\nu\notin S'$, $\varphi_1$ and $\varphi_2$ are right-invariant on $\{\langle y,1\rangle:y\in K_{G,\nu}\}$ and $f$ (and thereby $E^{U,\psi_U}(\cdot;s,f)$) is right-invariant on
$\{\langle y,1\rangle:y\in K_{H,\nu}\}$. Take $C_0'=\{xI_{c}:x\in F^*\A^{*r}\prod_{\nu\notin S'}\mathcal{O}^*_{\nu}\}$, then
$C_{r,c}(\A)<C_0'$, $[C_{G}(\A):C_0']<\infty$ and \eqref{Appendixintegrand} is invariant under $(g_1,g_2)\mapsto(zg_1,zg_2)$ for $z\in C_0'$.
Let $C'=\{xI_{2c}:x\in F^*\A^{*r}\prod_{\nu\notin S'}\mathcal{O}^*_{\nu}\}$. Additionally, take a compactly supported Schwartz function on $\R^*_{>0}$.

The global integral is first ``approximated" by the following integral:
\begin{align}\label{Appendixglglobal1}
Z(s,\varphi_1,\varphi_2,f,\varrho)=&\int\limits_{C'G(F)\times G(F)\backslash G({\A})\times G({\A})}\,
\varphi_1^{L}(\langle g_1,1\rangle)\varphi_2^R(\langle g_2,1\rangle)
\\&\times E^{U,\psi_U}(\nonumber
\langle \mathfrak{e}_1(g_1),1\rangle\langle \mathfrak{e}_2(g_2),1\rangle;s,f)\varrho(|\det(g_2g_1^{-1})|)\,dg_1\,dg_2.
\end{align}
(Cf., \cite[\S~3.4]{CFK} and \cite[\S~4.2]{PSR}.) Below we shall see how to remove the dependency on $\varrho$ to
define the global integral $Z(s,\varphi_1,\varphi_2,f)$.

\begin{apptheorem}\label{Appendixtheorem:main theorem classical groups}
Integral \eqref{Appendixglglobal1} is formally well defined, absolutely convergent away from the poles of the series, and admits meromorphic continuation to $\C$.
\end{apptheorem}
\begin{proof}
The integral is well defined by the discussion above and Proposition~\ref{Appendixproposition:global toy integral automorphic}.
Convergence and continuation follow from the rapid decay of cusp forms and moderate growth and meromorphic continuation
of the Eisenstein series, and the presence of $\varrho$ (see also \cite[\S~4.2]{PSR}).
\end{proof}

Next we prove \eqref{Appendixglglobal1} is Eulerian. Recall that for $\xi$ in the space of $\mathcal{E}_{\tau_0}$,
the $(rk,c)$ functional $\Lambda(\xi)$ was defined by \eqref{int:a c general Fourier coeff}.
The global $(rk,c)$ model is the space spanned by $\Lambda(\xi)$ as $\xi$ varies.
We then have the similar functional $\Lambda^*$ and model $W_{\psi}(\mathcal{E}_{\tau_0^*})$ for $\mathcal{E}_{\tau_0^*}$. Put
$W_{\psi}(\mathcal{E}_{\tau})=W_{\psi}(\mathcal{E}_{\tau_0})\otimes W_{\psi}(\mathcal{E}_{\tau_0^*})$.
Let $f_{W_{\psi}(\mathcal{E}_{\tau})}$ be the composition of $f$ with $\Lambda\otimes\Lambda^*$:
\begin{align*}
f_{W_{\psi}(\mathcal{E}_{\tau})}(s,h)=&\int\limits_{V_{(c^{rk})}(F)\backslash V_{(c^{rk})}(\A)}
\int\limits_{V_{(c^{rk})}(F)\backslash V_{(c^{rk})}(\A)}\\&\quad f(s,\langle \diag(v,v'),(\eta_{2rkc}^{\diamondsuit})^{-1}(\diag(v,v'))\rangle h)\psi^{-1}(v)\psi^{-1}(v')\,dv\,dv'.\nonumber
\end{align*}
Further let
\begin{align*}
\{\varphi_1^R,\varphi_2^R\}=
\int\limits_{C_0'G(F)\backslash G(\A)}\varphi_1^R(\langle g,1\rangle)\varphi_2^R(\langle g,1\rangle)\,dg.
\end{align*}
\begin{apptheorem}\label{Appendixtheorem:main gbl identity}
Assume the $(rk,c)$ functionals \eqref{int:a c general Fourier coeff} on $\mathcal{E}_{\tau_0}$ and $\mathcal{E}_{\tau_0^*}$ are factorizable
(this is only needed for Lemma~\ref{AppLemma:invariance of inner under G iota delta} below).
For $\Real(s)\gg0$, \eqref{Appendixglglobal1} equals
\begin{align}\label{Appendixglobal2}
\int\limits_{G({\A})}\int\limits_{U_0({\A})}
\{\varphi_1^R,\pi(\langle g,1\rangle)\varphi_2^R\}f_{W_{\psi}(\mathcal{E}_{\tau})}(s,\langle\delta u_0,(\eta_{2rkc}^{\diamondsuit})^{-1}(\delta u_0)\rangle
\langle\mathfrak{e}_2(g),1\rangle)
\,\psi_U(u_0)\varrho(|\det(g)|)\,du_0\,dg.
\end{align}
\end{apptheorem}
\begin{proof}
Let $D=(G\times G)\ltimes U<Q$ and for $h\in H$,
$D_h={}^{h^{-1}}P\cap D$. To lighten the notation, set $G^{\times}(\A)=G(\A)\times G(\A)$. For $\Real(s)\gg0$, $Z(s,\varphi_1,\varphi_2,f)$ equals
\begin{align*}
&\sum\limits_{h\in P(F)\backslash H(F)/D(F)}\quad\int\limits_{C'G^{\times}(F)\backslash G^{\times}(\A)}
\quad\int\limits_{U(F)\backslash U({\A})}\varrho(|\det(g_2g_1^{-1})|)\varphi_1^{L}(\langle g_1,1\rangle)\varphi_2^{R}(\langle g_2,1\rangle)\,\sum\limits_{y\in D_{h}(F)\backslash D(F)}
\\&\times f(s,\langle h,(\eta_{2rkc}^{\diamondsuit})^{-1}(h)\rangle\langle y,(\eta_{2rkc}^{\diamondsuit})^{-1}(y)\rangle\langle u,(\eta_{2rkc}^{\diamondsuit})^{-1}(u)\rangle\nonumber
\langle \mathfrak{e}_1(g_1),1\rangle\langle \mathfrak{e}_2(g_2),1\rangle)\,\psi_U(u)\,du\,dg_1\,dg_2.
\end{align*}
(See the argument for $\Sp_c^{(m)}(\A)$ in \S~\ref{global symplectic}.)
Write $y=y_0\mathfrak{e}_1(y_1)\mathfrak{e}_1(y_2)$ with $y_0\in U(F)$ and $y_1,y_2\in G(F)$. Since
$H^{(m,r)}(\A)$ is split over $H(F)$ with respect to $h\mapsto\langle h,(\eta_{2rkc}^{\diamondsuit})^{-1}(h)\rangle$,
using the global analog of \eqref{eq:sigma on h and v} (see Corollary~\ref{corollary:rho and eta on H and N without conjugation}), the fact that
$G\times G$ normalizes $U$ and stabilizes $\psi_U$, and reversing \eqref{Appendixeq:y_1 y_2}, we obtain
\begin{align*}
&\sum\limits_{h\in P(F)\backslash H(F)/D(F)}\quad\int\limits_{C'G^{\times}(F)\backslash G^{\times}(\A)}
\quad\int\limits_{U(F)\backslash U({\A})}\varrho(|\det(g_2g_1^{-1})|)\varphi_1^{L}(\langle y_1g_1,1\rangle)\varphi_2^{R}(\langle y_2g_2,1\rangle)\,\sum\limits_{y\in D_{h}(F)\backslash D(F)}
\\&\times f(s,\langle h,(\eta_{2rkc}^{\diamondsuit})^{-1}(h)\rangle\langle y_0u,(\eta_{2rkc}^{\diamondsuit})^{-1}(y_0u)\rangle\nonumber
\langle \mathfrak{e}_1(y_1g_1),1\rangle\langle \mathfrak{e}_2(y_2g_2),1\rangle)\,\psi_U(u)\,du\,dg_1\,dg_2.
\end{align*}
Thus when we collapse the inner sum into the integral we have $\sum_{h\in P(F)\backslash H(F)/D(F)}\mathrm{I}(h)$, where
\begin{align*}
\mathrm{I}(h)=
&\int\limits_{C'D_h(F)\backslash D(\A)}\varrho(|\det(g_2g_1^{-1})|)\varphi_1^{L}(\langle g_1,1\rangle)\varphi_2^{R}(\langle g_2,1\rangle)
\\&f(s,\langle h,(\eta_{2rkc}^{\diamondsuit})^{-1}(h)\rangle\langle u,(\eta_{2rkc}^{\diamondsuit})^{-1}(u)\rangle\nonumber
\langle \mathfrak{e}_1(g_1),1\rangle\langle \mathfrak{e}_2(g_2),1\rangle)\,\psi_U(u)\,du\,dg_1\,dg_2.
\end{align*}
For all $h$ such that $PhD\ne P\delta D$, $\mathrm{I}(h)=0$. This follows from \cite[Theorems~2.1, 3.1]{DimaKaplan} as explained
in \cite[\S~3.2]{DimaKaplan} (with $(k,c)$ of \cite[\S~3.1]{DimaKaplan} replaced by $(rk,c)$ here). In particular, as noted in \cite[\S~3.2]{DimaKaplan}, one has to use the constant term computation of $\mathcal{E}_{\tau}$ given by \cite[Lemma~4.1]{JngLiu}, which can be carried out for $\GL_{rkc}^{(m,r)}(\A)$ (with $a=rk$ in the notation of \textit{loc. cit.}).

Now we consider the remaining summand $\mathrm{I}(\delta)$ and show that it equals \eqref{Appendixglobal2}.
Denote an element in the direct product $G\times G$ by $[g,g']$ and
set $G^{\triangle}=\{[g,g]:g\in G\}$.
We have
\begin{align*}
&D_{\delta}=G^{\triangle}\ltimes{}^{\delta^{-1}}\diag(V_{(c^{rk})},V_{(c^{rk})}),\\
&U={}^{\delta^{-1}}\diag(V_{(c^{rk})},V_{(c^{rk})})\ltimes(U\cap U_P)={}^{\delta^{-1}}\diag(V_{(c^{rk})},V_{(c^{rk})})\ltimes U_0.
\end{align*}
($\delta$ does not normalize $\diag(V_{(c^{rk})},V_{(c^{rk})})$.)
For $u\in{}^{\delta^{-1}}\diag(V_{(c^{rk})},V_{(c^{rk})})$, set ${}^{\delta}u=\diag(v,v')$. Then
because both $\diag(V_{(c^{rk})},V_{(c^{rk})})$ and ${}^{\delta^{-1}}\diag(V_{(c^{rk})},V_{(c^{rk})})$ are subgroups of $N_{2rkc}$,
by Lemma~\ref{lemma:conjugation of N by H} and Corollary~\ref{corollary:rho and eta on H and N without conjugation} we obtain
\begin{align}\label{Appendixeq:conjugation of v_g}
{}^{\delta}\langle u,(\eta_{2rkc}^{\diamondsuit})^{-1}(u)\rangle=
\langle \diag(v,v'),(\eta_{2rkc}^{\diamondsuit})^{-1}(\diag(v,v'))\rangle.
\end{align}
Since \eqref{eq:eta beta nu} implies
\begin{align*}
\eta_{(rkc,rkc)}^{-1}(\diag(v,v'))(\eta_{2rkc}^{\diamondsuit})^{-1}(\diag(v,v'))=(\eta_{rkc}^{\diamondsuit})^{-1}(v)(\eta_{rkc}^{\diamondsuit})^{-1}(v'),
\end{align*}
when we combine \eqref{Appendixeq:conjugation of v_g} with \eqref{Appendixrep:gbl parabolic induction on GL properties}, and also use the fact that
$\psi_U({}^{\delta^{-1}}\diag(v,v'))$ equals the product of $(rk,c)$ characters $\psi^{-1}(v)\psi^{-1}(v')$,
we deduce
\begin{align*}
&\int\limits_{{}^{\delta^{-1}}\diag(V_{(c^{rk})}(F),V_{(c^{rk})}(F))\backslash U(\A)}
f(s,\langle\delta,(\eta_{2rkc}^{\diamondsuit})^{-1}(\delta)\rangle\langle u,(\eta_{2rkc}^{\diamondsuit})^{-1}(u)\rangle)\,\psi_U(u)\,du
\\&\quad=\int\limits_{U_0(\A)}
f_{W_{\psi}(\mathcal{E}_{\tau})}(s,\langle\delta,\eta^{-1}(\delta)\rangle\langle u_0,\eta^{-1}(u_0)\rangle
)\,\psi_U(u_0)\,du_0.
\end{align*}
Now $\mathrm{I}(\delta)$ equals
\begin{align*}
&\int\limits_{(C'\,G^{\triangle}(F))\backslash G^{\times}(\A)}
\quad\int\limits_{U_0(\A)}
\varrho(|\det(g_2g_1^{-1})|)\varphi_1^{L}(\langle g_1,1\rangle)\varphi_2^{R}(\langle g_2,1\rangle)
\\&f_{W_{\psi}(\mathcal{E}_{\tau})}(s,\langle \delta,(\eta_{2rkc}^{\diamondsuit})^{-1}(\delta)\rangle\langle u_0,(\eta_{2rkc}^{\diamondsuit})^{-1}(u_0)\rangle\nonumber
\langle \mathfrak{e}_1(g_1),1\rangle\langle \mathfrak{e}_2(g_2),1\rangle)\,\psi_U(u_0)\,du_0\,dg_1\,dg_2.
\end{align*}

Next, change variables $g_2\mapsto g_1g_2$, then $G^{\triangle}$ is mapped into $\{[g,I_c]:g\in G(F)\}$ which we identify with $G(F)$,
and $C'\mapsto \{[xI_c,I_c]:x\in F^*\A^{*r}\prod_{\nu\notin S'}\mathcal{O}^*_{\nu}\}=C_0'$. Also apply the definition \eqref{Appendixeq:gbl iso rhoR rhoL for functions} twice to $\varphi_1^{L}$. Then
\begin{align*}
\mathrm{I}(\delta)=
&\int\limits_{C_0'G(F)\backslash G^{\times}(\A)}
\quad\int\limits_{U_0(\A)}
\varrho(|\det g_2|)\varphi_1^{R}(\langle g_1,\eta_{(c^{2rk})}(\mathfrak{e}_1(g_1))\eta_{(c^{2rk})}(\mathfrak{e}_2(g_1))\rangle)\varphi_2^{R}(\langle g_1g_2,1\rangle)
\\&f_{W_{\psi}(\mathcal{E}_{\tau})}(s,\langle \delta,(\eta_{2rkc}^{\diamondsuit})^{-1}(\delta)\rangle\langle u_0,(\eta_{2rkc}^{\diamondsuit})^{-1}(u_0)\rangle\nonumber
\langle \mathfrak{e}_1(g_1),1\rangle\langle \mathfrak{e}_2(g_1g_2),1\rangle)\,\psi_U(u_0)\,du_0\,dg_1\,dg_2.
\end{align*}
Separating $\langle g_1g_2,1\rangle$ in $G^{(m,r)}(\A)[\rho_R]$ and $\langle \mathfrak{e}_2(g_1g_2),1\rangle$ in $H^{(m,r)}(\A)$,
the integral becomes
\begin{align}\label{Appendixint:I delta factoring change g_1 g_2}
&\int\limits_{C_0'G(F)\backslash G^{\times}(\A)}
\quad\int\limits_{U_0(\A)}\varrho(|\det g_2|)
\varphi_1^{R}(\langle g_1,\prod_{i=1}^2\eta_{(c^{2rk})}(\mathfrak{e}_i(g_1))\rangle)\varphi_2^{R}(\langle g_1,1\rangle\langle g_2,1\rangle)f_{W_{\psi}(\mathcal{E}_{\tau})}(s,
\\&\langle \delta,(\eta_{2rkc}^{\diamondsuit})^{-1}(\delta)\rangle\langle u_0,(\eta_{2rkc}^{\diamondsuit})^{-1}(u_0)\rangle\nonumber
\langle (g_1,g_1),\rho_{2rkc}^{\diamondsuit}(\mathfrak{e}_1(g_1),\mathfrak{e}_2(g_1))\rangle\langle \mathfrak{e}_2(g_2),1\rangle)\,\psi_U(u_0)\,du_0\,dg_1\,dg_2.\nonumber
\end{align}
When we apply \eqref{Appendixeq:rho beta and rho square globally} with $\beta=(c^{2rk})$, $m=\mathfrak{e}_1(g_1)$ and $m'=\mathfrak{e}_2(g_1)$, then use \eqref{eq:rho beta},
\begin{align*}
\prod_{i=1}^2\eta_{(c^{2rk})}(\mathfrak{e}_i(g_1))\rho_{2rkc}^{\diamondsuit}(\mathfrak{e}_1(g_1),\mathfrak{e}_2(g_1))
=\rho_{(c^{2rk})}(\mathfrak{e}_1(g_1),\mathfrak{e}_2(g_1))\eta_{(c^{2rk})}((g_1,g_1))=\eta_{(c^{2rk})}((g_1,g_1))\,(!).
\end{align*}
Conjugate $U_0$ by $(g_1,g_1)$, ${}^{(g_1,g_1)^{-1}}u_0=v_{g_1}u_{g_1}$ with $v_{g_1}\in {}^{\delta^{-1}}\diag(V_{(c^{rk})}(\A),V_{(c^{rk})}(\A))$ and $u_{g_1}\in U_0(\A)$.
Since $v_{g_1},u_{g_1}\in N_{rkc}(\A)$, we can write this in $H^{(m,r)}(\A)$ as in \eqref{Appendixeq:conjugation of v_g} and
the $du_0$-integral of \eqref{Appendixint:I delta factoring change g_1 g_2} equals
\begin{align*}
&\int\limits_{U_0(\A)}
f_{W_{\psi}(\mathcal{E}_{\tau})}(s,{}^{\delta}\langle (g_1,g_1),\eta_{(c^{2rk})}((g_1,g_1))\rangle\langle \delta,(\eta_{2rkc}^{\diamondsuit})^{-1}(\delta)\rangle\\&\langle v_{g_1},(\eta_{2rkc}^{\diamondsuit})^{-1}(v_{g_1})\rangle\langle u_{g_1},(\eta_{2rkc}^{\diamondsuit})^{-1}(u_{g_1})\rangle\langle\mathfrak{e}_2(g_2),1\rangle)\,\psi_U(u_0)\,du_0.
\end{align*}

\begin{applemma}\label{AppLemma:invariance of inner under G iota delta}
For any $g\in G(\A)$ and $h\in H^{(m,r)}(\A)$,
\begin{align*}
&f_{W_{\psi}(\mathcal{E}_{\tau})}
(s,{}^{\delta}\langle (g,g),\eta_{(c^{2rk})}((g,g))\rangle h
)=f_{W_{\psi}(\mathcal{E}_{\tau})}(s,h).
\end{align*}
\end{applemma}
Before proving the lemma we deduce that \eqref{Appendixint:I delta factoring change g_1 g_2} equals
\begin{align*}
&\int\limits_{C_0'G(F)\backslash G^{\times}(\A)}
\quad\int\limits_{U_0(\A)}
\varrho(|\det g_2|)\varphi_1^{R}(\langle g_1,1\rangle)\varphi_2^{R}(\langle g_1,1\rangle\langle g_2,1\rangle)f_{W_{\psi}(\mathcal{E}_{\tau})}
\\&(s,\langle \delta,(\eta_{2rkc}^{\diamondsuit})^{-1}(\delta)\rangle\langle v_{g_1},(\eta_{2rkc}^{\diamondsuit})^{-1}(v_{g_1})\rangle\langle u_{g_1},(\eta_{2rkc}^{\diamondsuit})^{-1}(u_{g_1})\rangle\nonumber
\langle \mathfrak{e}_2(g_2),1\rangle)\,\psi_U(u_0)\,du_0\,dg_1\,dg_2.\nonumber
\end{align*}
Changing variables $u_{g_1}\mapsto u_0$, the change to $\psi_U$ is cancelled by the left equivariance property of
$f_{W_{\psi}(\mathcal{E}_{\tau})}$ under $v_{g_1}$ (see \eqref{Appendixeq:conjugation of v_g}). One can also combine $\delta$ and $u_0$ in $H^{(m,r)}(\A)$ (see Corollary~\ref{corollary:rho and eta on H and N without conjugation}). The integral equals

\begin{align*}
&\int\limits_{C_0'G(F)\backslash G^{\times}(\A)}
\quad\int\limits_{U_0(\A)}
\varrho(|\det g_2|)\varphi_1^{R}(\langle g_1,1\rangle)\varphi_2^{R}(\langle g_1,1\rangle\langle g_2,1\rangle)
\\&f_{W_{\psi}(\mathcal{E}_{\tau})}(s,\langle \delta u_0,(\eta_{2rkc}^{\diamondsuit})^{-1}(\delta u_0)\rangle\nonumber
\langle \mathfrak{e}_2(g_2),1\rangle)\,\psi_U(u_0)\,du_0\,dg_1\,dg_2.\nonumber
\end{align*}
Finally, factoring through $\{[g,I_c]:g\in G(\A)\}$ (which contains $C_0'$) and using the definition of $\{\varphi_1^R,\varphi_2^R\}$ we conclude that $\mathrm{I}(\delta)$ equals
\begin{align*}
&\int\limits_{G(\A)}\{\varphi_1^R,\pi(\langle g_2,1\rangle)\varphi_2^R\}
\int\limits_{U_0(\A)}
f_{W_{\psi}(\mathcal{E}_{\tau})}(s,\langle \delta u_0,(\eta_{2rkc}^{\diamondsuit})^{-1}(\delta u_0)\rangle\langle \mathfrak{e}_2(g_2),1\rangle)\varrho(|\det g_2|)\,\psi_U(u_0)\,du_0\,dg_2
\end{align*}
($\pi$ acts by right-multiplication).
Therefore \eqref{Appendixglglobal1} and \eqref{Appendixglobal2} are equal for $\Real(s)\gg0$.
\end{proof}
\begin{proof}[Proof of Lemma~\ref{AppLemma:invariance of inner under G iota delta}]
Let $g\in G(\A)$. First observe ${}^{\delta}(g,g)=(g,g)$. At each $\nu$, $\langle\delta_{\nu},1\rangle$ and $\langle(g_{\nu},g_{\nu}),1\rangle$ commute in $G^{(m,r)}(F_{\nu})$ by the proof of Proposition~\ref{proposition:equiv props GL} (for the proof we used $\sigma_{c,\nu}^{\diamondsuit}$, the property holds for any isomorphic covering, in particular with $\rho_{c,\nu}^{\diamondsuit}$, hence also holds globally). Now by \eqref{Appendixrep:gbl parabolic induction on GL properties}, we need to show that
$\langle (g,g),\eta_{(rkc,rkc)}^{-1}((g,g))\eta_{(c^{2rk})}((g,g))\rangle$ acts trivially on
$W_{\psi}(\mathcal{E}_{\tau})$.

Recall $\widetilde{M}_P(\A)$ is realized using $\rho_{(rkc,rkc)}$. Identity~\eqref{Appendixeq:rho beta and rho square globally} implies that for $g'\in G(\A)$,
\begin{align*}
&\rho_{(rkc,rkc)}((g,g),(g',g'))\eta_{(rkc,rkc)}^{-1}((g,g))\eta_{(rkc,rkc)}^{-1}((g',g'))\\&=
\eta_{(rkc,rkc)}^{-1}((gg',gg'))\rho_{2rkc}^{\diamondsuit}((g,g),(g',g')).
\end{align*}
In addition $\rho_{(c^{2rk})}((g,g),(g',g'))=\rho_c^{\diamondsuit}(g,g')^{2rk}=1$, whence \eqref{Appendixeq:rho beta and rho square globally} also shows
\begin{align*}
&\rho_{2rkc}^{\diamondsuit}((g,g),(g',g'))\eta_{(c^{2rk})}((g,g))\eta_{(c^{2rk})}((g',g'))
=\eta_{(c^{2rk})}((gg',gg')).
\end{align*}
Combining these identities,
\begin{align*}
&\langle (g,g),\eta_{(rkc,rkc)}^{-1}((g,g))\eta_{(c^{2rk})}((g,g))\rangle
\langle (g',g'),\eta_{(rkc,rkc)}^{-1}((g',g'))\eta_{(c^{2rk})}((g',g'))\rangle\\&=
\langle (gg',gg'),\eta_{(rkc,rkc)}^{-1}((gg',gg'))\eta_{(c^{2rk})}((gg',gg'))\rangle.
\end{align*}
Thus $(g,g)\mapsto\langle (g,g),\eta_{(rkc,rkc)}^{-1}((g,g))\eta_{(c^{2rk})}((g,g))\rangle$ is a splitting into $\widetilde{M}_P(\A)$.
Hence we can consider separately $g\in\SL_c(\A)$ and $g=\diag(t,I_{c-1})\in T_{G}(\A)$. We provide a global argument for
$\SL_c(\A)$.

Let $g\mapsto d_g=\diag(g,\ldots,g)$ denote the
diagonal embedding of $G$ in $\GL_{rkc}$. By Corollary~\ref{corollary:gbl splitting for SL_c} $\GL_{rkc}^{(m,r)}(\A)$ is split over
$\{d_g:g\in\SL_c(\A)\}$, and because $\SL_c(\A)$ is perfect as a group, this splitting is unique and we denote it by
$d_g\mapsto\langle d_g,\varrho(d_g)\rangle$
($\varrho(d_g)$ was explicitly given in \textit{loc. cit.} with respect to $\rho_{rkc}^{\diamondsuit}$). Then because $\rho_{(rkc,rkc)}((g,g),(g',g'))=\rho_{rkc}^{\diamondsuit}(g,g')^2$,
$(g,g)=\diag(d_g,d_g)\mapsto\langle\diag(d_g,d_g),\varrho(d_g)\varrho(d_g)\rangle$ is the splitting of $\{(g,g):g\in\SL_c(\A)\}$ in
$\widetilde{M}_P(\A)$, and coincides with the splitting above. Thus
$\langle (g,g),\eta_{(rkc,rkc)}^{-1}((g,g))\eta_{(c^{2rk})}((g,g))\rangle
=\langle d_g,\varrho(d_g)\rangle\langle d_g,\varrho(d_g)\rangle$. According to Proposition~\ref{proposition:extra invariance} $\langle d_g,\varrho(d_g)\rangle$ acts trivially on $W_{\psi}(\mathcal{E}_{\tau_0})$ and on
$W_{\psi}(\mathcal{E}_{\tau_0^*})$.

It remains to consider $\diag(t,I_{c-1})$, but we provide a local argument which applies to all $g\in G$.
Since the $(rk,c)$ functional $\xi\mapsto\Lambda(\pi(\langle d_g,1\rangle)\xi)$ ($g\in G(\A)$) is proportional to
$\xi\mapsto\Lambda(\xi)$ and under our assumption that $\Lambda$ is decomposable,
it suffices to prove a local statement: namely the action of
$\langle (g_{\nu},g_{\nu}),\eta_{(rkc,rkc),\nu}^{-1}((g_{\nu},g_{\nu}))\eta_{(c^{2rk}),\nu}((g_{\nu},g_{\nu}))\rangle$
on $\rho_c(\tau_{\nu})\otimes\rho_c(\tau_{\nu}^*)$ is trivial.

We switch to local notation and omit $\nu$.
Since $\eta_{c}^{\diamondsuit}((g,g))^{2rk}=1$, by \eqref{eq:eta beta nu} we have
$\eta_{(c^{2rk})}((g,g))=\eta_{2rkc}^{\diamondsuit}((g,g))^{-1}$ (use $\beta=(c^{2rk})$) and also
$\eta_{(rkc,rkc)}((g,g))=\eta_{rkc}^{\diamondsuit}(d_{g})^2\eta_{2rkc}^{\diamondsuit}((g,g))^{-1}$, hence
\begin{align*}
\langle (g,g),\eta_{(rkc,rkc)}^{-1}((g,g))\eta_{(c^{2rk})}((g,g))\rangle
&=\langle (g,g),\eta_{rkc}^{\diamondsuit}(d_{g})^{-2}\rangle.
\end{align*}
As an element of $\widetilde{M}_P$ (realized using $\rho_{(rkc,rkc)}$) the last element equals the product
\begin{align*}
\langle d_{g},(\eta_{rkc}^{\diamondsuit})^{-1}(d_{g})\rangle
\langle d_{g},(\eta_{rkc}^{\diamondsuit})^{-1}(d_{g})\rangle.
\end{align*}
The map $\langle b,\epsilon\rangle\mapsto\langle b,\eta_{rkc}^{\diamondsuit}(b)\epsilon\rangle$ is the isomorphism
$G[\rho_c^{\diamondsuit}]\rightarrow G[\sigma_c^{\diamondsuit}]$ (see \eqref{eq:nu and sigma for covering of H}).
It thus remains to verify $\langle d_g,1\rangle\langle d_g,1\rangle$ acts trivially on
$W_{\psi}(\rho_c(\tau_0))\otimes W_{\psi}(\rho_c(\tau_0^*))$, and (only) here we use the assumption $\mu_{2m}<F^*$.
Over non-archimedean fields this was proved in Proposition~\ref{proposition:applicability of the GL GL construction}; the archimedean case
is clear because then the diagonal embedding of $G$ in $\GL_{rkc}$ acts by a character on the Jacquet module (see \cite[\S~1.4]{DimaKaplan}).
\end{proof}
\begin{appremark}
One can construct \eqref{Appendixglglobal1} with cusp forms $\varphi_1$ and $\varphi_2$ in the spaces of $\pi_1$ and $\pi_2^{\vee}$, for genuine cuspidal representations $\pi_1$ and $\pi_2$.
The proof of the theorem implies $\mathrm{I}(\delta)$ and thereby \eqref{Appendixglglobal1} vanish,
unless $\pi_1=\pi_2=\pi$.
\end{appremark}

As explained in \cite[\S~4.2]{PSR} one can choose a right half-plane which is independent of $\varrho$, where \eqref{Appendixglobal2} is absolutely convergent. Then if we define $Z(s,\varphi_1,\varphi_2,f)=\lim\limits_{l\to\infty}Z(s,\varphi_1,\varphi_2,f,\varrho_l)$ for a monotonic increasing sequence $\varrho_l\rightarrow 1$, in this half-plane
by the Monotone Convergence Theorem $Z(s,\varphi_1,\varphi_2,f)$ equals
\begin{align*}
\int\limits_{G({\A})}\int\limits_{U_0({\A})}
\{\varphi_1^R,\pi(\langle g,1\rangle)\varphi_2^R\}f_{W_{\psi}(\mathcal{E}_{\tau})}(s,\langle\delta u_0,(\eta_{2rkc}^{\diamondsuit})^{-1}(\delta u_0)\rangle
\langle\mathfrak{e}_2(g),1\rangle)
\,\psi_U(u_0)\,du_0\,dg.
\end{align*}
Now as in \S~\ref{global symplectic} when we take decomposable data, under the conjectures of \S~\ref{speh gbl} but assuming $\Lambda$ is factorizable,
$Z(s,\varphi_1,\varphi_2,f)$ is Eulerian and the integrals with unramified data are then given by Theorem~\ref{theorem:unr GL}.

\section{List of common notation and definitions}\label{list of common notation}
\begin{itemize}[leftmargin=*]
\item $F$ - a local field of characteristic $0$, or a number field; \S~\ref{Groups}.
\item $G$ - usually $\Sp_{2n}$; \S~\ref{Groups}.
\item $J_l$ - the permutation matrix in $\GL_l$ with $1$ along the anti-diagonal; \S~\ref{Groups}.
\item $B_l=T_l\ltimes N_l$ - the fixed Borel subgroup of $\Sp_{2l}$, where $T_l$ is the torus; \S~\ref{Groups}.
\item $W_{G}$ - the Weyl group of $G$; \S~\ref{Groups}.
\item $B_{\GL_d}=T_{\GL_d}\ltimes N_{\GL_d}$ - the Borel subgroup of upper triangular invertible matrices; \S~\ref{Groups}.
\item $P_{\beta}=M_{\beta}\ltimes V_{\beta}$ - a standard parabolic subgroup of $\GL_d$ where $\beta$ is a composition of $d$; \S~\ref{Groups}.

\item $\Phi_d$, $\Phi_d^+$ - simple and positive roots of $\GL_d$, usually regarded as pairs $(i,j)$; \S~\ref{Groups}.

\item $W_{\GL_d}$ - the Weyl group of $\GL_d$; \S~\ref{Groups}.

\item $\ell(w)$ - the length of $w\in W_{\GL_d}$; \S~\ref{Groups}.

\item $b^*=J_d{}^tb^{-1}J_d$ for $b\in\GL_d$ (${}^tb$ - the transpose of $b$); \S~\ref{Groups}.
\item ${}^xy=xyx^{-1}$; \S~\ref{Groups}.
\item $\mathcal{O}$, $\mathcal{P}$, $\varpi$, $q$ - $p$-adic fields, $\mathcal{O}$ is the ring of integers, $\varpi\in\mathcal{O}$, $|\varpi|=q^{-1}$ and $\mathcal{P}=\varpi\mathcal{O}$; \S~\ref{Groups}.
\item $K_G$ - over $p$-adic fields $K_G=G(\mathcal{O})$, globally a maximal compact subgroup of $G(\A)$; \S~\ref{Groups}.
\item $\mu_m$ - the $m$-th roots of unity; \S~\ref{Covering groups}.
\item $(\cdot,\cdot)_m$ - the $m$-th Hilbert symbol; \S~\ref{Covering groups}.
\item $G^{(m)}$ - the $m$-fold covering of $G$ when $G$ is a symplectic group; \S~\ref{Covering groups}.
\item $\langle g,\epsilon\rangle$ - a general element of a covering group; \S~\ref{Covering groups}.
\item $\widetilde{X}$ - when a covering of $G$ is fixed and $X<G$, the covering obtained by restriction; \S~\ref{Covering groups}.
\item $\varepsilon$ - a fixed faithful character of $\mu_m$ used to define genuine representations; \S~\ref{Covering groups}.
\item $c$ - $c=2n$ except when the integrals for $\GL_n\times\GL_k$ are considered, then $c=n$; \S~\ref{embedding}.
\item $r=m$ if $m$ is odd, otherwise $r=m/2$; \S~\ref{embedding}.
\item $H$ - usually $\Sp_{2rkc}$; \S~\ref{embedding}.
\item $P=M_P\ltimes U_P$ - the Siegel parabolic subgroup of $H$; \S~\ref{embedding}.
\item $Q=M_Q\ltimes U$ - a standard parabolic subgroup of $H$, $G\times G$ is embedded in $M_Q$ as the stabilizer of a character $\psi_U$ of $U$; \S~\ref{embedding}.
\item $\psi$ - a nontrivial additive character of a local field or of $F\backslash \A$; \S~\ref{embedding}.
\item $\psi_U$ - the character of $U$ stabilized by $G\times G$; \eqref{eq:character of U}.
\item $(g_1,g_2)$, $\mathfrak{e}_i$ - the embedding $G\times G\hookrightarrow H$, $\mathfrak{e}_1(g)=(g,1)$ and $\mathfrak{e}_2(g)=(1,g)$; \S~\ref{embedding}.
\item $\sigma_{2rkc}$, $\sigma_{d}$ - the $2$-cocycle of \cite{BLS} for $\GL_{2rkc}$, $\GL_d$, which we can compute explicitly in many cases, but does not globalize. In a local-global context we denote $\sigma_{2rkc,\nu}$ or $\sigma_{\nu}$; \S~\ref{local covering}.
\item $\mathfrak{W}_d$ - a set of representatives for $W_{\GL_d}$ constructed in \cite{BLS}, and used for $\sigma_d$;
    \S~\ref{local covering}.
\item $\mathfrak{W}^+_d$ - a group containing $\mathfrak{W}_d$ and diagonal matrices whose coordinates are $\pm1$; \S~\ref{local covering}.

\item $F$ is unramified - $F$ is non-archimedean, $|m|=1$, $q$ is odd and $q>3$; \S~\ref{local covering}.

\item $\eta_d$ - in a purely local context and when $F$ is unramified, $\langle y,\eta_d(y)\rangle$ is a splitting of $K_{\GL_{d}}$, and in particular $\langle y,\eta_{2rkc}(y)\rangle$ is the canonical splitting of $K_{H}$; \S~\ref{local covering}.

\item $\sigma^*_{c}$, $\sigma^{*,rk}_{c}$, $\varsigma_{*,c}$: in a local context $\sigma^*_{c}$ and $\sigma^{*,rk}_{c}$ are $2$-cocycles of $\GL_c$, cohomologous to $\sigma_c$ on $\SL_c$ (even on $\GL_c$), and $\sigma^*_{c}(g,g')=\sigma_{c}(g^*,{g'}^*)$ where $g^*=J_c{}^tg^{-1}J_c$. The $1$-cochain relating $\sigma_c$ to $\sigma_c^*$ is $\varsigma_{*,c}$;
    \eqref{eq:sigma c *}, \eqref{eq:sigma c * rk}, \S~\ref{local covering}, Proposition~\ref{proposition:sigma * and sigma on GLd}.

\item $\varphi\mapsto\varphi^{\varsigma_{*,c}^{rk+1}}$ - the local mapping of functions from the right copy of $G^{(m)}$ to the left; \eqref{eq:iso sigma sigma * rk for functions}.

\item $\rho$ - the $2$-cocycle for $H(\A)$, $\rho_{\nu}$ is the local version which is cohomologous to $\sigma_{\nu}$ where $\eta_{\nu}$ is the $1$-cochain. We write $\rho_{2l}$ for this $2$-cocycle on $\Sp_{2l}(\A)$; \eqref{eq:nu and sigma for covering of H}, \S~\ref{global covering}.

\item $\eta$ - the global $1$-cochain $\eta=\prod_{\nu}\eta_{\nu}$, which is well defined on $H(F)$ but not on $H(\A)$; \eqref{eq:nu and sigma for covering of H}, \S~\ref{global covering}.

   \item $\rho_L$ and $\rho_R$ - the global $2$-cocycles for the left and right copies of $\Sp_{c}(\A)$ in $H(\A)$;
    \eqref{gbl rho on left and right}.

\item $\varphi\mapsto\varphi^{(\eta^{\times})^{-1}}$ - the global version of $\varphi\mapsto\varphi^{\varsigma_{*,c}^{rk+1}}$ (from the right copy to the left); \eqref{eq:gbl iso rhoR rhoL for functions}.

    \item $\sigma$ - globally this is the product of $2$-cocycles $\sigma_{\nu}$ (usually $\sigma_{2rkc,\nu}$), e.g., on $H(F)$; \S~\ref{global covering}.
\item $\langle h,\eta^{-1}(h)\rangle$ - the splitting of $H(F)$, used for the definition of automorphic forms; \S~\ref{global covering}.
\item $\langle g,\eta(\mathfrak{e}_1(g))\rangle$, $\langle g,\eta^{-1}(\mathfrak{e}_2(g))\rangle$ - the splittings of $G(F)$ in the two realizations of $G^{(m)}(\A)$; \S~\ref{global covering}.

\item $\eta^{\times}$ - a $1$-cochain relating between the coverings on both copies of $G(\A)$ in $H^{(m)}(\A)$; \S~\ref{global covering}.

\item ${}^{\iota}$ - the involution $g\mapsto{}^{\iota}g$ of $G$, $\iota=\left(\begin{smallmatrix}&I_{c/2}\\I_{c/2}\end{smallmatrix}\right)$,
     lifted to $G^{(m)}$ locally and globally; \S~\ref{extension of the involution}.

\item $\varsigma_{\iota,c}$ - the $1$-cochain used for the local lift of $\iota$ to $G^{(m)}$; \eqref{eq:sigma iota c}, \eqref{eq:iota on the local coverings}.

\item $\pi^{\iota}$ - for a representation $\pi$ of $G^{(m)}$, $\pi^{\iota}$ acts on the space of $\pi$ by $\pi^{\iota}(g)=\pi({}^{\iota}g)$; \S~\ref{extension of the involution}.

\item $\eta_{\iota,R}$ - the global analog of $\varsigma_{\iota,c}$; \eqref{eq:eta iota R}.

\item $\GL_{rkc}^{(m,r)}$ (or $\GL_{d}^{(m,r)}$) - $\GL_{rkc}$ is embedded in $H=\Sp_{2rkc}$ by $\{\diag(g,g^*):g\in\GL_{rkc}\}$,
$\GL_{rkc}^{(m,r)}$ is the covering obtained by restriction from $H^{(m)}$.
Similarly for $\GL_{d}^{(m,r)}$ using $\Sp_{2d}^{(m)}$; \S~\ref{covering of the Levi}.

\item $C_{r,d}=\{xI_d:x\in F^{*r}\}$, $\widetilde{C}_{r,d}$ is the center of $\GL_d^{(m,r)}$; \S~\ref{covering of the Levi}.

\item $\sigma^{\diamondsuit}_{d}$ - the local $2$-cocycle of $\GL_{d}^{(m,r)}$,
$\sigma^{\diamondsuit}_{d}(b,b')=\sigma_{2d}(\diag(b,b^*),\diag(b',{b'}^*))$; \eqref{eq:sigma square}.

\item $\rho^{\diamondsuit}_{d}$ - the global $2$-cocycle of $\GL_{rkc}^{(m,r)}(\A)$ defined using $\rho_{2d}$, locally $\rho^{\diamondsuit}_{d,\nu}$; \eqref{eq:rho square}.

\item ${}^*\langle b,1\rangle$ - an involution of $\GL_{rkc}^{(m,r)}$ over a local field; \eqref{eq:involution b*0}.

\item $\pi^*$ - for a representation $\pi$ of $\GL_{rkc}^{(m,r)}$, $\pi^*$ acts on the space of $\pi$ by $\pi^*(b)=\pi(b^*)$;
\eqref{eq:involution b*0}.

\item $\eta_{d}^{\diamondsuit}$ - the local $1$-cochain relating between $\rho^{\diamondsuit}_{d}$
and $\sigma^{\diamondsuit}_{d}$; \eqref{eq:cohomologous $2$-cocycles on GL}.

\item $(\eta^{\diamondsuit}_{rkc})^{-1}(b)$ - the $1$-cochain used to define a splitting of $\GL_d(F)$, $N_{\GL_d}(\A)$ in $\GL_{rkc}^{(m,r)}(\A)$; \S~\ref{covering of the Levi}.

\item $\rho_{\beta}$ - a global block-compatible $2$-cocycle on $M_{\beta}$, $\rho_{\beta}(b,b')=\prod_{i=1}^l\rho^{\diamondsuit}_{\beta_i}(b_i,b_i')$; \eqref{eq:rho beta}.
\item $\eta_{\beta}$, $\eta_{\beta,\nu}$ - the $1$-cochain relating between $\rho_{\beta}$ (resp., $\rho_{\beta,\nu}$) and $\rho_d^{\diamondsuit}$ (resp., $\rho_{d,\nu}^{\diamondsuit}$); \eqref{eq:eta beta nu}.

\item $A$ - $A=F^{*r}\mathcal{O}^*$ is a maximal abelian subgroup of $\GL_1^{(m,r)}$; \S~\ref{unramified reps}.
\item $\gamma_{\psi}$ - the Weil factor; \eqref{eq:some props of gamma psi'}, \S~\ref{unramified reps}.
\item $T_{d,*}$, $T_{\GL_d,*}$ - the image in $T_d$, $T_{\GL_d}$ of certain maximal abelian subgroups of $\widetilde{T}_{d}$, $\widetilde{T}_{\GL_d}$; \S~\ref{unramified reps}.
\item $\mathrm{I}_{G^{(m)}}(\vartheta,\mu)$, $\mathrm{I}_{\GL_d^{(m,r)}}(\vartheta,\chi)$ - unramified principal series representations; \S~\ref{unramified reps}.
\item $t_{\pi}$, $t_{\tau}$ - Satake parameters for representations $\pi$ and $\tau$; \eqref{eq:Satake symplectic}, \S~\ref{unramified reps}.

\item $L_{\vartheta}(s,\tau,\sigma)$, $L_{\vartheta,\vartheta'}(s,\tau\times\tau')$, $L_{\vartheta_{\pi},\vartheta_{\tau}}(s,\pi\times\tau)$ - unramified $L$-factors; \S~\ref{unramified reps}.

\item $\Lambda_t$, $\Lambda$ - the Whittaker functional on $\mathrm{I}_{\GL_d^{(m,r)}}(\vartheta,\chi)$ given by a Jacquet integral with left-translation by $t$, $\Lambda=\Lambda_{I_d}$; \S~\ref{whittaker functionals}.

\item $\xi^0$ - a normalized unramified vector in a space of a principal series representation;
 \S~\ref{Casselman--Shalika formula}.

\item $\mathfrak{g}$ - a Gauss sum, $\mathfrak{g}(l)=-q^{-1}$ when $l\equiv 0\,(m)$; \eqref{eq:Gauss Sum}.

\item $\mathbf{x}$ - the ``linear parameter" of an unramified principal series representation of $\GL_d^{(m,r)}$; \S~\ref{Casselman--Shalika formula}.

\item $\mathbf{a}^*=(-a_1,\ldots,-a_d)$ where $\mathbf{a}=(a_1,\ldots,a_d)\in\Z^d$; \S~\ref{Casselman--Shalika formula}.
\item $\varpi^{\mathbf{a}}=\diag(\varpi^{a_1},\ldots,\varpi^{a_d})$,
    $t_{\mathbf{a}}=\langle\varpi^{\mathbf{a}},1\rangle$; \S~\ref{Casselman--Shalika formula}.

\item $\equiv$ on $\Z^r$ - $\mathbf{a}\equiv \mathbf{b}$ if $\mathbf{a}-\mathbf{b}\in r\Z^d$; \S~\ref{Casselman--Shalika formula}.

\item $W_{\mathbf{a}}(\mathbf{b},\vartheta,\chi)=\Lambda_{t_{\mathbf{a}}}(t_{\mathbf{b}}\cdot\xi^0)$ - the value of
$\Lambda_{t_{\mathbf{a}}}$ on the right-translation of $\xi^0$ by $t_{\mathbf{b}}$;
\eqref{eq:def W Lambda}, \eqref{eq:CS formula}.

\item $\mathbf{v}(t)$ - the valuations vector of a torus element $t$; \S~\ref{Casselman--Shalika formula}.

\item $\mathbf{x}_{\alpha}=x_ix_j^{-1}$; \S~\ref{Casselman--Shalika formula}.
\item ${}^w\mathbf{x}$ - the action of $W_{\GL_d}$ on $\mathbf{x}$, $({}^w\mathbf{x})_i=x_{w^{-1}(i)}$; \S~\ref{Casselman--Shalika formula}.

\item $\mathbf{x}(\mathbf{a})=(x_1^{a_1/r},\ldots,x_d^{a_d/r})$ for $\mathbf{a}\in r\Z^d$; \S~\ref{Casselman--Shalika formula}.

\item $w[\mathbf{a}]$, $w_{\alpha}[\mathbf{a}]$ - an action of $W_{\GL_d}$ on $\Z^d$; \S~\ref{Casselman--Shalika formula}.

\item $M(w_{\alpha})$, $M(w)$ - an intertwining operator between genuine unramified principal series; \S~\ref{Casselman--Shalika formula}.

\item $\tau_{t,t'}$, $\tau_{t,t'}^1$, $\tau_{t,t'}^2$ - the coefficients appearing in the Casselman--Shalika formula for coverings of $\GL_d^{(m,r)}$, the notation $\tau_{\mathbf{a},\mathbf{b}}$ is also used;
    \eqref{eq:general equation defining tau a b}--\eqref{eq:tau for simple ref}, Proposition~\ref{proposition:coefficients tau t t' 1 2}, \eqref{eq:tau 1 nonzero}, \eqref{eq:tau 2 nonzero}.

\item $\lceil x \rceil$ - the smallest integer greater than or equal to $x$; Proposition~\ref{proposition:coefficients tau t t' 1 2}.

\item $\Theta_{d,m,r,\vartheta}$ - an exceptional representation of $\GL_d^{(m,r)}$; \S~\ref{exceptional}.

\item $\beta\succsim \beta'$ - the partition $\beta$ is greater than or not comparable with the partition $\beta'$; \S~\ref{fourier coeff on orbits}.

\item $\mathcal{O}(\rho,\beta,\psi)$ - locally $\Hom_{V(\beta)}(\rho,\psi)$, globally a set of Fourier coefficients; \eqref{int:general Fourier coeff}, \S~\ref{fourier coeff on orbits}.

\item $\psi_{\lambda}$ - a (possibly degenerate) character of $N_{\GL_d}$ defined for a composition $\lambda$ of $d$; \S~\ref{Semi Whittaker coeff}.

\item $W^0$ - the normalized unramified Whittaker function in the space of $\mathcal{W}(\Theta_{r,m,r,\vartheta})$; \S~\ref{local theta speh for c=1}.

\item $p_l(\mathbf{x})$ - the $l$-th complete symmetric polynomial in $\mathbf{x}$; \S~\ref{local theta speh for c=1}.

\item $\Lambda[S]$ - an $(rk,c)$ functional over a finite set of archimedean or ramified places $S$; \eqref{eq:partial Lambda S}.

\item $\eta_{rkc}^{\triangle}$ - the splitting of $\SL_c^{\triangle}(\A)$ in $\GL_{rkc}^{(m,r)}(\A)$; Corollary~\ref{corollary:gbl splitting for SL_c}.

\item $\Theta_{rc,m,r,\vartheta}(\chi)$ - a representation parabolically induced from the tensor of $k$ copies of
exceptional representations $\Theta_{rc,m,r,\vartheta}$, each twisted by $\chi_i$; \eqref{gen of Suzuki new}.

\item $\chi_{\Theta,c}$, $\chi_{\Theta}$ - the ``linear part" of the inducing character of $\Theta_{rc,m,r,\vartheta}(\chi)$ as a subrepresentation, $\chi_{\Theta}=\chi_{\Theta,1}$; \eqref{eq:Thete rc m as a subrep}.

\item $\rho_c(\tau)$ - the local component of the $(rk,c)$ representation $\mathcal{E}_{\tau}$ at an unramified place; \S~\ref{Local components of rk c speh}.

\item $\mathcal{W}(\rho_c(\tau))$ - the $(rk,c)$ model of $\rho_c(\tau)$; \S~\ref{Local components of rk c speh}.

\item $\{\varphi_1,\varphi_2\}$ - the $G(\A)$-invariant pairing; \S~\ref{global symplectic}.

\item $f_{W(\mathcal{E}_{\tau})}$, $f_{W(\mathcal{E})}$ - globally the composition of a section $f$ with the $(rk,c)$ functional, locally a section taking values in the $(rk,c)$ model; \eqref{eq:def of f W global}, \S~\ref{global symplectic}.

\item $\lfloor x \rfloor$ - the largest integer smaller than or equal to $x$; \eqref{eq:d tau(s)}.

\end{itemize}

%\bibliographystyle{alpha-abbrvsort}
%\bibliography{bib}
\newcommand{\etalchar}[1]{$^{#1}$}
\def\cprime{$'$} \def\cprime{$'$} \def\cprime{$'$}

\end{document}